\documentclass[11pt,english]{amsart}
\usepackage[T1]{fontenc}
\usepackage[latin9]{inputenc}
\usepackage{babel}
\usepackage{amsbsy}
\usepackage{amsthm}
\usepackage{amssymb}
\usepackage{graphicx}
\usepackage[letterpaper]{geometry}
\usepackage[pdfusetitle,
 bookmarks=true,bookmarksnumbered=false,bookmarksopen=false,
 breaklinks=false,pdfborder={0 0 1},backref=false,colorlinks=false]
 {hyperref}

\makeatletter

\providecommand{\tabularnewline}{\\}

\numberwithin{equation}{section}
\numberwithin{figure}{section}
\theoremstyle{plain}
\newtheorem{thm}{\protect\theoremname}[section]
\theoremstyle{remark}
\newtheorem{rem}[thm]{\protect\remarkname}
\theoremstyle{plain}
\newtheorem{cor}[thm]{\protect\corollaryname}
\theoremstyle{definition}
\newtheorem{defn}[thm]{\protect\definitionname}
\theoremstyle{remark}
\newtheorem{notation}[thm]{\protect\notationname}
\theoremstyle{plain}
\newtheorem{lem}[thm]{\protect\lemmaname}
\theoremstyle{definition}
\newtheorem{condition}[thm]{\protect\conditionname}
\theoremstyle{plain}
\newtheorem{prop}[thm]{\protect\propositionname}

\@ifundefined{date}{}{\date{}}
\usepackage{amsmath, tikz-cd, amssymb, placeins}
\DeclareRobustCommand*\cal{\@fontswitch\relax\mathcal}

\usetikzlibrary{calc}
\usetikzlibrary{decorations.pathmorphing}

\tikzset{curve/.style={settings={#1},to path={(\tikztostart)
    .. controls ($(\tikztostart)!\pv{pos}!(\tikztotarget)!\pv{height}!270:(\tikztotarget)$)
    and ($(\tikztostart)!1-\pv{pos}!(\tikztotarget)!\pv{height}!270:(\tikztotarget)$)
    .. (\tikztotarget)\tikztonodes}},
    settings/.code={\tikzset{quiver/.cd,#1}
        \def\pv##1{\pgfkeysvalueof{/tikz/quiver/##1}}},
    quiver/.cd,pos/.initial=0.35,height/.initial=0}

\tikzset{tail reversed/.code={\pgfsetarrowsstart{tikzcd to}}}
\tikzset{2tail/.code={\pgfsetarrowsstart{Implies[reversed]}}}
\tikzset{2tail reversed/.code={\pgfsetarrowsstart{Implies}}}
\tikzset{no body/.style={/tikz/dash pattern=on 0 off 1mm}}

\makeatother

\providecommand{\conditionname}{Condition}
\providecommand{\corollaryname}{Corollary}
\providecommand{\definitionname}{Definition}
\providecommand{\lemmaname}{Lemma}
\providecommand{\notationname}{Notation}
\providecommand{\propositionname}{Proposition}
\providecommand{\remarkname}{Remark}
\providecommand{\theoremname}{Theorem}

\begin{document}
\title{Spectral sequences in unoriented link Floer homology}
\author{Gheehyun Nahm}
\thanks{The author was partially supported by the Simons Grant \emph{New structures in low-dimensional topology}.}
\begin{abstract}
In a previous work, we defined an unoriented skein exact triangle
in unoriented link Floer homology. In this paper, we iterate a modified
version of this exact triangle and obtain a spectral sequence from
various versions of Khovanov homology to various versions of unoriented
link Floer homology, over the field with two elements. In particular,
for knots in $S^{3}$, we obtain a spectral sequence from the reduced
Khovanov homology of the mirror of the knot to the knot Floer homology
of the knot.
\end{abstract}

\address{Department of Mathematics, Princeton University, Princeton, New Jersey
08544, USA}
\email{gn4470@math.princeton.edu}

\maketitle
\tableofcontents{}

\section{\label{sec:Introduction}Introduction}

In this paper, we prove Rasmussen's conjecture \cite{MR2189938} over
the field $\mathbb{F}=\mathbb{Z}/2$ with two elements.
\begin{thm}
\label{thm:knot-S3}Let $K$ be a knot in $S^{3}$. Then, there exists
a spectral sequence 
\[
\widetilde{Kh}(m(K))\Rightarrow\widehat{HFK}(S^{3},K)
\]
over the field $\mathbb{F}=\mathbb{Z}/2$ with two elements, where
$\widetilde{Kh}(m(K))$ is the reduced Khovanov homology of the mirror
of the knot.
\end{thm}

Note that a similar result over $\mathbb{Q}$ was proven by Dowlin
\cite{MR4777638}. Combined with that the unknot is the unique knot
in $S^{3}$ whose $\widehat{HFK}$ has rank $\le1$ \cite{MR2023281},
we obtain another proof of that Khovanov homology detects the unknot,
which was originally proven by Kronheimer and Mrowka \cite{MR2805599}.

Theorem \ref{thm:knot-S3} is a special case of the next result (Theorem
\ref{thm:main-thm}). In \cite{nahm2025unorientedskeinexacttriangle},
we studied an algebraic variant of the unoriented link Floer homology
defined by Ozsv\'{a}th, Stipsicz, and Szab\'{o} in \cite{MR3694597}.
For links $L$ with a distinguished point\footnote{The unreduced hat version and the minus version do not depend on this
distinguished point.} in a three-manifold $Y$, we defined three versions of unoriented
link Floer homology: the reduced hat version $\widetilde{HFL'}$,
the unreduced hat version $\widehat{HFL'}$, and the minus version
$HFL'{}^{-}$. They resemble the various versions of Khovanov homology:
the reduced version $\widetilde{Kh}$, the unreduced version $\widehat{Kh}$
(usually denoted as $Kh$), and the equivariant version (given by
the Lee deformation) $Kh^{-}$. We obtain a spectral sequence for
each of these versions. We recover Theorem \ref{thm:knot-S3} from
the reduced hat version, since $\widetilde{HFL'}(Y,K)=\widehat{HFK}(Y,K)$
for knots.
\begin{thm}[Corollary \ref{cor:khovanov-main}]
\label{thm:main-thm}If $L$ is a link with a distinguished point
in $S^{3}$, then there exist spectral sequences
\[
\widetilde{Kh}(m(L))\Rightarrow\widetilde{HFL'}(S^{3},L),\ \widehat{Kh}(m(L))\Rightarrow\widehat{HFL'}(S^{3},L),\ Kh^{-}(m(L))\Rightarrow HFL'{}^{-}(S^{3},L).
\]
\end{thm}

These spectral sequences respect the relative ``\emph{$\delta$ gradings},''
and the relative Alexander $\mathbb{Z}/2$ grading on unoriented link
Floer homology corresponds to the relative $q/2$ (quantum grading
divided by $2$) grading modulo $2$ on Khovanov homology; see Corollary
\ref{cor:khovanov-main}.
\begin{rem}
Since we work over $\mathbb{F}$, $Kh^{-}\simeq\widehat{Kh}\otimes_{\mathbb{F}}\mathbb{F}[U]$
as $\mathbb{F}[U]$-modules. It seems interesting to ask which $\mathbb{F}[U]$-towers
survive to the $E_{\infty}$-page: see Subsection \ref{subsec:Future-directions}.
\end{rem}

We obtain the following as a corollary (see Lemma \ref{lem:rank-ineq}).
\begin{cor}
\label{cor:link-inequality}Let $L$ be an $\ell$-component link
with a distinguished point in $S^{3}$. Then, we have 
\[
2^{\ell-1}\dim_{\mathbb{F}}\widetilde{Kh}(L)\ge2^{\ell-1}\dim_{\mathbb{F}}\widetilde{HFL'}(S^{3},L)\ge\dim_{\mathbb{F}}\widehat{HFL}(S^{3},L).
\]
\end{cor}

We prove Theorem \ref{thm:main-thm} by iterating a modified version
of the unoriented skein exact triangle that we studied in \cite{nahm2025unorientedskeinexacttriangle},
following the general strategy of obtaining a spectral sequence by
iterating an exact triangle; compare \cite{MR2141852,MR2805599,MR2964628}.
In particular, we get spectral sequences for links in any three-manifold
(Theorems \ref{thm:spectral-1}, \ref{thm:spectral-1-1}, and \ref{thm:spectral-1-2}).

\subsection{Organization}

To see what goes into the proof, we suggest reading Sections \ref{sec:Unoriented-link-Floer}
(ignoring discussions on ${\rm Spin}^{c}$-structures) and \ref{sec:The-setup-and},
Subsection \ref{subsec:Standard-generators}, and then Section \ref{sec:Iterating-the-unoriented}.
See Section \ref{sec:Some-simple-examples} for some simple examples.

In Section \ref{sec:Iterating-the-unoriented}, we assume the results
of the later sections and iterate an unoriented skein exact triangle
to build a \emph{``cube of resolutions''} and obtain a spectral
sequence. To build the cube of resolutions, we define the \emph{edge
maps}, check that the edge maps commute (up to homotopy, and so add
in the homotopies to the cube of resolution), and also check \emph{``higher
commutativity''} (and add in the ``higher homotopies'' as well).
Defining the edge maps and showing that they commute up to homotopy,
which boils downs to a series of local computations, take up most
of the paper.

Note that we work in the minus version throughout most of the paper
and consider the reduced and unreduced hat versions in Section \ref{sec:Iterating-the-unoriented}.

\subsubsection*{Preliminaries}

We build the cube of resolutions by building a twisted complex in
an $A_{\infty}$-category. In Section \ref{sec:Unoriented-link-Floer},
we define the $A_{\infty}$-categories we work in, and also recall
the definitions of balled links and unoriented link Floer homology.
In Section \ref{sec:More-Heegaard-Floer}, we recall the tools that
we use to carry out the computations. We recall the relevant definitions
of Khovanov homology in Section \ref{sec:Khovanov-homology}.

\subsubsection*{The setup}

In Section \ref{sec:The-setup-and}, we introduce a setup which will
be used throughout the remainder of the paper and describe how to
get a multi Heegaard diagram from that setup. Some properties, including
gradings, of the corresponding $A_{\infty}$-category are studied
in Section \ref{sec:Gradings-for-the}. 

\subsubsection*{The definition of the edge maps}

We define the edge maps in multiple steps. They are finally defined
in Definition \ref{def:For-each-consecutive} by modifying\footnote{The band maps defined in \cite{nahm2025unorientedskeinexacttriangle}
do not commute, but it is possible to modify them so that they commute.
See Remark \ref{rem:Similarly,-we-can}.} the \emph{``general band maps''}\footnote{We simply call these band maps later.}
(\emph{``canonical elements''}) defined in Section \ref{sec:Canonical-generators}.

Some simple cases of the general band maps are the band maps defined
in \cite{nahm2025unorientedskeinexacttriangle}, which we recall in
Section \ref{sec:Band-maps-for}. In Section \ref{sec:Swap-maps},
we define \emph{``swap maps''}; we later show that the general band
maps are compositions of band maps and swap maps (discussion right
before Proposition \ref{prop:The-swap,-merge,}).

We needed to understand the unoriented link Floer homology of some
simple links in $S^{1}\times S^{2}$ (Figure \ref{fig:s1s2-important-1-1})
to define the band maps. To define the general band maps, we also
have to study some links in $\#^{2}S^{1}\times S^{2}$ (Figure \ref{fig:Some-important-balled}).
We present two methods of computation: one method, carried out in
Appendix \ref{sec:Direct-computations}, is a direct computation using
Heegaard diagrams. The other method, presented in Section \ref{sec:Some-important-balled},
uses the unoriented skein exact sequence and some algebraic structures.

\subsubsection*{Edge maps commute}

Proposition \ref{prop:We-have-the} (\ref{enu:4}) (also see Proposition
\ref{prop:(Modified-band-maps-2}) states that the edge maps (\emph{``modified
band maps''}) commute. We show this through some local computations
and algebraic arguments. We state the local computation results in
terms of schematics: see Section \ref{sec:Interpreting-schematics}.
We introduce a general nonvanishing argument in Section \ref{sec:A-nonvanishing-argument}.
The computations are carried out in Sections \ref{sec:Band-maps-and},
\ref{sec:Modified-band-maps}, \ref{sec:Composition-of-band}, and
\ref{sec:Computations-for-the}.

\subsubsection*{Higher commutativity}

Similarly to \cite{MR2964628}, higher commutativity follows from
grading considerations; we study gradings in Section \ref{sec:Gradings-for-the}. 

\subsubsection*{The cube of resolutions is quasi-isomorphic to unoriented link Floer
homology}

This follows from Proposition \ref{prop:We-have-the} (\ref{enu:6})
(also see Proposition \ref{prop:(Modified-band-maps-2}), by a standard
argument which we recall right before Remark \ref{rem:unreduced-hat}.
We show these propositions in Section \ref{sec:Proofs-of-claims}.

\subsection{\label{subsec:Future-directions}Future directions}

We do not deal with signs and work over the field $\mathbb{F}$ with
two elements. In contrast, Dowlin \cite{MR4777638} works over $\mathbb{Q}$,
and has to divide by $2$. It would be interesting to work out whether
an analogous spectral sequence works over $\mathbb{Z}$. We do not
know any direct ways of computing our spectral sequence; it would
be interesting to compute it directly, using bordered methods or by
defining analogous spectral sequences in grid homology. It would also
be interesting to compare our spectral sequence with Dowlin's spectral
sequence; Dowlin's spectral sequence uses an oriented cube of resolutions
which involves singular links, and it is not clear to the author how
they would be related.

Furthermore, since we work over $\mathbb{F}$, our spectral sequence
for the reduced version is trivial for alternating knots, but is nontrivial
for the unreduced version and the minus version, even for the trefoils
(see Section \ref{sec:Some-simple-examples}). Indeed, for the minus
version, the $E_{2}$ page is $Kh^{-}\simeq\widehat{Kh}\otimes_{\mathbb{F}}\mathbb{F}[U]$
as $\mathbb{F}[U]$-modules, which is a free, rank $6$ $\mathbb{F}[U]$-module.
On the other hand, the $E_{\infty}$ page is $\mathbb{F}[U^{1/2}]\oplus\mathbb{F}[U^{1/2}]/(U^{1/2})$.
We can compute that the $\mathbb{F}[U]$-towers on the $E_{2}$ page
that survive to the $E_{\infty}$ page lie in homological grading
$\pm2$. This is perhaps similar to Rasmussen's $E(-1)$ spectral
sequence \cite{MR3447099} and Ballinger's $t$-invariant \cite{ballinger2020concordance};
it would be interesting to understand possible connections, and hence
it would be interesting to compute our spectral sequence even for
the infinity version.

\subsection{\label{subsec:Conventions-and-notations}Conventions and notations}

Three-manifolds are closed and oriented, but may be disconnected.
Handlebodies and Heegaard surfaces may be disconnected as well.

If a Heegaard surface is drawn on a plane, then almost complex structures
rotate in a counterclockwise direction. If a Heegaard surface is drawn
such that it bounds a (part of a) handlebody in a bounded region,
then it is oriented as the boundary of the handlebody (the shaded
regions in Figure \ref{fig:zstandard-heegaard-triple} have holomorphic
representatives for the order (red, blue, green)). The Heegaard surface
is $-\partial H_{\boldsymbol{\beta}}$, where $H_{\boldsymbol{\beta}}$
is the $\boldsymbol{\beta}$-handlebody. See Footnote \ref{fn:The-attaching-curves}.

We sometimes write $R\{a\}\oplus R\{b\}$, for instance, to mean the
rank $2$ free $R$-module generated by $a,b$.

We describe how some maps act on a module by describing their actions
on a set of generators. For example, the left hand side of Proposition
\ref{prop:For-each-balled} means $\Phi_{1}(a)=c$, $\Phi_{2}(a)=b+c$,
and so on. Beware that for instance in the last paragraph of the proof
of Lemma \ref{lem:structure-z11-2}, the left hand side diagram means
$A_{\kappa,-1}(d)=Zay+Wbx$, instead of $Z(d)=ay$.

The indeterminates $U_{p}$ are indexed by \emph{basepoint pairs}
$p$ and also \emph{baseballs} $BB$, and we identify $U_{p}=U_{BB}$
if $p$ belongs to $BB$. See Notations \ref{nota:Instead-of-writing}
and \ref{nota:Recall-from-Notation}, and Footnote \ref{fn:Similarly-to-Notation}.

\subsection{Acknowledgements}

We thank Peter Ozsv\'{a}th for his continuous support, explaining
a lot of the arguments in this paper, and helpful discussions. We
also thank Ian Zemke for his continuous support, teaching the author
a lot of previous works, especially \cite{2308.15658}, and helpful
discussions. We thank John Baldwin, Deeparaj Bhat, Fraser Binns, Evan
Huang, Yonghwan Kim, Jae Hee Lee, Adam Levine, Jiakai Li, Ayodeji
Lindblad, Robert Lipshitz, Marco Marengon, Yuta Nakayama, Sucharit
Sarkar, Zolt\'{a}n Szab\'{o}, Alison Tatsuoka, Jacob Rasmussen,
Joshua Wang, and Fan Ye for helpful discussions. We also thank Peter
Ozsv\'{a}th and Jacob Rasmussen for helpful comments on earlier drafts
of this paper.

\section{\label{sec:Unoriented-link-Floer}Unoriented link Floer homology}

In \cite{nahm2025unorientedskeinexacttriangle}, we considered certain
rank $2$ local systems that a priori involve negative powers of $U$.
We proved that if there are at most two attaching curves with such
a nontrivial local system, then it is possible to work in a minus
theory (i.e. there are no negative powers of $U$). Also, we observed
that if there is at most one attaching curve with such a nontrivial
local system, then we can express the chain complexes and the $\mu_{n}$'s
without mentioning local systems. In this paper, we push this observation
further: we morally use local systems, but phrase them in a way that
does not mention any local systems.

\subsection{\label{subsec:Pair-pointed-Heegaard-diagrams}Pair-pointed Heegaard
diagrams}
\begin{defn}
\label{def:A-pair-pointed-Heegaard}A\emph{ pair-pointed Heegaard
surface} $(\Sigma,\boldsymbol{p},\boldsymbol{G})$ consists of
\begin{itemize}
\item a closed, oriented surface $\Sigma$\footnote{We allow the Heegaard surface to be disconnected, as in for instance
\cite{1512.01184} (see Remark \ref{rem:If-the-Heegaard}).},
\item a set $\boldsymbol{p}$ of pairs of basepoints in $\Sigma$ (\emph{``basepoint
pairs''}),
\item and an arc $G$ (\emph{``forbidding arc''}) between $z$ and $w$
for each basepoint pair $(z,w)$.
\end{itemize}
An \emph{attaching curve }$(\boldsymbol{\gamma},E_{\boldsymbol{\gamma}})$
(we usually omit $E_{\boldsymbol{\gamma}}$ and just write $\boldsymbol{\gamma}$)
on a pair-pointed Heegaard surface $(\Sigma,\boldsymbol{p},\boldsymbol{G})$
is
\begin{itemize}
\item a set $\boldsymbol{\gamma}$ of pairwise disjoint, simple closed curves
on $\Sigma\backslash\boldsymbol{p}$ such that each connected component
of $\Sigma\backslash\boldsymbol{\gamma}$ is planar and contains exactly
one basepoint pair (and no other basepoints),
\item together with an assignment $E_{\boldsymbol{\gamma}}:\boldsymbol{p}\to\{{\rm free},{\rm link}\}$
such that if $E_{\boldsymbol{\gamma}}(p)={\rm free}$, then the corresponding
forbidding arc does not intersect $\boldsymbol{\gamma}$. We say that
$p\in\boldsymbol{p}$ is a \emph{free basepoint pair} (resp.\emph{
link basepoint pair}) for $\boldsymbol{\gamma}$ if $E_{\boldsymbol{\gamma}}(p)={\rm free}$
(resp. ${\rm link}$).
\end{itemize}
A \emph{pair-pointed Heegaard diagram} consists of a pair-pointed
Heegaard surface $(\Sigma,\boldsymbol{p},\boldsymbol{G})$, together
with a partially ordered set $(\boldsymbol{AC},\le)$ of attaching
curves, such that for every chain $\boldsymbol{\gamma}_{1}<\cdots<\boldsymbol{\gamma}_{m}$
and for every basepoint pair, the subset of $\boldsymbol{\gamma}_{k}$'s
for which the basepoint pair is a link basepoint pair is $\{\boldsymbol{\gamma}_{k}:i\le k\le j\}$
for some $i,j$.
\end{defn}

\begin{rem}
\label{rem:The-last-condition}The last condition in Definition \ref{def:A-pair-pointed-Heegaard}
on the chains and basepoint pairs is important for the (higher) compositions
$\mu_{n}$ (see Subsection \ref{subsec:ainf}) to satisfy the $A_{\infty}$-relations.
\end{rem}

An \emph{elementary two-chain} is a connected component of the Heegaard
surface minus the attaching curves, and a \emph{two-chain} is a linear
combination of elementary two-chains. A two-chain ${\cal D}$ is \emph{cornerless}
if $\partial_{\boldsymbol{\gamma}}{\cal D}$ is a cycle for every
attaching curve $\boldsymbol{\gamma}$. A \emph{domain} is a two-chain
together with its vertices, and $D({\bf x}_{0},\cdots,{\bf x}_{n})$
is the set of domains whose vertices are ${\bf x}_{0},\cdots,{\bf x}_{n}$.
Compare \cite[Subsection 2.1]{nahm2025unorientedskeinexacttriangle}.
\begin{defn}
\label{def:translate}If $(\Sigma,\boldsymbol{\alpha},\boldsymbol{\beta},\boldsymbol{p},\boldsymbol{G})$
is a pair-pointed Heegaard diagram with two attaching curves, then
identify it with the \emph{associated Heegaard diagram $(\Sigma,\boldsymbol{\alpha},\boldsymbol{\beta},\boldsymbol{u}\sqcup\boldsymbol{v})$},
where $\boldsymbol{u}$, $\boldsymbol{v}$ are as follows. Let $\boldsymbol{v}$
be the set of the first basepoints of each basepoint pair that is
a free basepoint pair for both $\boldsymbol{\alpha}$ and $\boldsymbol{\beta}$,
and let $\boldsymbol{u}$ be the set of both basepoints of each basepoint
pair that is a link basepoint pair for $\boldsymbol{\alpha}$ or $\boldsymbol{\beta}$.
\end{defn}

We will often simply call pair-pointed Heegaard diagrams Heegaard
diagrams.

\subsection{\label{subsec:Unoriented-links-and}Unoriented links and ${\rm Spin}^{c}$-structures}

In this subsection, we recall \cite[Subsections 2.7 and 2.8]{nahm2025unorientedskeinexacttriangle}.
Given an unoriented link $L$ in a three-manifold $Y$, we define
the space of ${\rm Spin}^{c}$-structures ${\rm Spin}^{c}(Y(L))$
as an $H_{1}^{orb}(Y(L);\mathbb{Z}):=H_{1}(Y\backslash L;\mathbb{Z})/2M$-torsor,
where $M\subset H_{1}(Y\backslash L;\mathbb{Z})$ is the subspace
spanned by the meridians of $L$.
\begin{defn}
\label{def:suture-data}A \emph{minimally pointed link $(L,\boldsymbol{u})$}
is a link $L\subset Y$ together with a set of basepoints $\boldsymbol{u}\subset L$
(\emph{``link basepoints''}) such that for each connected component
$L_{i}$ of $L$, we have $|L_{i}\cap\boldsymbol{u}|=2$. A \emph{suture
datum} $\alpha_{L}$ on a minimally pointed link $L\subset Y$ is
a choice of one of the two connected components of $L_{i}\backslash\boldsymbol{u}$
(which corresponds to the component inside the $\alpha$-handlebody)
for each connected component $L_{i}$.
\end{defn}

If $L^{sut}=(L,\boldsymbol{u},\alpha_{L})$ is a minimally pointed
link equipped with a suture datum in a pointed three-manifold $(Y,\boldsymbol{v})$,
consider the following sutured manifold $Y_{L^{sut}}$: the underlying
three-manifold is $Y\backslash N(L\cup\boldsymbol{v})$ where $N$
means tubular neighborhood, and each toroidal boundary component has
two meridional sutures and each sphere boundary component has one
suture. The sutures are oriented in a way such that $R_{-}$\footnote{Our convention is that $R_{-}$ corresponds to the $\alpha$-handlebody
part.} is the part that lies over the chosen components of $L\backslash\boldsymbol{u}$.
We say that a \emph{pair-pointed Heegaard diagram $(\Sigma,\boldsymbol{\alpha},\boldsymbol{\beta},\boldsymbol{p},\boldsymbol{G})$
represents $L^{sut}$ in $(Y,\boldsymbol{v})$} if the associated
Heegaard diagram $(\Sigma,\boldsymbol{\alpha},\boldsymbol{\beta},\boldsymbol{u}\sqcup\boldsymbol{v})$
represents the sutured manifold $Y_{L^{sut}}$.

Define ${\rm Spin}^{c}(Y_{L^{sut}})$ as the space of ${\rm Spin}^{c}$-structures
on the sutured manifold $Y_{L^{sut}}$ (\cite[Section 4]{MR2253454}).
Then, ${\rm Spin}^{c}(Y_{L^{sut}})$ is an $H^{2}(Y,L;\mathbb{Z})\simeq H_{1}(Y\backslash L;\mathbb{Z})$-torsor,
and an intersection point ${\bf x}\in\boldsymbol{\alpha}\cap\boldsymbol{\beta}$
determines a ${\rm Spin}^{c}$-structure $\mathfrak{s}({\bf x})\in{\rm Spin}^{c}(Y_{L^{sut}})$.

Given an orientation $\mathfrak{o}$ on $L$, \cite[Section 3]{MR2443092}
defines the first Chern class\footnote{\cite[Section 3]{MR2443092} defines the \emph{space of relative ${\rm Spin}^{c}$-structures}
$\underline{{\rm Spin}^{c}}(Y,L)$ and a map $c_{1}:\underline{{\rm Spin}^{c}}(Y,L)\to H^{2}(Y,L;\mathbb{Z})$,
and given an orientation $\mathfrak{o}$ on $L$, also an isomorphism
of $H^{2}(Y,L;\mathbb{Z})$-torsors ${\rm Spin}^{c}(Y_{L^{sut}})\to\underline{{\rm Spin}^{c}}(Y,L)$.} $c_{1}^{\mathfrak{o}}:{\rm Spin}^{c}(Y_{L^{sut}})\to H^{2}(Y,L;\mathbb{Z})$.
We have $c_{1}^{\mathfrak{o}}(\mathfrak{s}+\omega)=c_{1}^{\mathfrak{o}}(\mathfrak{s})+2\omega$
for $\mathfrak{s}\in{\rm Spin}^{c}(Y_{L^{sut}})$ and $\omega\in H^{2}(Y,L;\mathbb{Z})$.
For different choices of $\mathfrak{o}$, the Chern class $c_{1}^{\mathfrak{o}}$
differs by the Poincare dual of a sum of some meridians of $L$. Hence,
we get a map $c_{1}:{\rm Spin}^{c}(Y_{L^{sut}})\to H^{2}(Y;\mathbb{Z})$
that does not depend on $\mathfrak{o}$.

Define the \emph{space of ${\rm Spin}^{c}$-structures} as ${\rm Spin}^{c}(Y(L^{sut})):={\rm Spin}^{c}(Y_{L^{sut}})/2M$.
An intersection point ${\bf x}\in\boldsymbol{\alpha}\cap\boldsymbol{\beta}$
determines a ${\rm Spin}^{c}$-structure $\mathfrak{s}({\bf x})\in{\rm Spin}^{c}(Y(L^{sut}))$.
Also, ${\rm Spin}^{c}(Y(L^{sut}))$ is an $H_{1}^{orb}(Y(L);\mathbb{Z}):=H_{1}(Y\backslash L;\mathbb{Z})/2M$-torsor,
and we have the induced map $c_{1}:{\rm Spin}^{c}(Y(L^{sut}))\to H^{2}(Y;\mathbb{Z})$.
\begin{rem}
\label{rem:compute-chern}Here is one way to compute $c_{1}(\mathfrak{s}({\bf x}))$.
Let $\boldsymbol{z}$ (resp. $\boldsymbol{w}$) be the set of the
first (resp. second) basepoints of each basepoint pair. Then, each
connected component of $\Sigma\backslash\boldsymbol{\alpha}$ and
$\Sigma\backslash\boldsymbol{\beta}$ contains exactly one element
of $\boldsymbol{z}$ and one element of $\boldsymbol{w}$, and so
we can consider $\mathfrak{s}_{\boldsymbol{z}}({\bf x})\in{\rm Spin}^{c}(Y)$
(resp. $\mathfrak{s}_{\boldsymbol{w}}({\bf x})$), i.e. the ${\rm Spin}^{c}$-structure
of ${\bf x}\in\boldsymbol{\alpha}\cap\boldsymbol{\beta}$ with respect
to $\boldsymbol{z}$ (resp. $\boldsymbol{w}$). Furthermore, $\boldsymbol{z}$
and $\boldsymbol{w}$ define an orientation of $L$: call this oriented
link $\overrightarrow{L}$. Then, by \cite[Lemma 3.13]{MR2443092}
(\cite[Equation (2.3)]{nahm2025unorientedskeinexacttriangle}), 
\begin{equation}
c_{1}(\mathfrak{s}({\bf x}))=c_{1}(\mathfrak{s}_{\boldsymbol{w}}({\bf x}))-PD(\overrightarrow{L})=c_{1}(\mathfrak{s}_{\boldsymbol{z}}({\bf x}))+PD(\overrightarrow{L}).\label{eq:c1}
\end{equation}
\end{rem}

\begin{rem}
\label{rem:Note-that-different}Note that different assignments $E_{\boldsymbol{\alpha}},E_{\boldsymbol{\beta}}$
result in different $L$'s: the corresponding links differ by some
number of unknots. However, for ${\bf x}\in\boldsymbol{\alpha}\cap\boldsymbol{\beta}$,
the Chern class $c_{1}(\mathfrak{s}({\bf x}))$ does not depend on
$E_{\boldsymbol{\alpha}},E_{\boldsymbol{\beta}}$, by Equation (\ref{eq:c1}).
\end{rem}

\begin{rem}
For simplicity, we assumed that links are minimally pointed since
that is the setting that we will use throughout most of the paper,
with the exception being Subsection \ref{sec:Higher-compositions-vanish}.
However, the discussion works in general, where each link component
has an even number of link basepoints.
\end{rem}

\subsection{\label{subsec:Strong-admissibility-and}Strong admissibility and
unoriented link Floer homology}

In this subsection, we define strong admissibility and unoriented
link Floer homology. Compare \cite[Subsections 2.9 and 2.10]{nahm2025unorientedskeinexacttriangle}.

Let $R:=\mathbb{F}[\{U_{p}^{k_{p}}\}_{p\in\boldsymbol{p}}]$\footnote{This means that it is the polynomial ring generated by indeterminates
$U_{p}^{k_{p}}$ for $p\in\boldsymbol{p}$.} for some $k_{p}=\{1,1/2\}$ such that whenever $p$ is a link basepoint
pair for $\boldsymbol{\alpha}$ or $\boldsymbol{\beta}$, then $k_{p}=1/2$. 
\begin{notation}
\label{nota:Instead-of-writing}If $v$ is a basepoint in a basepoint
pair $p$, then let $U_{v}:=U_{p}$. Also, instead of writing $U_{(z,w)}$
for a basepoint pair $(z,w)$, we sometimes write $U_{i}$ if the
basepoint pair has index $i$ (for instance, if it is called $(z_{i},w_{i})$).
We use the same notation conventions for other objects as well.
\end{notation}

\begin{defn}
Let ${\cal D}$ be a two-chain. If $p=(z,w)$ is a basepoint pair,
let $n_{p}({\cal D}):=n_{z}({\cal D})+n_{w}({\cal D})$. The \emph{total
multiplicity} $P({\cal D})$ of ${\cal D}$ is 
\[
P({\cal D})=\sum_{p\in\boldsymbol{p}}n_{p}({\cal D}).
\]
\end{defn}

\begin{defn}
\label{def:strongly-admissible-multi}A pair-pointed Heegaard diagram
with attaching curves $\boldsymbol{\alpha},\boldsymbol{\beta}$ is
\emph{$c$-strongly admissible} for $c\in H^{2}(Y;\mathbb{Z})$, where
$Y$ is the three-manifold given by $\boldsymbol{\alpha},\boldsymbol{\beta}$,
if all cornerless domains ${\cal D}$ such that 
\[
\left\langle c,H({\cal D})\right\rangle +P({\cal D})=0
\]
have both positive and negative local multiplicities, where $H({\cal D})\in H_{2}(Y;\mathbb{Z})$
is the homology class that ${\cal D}$ represents. A Heegaard diagram
is \emph{weakly admissible }if it is $0$-strongly admissible.
\end{defn}

\begin{lem}[{\cite[Lemma 2.47]{nahm2025unorientedskeinexacttriangle}}]
\label{lem:If--is}If ${\cal D}\in D({\bf x},{\bf x})$ is a cornerless
domain, then 
\[
\mu({\cal D})=\left\langle c_{1}(\mathfrak{s}({\bf x})),H({\cal D})\right\rangle +P({\cal D}).
\]
\end{lem}

\begin{defn}
If a pair-pointed Heegaard diagram with attaching curves $\boldsymbol{\alpha}$
and $\boldsymbol{\beta}$ is $c$-strongly admissible, then define
$CF_{R}^{-}(\boldsymbol{\alpha},\boldsymbol{\beta};c)$ as the free
$R$-module generated by intersection points (\emph{``generators''})
${\bf x}\in\boldsymbol{\alpha}\cap\boldsymbol{\beta}$ such that $c_{1}(\mathfrak{s}({\bf x}))=c$.
Define the differential $\partial$ as the $R$-linear map such that
\[
\partial{\bf x}=\sum_{{\cal D}\in D({\bf x},{\bf y}),\ \mu({\cal D})=1}\#{\cal M}({\cal D})\prod_{p\in\boldsymbol{p}}U_{p}^{\frac{1}{2}n_{p}({\cal D})}{\bf y}.
\]
\end{defn}

By Lemma \ref{lem:If--is}, $CF_{R}^{-}(\boldsymbol{\alpha},\boldsymbol{\beta};c)$
is a well-defined chain complex, and if $\mathfrak{d}(c)$ is the
\emph{divisibility} of $c$, then $CF_{R}^{-}(\boldsymbol{\alpha},\boldsymbol{\beta};c)$
can be given a (relative) homological $\mathbb{Z}/\mathfrak{d}(c)$-grading,
where every intersection point ${\bf x}\in\boldsymbol{\gamma}_{1}\cap\boldsymbol{\gamma}_{2}$
is homogeneous and $U_{p}^{k}$ has homological grading $-2k$. (Compare
\cite[Subsubsections 2.8.3 and 2.10.1]{nahm2025unorientedskeinexacttriangle})
Also, we sometimes lift this relative homological grading to an absolute
homological grading; in that case, we write $CF_{R,i}^{-}$, $HF_{R,i}^{-}$
to denote the $\mathbb{F}$-linear summands in grading $i$. Similarly,
if there is a relative homological $\mathbb{Z}$-grading, we write
$CF_{R,top}^{-}$, $HF_{R,top}^{-}$ to denote the top relative homological
$\mathbb{Z}$-grading summands. 

Also, if the Heegaard surface is connected, then multiplication by
$U_{p}$ for different $p\in\boldsymbol{p}$ are homotopic (\cite[Equation (2.8)]{MR4845975}),
and so we can view $HF_{R}^{-}$ as an $R[U]/(\{U-U_{p}\}_{p\in\boldsymbol{p}})$-module,
and in particular as an $\mathbb{F}[U]$-module, where $U$ acts by
multiplying by $U_{p}$.
\begin{rem}
\label{rem:In-particular,-}In particular, $HF_{R}^{-}(\boldsymbol{\alpha},\boldsymbol{\beta};0)$
can be given a relative homological $\mathbb{Z}$-grading, but it
is in general not unique. However, if $H^{2}(Y;\mathbb{Z})$ does
not have $2$-torsion, then it is unique since there is a domain in
$D({\bf x},{\bf y})$ for any two intersection points ${\bf x},{\bf y}$
with the same $c_{1}$, and so we can say \emph{the} relative homological
$\mathbb{Z}$-grading, or \emph{the} top homological $\mathbb{Z}$-grading.
\end{rem}

Now, let us specialize to a specific ring $R_{\boldsymbol{\alpha},\boldsymbol{\beta}}$.
\begin{defn}
Define $R_{\boldsymbol{\alpha},\boldsymbol{\beta}}:=\mathbb{F}[\{U_{p}^{k_{p}}\}_{p\in\boldsymbol{p}}]$
(\emph{``the coefficient ring''}), where $k_{p}=1$ if the $p$
is a free basepoint pair for both $\boldsymbol{\alpha}$ or $\boldsymbol{\beta}$,
and $k_{p}=1/2$ otherwise. Define $CF^{-}(\boldsymbol{\alpha},\boldsymbol{\beta};c):=CF_{R_{\boldsymbol{\alpha},\boldsymbol{\beta}}}^{-}(\boldsymbol{\alpha},\boldsymbol{\beta};c)$.
\end{defn}

Let $S:=\mathbb{F}[\{U_{p}\}_{p\in\boldsymbol{p}}]$. Then, $R_{\boldsymbol{\alpha},\boldsymbol{\beta}}=\mathbb{F}[\{U_{p}^{k_{p}}\}_{p\in\boldsymbol{p}}]$
is a free $S$-module, with basis 
\[
B:=\left\{ \prod_{k_{p}=1/2}U_{p}^{\ell_{p}}:\ell_{p}\in\{0,1/2\}\right\} ,
\]
where the product ranges over $p$'s such that $k_{p}=1/2$. An \emph{$S$-generator
}for $\boldsymbol{\alpha},\boldsymbol{\beta}$ is an element of the
form $b{\bf x}$ for $b\in B$, ${\bf x}\in\boldsymbol{\alpha}\cap\boldsymbol{\beta}$.
The ${\rm Spin}^{c}$-structure of an $S$-generator is defined as
\[
\mathfrak{s}\left(\prod U_{p}^{\ell_{p}}{\bf x}\right):=\mathfrak{s}({\bf x})+\sum_{p}2\ell_{p}\mu_{p}\in{\rm Spin}^{c}(Y(L^{sut})),
\]
where $\mu_{p}\in H_{1}^{orb}(Y(L);\mathbb{Z})$ is defined as follows:
\begin{itemize}
\item if $p$ is a link basepoint pair, then $\mu_{p}$ is the meridian
of the link component that $p$ is on.
\item if $p$ is a free basepoint pair, then $\mu_{p}=0$.
\end{itemize}
\begin{defn}
If a Heegaard diagram with attaching curves $\boldsymbol{\alpha}$
and $\boldsymbol{\beta}$ is $c_{1}(\mathfrak{s})$-strongly admissible
for some ${\rm Spin}^{c}$-structure $\mathfrak{s}\in{\rm Spin}^{c}(Y(L^{sut}))$,
define $CF^{-}(\boldsymbol{\alpha},\boldsymbol{\beta};\mathfrak{s})$
as the free $S$-sub chain complex of $CF^{-}(\boldsymbol{\alpha},\boldsymbol{\beta};c_{1}(\mathfrak{s}))$
generated by the $S$-generators $b{\bf x}$ for $b\in B$, ${\bf x}\in\boldsymbol{\alpha}\cap\boldsymbol{\beta}$
such that $\mathfrak{s}(b{\bf x})=\mathfrak{s}$.
\end{defn}

Hence, we have an $S$-linear splitting of chain complexes
\[
CF^{-}(\boldsymbol{\alpha},\boldsymbol{\beta};c)=\bigoplus_{c_{1}(\mathfrak{s})=c}CF^{-}(\boldsymbol{\alpha},\boldsymbol{\beta};\mathfrak{s}),
\]
which we call the \emph{${\rm Spin}^{c}$-splitting}.

\begin{defn}[Unoriented link Floer homology]
\label{def:unori}Let $L^{sut}:=(L,\boldsymbol{u},\alpha_{L})$ be
a minimally pointed link equipped with a suture datum inside a pointed
three-manifold $(Y,\boldsymbol{v})$ such that every connected component
of $Y$ contains at least one basepoint $v\in\boldsymbol{u}\sqcup\boldsymbol{v}$.
Define the \emph{unoriented link Floer chain complex }as 
\[
CFL'{}^{-}((Y,\boldsymbol{v}),L^{sut};c):=CF^{-}(\boldsymbol{\alpha},\boldsymbol{\beta};c),\ CFL'{}^{-}((Y,\boldsymbol{v}),L^{sut};\mathfrak{s}):=CF^{-}(\boldsymbol{\alpha},\boldsymbol{\beta};\mathfrak{s})
\]
where $(\Sigma,\boldsymbol{\alpha},\boldsymbol{\beta},\boldsymbol{p},\boldsymbol{G})$
is a pair-pointed Heegaard diagram that represents $L^{sut}$ in $(Y,\boldsymbol{v})$.
\end{defn}

\subsection{\label{subsec:ainf}The \texorpdfstring{$A_{\infty}$}{Ainf}-categories
that we work in}

When we consider more than two attaching curves, we have to in general
fix a ${\rm Spin}^{c}$-structure of the associated four-manifold
to talk about strong admissibility. In \cite[Subsection 2.10]{nahm2025unorientedskeinexacttriangle},
we introduced an assumption for which this becomes simple.
\begin{defn}
\label{def:Two-attaching-curves}Two attaching curves $\boldsymbol{\beta}_{1},\boldsymbol{\beta}_{2}$
on a pair-pointed Heegaard surface $(\Sigma,\boldsymbol{p},\boldsymbol{G})$
are \emph{handlebody-equivalent} if they define the same handlebody
(ignoring $\boldsymbol{p}$).
\end{defn}

We assume that the poset $\boldsymbol{AC}$ of attaching curves is
a disjoint union of alpha attaching curves $\boldsymbol{A}$ and beta
attaching curves $\boldsymbol{B}$; that $\boldsymbol{\alpha}<\boldsymbol{\beta}$
for every alpha attaching curve $\boldsymbol{\alpha}\in\boldsymbol{A}$
and beta attaching curve $\boldsymbol{\beta}\in\boldsymbol{B}$; that
all the alpha attaching curves are pairwise handlebody-equivalent;
and that all the beta attaching curves are also pairwise handlebody-equivalent.
Identify the alpha-handlebodies, and identify the beta-handlebodies.
If $\boldsymbol{A},\boldsymbol{B}\neq\emptyset$, then let $Y$ be
the three-manifold given by the alpha-handlebody and the beta-handlebody.
\begin{defn}
\label{def:strongly-admissible-multi-1}Let $\boldsymbol{A},\boldsymbol{B}\neq\emptyset$.
The Heegaard diagram is \emph{$c$-strongly admissible} for $c\in H^{2}(Y;\mathbb{Z})$
if all cornerless two-chains ${\cal D}$ such that 
\[
\left\langle c,H({\cal D})\right\rangle +P({\cal D})=0
\]
have both positive and negative local multiplicities.
\end{defn}

\begin{defn}
A pair-pointed Heegaard diagram is \emph{weakly admissible} if all
cornerless two-chains ${\cal D}$ such that $n_{v}({\cal D})=0$ for
all basepoints $v$ have both positive and negative local multiplicities.
\end{defn}

If $\boldsymbol{A},\boldsymbol{B}\neq\emptyset$, then fix $c\in H^{2}(Y;\mathbb{Z})$
and consider a $c$-strongly admissible Heegaard diagram with underlying
pair-pointed Heegaard surface $(\Sigma,\boldsymbol{p},\boldsymbol{G})$.
If $\boldsymbol{A}$ or $\boldsymbol{B}$ is empty, then consider
a weakly admissible Heegaard diagram with underlying pair-pointed
Heegaard surface $(\Sigma,\boldsymbol{p},\boldsymbol{G})$ (and formally
let $c=0$). For attaching curves $\boldsymbol{\gamma}_{1},\boldsymbol{\gamma}_{2}\in\boldsymbol{AC}$
such that $\boldsymbol{\gamma}_{1}<\boldsymbol{\gamma}_{2}$, we write
\[
CF_{R}^{-}(\boldsymbol{\gamma}_{1},\boldsymbol{\gamma}_{2}):=\begin{cases}
CF_{R}^{-}(\boldsymbol{\gamma}_{1},\boldsymbol{\gamma}_{2};0) & {\rm if}\ \boldsymbol{\gamma}_{1},\boldsymbol{\gamma}_{2}\in\boldsymbol{A}\ {\rm or}\ \boldsymbol{\gamma}_{1},\boldsymbol{\gamma}_{2}\in\boldsymbol{B}\\
CF_{R}^{-}(\boldsymbol{\gamma}_{1},\boldsymbol{\gamma}_{2};c) & {\rm if}\ \boldsymbol{\gamma}_{1}\in\boldsymbol{A}\ {\rm and}\ \boldsymbol{\gamma}_{2}\in\boldsymbol{B}
\end{cases}.
\]

The associated $A_{\infty}$-category has objects $\boldsymbol{\gamma}\in\boldsymbol{AC}$,
and ${\rm Hom}(\boldsymbol{\gamma}_{1},\boldsymbol{\gamma}_{2})$
is $0$ if $\boldsymbol{\gamma}_{1}\not<\boldsymbol{\gamma}_{2}$,
and if $\boldsymbol{\gamma}_{1}<\boldsymbol{\gamma}_{2}$, then it
is
\[
CF^{-}(\boldsymbol{\gamma}_{1},\boldsymbol{\gamma}_{2}):=CF_{R_{\boldsymbol{\gamma}_{1},\boldsymbol{\gamma}_{2}}}^{-}(\boldsymbol{\gamma}_{1},\boldsymbol{\gamma}_{2}).
\]

We sometimes refer to an element $\sigma\in CF^{-}(\boldsymbol{\gamma}_{1},\boldsymbol{\gamma}_{2})$
as a \emph{map $\sigma:\boldsymbol{\gamma}_{1}\to\boldsymbol{\gamma}_{2}$}.
We will consider twisted complexes in this $A_{\infty}$-category,
following the conventions of \cite[Subsection 2.4]{nahm2025unorientedskeinexacttriangle}.
Similarly, given twisted complexes $\underline{\boldsymbol{\gamma}_{1}}$
and $\underline{\boldsymbol{\gamma}_{2}}$, we refer to an element
$\sigma\in CF^{-}(\underline{\boldsymbol{\gamma}_{1}},\underline{\boldsymbol{\gamma}_{2}})$
as a \emph{map $\sigma:\underline{\boldsymbol{\gamma}_{1}}\to\underline{\boldsymbol{\gamma}_{2}}$}.
\begin{rem}
\label{rem:Recall-Remark-.}Recall Remark \ref{rem:In-particular,-}.
Hence, in particular if $\boldsymbol{\gamma}_{1},\boldsymbol{\gamma}_{2}\in\boldsymbol{A}$
or $\boldsymbol{\gamma}_{1},\boldsymbol{\gamma}_{2}\in\boldsymbol{B}$,
then $CF^{-}(\boldsymbol{\gamma}_{1},\boldsymbol{\gamma}_{2})$ has
a well-defined relative homological $\mathbb{Z}$-grading.
\end{rem}

Let $R_{tot}:=\mathbb{F}[\{U_{p}^{1/2}\}_{p\in\boldsymbol{p}}]$.
Let $P:=\{\boldsymbol{\gamma}_{0}<\cdots<\boldsymbol{\gamma}_{d}\}$
be a sequence of attaching curves. Let $F_{\boldsymbol{\gamma}_{0},\boldsymbol{\gamma}_{d}}$
be the set of basepoint pairs that is a free basepoint pair for both
$\boldsymbol{\gamma}_{0}$ and $\boldsymbol{\gamma}_{d}$. Then, $R_{tot}$
is a rank $2^{|F_{\boldsymbol{\gamma}_{0},\boldsymbol{\gamma}_{d}}|}$,
free $R_{\boldsymbol{\gamma}_{0},\boldsymbol{\gamma}_{d}}$-module,
with basis 
\[
B_{\boldsymbol{\gamma}_{0},\boldsymbol{\gamma}_{d}}:=\left\{ \prod_{p\in F_{\boldsymbol{\gamma}_{0},\boldsymbol{\gamma}_{d}}}U_{p}^{\ell_{p}}:\ell_{p}\in\{0,1/2\}\right\} .
\]
This basis induces a splitting of $R_{\boldsymbol{\gamma}_{0},\boldsymbol{\gamma}_{d}}$-chain
complexes 
\begin{equation}
CF_{R_{tot}}^{-}(\boldsymbol{\gamma}_{0},\boldsymbol{\gamma}_{d})=\bigoplus_{b\in B_{\boldsymbol{\gamma}_{0},\boldsymbol{\gamma}_{d}}}bCF^{-}(\boldsymbol{\gamma}_{0},\boldsymbol{\gamma}_{d}).\label{eq:splitting}
\end{equation}
Let $U_{F}^{1/2}$ be the product of the $U_{p}^{1/2}$'s where $p$
ranges over the basepoints pairs in $F_{\boldsymbol{\gamma}_{0},\boldsymbol{\gamma}_{d}}$
that is a link basepoint pair for at least one of the $\boldsymbol{\gamma}_{i}$'s.
Define 
\[
\mu_{d}^{\boldsymbol{\gamma}_{0},\boldsymbol{\gamma}_{1},\cdots,\boldsymbol{\gamma}_{d}}:CF^{-}(\boldsymbol{\gamma}_{0},\boldsymbol{\gamma}_{1})\otimes_{S}\cdots\otimes_{S}CF^{-}(\boldsymbol{\gamma}_{d-1},\boldsymbol{\gamma}_{d})\to CF^{-}(\boldsymbol{\gamma}_{0},\boldsymbol{\gamma}_{d}),
\]
where $S:=\mathbb{F}[\{U_{p}\}_{p\in\boldsymbol{p}}]$, as the composite
of the following maps:
\begin{gather}
CF^{-}(\boldsymbol{\gamma}_{0},\boldsymbol{\gamma}_{1})\otimes_{S}\cdots\otimes_{S}CF^{-}(\boldsymbol{\gamma}_{d-1},\boldsymbol{\gamma}_{d})\rightarrow CF_{R_{tot}}^{-}(\boldsymbol{\gamma}_{0},\boldsymbol{\gamma}_{1})\otimes_{R_{tot}}\cdots\otimes_{R_{tot}}CF_{R_{tot}}^{-}(\boldsymbol{\gamma}_{d-1},\boldsymbol{\gamma}_{d})\nonumber \\
\xrightarrow{\mu_{d,R_{tot}}^{\boldsymbol{\gamma}_{0},\boldsymbol{\gamma}_{1},\cdots,\boldsymbol{\gamma}_{d}}}CF_{R_{tot}}^{-}(\boldsymbol{\gamma}_{0},\boldsymbol{\gamma}_{d})\xrightarrow{\pi}U_{F}^{1/2}CF^{-}(\boldsymbol{\gamma}_{0},\boldsymbol{\gamma}_{d})\xrightarrow{\cdot U_{F}^{-1/2}}CF^{-}(\boldsymbol{\gamma}_{0},\boldsymbol{\gamma}_{d}),\label{eq:mud-defn}
\end{gather}
where the first map is induced by the inclusions $\iota:CF^{-}(\boldsymbol{\gamma}_{i},\boldsymbol{\gamma}_{i+1})\hookrightarrow CF_{R_{tot}}^{-}(\boldsymbol{\gamma}_{i},\boldsymbol{\gamma}_{i+1})$
into the summand $CF^{-}(\boldsymbol{\gamma}_{i},\boldsymbol{\gamma}_{i+1})$
of Equation (\ref{eq:splitting}) and the third map $\pi$ is projection
onto the $U_{F}^{1/2}CF^{-}(\boldsymbol{\gamma}_{0},\boldsymbol{\gamma}_{d})$
summand of Equation (\ref{eq:splitting}). As in this formula, we
will sometimes denote the higher composition map on the $CF_{R}^{-}(\boldsymbol{\gamma}_{i},\boldsymbol{\gamma}_{j})$'s
as $\mu_{d,R}$. Also, we mostly omit the superscript $\boldsymbol{\gamma}_{0},\cdots,\boldsymbol{\gamma}_{d}$.

Lemma \ref{lem:If--is-1} ensures that the higher multiplications
$\mu_{d,R_{tot}}$ are well-defined and that the $A_{\infty}$-relations
hold for $\mu_{d,R_{tot}}$. Using the last condition of Definition
\ref{def:A-pair-pointed-Heegaard}, one can show that 
\[
\mu_{d-j+1}(\overbrace{{\rm Id}\otimes\cdots\otimes{\rm Id}}^{i}\otimes\mu_{j}({\rm Id}\otimes\cdots\otimes{\rm Id})\otimes\cdots\otimes{\rm Id})=U_{F}^{-1/2}\pi(\mu_{d-j+1,R_{tot}}(\overbrace{\iota\otimes\cdots\otimes\iota}^{i}\otimes\mu_{j,R_{tot}}(\iota\otimes\cdots\otimes\iota)\otimes\cdots\otimes\iota))
\]
as maps 
\[
CF^{-}(\boldsymbol{\gamma}_{0},\boldsymbol{\gamma}_{1})\otimes_{S}\cdots\otimes_{S}CF^{-}(\boldsymbol{\gamma}_{d-1},\boldsymbol{\gamma}_{d})\to CF^{-}(\boldsymbol{\gamma}_{0},\boldsymbol{\gamma}_{d}),
\]
and so the $A_{\infty}$-relations hold for $\mu_{d}$ as well.

Also, Lemma \ref{lem:If--is-1} shows that the above $A_{\infty}$-category
is \emph{homologically $\mathbb{Z}/\mathfrak{d}(c)$-gradable}: in
other words, there exist homological $\mathbb{Z}/\mathfrak{d}(c)$-grdings
on $CF^{-}(\boldsymbol{\gamma}_{1},\boldsymbol{\gamma}_{2})$ for
each $\boldsymbol{\gamma}_{1}<\boldsymbol{\gamma}_{2}$ as in Subsection
\ref{subsec:Strong-admissibility-and} such that $\mu_{d}$ has degree
$d-2$. In particular, if $c$ is torsion, then the $A_{\infty}$-category
is homologically $\mathbb{Z}$-gradable.
\begin{lem}
\label{lem:If--is-1} If $\boldsymbol{A},\boldsymbol{B}\neq\emptyset$,
then fix $c\in H^{2}(Y;\mathbb{Z})$. Otherwise, formally let $c=0$.
Say ${\bf x}\in\boldsymbol{\gamma}_{1}\cap\boldsymbol{\gamma}_{2}$
is \emph{relevant} if the following hold.
\begin{itemize}
\item If $\boldsymbol{\gamma}_{1},\boldsymbol{\gamma}_{2}\in\boldsymbol{A}$
or $\boldsymbol{\gamma}_{1},\boldsymbol{\gamma}_{2}\in\boldsymbol{B}$,
then $c_{1}(\mathfrak{s}({\bf x}_{i}))=0$.
\item If $\boldsymbol{\gamma}_{1}\in\boldsymbol{A}$ and $\boldsymbol{\gamma}_{2}\in\boldsymbol{B}$,
then $c_{1}(\mathfrak{s}({\bf x}_{i}))=c$.
\end{itemize}
Let ${\cal D}\in D({\bf x}_{0},\cdots,{\bf x}_{d})$ be a domain such
that ${\bf x}_{0},\cdots,{\bf x}_{d-1}$ are relevant. Then, ${\bf x}_{d}$
is relevant.

If ${\cal D},{\cal D}'\in D({\bf x}_{0},\cdots,{\bf x}_{d})$ are
domains with the same vertices and if ${\bf x}_{0},\cdots,{\bf x}_{d}$
are relevant, then
\[
\mu({\cal D})-\mu({\cal D}')=\left\langle c,H({\cal D}-{\cal D}')\right\rangle +P({\cal D})-P({\cal D}').
\]

\end{lem}

\begin{proof}
By Remark \ref{rem:Note-that-different}, we can assume, without loss
of generality, that every basepoint pair is a link basepoint pair
for every attaching curve. The first part is \cite[Lemma 2.66]{nahm2025unorientedskeinexacttriangle},
and the second part is \cite[Lemma 2.65]{nahm2025unorientedskeinexacttriangle}.
\end{proof}
\begin{rem}
\label{rem:If-the-Heegaard}If the Heegaard diagram ${\cal H}$ is
the disjoint union of ${\cal H}_{1}$ and ${\cal H}_{2}$, then let
the coefficient ring $R_{\boldsymbol{\alpha},\boldsymbol{\beta}}$
for ${\cal H},{\cal H}_{1},{\cal H}_{2}$ be $R,R_{1},R_{2}$ respectively.
Then $R=R_{1}\otimes_{\mathbb{F}}R_{2}$, and 
\[
CF_{{\cal H}}^{-}(\boldsymbol{\alpha},\boldsymbol{\beta})\simeq CF_{{\cal H}_{1}}^{-}(\boldsymbol{\alpha},\boldsymbol{\beta})\otimes_{\mathbb{F}}CF_{{\cal H}_{2}}^{-}(\boldsymbol{\alpha},\boldsymbol{\beta})
\]
as chain complexes over $R$. Furthermore, if there are three attaching
curves, then the following commutes (let the rings $\mathbb{F}[\{U_{p}\}_{p\in\boldsymbol{p}}]$
for ${\cal H},{\cal H}_{1},{\cal H}_{2}$ be $S,S_{1},S_{2}$ respectively):
\[\begin{tikzcd}
	{\left(CF_{{\cal H}_{1}}^{-}(\boldsymbol{\alpha},\boldsymbol{\beta})\otimes_{\mathbb{F}}CF_{{\cal H}_{2}}^{-}(\boldsymbol{\alpha},\boldsymbol{\beta})\right)\otimes_{S}\left(CF_{{\cal H}_{1}}^{-}(\boldsymbol{\beta},\boldsymbol{\gamma})\otimes_{\mathbb{F}}CF_{{\cal H}_{2}}^{-}(\boldsymbol{\beta},\boldsymbol{\gamma})\right)} & {CF_{{\cal H}}^{-}(\boldsymbol{\alpha},\boldsymbol{\beta})\otimes_{S}CF_{{\cal H}}^{-}(\boldsymbol{\beta},\boldsymbol{\gamma})} \\
	{\left(CF_{{\cal H}_{1}}^{-}(\boldsymbol{\alpha},\boldsymbol{\beta})\otimes_{S_{1}}CF_{{\cal H}_{1}}^{-}(\boldsymbol{\beta},\boldsymbol{\gamma})\right)\otimes_{\mathbb{F}}\left(CF_{{\cal H}_{2}}^{-}(\boldsymbol{\alpha},\boldsymbol{\beta})\otimes_{S_{2}}CF_{{\cal H}_{2}}^{-}(\boldsymbol{\beta},\boldsymbol{\gamma})\right)} \\
	{CF_{{\cal H}_{1}}^{-}(\boldsymbol{\alpha},\boldsymbol{\gamma})\otimes_{\mathbb{F}}CF_{{\cal H}_{2}}^{-}(\boldsymbol{\alpha},\boldsymbol{\gamma})} & {CF_{{\cal H}}^-(\boldsymbol{\alpha},\boldsymbol{\gamma})}
	\arrow["\simeq", from=1-1, to=1-2]
	\arrow["\simeq", from=1-1, to=2-1]
	\arrow["{\mu _2 }", from=1-2, to=3-2]
	\arrow["{\mu_2}", from=2-1, to=3-1]
	\arrow["\simeq", from=3-1, to=3-2]
\end{tikzcd}\]
\end{rem}

\subsection{\label{subsec:An-Alexander--splitting}An Alexander $\mathbb{Z}/2$-splitting}

For simplicity, we do not discuss the decomposition of $\mu_{n}$
for $n\ge2$ as $\mu_{n}=\sum_{\mathfrak{s}}\mu_{n}^{\mathfrak{s}}$
where $\mathfrak{s}$ ranges over the ${\rm Spin}^{c}$-structures
of a suitable four-manifold. However, we consider a closely related\emph{
Alexander $\mathbb{Z}/2$-splitting} on the $A_{\infty}$-category,
which is sufficient for our purposes. Compare \cite[Subsubsection 2.9.2]{nahm2025unorientedskeinexacttriangle}.
\begin{defn}
A pair-pointed Heegaard diagram is \emph{Alexander $\mathbb{Z}/2$-splittable}
if $P({\cal D})$ is even for all cornerless two-chains ${\cal D}$.
\end{defn}

We have the following lemma for the case where the Heegaard diagram
has two attaching curves.
\begin{lem}
\label{lem:alexander-two}Let ${\cal H}=(\Sigma,\boldsymbol{\alpha},\boldsymbol{\beta},\boldsymbol{p},\boldsymbol{G})$
be a pair-pointed Heegaard diagram. Let ${\cal H}$ represent the
link $L$ in $Y$, and let $L_{1},\cdots,L_{n}$ be the components
of $L$. Then, the following are equivalent:
\begin{enumerate}
\item \label{enu:alex1}${\cal H}$ is Alexander $\mathbb{Z}/2$-splittable.
\item \label{enu:alex2}Let $\lambda_{i}\in H_{1}(Y;\mathbb{Z})$ be the
homology class of $L_{i}$. The intersection number $x\cdot\lambda$
is even for all $x\in H_{2}(Y;\mathbb{Z})$.
\item \label{enu:alex3}Let $\mu_{i}\in H_{1}(Y\backslash L;\mathbb{Z})$
be the homology class of the meridian of $L_{i}$. If $\sum n_{i}\mu_{i}=0$,
then $\sum n_{i}$ is even.
\end{enumerate}
Note that $\lambda_{i}$ and $\mu_{i}$ are well-defined up to sign,
and the conditions do not depend on the sign.
\end{lem}

\begin{proof}
Conditions (\ref{enu:alex1}) and (\ref{enu:alex2}) are equivalent
since $P({\cal D})=H({\cal D})\cdot\lambda$ modulo 2 for any cornerless
two-chain ${\cal D}$, and any $x\in H_{2}(Y;\mathbb{Z})$ is $H({\cal D})$
for some ${\cal D}$. It is easy to check that Condition (\ref{enu:alex2})
and (\ref{enu:alex3}) are equivalent.
\end{proof}
If a pair-pointed Heegaard diagram is Alexander $\mathbb{Z}/2$-splittable,
then we can assign to each intersection point ${\bf x}$ an (absolute)
\emph{Alexander $\mathbb{Z}/2$-grading }${\rm gr}_{A}^{\mathbb{Z}/2}({\bf x})\in\mathbb{Z}/2$,
such that if there is a domain ${\cal D}\in D({\bf x}_{0},\cdots,{\bf x}_{n})$,
then 
\[
\sum_{i=0}^{n}{\rm gr}_{A}^{\mathbb{Z}/2}({\bf x}_{i})=P({\cal D})\mod 2.
\]
Hence, we can equip the associated $A_{\infty}$-category with an
\emph{Alexander $\mathbb{Z}/2$-splitting}, in the following sense:
for each $\boldsymbol{\gamma}_{1}<\boldsymbol{\gamma}_{2}$, define
\[
{\rm gr}_{A}^{\mathbb{Z}/2}(\prod_{p}U_{p}^{\ell_{p}}{\bf x}):={\rm gr}_{A}^{\mathbb{Z}/2}({\bf x})+2\sum\ell_{p}.
\]
Let $CF_{i}^{-}(\boldsymbol{\gamma}_{1},\boldsymbol{\gamma}_{2})$
(\emph{``the Alexander $\mathbb{Z}/2$-grading $i$ summand''};
this notation is used only within this subsection) be the $S:=\mathbb{F}[\{U_{p}\}_{p\in\boldsymbol{p}}]$-submodule
of $CF^{-}(\boldsymbol{\gamma}_{1},\boldsymbol{\gamma}_{2})$ generated
by the $\prod U_{p}^{\ell_{p}}{\bf x}$'s with Alexander $\mathbb{Z}/2$-grading
$i$. Then, $CF^{-}(\boldsymbol{\gamma}_{1},\boldsymbol{\gamma}_{2})=CF_{0}^{-}(\boldsymbol{\gamma}_{1},\boldsymbol{\gamma}_{2})\oplus CF_{1}^{-}(\boldsymbol{\gamma}_{1},\boldsymbol{\gamma}_{2})$
as $S$-chain complexes, and $\mu_{d}$ has Alexander $\mathbb{Z}/2$-degree
$0$.

We claim that these Alexander $\mathbb{Z}/2$-grading summands $CF_{i}^{-}(\boldsymbol{\gamma}_{1},\boldsymbol{\gamma}_{2})$
are the direct sums of some ${\rm Spin}^{c}$-summands. Let $(\boldsymbol{\gamma}_{1},\boldsymbol{\gamma}_{2})$
represent the link $L$ in $Y$, and let $\mu_{1},\cdots,\mu_{n}\in H_{1}^{orb}(Y(L);\mathbb{Z})$
be the meridians of the components of $L$. For intersection points
${\bf x},{\bf y}\in\boldsymbol{\gamma}_{1}\cap\boldsymbol{\gamma}_{2}$,
if there is a domain ${\cal D}\in D({\bf x},{\bf y})$, then 
\begin{equation}
\mathfrak{s}({\bf x})=\mathfrak{s}({\bf y})+\sum_{i=1}^{n}n_{i}\mu_{i}\label{eq:spinc-sum}
\end{equation}
for some $n_{i}\in\mathbb{Z}$. The claim follows once we show that
for any such $n_{i}$ that satisfy Equation (\ref{eq:spinc-sum}),
$\sum n_{i}=P({\cal D})$ modulo $2$. This follows from the following
two statements: if $n_{i}:=n_{p}({\cal D})$ where $p$ is the basepoint
pair for $L_{i}$, then Equation (\ref{eq:spinc-sum}) holds by \cite[Lemma 4.7]{MR2253454}.
Also, the sum $\sum n_{i}$ modulo $2$ does not depend on the $n_{i}$'s
that satisfy Equation (\ref{eq:spinc-sum}), by Lemma \ref{lem:alexander-two}
(\ref{enu:alex3}).
\begin{rem}
If $\boldsymbol{\gamma}_{1},\boldsymbol{\gamma}_{2}\in\boldsymbol{A}$
or $\boldsymbol{\gamma}_{1},\boldsymbol{\gamma}_{2}\in\boldsymbol{B}$,
then if the sub Heegaard diagram with attaching curves $\boldsymbol{\gamma}_{1},\boldsymbol{\gamma}_{2}$
is Alexander $\mathbb{Z}/2$-splittable, then the relative Alexander
$\mathbb{Z}/2$-grading on $CF^{-}(\boldsymbol{\gamma}_{1},\boldsymbol{\gamma}_{2})$
is well-defined, as in Remarks \ref{rem:In-particular,-} and \ref{rem:Recall-Remark-.}.
We sometimes denote the relative Alexander $\mathbb{Z}/2$-grading
as ${\rm gr}_{A}^{\mathbb{Z}/2}$ as well.
\end{rem}

\subsection{Balled links and unoriented link Floer homology}

In \cite[Section 3]{nahm2025unorientedskeinexacttriangle}, we introduced
the notion of a \emph{balled link} that encodes a minimally pointed
link equipped with a suture datum inside a pointed three-manifold,
and behaves nicely with respect to band moves.

A \emph{balled link $L$ in a three-manifold $Y$} consists of:
\begin{itemize}
\item a (potentially empty) link $L^{link}$ in $Y$, referred to as the
\emph{underlying link} (also denoted $L$),
\item a finite collection of pairwise disjoint, embedded (closed) three-balls
in $Y$, called \emph{baseballs}, such that (1) each connected component
of $Y$ contains some baseball; (2) each link component of $L^{link}$
intersects some baseball; and (3) if a baseball $BB$ intersects $L^{link}$,
then there exists an open neighborhood $U\supset BB$ such that $(U,L^{link}\cap U,BB)$
is diffeomorphic to $(\mathbb{R}^{3},\mathbb{R}\times\{(0,0)\},\{{\bf x}:\left|{\bf x}\right|\le1\})$,
\item one of two \emph{types} (\emph{link} or \emph{free}) associated to
each baseball: for each link component of $L^{link}$, one baseball
intersecting the component is said to be a \emph{link baseball}, and
all baseballs not said to be link baseballs are said to be \emph{free
baseballs.}
\end{itemize}

To define the unoriented link Floer homology of a balled link $L\subset Y$,
we need to choose an \emph{auxiliary datum}. An \emph{auxiliary datum}
for $L\subset Y$ is a link $\widetilde{L}$ together with a choice
of \emph{link basepoints} $\boldsymbol{u}$ and \emph{free basepoints}
$\boldsymbol{v}$ as follows:
\begin{itemize}
\item $\widetilde{L}$ agrees with $L^{link}$ outside the union of the
baseballs.
\item The link $\widetilde{L}$ and the baseballs of $L$ form a balled
link.
\item For each link baseball $BB$, choose two distinct points (\emph{``link
basepoints''}) in $({\rm int}BB)\cap\widetilde{L}$. Let $\boldsymbol{u}$
be the set of these link basepoints.
\item For each free baseball $BB$, choose a point (\emph{``free basepoint''})
in $({\rm int}BB)\backslash\widetilde{L}$. Let $\boldsymbol{v}$
be the set of these free basepoints.
\end{itemize}
If $(\widetilde{L},\boldsymbol{u},\boldsymbol{v})$ is an auxiliary
datum for $L\subset Y$, then each link component of $\widetilde{L}$
has exactly two link basepoints (i.e. $(\widetilde{L},\boldsymbol{u})$
is minimally pointed), and they divide the component into two components,
exactly one of which is entirely contained in the interior of a link
baseball. Define the \emph{suture datum} $\alpha_{\widetilde{L}}$
on $(\widetilde{L},\boldsymbol{u})$ by choosing the components of
$\widetilde{L}\backslash\boldsymbol{u}$ that are contained in the
link baseballs.

The tuples $(Y,\widetilde{L},\boldsymbol{u},\boldsymbol{v})$ for
different choices of auxiliary datum $(\widetilde{L},\boldsymbol{u},\boldsymbol{v})$
are diffeomorphic, via a diffeomorphism supported inside the baseballs.
In \cite[Proposition 3.8]{nahm2025unorientedskeinexacttriangle},
we showed that any diffeomorphism of $(Y,\widetilde{L},\boldsymbol{u},\boldsymbol{v})$
that is supported inside the baseballs is isotopic to a diffeomorphism
that is obtained by moving the free basepoints $\boldsymbol{v}$ along
some meridional loops (inside the baseballs). Hence, the ${\rm Spin}^{c}(Y(\widetilde{L},\boldsymbol{u},\alpha_{\widetilde{L}}))$'s
for different choices of the auxiliary datum $(\widetilde{L},\boldsymbol{u},\boldsymbol{v})$
are canonically identified; call this the \emph{space of ${\rm Spin}^{c}$-structures
of the balled link $L\subset Y$}, and denote it as ${\rm Spin}^{c}(Y(L))$.

The \emph{unoriented link Floer chain complex of the balled link $L\subset Y$}
is 
\[
CFL'{}^{-}(Y,L;\mathfrak{s}):=CFL'{}^{-}((Y,\boldsymbol{v}),(\widetilde{L},\boldsymbol{u},\alpha_{\widetilde{L}});\mathfrak{s}),
\]
for $\mathfrak{s}\in{\rm Spin}^{c}(Y(L))$. By \cite[Proposition 3.8]{nahm2025unorientedskeinexacttriangle},
this definition is natural: two auxiliary data for $L\subset Y$ give
rise to two different chain complexes, but a diffeomorphism between
the tuples $(Y,\widetilde{L},\boldsymbol{u},\boldsymbol{v})$ induces
a chain homotopy equivalence between the chain complexes. We consider
such chain homotopy equivalences induced by diffeomorphisms that are
supported inside the baseballs; these chain homotopy equivalences
do not depend on the diffeomorphism up to chain homotopy since the
chain homotopy equivalences induced by moving the free basepoints
$\boldsymbol{v}$ along some meridional loops are chain homotopic
to the identity.
\begin{notation}
\label{nota:Recall-from-Notation}Recall from Notation \ref{nota:Instead-of-writing}
that the variables $U_{p}$ were indexed by basepoint pairs $p$.
For a baseball $BB$, let $U_{BB}:=U_{p}$ where $p$ corresponds
to the basepoint pair that is contained in $BB$. Also, if a baseball
has index $i$ (i.e. it is called $BB_{i}$), then let $U_{i}:=U_{BB_{i}}$.
\end{notation}

\begin{rem}
\label{rem:If--is}If $L$ is a balled link in $Y$ and $L'$ is obtained
from $L$ by adding a free baseball $BB$, then one can obtain a Heegaard
diagram of $L'\subset Y$ from a Heegaard diagram of $L\subset Y$
by a simple stabilization \cite[Subsection 6.1]{MR2443092} (called
free-stabilization in \cite[Section 6]{1512.01184}; also compare
Subsection \ref{subsec:Connected-sums}), and so $CFL'{}^{-}(Y,L';\mathfrak{s})$
is chain homotopy equivalent to the mapping cone 
\[
CFL'{}^{-}(Y,L;\mathfrak{s})[U_{BB}]\xrightarrow{U_{BB}-U_{BB'}}CFL'{}^{-}(Y,L;\mathfrak{s})[U_{BB}],
\]
where $BB'$ is any baseball of $L$ that is in the same connected
component of $Y$ as $BB$. Hence $HFL'{}^{-}(Y,L;\mathfrak{s})\simeq HFL'{}^{-}(Y,L';\mathfrak{s})$.
\end{rem}

To define the \emph{unreduced hat version}, choose a \emph{distinguished
baseball }$BB$, and define the \emph{unreduced hat version of the
unoriented link Floer chain complex }as 
\[
\widehat{CFL'}(Y,L,BB;\mathfrak{s}):=CFL'{}^{-}(Y,L;\mathfrak{s})/U_{BB}.
\]
If $BB,BB'$ belong to the same connected component of $Y$, then
multiplication by $U_{BB}$ and $U_{BB'}$ are homotopic, and so the
corresponding homology groups are isomorphic.

To define the \emph{reduced hat version}, choose a \emph{distinguished
link baseball }$BB^{link}\in L$, and define 
\[
\widetilde{CFL'}(Y,L,BB^{link};\mathfrak{s}):=CFL'{}^{-}(Y,L;\mathfrak{s})/U_{BB^{link}}^{1/2}.
\]
Instead of a distinguished link baseball, we can choose a \emph{distinguished
point }$q$ on $L$; in this case, let $BB^{link}$ be the link baseball
of the link component that $q$ is on, and define 
\[
\widetilde{CFL'}(Y,L,q;\mathfrak{s}):=\widetilde{CFL'}(Y,L,BB^{link};\mathfrak{s}).
\]
Different choices of $BB^{link}$ may result in different $\widetilde{HFL'}$.

\begin{rem}
\label{rem:For-knots-}For knots $K$ in $S^{3}$ (in fact, for any
three-manifold), the reduced hat version of unoriented link Floer
homology is the hat version of knot Floer homology:
\[
\widetilde{CFL'}(S^{3},K,BB^{link})=\widehat{CFK}(S^{3},K).
\]
\end{rem}

Corollary \ref{cor:link-inequality} follows from the following lemma.
\begin{lem}
\label{lem:rank-ineq}Let $L$ be an $\ell$-component balled link,
and let $BB^{link}$ be a link baseball. Then, 
\[
2^{\ell-1}\dim_{\mathbb{F}}\widetilde{HFL'}(S^{3},L,BB^{link})\ge\dim_{\mathbb{F}}\widehat{HFL}(S^{3},L).
\]
\end{lem}

\begin{proof}
Let the $U$-variables corresponding to the link basepoint pairs be
$U_{1},\cdots,U_{\ell}$, and let $U_{1}$ correspond to $BB^{link}$.
Define chain complexes $C_{k}$ for $k=1,\cdots,\ell$ as 
\[
C_{k}:=\widetilde{CFL'}(S^{3},L,BB^{link})/(U_{2}^{1/2},\cdots,U_{k}^{1/2}).
\]
Then, $\widetilde{CFL'}(Y,L,BB^{link})=C_{1}$ and $\widehat{CFL}(Y,L)=C_{\ell}$.
Furthermore, we have a short exact sequence of chain complexes 
\[
0\to C_{k}\xrightarrow{\cdot U_{k+1}^{1/2}}C_{k}\to C_{k+1}\to0
\]
for $k=1,\cdots,\ell-1$, and so $2\dim_{\mathbb{F}}H(C_{k})\ge\dim_{\mathbb{F}}H(C_{k+1})$.
\end{proof}
The following is an analogue of the basepoint actions in Khovanov
homology (see Subsections \ref{subsec:The-action} and \ref{subsec:The-Khovanov-chain})
for unoriented link Floer homology. Note that in Heegaard Floer homology,
the $\Phi$-actions (Subsection \ref{subsec:The--action}) are referred
to as basepoint actions (as in \cite[Subsection 4.2]{MR3905679});
however, we do not adopt this convention in this paper, in order to
match the conventions used in Khovanov homology.
\begin{defn}
\label{def:hflaction}For a point $q\in L$, define the \emph{action
of $q$ on $HFL'{}^{-}$, $\widehat{HFL'}$, $\widetilde{HFL'}$}
as multiplication by $U_{BB}^{1/2}$, if $BB$ is the link baseball
on the connected component of $L$ that $q$ is on.
\end{defn}

\section{\label{sec:More-Heegaard-Floer}More Heegaard Floer preliminaries}

\subsection{\label{subsec:The--action}The \texorpdfstring{$\Phi$}{Phi} action}

Let $q=(z,w)$ be a basepoint pair, and let $R=\mathbb{F}[\{U_{p}^{k_{p}}\}_{p\in\boldsymbol{p}}]$
for $k_{p}\in\{1,1/2\}$ be any ring such that $k_{q}=1/2$, and whenever
$p$ is a link basepoint pair for $\boldsymbol{\gamma}_{1}$ or $\boldsymbol{\gamma}_{2}$,
we have $k_{p}=1/2$. Define $R$-linear maps $\Phi_{R,v}:CF_{R}^{-}(\boldsymbol{\gamma}_{1},\boldsymbol{\gamma}_{2})\to CF_{R}^{-}(\boldsymbol{\gamma}_{1},\boldsymbol{\gamma}_{2})$
for $v\in\{z,w\}$ as follows (they are denoted as $\Phi$ and $\Psi$
in \cite{MR3426686,MR3709653}):
\[
\Phi_{R,v}{\bf x}=U_{q}^{-1/2}\sum_{{\bf y}\in\mathbb{T}_{\boldsymbol{\gamma}_{1}}\cap\mathbb{T}_{\boldsymbol{\gamma}_{2}}}\sum_{{\cal D}\in D({\bf x},{\bf y}),\ \mu({\cal D})=1}n_{v}({\cal D})\#\mathcal{M}({\cal D})\prod_{p\in\boldsymbol{p}}U_{p}^{\frac{1}{2}n_{p}({\cal D})}{\bf y}.
\]
These maps are chain maps and have homological degree $0$.

Furthermore, $\Phi_{R,w}+\Phi_{R,z}$ is given by differentiating
the coefficients of the differential $\partial$ considered as a matrix
with entries in $R$, with respect to $U^{1/2}$, as in \cite[Remark 14.12]{1512.01184}.
Hence, the discussion right before \cite[Corollary 14.19]{1512.01184}
shows that $\Phi_{R,w}=\Phi_{R,z}$ on homology $HF_{R}^{-}(\boldsymbol{\gamma}_{1},\boldsymbol{\gamma}_{2})$.
We simply denote these maps as 
\[
\Phi_{R,q}:=\Phi_{R,w}=\Phi_{R,z}:HF_{R}^{-}(\boldsymbol{\gamma}_{1},\boldsymbol{\gamma}_{2})\to HF_{R}^{-}(\boldsymbol{\gamma}_{1},\boldsymbol{\gamma}_{2}).
\]
Also, similarly to Notation \ref{nota:Instead-of-writing}, if the
basepoint pair $q$ has index $i$, i.e. is called $q_{i}$, then
we write $\Phi_{R,i}:=\Phi_{R,q_{i}}$.

If $q$ is a free basepoint pair for both $\boldsymbol{\gamma}_{1}$
and $\boldsymbol{\gamma}_{2}$, then $\Phi_{R,q}=0$ on $HF_{R}^{-}(\boldsymbol{\gamma}_{1},\boldsymbol{\gamma}_{2})$,
since $\Phi_{R,q}$ is $U_{q}^{1/2}$ times the degree $1$ $\Phi$
action (defined in \cite[Subsection 14.3]{1512.01184}) given by a
free basepoint, and this degree $1$ $\Phi$ action vanishes by the
argument right before \cite[Corollary 14.19]{1512.01184}.

If $q$ is a link basepoint pair for $\boldsymbol{\gamma}_{1}$ or
$\boldsymbol{\gamma}_{2}$, then define $\Phi_{q}:=\Phi_{R_{\boldsymbol{\gamma}_{1},\boldsymbol{\gamma}_{2}},q}:HF^{-}(\boldsymbol{\gamma}_{1},\boldsymbol{\gamma}_{2})\to HF^{-}(\boldsymbol{\gamma}_{1},\boldsymbol{\gamma}_{2})$\footnote{\label{fn:Similarly-to-Notation}Similarly to Notation \ref{nota:Instead-of-writing}
and the above, if the basepoint pair $q$ has index $i$, then we
write $\Phi_{i}:=\Phi_{q_{i}}$, and if $v$ is a basepoint in the
basepoint pair $q$, then $\Phi_{v}:=\Phi_{q}$.}. If $q$ is a free basepoint pair for both $\boldsymbol{\gamma}_{1}$
and $\boldsymbol{\gamma}_{2}$, then define $\Phi_{q}:=0$. Consider
$R$ as a free $R_{\boldsymbol{\gamma}_{1},\boldsymbol{\gamma}_{2}}$-module,
with basis 
\[
B:=\left\{ \prod_{p\in F_{\boldsymbol{\gamma}_{1},\boldsymbol{\gamma}_{2}},\ k_{p}=1/2}U_{p}^{\ell_{p}}:\ell_{p}\in\{0,1/2\}\right\} ,
\]
where $F_{\boldsymbol{\gamma}_{1},\boldsymbol{\gamma}_{2}}$ is the
set of basepoint pairs that is a free basepoint pair for $\boldsymbol{\gamma}_{1}$
and $\boldsymbol{\gamma}_{2}$. Consider the induced splitting 
\[
HF_{R}^{-}(\boldsymbol{\gamma}_{1},\boldsymbol{\gamma}_{2})=\bigoplus_{b\in B}bHF^{-}(\boldsymbol{\gamma}_{1},\boldsymbol{\gamma}_{2}).
\]
The $\Phi_{R,q}$ action acts on each summand $bHF^{-}(\boldsymbol{\gamma}_{1},\boldsymbol{\gamma}_{2})$,
and agrees with $\Phi_{q}$: if $q$ is a link basepoint pair for
$\boldsymbol{\gamma}_{1}$ or $\boldsymbol{\gamma}_{2}$, then this
is immediate. If $q$ is a free basepoint pair for both $\boldsymbol{\gamma}_{1}$
and $\boldsymbol{\gamma}_{2}$, then this follows since $\Phi_{R,q}=\Phi_{q}=0$.

By \cite[Lemma 3.2]{MR3709653}, the $\Phi$ actions are natural (i.e.
$\Phi_{R,v}$ commutes with the change of diagram maps up to chain
homotopy). Hence, we introduce the following definition for balled
links.
\begin{defn}
\label{def:phi-balled}Let $L$ be a balled link in $Y$, and let
$c\in H^{2}(Y;\mathbb{Z})$. If $BB$ is a baseball, then define $\Phi_{BB}:HFL'{}^{-}(Y,L;c)\to HFL'{}^{-}(Y,L;c)$
as the $\Phi$ action of the basepoint pair that is contained in $BB$.
If $L_{C}$ is a connected component of $L$, then the\emph{ $\Phi$
action of $L_{C}$} is defined as $\Phi_{BB^{link}}:HFL'{}^{-}(Y,L;c)\to HFL'{}^{-}(Y,L;c)$
where $BB^{link}$ is the link baseball on $L_{C}$.
\end{defn}

\begin{rem}
By the above discussion, if $BB$ is a free baseball, then $\Phi_{BB}=0$.
\end{rem}

Now, consider three attaching curves $\boldsymbol{\gamma}_{1},\boldsymbol{\gamma}_{2},\boldsymbol{\gamma}_{3}$,
and further assume that whenever $p$ is a link basepoint pair for
at least one $\boldsymbol{\gamma}_{i}$, then $k_{p}=1/2$. Define
the map $\Phi_{R,q}$ on $HF_{R}^{-}(\boldsymbol{\gamma}_{1},\boldsymbol{\gamma}_{2})\otimes HF_{R}^{-}(\boldsymbol{\gamma}_{2},\boldsymbol{\gamma}_{3})$
as $\Phi_{R,q}\otimes{\rm Id}+{\rm Id}\otimes\Phi_{R,q}$. Then, $\mu_{2,R}:HF_{R}^{-}(\boldsymbol{\gamma}_{1},\boldsymbol{\gamma}_{2})\otimes HF_{R}^{-}(\boldsymbol{\gamma}_{2},\boldsymbol{\gamma}_{3})\to HF_{R}^{-}(\boldsymbol{\gamma}_{1},\boldsymbol{\gamma}_{3})$
is $\Phi_{R,q}$-equivariant, by \cite[Equation (37)]{MR3905679}.

Similarly, define $\Phi_{q}$ on $HF^{-}(\boldsymbol{\gamma}_{1},\boldsymbol{\gamma}_{2})\otimes HF^{-}(\boldsymbol{\gamma}_{2},\boldsymbol{\gamma}_{3})$
as $\Phi_{q}\otimes{\rm Id}+{\rm Id}\otimes\Phi_{q}$. By the above,
\[
\mu_{2}:HF^{-}(\boldsymbol{\gamma}_{1},\boldsymbol{\gamma}_{2})\otimes HF^{-}(\boldsymbol{\gamma}_{2},\boldsymbol{\gamma}_{3})\to HF^{-}(\boldsymbol{\gamma}_{1},\boldsymbol{\gamma}_{3})
\]
is $\Phi_{q}$-equivariant.

We record the following lemma.
\begin{lem}[{\cite[Lemmas 9.1, 9.2, and 9.7]{MR3709653}}]
Let $p,p'$ be basepoint pairs. Then on homology,
\[
\Phi_{R,p}^{2}=0,\ \Phi_{R,p}\Phi_{R,p'}=\Phi_{R,p'}\Phi_{R,p}.
\]
Hence, $\Phi_{p}^{2}=0$ and $\Phi_{p}\Phi_{p'}=\Phi_{p'}\Phi_{p}$
as well.
\end{lem}

\subsection{\label{subsec:Relative-homology-actions}Relative homology actions}

We will consider relative homology actions as in \cite[Section 5]{1512.01184}
and \cite[Section 2.3]{MR4845975} (compare \cite[Section 4.2.5]{MR2113019},
\cite{MR3294567}). Let $R=\mathbb{F}[\{U_{p}^{k_{p}}\}_{p\in\boldsymbol{p}}]$
for $k_{p}\in\{1,1/2\}$ be any ring such that whenever $p$ is a
link basepoint pair for $\boldsymbol{\gamma}_{1}$ or $\boldsymbol{\gamma}_{2}$,
we have $k_{p}=1/2$.
\begin{defn}
\label{def:unoriented-collapse}View $\mathbb{F}[U^{1/2}]$ as an
$R$-module by letting $U_{p}^{k_{p}}$ act on $\mathbb{F}[U^{1/2}]$
as $U^{k_{p}}$. Let 
\[
CF_{\mathbb{F}[U^{1/2}]}^{-}(\boldsymbol{\gamma}_{1},\boldsymbol{\gamma}_{2}):=CF_{R}^{-}(\boldsymbol{\gamma}_{1},\boldsymbol{\gamma}_{2})\otimes_{R}\mathbb{F}[U^{1/2}].
\]
Denote the base change map as
\[
\rho:HF_{R}^{-}(\boldsymbol{\gamma}_{1},\boldsymbol{\gamma}_{2})\to HF_{\mathbb{F}[U^{1/2}]}^{-}(\boldsymbol{\gamma}_{1},\boldsymbol{\gamma}_{2}).
\]
\end{defn}

Let $\lambda$ be an immersed loop on the Heegaard surface $\Sigma$,
or an immersed path between two basepoints $u_{1}$ and $u_{2}$ on
$\Sigma$. We always assume that $\lambda$ intersects the attaching
curves transversely and is disjoint from all the intersection points.
Define $A_{R,\lambda,-1}:CF_{R}^{-}(\boldsymbol{\gamma}_{1},\boldsymbol{\gamma}_{2})\to CF_{R}^{-}(\boldsymbol{\gamma}_{1},\boldsymbol{\gamma}_{2})$
as the $R$-linear map such that
\begin{gather*}
A_{R,\lambda,-1}{\bf x}=\sum_{{\bf y}\in\mathbb{T}_{\boldsymbol{\gamma}_{1}}\cap\mathbb{T}_{\boldsymbol{\gamma}_{2}}}\sum_{{\cal D}\in D({\bf x},{\bf y}),\ \mu({\cal D})=1}\#(\lambda\cap\partial_{\boldsymbol{\gamma}_{1}}{\cal D})\#\mathcal{M}({\cal D})\prod_{p\in\boldsymbol{p}}U_{p}^{\frac{1}{2}n_{p}({\cal D})}{\bf y},
\end{gather*}
and define $B_{R,\lambda,-1}$ similarly, by replacing $\partial_{\boldsymbol{\gamma}_{1}}{\cal D}$
with $\partial_{\boldsymbol{\gamma}_{2}}{\cal D}$.
\begin{rem}
The map $A_{R,\lambda,-1}$ is denoted as $A_{\lambda}$ in \cite[Section 5]{1512.01184}.
We write $A_{R,\lambda,-1}$ or $A_{\lambda,-1}$ to emphasize that
it has homological degree $-1$, as we later consider closely related
maps that have homological degree $0$ (see Definition \ref{def:relative-hom}).
\end{rem}

By the same argument as \cite[Lemma 5.1]{1512.01184} (compare \cite[Equation (2.8)]{MR4845975}),
if $\lambda$ is a loop, then 
\[
A_{R,\lambda,-1}\partial+\partial A_{R,\lambda,-1}=B_{R,\lambda,-1}\partial+\partial B_{R,\lambda,-1}=0;
\]
if $\lambda$ is a path between $u_{1}$ and $u_{2}$, then 
\begin{gather*}
A_{R,\lambda,-1}\partial+\partial A_{R,\lambda,-1}=B_{R,\lambda,-1}\partial+\partial B_{R,\lambda,-1}=U_{u_{1}}+U_{u_{2}}.
\end{gather*}
Hence, $A_{R,\lambda,-1}$ and $B_{R,\lambda,-1}$ induce homological
degree $-1$ chain maps
\[
A_{\mathbb{F}[U^{1/2}],\lambda,-1},B_{\mathbb{F}[U^{1/2}],\lambda,-1}:CF_{\mathbb{F}[U^{1/2}]}^{-}(\boldsymbol{\gamma}_{1},\boldsymbol{\gamma}_{2})\to CF_{\mathbb{F}[U^{1/2}]}^{-}(\boldsymbol{\gamma}_{1},\boldsymbol{\gamma}_{2}),
\]
and hence a homological degree $-1$ map on $HF_{\mathbb{F}[U^{1/2}]}^{-}$.

We record the following lemmas.

\begin{lem}[{\cite[Lemma 5.3]{1512.01184}}]
\label{lem:ab-additive}The maps $A_{R,\lambda,-1}$ and $B_{R,\lambda,-1}$
are additive in $\lambda$:
\begin{itemize}
\item If $\lambda_{1}$ is a path from $u_{1}$ to $u_{2}$ and $\lambda_{2}$
is a path from $u_{2}$ to $u_{3}$, let $\lambda_{2}\ast\lambda_{1}$
be the concatenation.
\item If $\lambda_{1}$ is a loop or a path and $\lambda_{2}$ is a loop,
then let $\lambda_{2}\ast\lambda_{1}$ be the loop or path obtained
by splicing $\lambda_{2}$ into $\lambda_{1}$.
\end{itemize}
Then, 
\[
A_{R,\lambda_{2}\ast\lambda_{1},-1}=A_{R,\lambda_{2},-1}+A_{R,\lambda_{1},-1},\ B_{R,\lambda_{2}\ast\lambda_{1},-1}=B_{R,\lambda_{2},-1}+B_{R,\lambda_{1},-1}.
\]

\end{lem}

\begin{lem}[{\cite[Lemma 5.7]{1512.01184}}]
\label{lem:a+b}If $\lambda$ is a loop, then $A_{R,\lambda,-1}=B_{\lambda,-1}$.
If $\lambda$ is a path between $u_{1}$ and $u_{2}$, then $A_{R,\lambda,-1}+B_{R,\lambda,-1}=U_{u_{1}}^{1/2}\Phi_{R,u_{1}}+U_{u_{2}}^{1/2}\Phi_{R,u_{2}}$.
\end{lem}

\begin{lem}[{\cite[Lemma 5.2]{1512.01184}}]
\label{lem:mu2-ab}We have 
\[
\mu_{2,R}(A_{R,\lambda,-1}\otimes{\rm Id})\simeq A_{R,\lambda,-1}\mu_{2,R}({\rm Id}\otimes{\rm Id}),\ \mu_{2,R}({\rm Id}\otimes B_{R,\lambda,-1})\simeq B_{R,\lambda,-1}\mu_{2,R}({\rm Id}\otimes{\rm Id}).
\]
\end{lem}

By the same argument as \cite[Section 5.3]{1512.01184}, one can show
that the maps $A_{R,\lambda,-1}$ and $B_{R,\lambda,-1}$ only depend
on the homology class $[\lambda]\in H_{1}(Y;\mathbb{Z})$ or $[\lambda]\in H_{1}(Y,\{u_{1},u_{2}\};\mathbb{Z})$
up to chain homotopy (note that any such homology class can be represented
by a loop or a path on $\Sigma$), and are natural (i.e. commute with
the change of diagram maps up to chain homotopy). In the context of
naturality, we assume that if $u_{i}$ belongs to a basepoint pair
$p$ that is a free basepoint pair for both $\boldsymbol{\gamma}_{1}$
and $\boldsymbol{\gamma}_{2}$, then $u_{i}$ is the first basepoint
in the basepoint pair $p$ (note that this does not lose any generality
since if $\kappa$ is the forbidding arc for $p$, then $A_{R,\kappa,-1},B_{R,\kappa,-1}=0$).
Hence, we in particular get maps $A_{\lambda,-1}$ and $B_{\lambda,-1}$
on $CFL'{}^{-}((Y,\boldsymbol{v}),(L,\boldsymbol{u},\alpha_{L});c)$,
and $A_{\mathbb{F}[U^{1/2}],\lambda,-1}$ and $B_{\mathbb{F}[U^{1/2}],\lambda,-1}$
on $CFL'{}_{\mathbb{F}[U^{1/2}]}^{-}((Y,\boldsymbol{v}),(L,\boldsymbol{u},\alpha_{L});c)$
that are well-defined up to chain homotopy.

\subsection{\label{subsec:Connected-sums}Stabilizations}

\begin{figure}[h]
\begin{centering}
\includegraphics[scale=2]{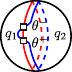}
\par\end{centering}
\caption{\label{fig:stabilization}Connected summing region}
\end{figure}

In this subsection, we introduce an operation called \emph{stabilization}.
Stabilization does not change the top homological $\mathbb{Z}$-grading
part (if well-defined) of the homology group. We follow \cite[Section 8]{1512.01184}.
Note that our notions are slightly different from \cite[Subsection 2.6]{nahm2025unorientedskeinexacttriangle}.

Consider a Heegaard diagram ${\cal H}$ with Heegaard surface $\Sigma$.
Given two points $q_{1},q_{2}\in\Sigma$ that are disjoint from the
attaching curves and forbidding arcs, define the Heegaard diagram
${\cal H}^{zs}$ \emph{obtained from ${\cal H}$ by performing a $0$-surgery
(along $q_{1}$ and $q_{2}$)} as follows. The underlying surface
is $\Sigma^{zs}=\Sigma\#_{q_{1},q_{2}}S^{2}$ where the $S^{2}$ is
as in Figure \ref{fig:stabilization}, and we perform a ``generalized
connected sum'' of $S^{2}$ and $\Sigma$ along $q_{1}\in\Sigma,S^{2}$
and $q_{2}\in\Sigma,S^{2}$. The set of basepoint pairs and forbidding
arcs of ${\cal H}^{zs}$ are the same as ${\cal H}$. For each attaching
curve $\boldsymbol{\alpha}$ on $\Sigma$, the corresponding attaching
curve on $\Sigma^{zs}$ is $\boldsymbol{\alpha}$ together with a
circle that separates $q_{1}$ and $q_{2}$ in $S^{2}$, as in Figure
\ref{fig:stabilization} (for different attaching curves, let these
corresponding circles be standard translates of each other). We denote
this new attaching curve as $\boldsymbol{\alpha}^{zs}$.

Let $(Y,L)$ (resp. $(Y^{zs},L)$) be the three-manifold and link
that $\boldsymbol{\alpha},\boldsymbol{\beta}$ (resp. $\boldsymbol{\alpha}^{zs},\boldsymbol{\beta}^{zs}$)
represent. Then, $Y^{zs}$ is obtained from $Y$ by performing a three-dimensional
$0$-surgery along $q_{1},q_{2}$. Call the boundary of the cocore
of the corresponding $1$-handle of the trace of the $0$-surgery
the \emph{cocore of the $0$-surgery}. This is a two-sphere. If $\mathfrak{s}\in{\rm Spin}^{c}(Y(L))$,
then let $\mathfrak{s}^{zs}\in{\rm Spin}^{c}(Y^{zs}(L))$ be the unique
${\rm Spin}^{c}$-structure that restricts to $\mathfrak{s}$ outside
the $0$-surgery region and evaluates trivially on the cocore of the
$0$-surgery.

Assume that ${\cal H}$ is $c_{1}(\mathfrak{s})$-admissible. Then,
${\cal H}^{zs}$ is $c_{1}(\mathfrak{s}^{zs})$-admissible, and we
understand the chain complex $CF^{-}(\boldsymbol{\alpha}^{zs},\boldsymbol{\beta}^{zs};\mathfrak{s}^{zs})$
(\cite[Proposition 8.5]{1512.01184}). First, the intersection points
$\boldsymbol{\alpha}^{zs}\cap\boldsymbol{\beta}^{zs}=(\boldsymbol{\alpha}\cap\boldsymbol{\beta})\times\{\theta^{+},\theta^{-}\}$,
and for ${\bf x}\in\boldsymbol{\alpha}\cap\boldsymbol{\beta}$, $\mathfrak{s}(({\bf x},\theta^{+}))=\mathfrak{s}(({\bf x},\theta^{-}))=\mathfrak{s}({\bf x})^{zs}$.
Hence the underlying module of $CF^{-}(\boldsymbol{\alpha}^{zs},\boldsymbol{\beta}^{zs};\mathfrak{s}^{zs})$
is 
\[
CF^{-}(\boldsymbol{\alpha},\boldsymbol{\beta};\mathfrak{s})\{\theta^{+}\}\oplus CF^{-}(\boldsymbol{\alpha},\boldsymbol{\beta};\mathfrak{s})\{\theta^{-}\}.
\]
If the connected summing regions are sufficiently stretched, then
the differential is upper triangular, i.e. 
\[
\partial_{{\cal H}^{zs}}=\begin{pmatrix}\partial_{{\cal H}} & F\\
0 & \partial_{{\cal H}}
\end{pmatrix}
\]
for some $F:CF^{-}(\boldsymbol{\alpha},\boldsymbol{\beta};\mathfrak{s})\{\theta^{-}\}\to CF^{-}(\boldsymbol{\alpha},\boldsymbol{\beta};\mathfrak{s})\{\theta^{+}\}$.
\begin{rem}
\label{rem:We-also-understand}By the proof of \cite[Proposition 8.5]{1512.01184},
we in fact understand the domains that contribute to the $\theta^{+}\to\theta^{+}$,
$\theta^{+}\to\theta^{-}$, $\theta^{-}\to\theta^{-}$ parts of the
differential. For $\theta^{+}\to\theta^{-}$, the domains that contribute
are the two bigons in $D(({\bf x},\theta^{+}),({\bf x},\theta^{-}))$
supported in the $S^{2}$ that we ``generalized connected summed''
with. They both have an odd number of holomorphic representatives.

For $\theta^{+}\to\theta^{+}$, for each Maslov index $1$ domain
${\cal D}\in D({\bf x},{\bf y})$ of ${\cal H}$, the sum $\sum_{{\cal D}_{0}}\#{\cal M}({\cal D}\#{\cal D}_{0})$
is equal to $\#{\cal M}({\cal D})$, where ${\cal D}_{0}$ ranges
over the domains ${\cal D}_{0}\in D_{S^{2}}(\theta^{+},\theta^{+})$
such that $\mu({\cal D}\#{\cal D}_{0})=1$. Furthermore, these are
the only domains that have holomorphic representatives. An analogous
statement holds for $\theta^{-}\to\theta^{-}$.
\end{rem}

Furthermore, if $p_{1}$ (resp. $p_{2}$) is the basepoint pair in
the connected component of $\Sigma\backslash\boldsymbol{\alpha}$
that $q_{1}$ (resp. $q_{2}$) is in, then the composition 
\[
CF^{-}(\boldsymbol{\alpha},\boldsymbol{\beta};\mathfrak{s})\{\theta^{-}\}\xrightarrow{F}CF_{{\cal H}}^{-}(\boldsymbol{\alpha},\boldsymbol{\beta};\mathfrak{s})\{\theta^{+}\}\xrightarrow{\iota}CF_{{\cal H}}^{-}(\boldsymbol{\alpha},\boldsymbol{\beta};\mathfrak{s})\{\theta^{-}\}
\]
is chain homotopic to multiplication by $U_{p_{1}}+U_{p_{2}}$, where
$\iota$ is given by $({\bf x},\theta^{-})\mapsto({\bf x},\theta^{+})$
for ${\bf x}\in\boldsymbol{\alpha}\cap\boldsymbol{\beta}$\footnote{To prove this, consider a path $\lambda$ between a basepoint in $p_{1}$
and a basepoint in $p_{2}$. which goes through the connected summing
region, and consider $A_{\lambda,-1}$. The $\theta^{+}\to\theta^{-}$
component of this is $\iota$. If $H$ is the $\theta^{-}\to\theta^{-}$
component of $A_{\lambda,-1}$, then the $\theta^{-}\to\theta^{-}$
component of $\partial A_{\lambda,-1}+A_{\lambda,-1}\partial=U_{p_{1}}+U_{p_{2}}$
reads $\iota F+U_{p_{1}}+U_{p_{2}}=\partial H+H\partial$.}.

Define the\emph{ stabilization map} (the $1$-handle map in \cite[Section 8]{1512.01184})
\[
S^{+}:CF^{-}(\boldsymbol{\alpha},\boldsymbol{\beta};\mathfrak{s})\to CF^{-}(\boldsymbol{\alpha}^{zs},\boldsymbol{\beta}^{zs};\mathfrak{s}^{zs})
\]
as the $R_{\boldsymbol{\alpha},\boldsymbol{\beta}}$-linear map given
by mapping ${\bf x}\in\boldsymbol{\alpha}\cap\boldsymbol{\beta}$
to $({\bf x},\theta^{+})$. By the above, this is a chain map, and
so we get a map on homology 
\begin{equation}
S^{+}:HF^{-}(\boldsymbol{\alpha},\boldsymbol{\beta};\mathfrak{s})\to HF^{-}(\boldsymbol{\alpha}^{zs},\boldsymbol{\beta}^{zs};\mathfrak{s}^{zs}).\label{eq:kunneth-1}
\end{equation}
This map Equation (\ref{eq:kunneth-1}) is equivariant with respect
the $\Phi$-actions \cite[Lemma 8.3]{MR3905679} and (relative) homology
actions \cite[Lemma 8.11]{1512.01184} (these follow from Remark \ref{rem:We-also-understand}).

For simplicity of exposition, let us assume that all the attaching
curves in the remainder of this subsection are handlebody-equivalent
(Definition \ref{def:Two-attaching-curves}) and let us only consider
the $c_{1}=0$ summands. Then, the map 
\[
S^{+}:HF^{-}(\boldsymbol{\alpha},\boldsymbol{\beta};0)\to HF^{-}(\boldsymbol{\alpha}^{zs},\boldsymbol{\beta}^{zs};0)
\]
is an isomorphism on the top homological $\mathbb{Z}$-grading parts.
If there are three attaching curves, then $\mu_{2}$ and $S^{+}$
commute on homology by \cite[Theorem 8.8]{1512.01184}, i.e. the following
square commutes.
\[\begin{tikzcd}
	{HF^{-}(\boldsymbol{\alpha},\boldsymbol{\beta};0) \otimes HF^{-}(\boldsymbol{\beta},\boldsymbol{\gamma};0)} & {HF^{-}(\boldsymbol{\alpha},\boldsymbol{\gamma};0)} \\
	{HF^{-}(\boldsymbol{\alpha}^{zs},\boldsymbol{\beta}^{zs};0) \otimes HF^{-}(\boldsymbol{\beta}^{zs},\boldsymbol{\gamma}^{zs};0)} & {HF^{-}(\boldsymbol{\alpha}^{zs},\boldsymbol{\gamma}^{zs};0)}
	\arrow["{\mu _2 }", from=1-1, to=1-2]
	\arrow["{S^+ \otimes S^+}", from=1-1, to=2-1]
	\arrow["{S^+}", from=1-2, to=2-2]
	\arrow["{\mu _2 }", from=2-1, to=2-2]
\end{tikzcd}\]
\begin{defn}
\label{def:A-pair-pointed-Heegaard-1}A pair-pointed Heegaard diagram
${\cal H}^{stab}$ is a\emph{ stabilization of ${\cal H}$} if it
is obtained from ${\cal H}^{disj}$ by performing some number of $0$-surgeries,
where ${\cal H}^{disj}$ is as follows: ${\cal H}^{disj}$ is a disjoint
union of ${\cal H}$ with some number of $S^{2}$'s with exactly one
basepoint pair on each $S^{2}$, where for each such basepoint pair,
it is either a free basepoint pair for every attaching curve or a
link basepoint pair for every attaching curve.

In particular, if we consider two attaching curves $\boldsymbol{\alpha},\boldsymbol{\beta}$
and if $R_{\boldsymbol{\alpha},\boldsymbol{\beta}}^{disj}$ is the
corresponding coefficient ring, then $CF_{{\cal H}^{disj}}^{-}(\boldsymbol{\alpha},\boldsymbol{\beta};0)=CF_{{\cal H}}^{-}(\boldsymbol{\alpha},\boldsymbol{\beta};0)\otimes_{R_{\boldsymbol{\alpha},\boldsymbol{\beta}}}R_{\boldsymbol{\alpha},\boldsymbol{\beta}}^{disj}.$
Define $S^{+}:CF_{{\cal H}}^{-}(\boldsymbol{\alpha},\boldsymbol{\beta};0)\to CF_{{\cal H}^{disj}}^{-}(\boldsymbol{\alpha},\boldsymbol{\beta};0)$
as $x\mapsto x\otimes1$.

Define the \emph{stabilization map} $S^{+}:CF_{{\cal H}}^{-}(\boldsymbol{\alpha},\boldsymbol{\beta};0)\to CF_{{\cal H}^{stab}}^{-}(\boldsymbol{\alpha},\boldsymbol{\beta};0)$
as the composition 
\[
CF_{{\cal H}}^{-}(\boldsymbol{\alpha},\boldsymbol{\beta};0)\xrightarrow{S^{+}}CF_{{\cal H}^{disj}}^{-}(\boldsymbol{\alpha},\boldsymbol{\beta};0)\xrightarrow{S^{+}}CF_{{\cal H}^{stab}}^{-}(\boldsymbol{\alpha},\boldsymbol{\beta};0).
\]
\end{defn}

By the above discussion, $S^{+}$ is homogeneous with respect to the
${\rm Spin}^{c}$-splitting (hence also \emph{the} relative Alexander
$\mathbb{Z}/2$-splitting), is equivariant with respect the $\Phi$-actions
and (relative) homology actions, and $\mu_{2}$ and $S^{+}$ commute
on homology. Furthermore, $S^{+}$, on homology, is an isomorphism
on the top homological $\mathbb{Z}$-grading.

\subsection{\label{subsec:Model-computations}Model computations}

\begin{figure}[h]
\begin{centering}
\includegraphics{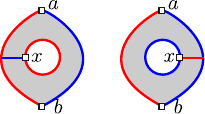}
\par\end{centering}
\caption{\label{fig:annulus}These annuli, viewed as domains in $D(ax,bx)$,
always have an odd number of holomorphic representatives.}
\end{figure}

\begin{figure}[h]
\begin{centering}
\includegraphics{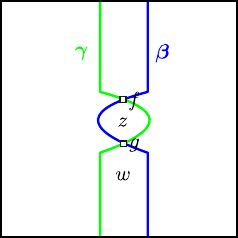}\qquad{}\includegraphics{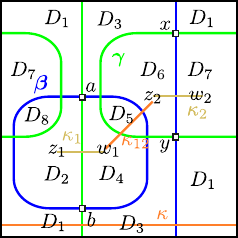}
\par\end{centering}
\caption{\label{fig:local-ori}Left: a local Heegaard diagram for the non-orientable
band map. Right: a local Heegaard diagram for the orientable band
maps. The basepoint pairs $(z,w)$ and $(z_{i},w_{i})$ are link basepoint
pairs for all the attaching curves, and the forbidding arcs are omitted.}
\end{figure}

\begin{lem}
\label{lem:ori-diagram}Consider the Heegaard diagrams of Figure \ref{fig:local-ori}.
For the left hand side, work over $\mathbb{F}[W,Z]$ and assign weights
$W,Z$ to $w,z$, respectively. For the right hand side, work over
$\mathbb{F}[W_{1},Z_{1},W_{2},Z_{2}]$ and assign weights $W_{1},Z_{1},W_{2},Z_{2}$
to $w_{1},z_{1},w_{2},z_{2}$, respectively. Then ${\cal CF}^{-}(\boldsymbol{\beta},\boldsymbol{\gamma})$
is the following (respectively), regardless of the almost complex
structure.
\[\begin{tikzcd}[sep=huge]
	f & ax & ay \\
	g & bx & by
	\arrow["{z+w}"', from=1-1, to=2-1]
	\arrow["{w_1 + w_2 }", curve={height=-6pt}, from=1-2, to=1-3]
	\arrow["{z_1 + z_2}"', curve={height=6pt}, from=1-2, to=2-2]
	\arrow["{z_1 + z_2 }", curve={height=-6pt}, from=1-3, to=1-2]
	\arrow["{z_1 + z_2}"', curve={height=6pt}, from=1-3, to=2-3]
	\arrow["{w_1 + w_2}"', curve={height=6pt}, from=2-2, to=1-2]
	\arrow["{w_1 + w_2 }", curve={height=-6pt}, from=2-2, to=2-3]
	\arrow["{w_1 + w_2}"', curve={height=6pt}, from=2-3, to=1-3]
	\arrow["{z_1 + z_2 }", curve={height=-6pt}, from=2-3, to=2-2]
\end{tikzcd}\]
\end{lem}

\begin{proof}
The statement for the left hand side diagram is immediate. Let us
focus on the right hand side. The following are the two-chains of
all the Maslov index $1$ domains, and each of these domains have
an odd number of holomorphic representatives (they are either a bigon
or an annulus (Figure \ref{fig:annulus})), regardless of the almost
complex structure. 
\begin{itemize}
\item $ax\to bx,ay\to by$: $D_{2}+D_{8}$, $D_{3}+D_{6}$
\item $bx\to ax,by\to ay$: $D_{1}+D_{7}$, $D_{4}+D_{5}$
\item $ax\to ay,bx\to by$: $D_{3}+D_{4}$, $D_{7}+D_{8}$
\item $ay\to ax,by\to bx$: $D_{1}+D_{2}$, $D_{5}+D_{6}$
\end{itemize}
\end{proof}
Hence, the homology groups $HF^{-}(\boldsymbol{\beta},\boldsymbol{\gamma};0)$\footnote{In the sense of Subsection \ref{subsec:Strong-admissibility-and},
i.e. work over $\mathbb{F}[U^{1/2}]$ (resp. $\mathbb{F}[U_{1}^{1/2},U_{2}^{1/2}]$)
for the left (resp. right) side diagram.} are (all the intersection points have $c_{1}=0$)
\[
\mathbb{F}[U^{1/2}]\{f\}\oplus\mathbb{F}[U^{1/2}]\{g\},\ {\rm resp.}\ \mathbb{F}[U^{1/2}]\{ay+bx\}\oplus\mathbb{F}[U^{1/2}]\{ax+by\},
\]
where $U_{i}^{1/2}$ acts on $\mathbb{F}[U^{1/2}]$ by multiplication
by $U^{1/2}$. 

Let us describe the $\Phi$ actions. Recall that if $z_{i},w_{i}$
are in the same basepoint pair, then $\Phi_{z_{i}}=\Phi_{w_{i}}$
on homology, and so we simply denote it as $\Phi_{i}$. For the left
hand side, $\Phi f:=\Phi_{z}f=\Phi_{w}f=g$ and $\Phi_{z}g=\Phi_{w}g=0$.
For the right hand side, 
\begin{gather*}
\Phi_{i}(ay+bx):=\Phi_{z_{i}}(ay+bx)=\Phi_{w_{i}}(ay+bx)=ax+by,\\
\Phi_{i}(ax+by):=\Phi_{z_{i}}(ax+by)=\Phi_{w_{i}}(ax+by)=0,
\end{gather*}
for $i=1,2$. Recall our conventions from Subsection \ref{subsec:Conventions-and-notations};
we describe the above succinctly as 
\[
f\xrightarrow{\Phi}g,\ ay+bx\xrightarrow{\Phi_{1},\Phi_{2}}ax+by.
\]

Hence we can also describe the top homological $\mathbb{Z}$-grading
part of any stabilization of the Heegaard diagrams of Figure \ref{fig:local-ori}.
We fix our notations.
\begin{defn}
\label{def:f-and-g}Let ${\cal H}=(\Sigma,\boldsymbol{\beta},\boldsymbol{\gamma},\boldsymbol{p},\boldsymbol{G})$
be a stabilization of one of the Heegaard diagrams of Figure \ref{fig:local-ori}.
Then, ${\cal H}$ is Alexander $\mathbb{Z}/2$-splittable, and each
relative Alexander $\mathbb{Z}/2$-grading summand of $HF_{top}^{-}(\boldsymbol{\beta},\boldsymbol{\gamma};0)$
has rank $1$. Let $f,g\in HF_{top}^{-}(\boldsymbol{\beta},\boldsymbol{\gamma};0)$
be the nonzero elements that are uniquely characterized by the following
properties:
\begin{itemize}
\item $f,g$ are homogeneous with respect to the relative Alexander $\mathbb{Z}/2$-grading
(hence also homogeneous with respect to the ${\rm Spin}^{c}$-splitting).
\item $\Phi_{p}g=0$ for every basepoint pair $p$, and there exists some
basepoint pair $p$ such that $\Phi_{p}f\neq0$.
\end{itemize}
In particular, ${\rm gr}_{A}^{\mathbb{Z}/2}(f)\neq{\rm gr}_{A}^{\mathbb{Z}/2}(g)$
and $(f,g)$ form a basis of $HF_{top}^{-}(\boldsymbol{\beta},\boldsymbol{\gamma};0)$.
Also note that if $p$ is a basepoint pair such that $\Phi_{p}f\neq0$,
then $\Phi_{p}f=g$, and that such basepoint pairs $p$ are precisely
the ones that already exist in Figure \ref{fig:local-ori}.

If the Heegaard diagram is a stabilization of the left hand side (resp.
right hand side) of Figure \ref{fig:local-ori}, then define the \emph{canonical
element} $\theta\in HF_{top}^{-}(\boldsymbol{\beta},\boldsymbol{\gamma};0)$
as $g$ (resp. $f$).
\end{defn}

\begin{rem}
For the right hand side of Figure \ref{fig:local-ori}, the $f,g$
given by Definition \ref{def:f-and-g} are $f=ay+bx$, $g=ax+by$.
\end{rem}

Let us now go back to the right hand side of Figure \ref{fig:local-ori},
and compute the $H_{1}$ and relative $H_{1}$-actions on $HF_{\mathbb{F}[U^{1/2}]}^{-}(\boldsymbol{\beta},\boldsymbol{\gamma})$.
If we let $F,G$ as below, then $A_{\mathbb{F}[U^{1/2}],\kappa,-1}=B_{\mathbb{F}[U^{1/2}],\kappa,-1}=B_{\mathbb{F}[U^{1/2}],\kappa_{1},-1}=A_{\mathbb{F}[U^{1/2}],\kappa_{2},-1}=U^{1/2}F$,
$A_{\mathbb{F}[U^{1/2}],\kappa_{1},-1}=B_{\mathbb{F}[U^{1/2}],\kappa_{2},-1}=0$,
and $A_{\mathbb{F}[U^{1/2}],\kappa_{12},-1}=U^{1/2}G$, $B_{\mathbb{F}[U^{1/2}],\kappa_{12},-1}=U^{1/2}(F+G)$.
\[\begin{tikzcd}[sep=large]
	ax & ay & ax & ay \\
	bx & by & bx & by
	\arrow["F"{description}, curve={height=-12pt}, from=1-1, to=1-2]
	\arrow["F"{description}, curve={height=12pt}, from=1-1, to=2-1]
	\arrow["F"{description}, curve={height=-12pt}, from=1-2, to=1-1]
	\arrow["F"{description}, curve={height=12pt}, from=1-2, to=2-2]
	\arrow["G"{description}, curve={height=12pt}, from=1-3, to=2-3]
	\arrow["G"{description}, curve={height=12pt}, from=1-4, to=2-4]
	\arrow["F"{description}, curve={height=12pt}, from=2-1, to=1-1]
	\arrow["F"{description}, curve={height=-12pt}, from=2-1, to=2-2]
	\arrow["F"{description}, curve={height=12pt}, from=2-2, to=1-2]
	\arrow["F"{description}, curve={height=-12pt}, from=2-2, to=2-1]
	\arrow["G"{description}, curve={height=12pt}, from=2-3, to=1-3]
	\arrow["G"{description}, curve={height=12pt}, from=2-4, to=1-4]
\end{tikzcd}\]

\section{\label{sec:Khovanov-homology}Khovanov homology}

In this section, we record our conventions for Khovanov homology.
Note that in this paper, we consider the quantum and homological gradings
of Khovanov homology only as \emph{relative} gradings.

The \emph{equivariant Khovanov homology given by the Lee deformation}
(${\cal F}_{3}$ of \cite{MR2232858}; compare \cite{MR4504654})
uses the Frobenius algebra ${\cal A}=\mathbb{F}[x,U]/(x^{2}=U)$ over
$\mathbb{F}[U]$ and the multiplication and comultiplication maps
are $\mathbb{F}[U]$-linear maps given by the following.
\begin{gather*}
m:{\cal A}\otimes_{\mathbb{F}[U]}{\cal A}\to{\cal A}:1\otimes1\mapsto1,\ 1\otimes x\mapsto x,\ x\otimes1\mapsto x,\ x\otimes x\mapsto U1\otimes1\\
\Delta:{\cal A}\to{\cal A}\otimes_{\mathbb{F}[U]}{\cal A}:1\mapsto1\otimes x+x\otimes1,\ x\mapsto x\otimes x+U1\otimes1
\end{gather*}
In other words, the equivariant Khovanov homology of a planar link
(a link in the equatorial $S^{2}\subset S^{3}$; hence an unlink)
$PL$ is\footnote{We denote it as ${\cal A}$ instead of $Kh^{-}$ to avoid confusions.}
\[
{\cal A}(PL):=\mathbb{F}[\{x_{C}\}_{C\in\boldsymbol{C}},U]/(\{x_{C}^{2}-U\}_{C\in\boldsymbol{C}})
\]
where $\boldsymbol{C}$ is the set of connected components of $PL$.
Furthermore, if $B:PL\to PL'$ is a planar band\footnote{A \emph{band on a link $L\subset Y$} is an ambient $2$-dimensional
$1$-handle, i.e. it is an embedding $\iota:[0,1]\times[0,1]\hookrightarrow Y$
such that $\iota^{-1}(L)=[0,1]\times\{0,1\}$. We say $B:L\to L'$
if the link $L'$ is obtained by surgering $L$ along $B$.}, i.e. $PL\cup B$ is in the equatorial $S^{2}\subset S^{3}$, then
the associated band map is as follows:
\begin{itemize}
\item If $B$ is a split band, i.e. if $|PL|+1=|PL'|$, then the band map
is the ring map
\[
{\cal A}(B):{\cal A}(PL)\to{\cal A}(PL'):\begin{cases}
x_{C}\mapsto x_{C} & C\cap B=\emptyset\\
x_{D}\mapsto x_{D_{1}}+x_{D_{2}}
\end{cases}
\]
where $D$ is the connected component of $PL$ that intersects $B$,
and $D_{1},D_{2}$ are the connected components of $PL'$ that intersect
$B$.
\item If $B$ is a merge band, i.e. if $|PL|=|PL'|+1$, then the band map
is the ring map
\[
{\cal A}(B):{\cal A}(PL)\to{\cal A}(PL'):\begin{cases}
x_{C}\mapsto x_{C} & C\cap B=\emptyset\\
x_{C}\mapsto x_{D'} & C\cap B\neq\emptyset
\end{cases}
\]
where $D'$ is the connected component of $PL'$ that intersects $B$.
\end{itemize}
\begin{figure}[h]
\begin{centering}
\includegraphics[scale=1.5]{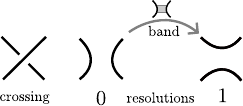}
\par\end{centering}
\caption{\label{fig:skein-moves-1-1}Our convention for Khovanov homology}
\end{figure}

Given a planar diagram of a link $L$ in $S^{3}$, one considers the
cube of resolutions $\{0,1\}^{n}$ (if the planar diagram has $n$
crossings), where the convention is as in Figure \ref{fig:skein-moves-1-1},
and uses the Khovanov homology for planar links and bands (i.e. the
Frobenius algebra ${\cal A}$) to build the \emph{equivariant Khovanov
(cube of resolutions) chain complex $CKh^{-}(L)$}. The underlying
$\mathbb{F}[U]$-module of the chain complex is the direct sum 
\[
\bigoplus_{v\in\{0,1\}^{n}}{\cal A}(L_{v})
\]
where $L_{v}$ is the complete $v$-resolution, which is a planar
link. For each edge $e:v\to v'$ of the cube (i.e. $v,v'$ differ
in one coordinate for which $v$ is $0$ and $v'$ is $1$), there
is an associated band $B_{e}:L_{v}\to L_{v'}$ (see Figure \ref{fig:skein-moves-1-1});
the differential is the sum of the corresponding band maps: $\partial=\sum_{e:v\to v'}{\cal A}(B_{e})$.
The homology of this cube of resolutions chain complex is $Kh^{-}(L)$,
the \emph{equivariant Khovanov homology} of $L$.
\begin{rem}
The convention is that if $L_{0},L_{1},L_{2}$ is an unoriented skein
triple (Definition \ref{def:unoriented-skein-triple}), then their
\emph{mirrors} form an exact triangle in Khovanov homology.
\end{rem}

Define the \emph{unreduced Khovanov chain complex}, or the unreduced
hat version, as
\[
\widehat{CKh}(L):=CKh^{-}(L)/U.
\]

Given a point $p\in L$ not on any of the crossings, define the action
of $p$ on each ${\cal A}(L_{v})$ as multiplication by $x_{C}$,
where $C$ is the connected component of $L_{v}$ that $x$ is on.
This defines an action on $CKh^{-}(L)$. Hence, given a set of points
$\boldsymbol{p}\in L$ that is disjoint from the crossings, we can
view $CKh^{-}(L)$ as a chain complex over $R_{\boldsymbol{p}}:=\mathbb{F}[\{X_{p}\}_{p\in\boldsymbol{p}}]$,
where $X_{p}$ acts by the action of $p$. Hence $Kh^{-}(L)$ can
be viewed as an $R_{\boldsymbol{p}}$-module.

To define the \emph{reduced Khovanov chain complex,} or the reduced
hat version, we need to choose a distinguished point $p\in L$; define
\[
\widetilde{CKh}(L,p):=CKh^{-}(L)/X_{p}.
\]

We collapse the relative quantum grading ($q$) and relative homological
grading $(h$) to a relative $\delta$-grading ${\rm gr}_{\delta}={\rm gr}_{q}/2-{\rm gr}_{h}$:
the actions $X_{p}$ have $\delta$-degree $-1$.

\section{\label{sec:The-setup-and}The setup and the Heegaard diagram}

In this section, we introduce a general setup. As in this section,
we work in the minus version throughout most of this paper, and specialize
to the reduced and unreduced hat versions in Subsection \ref{sec:Iterating-the-unoriented}.

\subsection{\label{subsec:The-general-setup}The general setup}

\begin{figure}[h]
\begin{centering}
\includegraphics[scale=2]{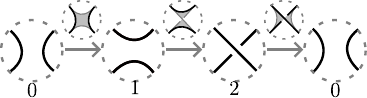}
\par\end{centering}
\caption{\label{fig:skein-moves-1}A local diagram for an unoriented skein
triple in $D^{3}$ and bands between them}
\end{figure}

Let $CB_{1},\cdots,CB_{s}$ (\emph{``crossing balls''}) be a collection
of pairwise disjoint embeddings of $D^{3}$ (identified with the $D^{3}$'s
of Figure \ref{fig:skein-moves-1}) in a three-manifold $Y$. Let
$\boldsymbol{L}$ be a partially ordered set of balled links in $Y$,
such that the following hold.
\begin{enumerate}
\item The underlying links of $L\in\boldsymbol{L}$ are identical outside
the crossing balls.
\item The baseballs of every $L\in\boldsymbol{L}$ are identical, if we
ignore their types (i.e. whether they are free or link baseballs).
\item The baseballs are disjoint from all the crossing balls.
\item For every $L\in\boldsymbol{L}$, every baseball intersects the underlying
link of $L$.
\item Every connected component of $Y$ has a baseball.
\item For each $L\in\boldsymbol{L}$ and $CB_{r}$, $CB_{r}\cap L$ is a
rational tangle in $CB_{r}$, and in particular is one of $0,1,2$
of Figure \ref{fig:skein-moves-1}. Call this $L(r)\in\{0,1,2\}$.
If $L<L'$, then $L(r)\le L'(r)$ for all $r$.
\item \label{enu:For-every-directed}For every chain $L_{1}<\cdots<L_{m}$
and for every baseball, the subset of the $L_{k}$'s for which the
baseball is a link baseball is $\{L_{k}:i\le k\le j\}$ for some $i,j$.
\end{enumerate}

\subsection{\label{subsec:The-Heegaard-diagram}The Heegaard diagram and the
\texorpdfstring{$A_{\infty}$}{Ainf}-category}

Given a setup as in Subsection \ref{subsec:The-general-setup}, we
define a Heegaard diagram. We will sometimes only consider beta-attaching
curves, and sometimes also consider an alpha-attaching curve. If we
only consider beta-attaching curves, then the ambient manifold $Y$
is not important. If we consider an alpha-attaching curve as well,
then we fix a second cohomology class $c\in H^{2}(Y;\mathbb{Z})$.

\subsubsection*{The Heegaard surface}

Fix an $L\in\boldsymbol{L}$\footnote{The following definition does not depend on this choice.},
and let $L^{link}$ be its underlying link. We define a ``$\beta$-handlebody\footnote{This might be disconnected.}''
$H^{oneh}(L)$ as follows (the Heegaard surface is $\Sigma:=-\partial H^{oneh}(L)$).
Let $\overline{N}(L^{link})$ be a sufficiently ``small'' closed
tubular neighborhood of $L^{link}$ in $Y$ such that for each crossing
ball $CB_{r}$, the intersection $\overline{N}(L^{link})\cap CB_{r}$
is a tubular neighborhood of $L^{link}\cap CB_{r}$ in $CB_{r}$;
and for each baseball $BB$, the intersection $\overline{N}(L^{link})\cap BB$
is a tubular neighborhood of $L^{link}\cap BB$ in $BB$. Let $H(L):=\overline{N}(L^{link})\cup\bigsqcup_{r}CB_{r}$.
Attach some number of ambient three-dimensional one-handles to $H(L)$,
disjoint from the crossing balls and the baseballs, and call the resulting
handlebody $H^{oneh}(L)$. If we want to consider an alpha-attaching
curve, then let the ambient one-handles be such that $Y\backslash{\rm int}(H^{oneh}(L))$
is a handlebody.

\begin{figure}[h]
\begin{centering}
\includegraphics{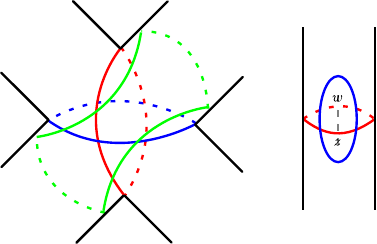}
\par\end{centering}
\caption{\label{fig:standard-diagram}Left: a standard diagram for $CB\cap\Sigma$.
Right: a standard diagram for $BB\cap\Sigma$ for a baseball $BB$}
\end{figure}

\subsubsection*{The basepoints, attaching curves, and forbidding arcs}

The beta-attaching curves are $\boldsymbol{\beta}(L)$ for each $L\in\boldsymbol{L}$
and the partial order of the $\boldsymbol{\beta}(L)$'s is given by
the partial order on $\boldsymbol{L}$. If we consider an alpha-attaching
curve $\boldsymbol{\alpha}$ as well, then $\boldsymbol{\alpha}<\boldsymbol{\beta}(L)$
for all $L$.

We first define an auxiliary Heegaard diagram ${\cal H}$ with Heegaard
surface $\Sigma$ as follows. The basepoints and the components of
the $\boldsymbol{\beta}(L)$'s are contained in the following (pairwise
disjoint) regions, and so we describe them in each region. We call
these regions \emph{special}.
\begin{itemize}
\item On the boundary of each one-handle we attached to get $H^{oneh}(L)$,
$\boldsymbol{\beta}(L)$ is a standard translate of the boundary of
the cocore of the one-handle.
\item For each crossing ball $CB_{r}$, identify $CB_{r}\cap\Sigma$ (which
is $S^{2}$ minus four disks) with the left hand side of Figure \ref{fig:standard-diagram}:
in this region, $\boldsymbol{\beta}(L)$ is a standard translate of
one of the curves in the left hand side of Figure \ref{fig:standard-diagram},
depending on the crossing.\footnote{\label{fn:The-attaching-curves}The attaching curves for the three
tangles $0$, $1$, and $2$ of Figure \ref{fig:skein-moves-1} can
be identified with red, blue, and green, respectively, in the left
hand side of Figure \ref{fig:standard-diagram}, although it looks
as if it suggests otherwise. This is because of our orientation conventions:
recall Subsection \ref{subsec:Conventions-and-notations}.}
\item For each baseball, identify its intersection with $\Sigma$ (which
is an annulus) with the right hand side of Figure \ref{fig:standard-diagram}:
put a basepoint pair ($(z,w)$ in the figure) and let the corresponding
forbidding arc be the grey arc in the figure. This basepoint pair
is a link (resp. free) basepoint pair for $\boldsymbol{\beta}(L)$
if the baseball is a link (resp. free) baseball for $L$. In this
region, $\boldsymbol{\beta}(L)$ is a standard translate of one of
the curves in the figure, depending on whether the baseball is a link
baseball (red circle) or a free baseball (blue circle). If it is a
free baseball, then $\boldsymbol{\beta}(L)$ is disjoint from the
corresponding forbidding arc.
\end{itemize}
If we want to consider an alpha-attaching curve, then define the attaching
curve $\boldsymbol{\alpha}$ such that it represents the handlebody
$Y\backslash{\rm int}(H^{oneh}(L))$ and is disjoint from all the
forbidding arcs. Let all the basepoint pairs be free basepoint pairs
for $\boldsymbol{\alpha}$.

To define the $A_{\infty}$-category, the Heegaard diagram has to
be weakly admissible if we only consider the beta-attaching curves.
If we consider an alpha-attaching curve as well, then we have to fix
some $c\in H^{2}(Y;\mathbb{Z})$, and the Heegaard diagram has to
be strongly $c$-admissible. To achieve weak admissibility or strong
$c$-admissibility, we choose a set of multi-arcs $\gamma_{\boldsymbol{\alpha}}$
and $\gamma_{\boldsymbol{\beta}(L)}$ that do not intersect any intersection
points of attaching curves, basepoints, or the forbidding arcs, and
wind the attaching curves along these multi-arcs. Call this Heegaard
diagram ${\cal H}^{wind}$; we mainly work with ${\cal H}^{wind}$.
Note that $(\boldsymbol{\alpha},\boldsymbol{\beta}(L))$ represents
the balled link $L$.

\subsubsection*{The $A_{\infty}$-category}

Let ${\cal B}$ be the $A_{\infty}$-category given by the weakly
admissible Heegaard diagram ${\cal H}^{wind}$ whose objects are the
beta-attaching curves $\boldsymbol{\beta}(L)$. By abuse of notion,
we will identify $L$ with $\boldsymbol{\beta}(L)$. Recall from Subsection
\ref{subsec:ainf} that ${\cal B}$ is homologically $\mathbb{Z}$-gradable.

For $c\in H^{2}(Y;\mathbb{Z})$, let ${\cal A}_{c}$ be the $A_{\infty}$-category
given by the strongly $c$-admissible Heegaard diagram ${\cal H}^{wind}$
whose objects are the $\boldsymbol{\beta}(L)$'s and $\boldsymbol{\alpha}$.

\section{\label{sec:Band-maps-for}Band maps for balled links}

In this section, we recall the band maps in unoriented link Floer
homology that we defined in \cite[Section 3]{nahm2025unorientedskeinexacttriangle}.
These band maps give rise to an unoriented skein exact triangle.

\subsection{\label{subsec:Band-maps}Band maps}
\begin{defn}
\label{def:Two-balled-links-1}Two balled links $L_{0},L_{1}$ are
\emph{related by a band map $B:L_{0}\to L_{1}$ }if there exists a
setup as in Subsection \ref{subsec:The-general-setup}, where $L_{0},L_{1}\in\boldsymbol{L}$,
their underlying links differ in exactly one crossing ball $CB$,
and $L_{i}\cap CB\subset CB$ is $i$ of Figure \ref{fig:skein-moves-1}.
Identify $B$ with the band drawn in Figure \ref{fig:skein-moves-1}.

The following are the three types of bands:
\begin{itemize}
\item $B$ is a \emph{non-orientable band} if $|L_{0}|=|L_{1}|;$
\item $B$ is a \emph{split band} if $|L_{0}|+1=|L_{1}|$; and
\item $B$ is a \emph{merge band} if $|L_{0}|=|L_{1}|+1$.
\end{itemize}
\end{defn}

\begin{condition}
\label{cond:If--is}If $L$ is a balled link, let $LB(L)$ be the
set of link baseballs of $L$. In this section (as in \cite[Definition 3.10]{nahm2025unorientedskeinexacttriangle}),
we further assume that
\begin{itemize}
\item If $B$ is a non-orientable band, then $LB(L_{0})=LB(L_{1})$.
\item If $B$ is a split band (resp. merge band), then there exists a baseball
$BB$ such that $LB(L_{1})=LB(L_{0})\sqcup\{BB\}$ (resp. $LB(L_{0})=LB(L_{1})\sqcup\{BB\}$).
\end{itemize}
\end{condition}

\begin{figure}[h]
\begin{centering}
\includegraphics{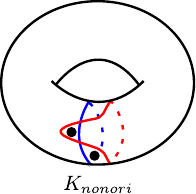}\qquad{}\includegraphics{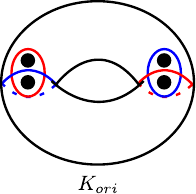}\qquad{}\includegraphics{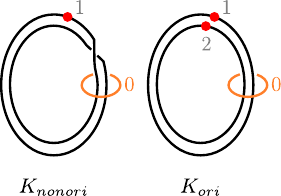}
\par\end{centering}
\caption{\label{fig:s1s2-important-1-1}Standard ($K_{ori}$ is not weakly
admissible) Heegaard diagrams and the corresponding links}
\end{figure}

Fix a second homology class $c\in H^{2}(Y;\mathbb{Z})$. Consider
a $c$-strongly admissible pair-pointed Heegaard diagram $(\Sigma,\boldsymbol{\alpha},\boldsymbol{\beta}_{0},\boldsymbol{\beta}_{1},\boldsymbol{p},\boldsymbol{G})$
as in Subsection \ref{subsec:The-Heegaard-diagram}. Then, the sub
Heegaard diagram for $\boldsymbol{\beta}_{0},\boldsymbol{\beta}_{1}$,
after isotopies, is a stabilization of the Heegaard diagram labelled
$K_{nonori}$ (resp. $K_{ori}$) of Figure \ref{fig:s1s2-important-1-1}
if $B$ is non-orientable (resp. a split or merge band), or equivalently,
is a Heegaard diagram that we considered in Definition \ref{def:f-and-g}.
Recall that we also defined the canonical element $\theta\in HF_{top}^{-}(\boldsymbol{\beta}_{0},\boldsymbol{\beta}_{1})$
in Definition \ref{def:f-and-g}. Define the band map as
\[
F_{B}:=\mu_{2}(-\otimes\theta):CF^{-}(\boldsymbol{\alpha},\boldsymbol{\beta}_{0})\to CF^{-}(\boldsymbol{\alpha},\boldsymbol{\beta}_{1}).
\]
This map is homogeneous with respect to the relative homological gradings
(see Subsection \ref{subsec:ainf}).
\begin{rem}
\label{rem:canonical-simple}For the two balled links $(S^{1}\times S^{2},L)$
($L=K_{nonori}$, resp. $K_{ori}$) of Figure \ref{fig:s1s2-important-1-1},
the $\Phi$ actions on the basis $(f,g)$ of $HFL'{}_{top}^{-}(S^{1}\times S^{2},L;0)$
defined in Definition \ref{def:f-and-g} are as follows:
\[
f\xrightarrow{\Phi_{1}}g,\ {\rm resp.}\ f\xrightarrow{\Phi_{1},\Phi_{2}}g.
\]
The canonical element $\theta$ is $g$, resp. $f$.
\end{rem}

\begin{rem}
\label{rem:relax-band}In Section \ref{sec:The-setup-and}, every
basepoint pair was a free basepoint pair for $\boldsymbol{\alpha}$.
We can relax this assumption: only assume that the basepoint pair
corresponding to the baseball $BB$ from Condition \ref{cond:If--is}
is a free baseball pair for $\boldsymbol{\alpha}$ (this is a vacuous
condition if the band is non-orientable). In this case, $(\boldsymbol{\alpha},\boldsymbol{\beta}_{0})$
and $(\boldsymbol{\alpha},\boldsymbol{\beta}_{1})$ represent balled
links (note that they are in general different from $L_{0},L_{1}$)
that are related by the band $B$, and the map 
\[
\mu_{2}(-\otimes\theta):CF^{-}(\boldsymbol{\alpha},\boldsymbol{\beta}_{0})\to CF^{-}(\boldsymbol{\alpha},\boldsymbol{\beta}_{1})
\]
is the band map for $B$.
\end{rem}

\subsection{\label{subsec:Merge-and-split}Merge and split band maps as the composition
of two maps}

We can consider merge and split band maps as the composition of two
maps, if $B$ is a merge or a split band (see Figure \ref{fig:s1s2-important-1-1-1}).

\begin{figure}[h]
\begin{centering}
\includegraphics{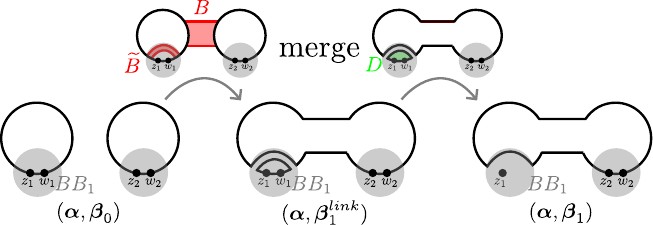}
\par\end{centering}
\begin{centering}
\includegraphics{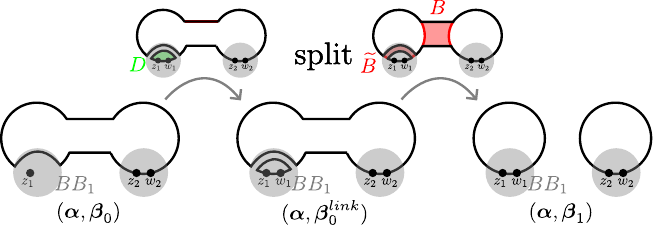}
\par\end{centering}
\caption{\label{fig:s1s2-important-1-1-1}Merge and split maps as the composition
of two maps}
\end{figure}

If $B$ is a merge band, then let $p_{1},p_{2}$ be the two link basepoint
pairs that correspond to the link components of $L_{0}$ that the
band intersects, such that $p_{2}$ is a link basepoint pair for $L_{1}$.
Then, let $\boldsymbol{\beta}_{1}^{link}$ be the same attaching curve
as $\boldsymbol{\beta}_{1}$, but where $p_{1}$ is also a link basepoint
pair. Then, the band map $F_{B}$ is the composition of the following
maps:
\begin{multline*}
CF^{-}(\boldsymbol{\alpha},\boldsymbol{\beta}_{0})\xrightarrow{\mu_{2}(-\otimes\theta)}CF^{-}(\boldsymbol{\alpha},\boldsymbol{\beta}_{1}^{link})\simeq CF^{-}(\boldsymbol{\alpha},\boldsymbol{\beta}_{1})\oplus U_{1}^{1/2}CF^{-}(\boldsymbol{\alpha},\boldsymbol{\beta}_{1})\\
\xrightarrow{\pi}U_{1}^{1/2}CF^{-}(\boldsymbol{\alpha},\boldsymbol{\beta}_{1})\xrightarrow{\cdot U_{1}^{-1/2}}CF^{-}(\boldsymbol{\alpha},\boldsymbol{\beta}_{1})
\end{multline*}
where $\pi$ is projection onto the second summand.

Let $BB_{1}$ be the baseball corresponding to $p_{1}$. Then, the
link that $(\boldsymbol{\alpha},\boldsymbol{\beta}_{1}^{link})$ represents
is the same link as $(\boldsymbol{\alpha},\boldsymbol{\beta}_{1})$,
but has an extra unknot inside $BB_{1}$ (that bounds a disk $D$
in $BB_{1}$ given by the forbidding arc $G_{1}$ between the two
basepoints of $p_{1}$; $D\subset BB_{1}$ is disjoint from everything
else, and such $D\subset BB_{1}$ is unique up to isotopy). The map
$U_{1}^{-1/2}\pi:CF^{-}(\boldsymbol{\alpha},\boldsymbol{\beta}_{1}^{link})\twoheadrightarrow CF^{-}(\boldsymbol{\alpha},\boldsymbol{\beta}_{1})$
is the \emph{annihilation map} (\cite[Subsubsection 3.3.1]{nahm2025unorientedskeinexacttriangle})
given by the disk $D$. Also, the map $\mu_{2}(-\otimes\theta):CF^{-}(\boldsymbol{\alpha},\boldsymbol{\beta}_{0})\to CF^{-}(\boldsymbol{\alpha},\boldsymbol{\beta}_{1}^{link})$
is the map for the two bands $B,\widetilde{B}$ described in Figure
\ref{fig:s1s2-important-1-1-1} (\emph{a mergesplit pair} \cite[Subsubsection 3.3.2]{nahm2025unorientedskeinexacttriangle}).

Similarly, if $B$ is a split band, then let $p_{1},p_{2}$ be the
two link basepoint pairs that correspond to the link components of
$L_{1}$ that the band intersects, such that $p_{2}$ is a link basepoint
pair for $L_{0}$. Let $\boldsymbol{\beta}_{0}^{link}$ be the same
attaching curve as $\boldsymbol{\beta}_{0}$, but where $p_{1}$ is
also a link basepoint pair. Then, the band map $F_{B}$ is the composition
of the following maps:
\[
CF^{-}(\boldsymbol{\alpha},\boldsymbol{\beta}_{0})\xrightarrow{\iota}CF^{-}(\boldsymbol{\alpha},\boldsymbol{\beta}_{0})\oplus U_{1}^{1/2}CF^{-}(\boldsymbol{\alpha},\boldsymbol{\beta}_{0})\simeq CF^{-}(\boldsymbol{\alpha},\boldsymbol{\beta}_{0}^{link})\xrightarrow{\mu_{2}(-\otimes\theta)}CF^{-}(\boldsymbol{\alpha},\boldsymbol{\beta}_{1}),
\]
where $\iota$ is inclusion into the first summand. The link that
$(\boldsymbol{\alpha},\boldsymbol{\beta}_{0}^{link})$ represents
is the same link as $(\boldsymbol{\alpha},\boldsymbol{\beta}_{0})$,
but has an extra unknot inside $BB_{1}$ (that bounds a disk $D$
in $BB_{1}$, which is disjoint from everything else); the map $\iota:CF^{-}(\boldsymbol{\alpha},\boldsymbol{\beta}_{0})\hookrightarrow CF^{-}(\boldsymbol{\alpha},\boldsymbol{\beta}_{0}^{link})$
is the \emph{creation map} (\cite[Subsubsection 3.3.1]{nahm2025unorientedskeinexacttriangle})
given by $D$. Also, the map $\mu_{2}(-\otimes\theta):CF^{-}(\boldsymbol{\alpha},\boldsymbol{\beta}_{0}^{free})\to CF^{-}(\boldsymbol{\alpha},\boldsymbol{\beta}_{1})$
is the map for the mergesplit pair $B,\widetilde{B}$ described in
Figure \ref{fig:s1s2-important-1-1-1}.

\subsection{Unoriented link Floer homology and Khovanov homology for planar bands}

Unoriented link Floer homology and Khovanov homology agree for planar
links and planar bands. More precisely, we have the following.
\begin{prop}[{\cite[Section 5]{nahm2025unorientedskeinexacttriangle}}]
\label{prop:Let--be-1}Let $L$ be a balled link in $S^{3}$, whose
underlying link is a planar link. Recall that $Kh^{-}(L)=\mathbb{F}[\{x_{C}\}_{C},U]/(\{x_{C}-U\}_{C})$,
where $C$ ranges over the components of $L$. Consider $HFL'{}^{-}(S^{3},L)$
as a $Kh^{-}(L)$-module where $x_{C}$ acts by multiplication by
$U_{p}^{1/2}$ where $p$ is the link basepoint pair on $C$. Then,
\[
HFL'{}^{-}(S^{3},L)\simeq Kh^{-}(L)
\]
as relatively graded $Kh^{-}(L)$-modules, where the relative homological
$\mathbb{Z}$-grading (resp. relative Alexander $\mathbb{Z}/2$-grading)
on $HFL'{}^{-}(S^{3},L)$ corresponds to the relative $\delta=q/2-h$
grading (resp. relative $q/2$ modulo $2$ grading) on $Kh^{-}(L)$.
(Note that such isomorphism is unique.) If $B:L\to L'$ is a planar
band, then the band map on $HFL'{}^{-}$ agrees with the band map
on $Kh^{-}$.
\end{prop}

\begin{rem}
Since $L$ is planar, its Khovanov homology is supported in one $h$
grading, and so the relative $\delta$ grading agrees with the relative
$q/2$ grading. We consider the $\delta$ grading since this is the
grading that descends to the homological grading on $HFL'$ under
the spectral sequence: see Subsection \ref{subsec:The-Khovanov-chain}.
\end{rem}

\begin{rem}
Perhaps it is more natural to consider the mirror of the link in Proposition
\ref{prop:Let--be-1} for Khovanov homology (compare with Theorems
\ref{thm:knot-S3} and \ref{thm:main-thm}), but it does not matter
since the link and the band are planar.
\end{rem}

\subsection{The basepoint actions}

Let $L_{0},L_{1}$, $CB$, and $BB$ be as in Definition \ref{def:Two-balled-links-1}
and Condition \ref{cond:If--is}.
\begin{prop}[{\cite[Proposition 3.25]{nahm2025unorientedskeinexacttriangle}}]
\label{prop:For-any-}For any $q\in L_{i}\backslash CB$ (or equivalently
a connected component $CC$ of $L_{i}\backslash CB$), the band map
\[
F_{B}:HFL'{}^{-}(Y,L_{0};c)\to HFL'{}^{-}(Y,L_{1};c)
\]
is equivariant with respect to the action of $q$ (Definition \ref{def:hflaction}).
\end{prop}

\begin{proof}
If the band is non-orientable, there is nothing to check. If the band
is a merge band, then let $BB,BB'$ be the link baseballs of the components
of $L_{0}$ that intersect the band. Then we have to check that $F_{B}(U_{BB}^{1/2}\cdot-)=U_{BB'}^{1/2}F_{B}(-)$
on homology. Similarly, if the band is a split band, then let $BB,BB'$
be the link baseballs of the components of $L_{1}$ that intersect
the band. Then we have to check that $F_{B}(U_{BB'}^{1/2}\cdot-)=U_{BB}^{1/2}F_{B}(-)$
on homology. Both are formal consequences of that the element $\theta\in HF^{-}(\boldsymbol{\beta}_{0},\boldsymbol{\beta}_{1})$
satisfies $U_{BB}^{1/2}\theta=U_{BB'}^{1/2}\theta$.
\end{proof}

\subsection{\label{subsec:An-unoriented-skein}An unoriented skein exact triangle}
\begin{defn}
\label{def:unoriented-skein-triple}Three balled links $L_{0},L_{1},L_{2}$
form an \emph{unoriented skein triple} if there exists a setup as
in Subsection \ref{subsec:The-general-setup}, where $L_{0},L_{1},L_{2}\in\boldsymbol{L}$,
such that:
\begin{itemize}
\item There is exactly one crossing ball $CB$, and $L_{i}\cap CB\subset CB$
is $i$ of Figure \ref{fig:skein-moves-1}, and
\item Let $i\in\{0,1,2\}$ be such that $|L_{i}|=|L_{j}|+1$ for both $j\in\{0,1,2\}$
such that $j\neq i$. Then, there exists a baseball $BB$ such that
$LB(L_{i})=LB(L_{j})\sqcup\{BB\}$ for both $j\in\{0,1,2\}$ such
that $j\neq i$.
\end{itemize}
By abuse of notion, we also say that the corresponding attaching curves
$\boldsymbol{\beta}_{0},\boldsymbol{\beta}_{1},\boldsymbol{\beta}_{2}$
form an \emph{unoriented skein triple}.
\end{defn}

Let $\boldsymbol{\beta}_{0},\boldsymbol{\beta}_{1},\boldsymbol{\beta}_{2}$
form an unoriented skein triple. Recall that the $A_{\infty}$-category
${\cal B}$ (Subsection \ref{subsec:The-Heegaard-diagram}) with objects
$\boldsymbol{\beta}_{0},\boldsymbol{\beta}_{1},\boldsymbol{\beta}_{2}$
is homologically $\mathbb{Z}$-gradable. Lift the relative homological
$\mathbb{Z}$-grading to an absolute homological $\mathbb{Z}$-grading
such that for $i=0,1$, the canonical element $\theta\in HF_{0}^{-}(\boldsymbol{\beta}_{i},\boldsymbol{\beta}_{i+1})$,
i.e. lies in homological grading $0$. Let us abuse notation and let
$\theta\in CF_{0}^{-}(\boldsymbol{\beta}_{i},\boldsymbol{\beta}_{i+1})$
be any cycle that represents the homology class $\theta\in HF_{0}^{-}(\boldsymbol{\beta}_{i},\boldsymbol{\beta}_{i+1})$.

The main theorem of \cite{nahm2025unorientedskeinexacttriangle} is
the following.
\begin{thm}
\label{thm:Let--be}Let $\boldsymbol{\beta}_{0},\boldsymbol{\beta}_{1},\boldsymbol{\beta}_{2}$
form an unoriented skein triple. Then, there exists some $\zeta\in CF_{1}^{-}(\boldsymbol{\beta}_{0},\boldsymbol{\beta}_{2})$
such that $\mu_{1}(\zeta)=\mu_{2}^{\boldsymbol{\beta}_{0},\boldsymbol{\beta}_{1},\boldsymbol{\beta}_{2}}(\theta\otimes\theta)$,
i.e. 
\[\begin{tikzcd}
	{\underline{\boldsymbol{\beta}_{012}}:=} & {\boldsymbol{\beta}_0} & {\boldsymbol{\beta}_1} & {\boldsymbol{\beta}_2}
	\arrow["\theta"{description}, from=1-2, to=1-3]
	\arrow["{\zeta }"{description, pos=0.3}, curve={height=-12pt}, from=1-2, to=1-4]
	\arrow["\theta"{description}, from=1-3, to=1-4]
\end{tikzcd}\]is a twisted complex. For any such $\zeta$ and any attaching curve
$\boldsymbol{\alpha}$ such that the basepoint pair contained in the
baseball $BB$ is a free basepoint pair\footnote{Since our attaching curves are ordered, we technically cannot consider
both $HF^{-}(\boldsymbol{\alpha},\underline{\boldsymbol{\beta}_{012}})$
and $HF^{-}(\underline{\boldsymbol{\beta}_{012}},\boldsymbol{\alpha})$
at once. This is not a problem since we can treat this as separate
statements.}, 
\begin{gather*}
HF^{-}(\boldsymbol{\alpha},\underline{\boldsymbol{\beta}_{012}})=0,\ HF^{-}(\underline{\boldsymbol{\beta}_{012}},\boldsymbol{\alpha})=0.
\end{gather*}
\end{thm}

\begin{proof}
This is mostly shown in the proof of \cite[Theorem 4.2]{nahm2025unorientedskeinexacttriangle}.
Let us focus on the case where $\boldsymbol{\alpha}<\boldsymbol{\beta}_{0},\boldsymbol{\beta}_{1},\boldsymbol{\beta}_{2}$,
as the other case is exactly the same. In the proof of \cite[Theorem 4.2]{nahm2025unorientedskeinexacttriangle},
we considered special kinds of $\boldsymbol{\beta}_{0},\boldsymbol{\beta}_{1},\boldsymbol{\beta}_{2}$,
which are obtained from \cite[Figure 4.1]{nahm2025unorientedskeinexacttriangle}
($\boldsymbol{\beta}_{0},\boldsymbol{\beta}_{1},\boldsymbol{\beta}_{2}$
is a cyclic permutation of $\boldsymbol{\beta}_{a},\boldsymbol{\beta}_{b},\boldsymbol{\beta}_{c}$
of the figure) by a\emph{ ``stabilization''} in the sense of \cite[Subsection 2.6]{nahm2025unorientedskeinexacttriangle},
which is, as written, slightly more restrictive than Subsection \ref{subsec:Connected-sums},
but they are related by handleslides and isotopies. Let us call the
case we considered in the proof of \cite[Theorem 4.2]{nahm2025unorientedskeinexacttriangle}
the \emph{special case}, and our case the \emph{general case}.

Let us first briefly explain where the theorem for the special case
is shown in the proof of \cite[Theorem 4.2]{nahm2025unorientedskeinexacttriangle}.
There are three cases to consider, depending on the type of the band
$B_{01}:\boldsymbol{\beta}_{0}\to\boldsymbol{\beta}_{1}$ (equivalently,
the cyclic permutation of $\boldsymbol{\beta}_{a},\boldsymbol{\beta}_{b},\boldsymbol{\beta}_{c}$
of \cite[Figure 4.1]{nahm2025unorientedskeinexacttriangle}). ``Hence,
$\mu_{2}(-,\underline{f})$ is a quasi-isomorphism'' in the penultimate
paragraph proves the case where $B_{01}$ is non-orientable, and that
$\mu_{2}(-,\underline{g})$ is a quasi-isomorphism implies the case
where $B_{01}$ is a merge band. The final paragraph explains how
to deduce the theorem for the case where $B_{01}$ is a split band.

We deduce the theorem for the general case from the special case;
note that we use the same argument in Section \ref{sec:Proofs-of-claims}
and a similar argument in Section \ref{sec:Iterating-the-unoriented}
(the argument after Proposition \ref{prop:We-have-the}).

The general case can be obtained from a special case by handleslides
and isotopies, and so $\mu_{2}^{\boldsymbol{\beta}_{0},\boldsymbol{\beta}_{1},\boldsymbol{\beta}_{2}}(\theta\otimes\theta)=0$
on homology in the general case as well. Let us call the attaching
curves of the corresponding special case $\widetilde{\boldsymbol{\beta}_{i}}$.
The $A_{\infty}$-category with elements $\widetilde{\boldsymbol{\beta}_{i}},\boldsymbol{\beta}_{i}$
for $i=0,1,2$ (let $\widetilde{\boldsymbol{\beta}_{i}}<\widetilde{\boldsymbol{\beta}_{j}}$,
$\boldsymbol{\beta}_{i}<\boldsymbol{\beta}_{j}$ for $i<j$, and $\widetilde{\boldsymbol{\beta}_{i}}<\boldsymbol{\beta}_{i}$)
is homologically $\mathbb{Z}$-gradable. We can fix the homological
$\mathbb{Z}$-gradings as follows (alternatively, use Proposition
\ref{prop:combi-claims} and fix the homological $\mathbb{Z}$-gradings
as in Proposition \ref{prop:combi-claims}; this is equivalent to
the following): for $i=0,1$, the canonical elements $\theta\in HF^{-}(\widetilde{\boldsymbol{\beta}_{i}},\widetilde{\boldsymbol{\beta}_{i+1}})$,
$\theta\in HF^{-}(\boldsymbol{\beta}_{i},\boldsymbol{\beta}_{i+1})$
lie in homological grading $0$, and for $j=0,1,2$, the elements
$\Theta^{+}\in CF^{-}(\widetilde{\boldsymbol{\beta}_{j}},\boldsymbol{\beta}_{j})$
that induce the canonical isomorphism $\mu_{2}(-\otimes\Theta^{+}):HF^{-}(\boldsymbol{\alpha},\widetilde{\boldsymbol{\beta}_{j}})\to HF^{-}(\boldsymbol{\alpha},\boldsymbol{\beta}_{j})$
also lie in homological grading $0$. Then for $i<j$, $i,j\in\{0,1,2\}$,
the top homological grading of $HF^{-}(\widetilde{\boldsymbol{\beta}_{i}},\widetilde{\boldsymbol{\beta}_{j}})$
is $0$, and so the top homological grading of $HF^{-}(\widetilde{\boldsymbol{\beta}_{i}},\boldsymbol{\beta}_{j})$
is also $0$.

Let $\zeta\in CF_{1}^{-}(\widetilde{\boldsymbol{\beta}_{0}},\widetilde{\boldsymbol{\beta}_{2}})$
be such that $\mu_{1}(\zeta)=\mu_{2}^{\widetilde{\boldsymbol{\beta}_{0}},\widetilde{\boldsymbol{\beta}_{1}},\widetilde{\boldsymbol{\beta}_{2}}}(\theta\otimes\theta)$,
let $\zeta\in CF_{1}^{-}(\boldsymbol{\beta}_{0},\boldsymbol{\beta}_{2})$
be such that $\mu_{1}(\zeta)=\mu_{2}^{\boldsymbol{\beta}_{0},\boldsymbol{\beta}_{1},\boldsymbol{\beta}_{2}}(\theta\otimes\theta)$,
and let $\underline{\widetilde{\boldsymbol{\beta}_{012}}},\underline{\boldsymbol{\beta}_{012}}$
be the corresponding twisted complexes (the horizontal parts of Diagram
(\ref{eq:unoriented-twisted})). We can define the diagonal maps of
Diagram (\ref{eq:unoriented-twisted}) (i.e. find elements of $CF_{1}^{-}(\widetilde{\boldsymbol{\beta}_{0}},\boldsymbol{\beta}_{1})$,
$CF_{1}^{-}(\widetilde{\boldsymbol{\beta}_{1}},\boldsymbol{\beta}_{2})$,
and $CF_{2}^{-}(\widetilde{\boldsymbol{\beta}_{0}},\boldsymbol{\beta}_{2})$)
such that the element $\underline{\Theta^{+}}\in CF^{-}(\underline{\widetilde{\boldsymbol{\beta}_{012}}},\underline{\boldsymbol{\beta}_{012}})$
that consists of the vertical and diagonal maps is a cycle. Here,
the vertical maps are the elements $\Theta^{+}\in CF_{0}^{-}(\widetilde{\boldsymbol{\beta}_{j}},\boldsymbol{\beta}_{j})$
from above. 
\begin{equation}\label{eq:unoriented-twisted}
\begin{tikzcd}
	{\widetilde{\boldsymbol{\beta}_0}} & {\widetilde{\boldsymbol{\beta}_1}} & {\widetilde{\boldsymbol{\beta}_2}} \\
	{\boldsymbol{\beta}_0} & {\boldsymbol{\beta}_1} & {\boldsymbol{\beta}_2}
	\arrow["\theta"{description}, from=1-1, to=1-2]
	\arrow["{\zeta }"{description, pos=0.3}, curve={height=-12pt}, from=1-1, to=1-3]
	\arrow["{\Theta^+}"{description, pos=0.3}, from=1-1, to=2-1]
	\arrow[from=1-1, to=2-2]
	\arrow[from=1-1, to=2-3]
	\arrow["{\theta }"{description}, from=1-2, to=1-3]
	\arrow["{\Theta^+}"{description, pos=0.3}, from=1-2, to=2-2]
	\arrow[from=1-2, to=2-3]
	\arrow["{\Theta^+}"{description, pos=0.3}, from=1-3, to=2-3]
	\arrow["\theta"{description}, from=2-1, to=2-2]
	\arrow["{\zeta }"{description, pos=0.3}, curve={height=12pt}, from=2-1, to=2-3]
	\arrow["\theta"{description}, from=2-2, to=2-3]
\end{tikzcd}
\end{equation} Hence, $\mu_{2}(-\otimes\underline{\Theta^{+}}):HF^{-}(\boldsymbol{\alpha},\underline{\widetilde{\boldsymbol{\beta}_{012}}})\to HF^{-}(\boldsymbol{\alpha},\underline{\boldsymbol{\beta}_{012})}$
is an isomorphism, and so the theorem follows.
\end{proof}
If $L_{0},L_{1},L_{2}$ form an unoriented skein triple, then $L_{1},L_{2},L_{0}$
and $L_{2},L_{0},L_{1}$ form unoriented skein triples. Hence, we
get an exact triangle 
\[
\cdots\to HF^{-}(\boldsymbol{\alpha},\boldsymbol{\beta}_{0})\to HF^{-}(\boldsymbol{\alpha},\boldsymbol{\beta}_{1})\to HF^{-}(\boldsymbol{\alpha},\boldsymbol{\beta}_{2})\to HF^{-}(\boldsymbol{\alpha},\boldsymbol{\beta}_{0})\to\cdots.
\]
Lemma \ref{lem:Let--be} says that the sum of the homological degrees
of three consecutive maps is $-1$. Compare \cite[Subsection 9.1]{nahm2025unorientedskeinexacttriangle}.
\begin{lem}
\label{lem:Let--be}Let $L_{0},L_{1},L_{2}$ be an unoriented skein
triple as above. Let $\boldsymbol{\alpha}$ be an attaching curve
such that the basepoint pair for the baseball $BB$ is a free basepoint
pair. Then, it is possible to assign the homological gradings such
that the gradings are as follows:
\[
HF_{a}^{-}(\boldsymbol{\alpha},\boldsymbol{\beta}_{0})\to HF_{b}^{-}(\boldsymbol{\alpha},\boldsymbol{\beta}_{1})\to HF_{c}^{-}(\boldsymbol{\alpha},\boldsymbol{\beta}_{2})\to HF_{a-1}^{-}(\boldsymbol{\alpha},\boldsymbol{\beta}_{0}).
\]
\end{lem}

\begin{proof}
Without loss of generality, cyclically permute $\boldsymbol{\beta}_{0},\boldsymbol{\beta}_{1},\boldsymbol{\beta}_{2}$
such that $BB$ is a free basepoint pair for $\boldsymbol{\beta}_{0}$.
Then, consider the $A_{\infty}$-category with objects $\boldsymbol{\alpha}<\boldsymbol{\beta}_{0}<\boldsymbol{\beta}_{1}<\boldsymbol{\beta}_{2}<\boldsymbol{\beta}_{0}'$
where $\boldsymbol{\beta}_{0}'$ is a standard translate of $\boldsymbol{\beta}_{0}$,
and consider a homological grading on this $A_{\infty}$-category.
Shift the homological gradings such that the top homological $\mathbb{Z}$-grading
element $\Theta^{+}\in HF^{-}(\boldsymbol{\beta}_{0},\boldsymbol{\beta}_{0}')$
has grading $0$, i.e. the canonical isomorphism $HF^{-}(\boldsymbol{\alpha},\boldsymbol{\beta}_{0})\to HF^{-}(\boldsymbol{\alpha},\boldsymbol{\beta}_{0}')$
has degree $0$. We only have to show that the sum of the homological
degrees of the band maps $\mu_{2}^{\boldsymbol{\alpha},\boldsymbol{\beta}_{0},\boldsymbol{\beta}_{1}}(-\otimes\theta)$,
$\mu_{2}^{\boldsymbol{\alpha},\boldsymbol{\beta}_{1},\boldsymbol{\beta}_{2}}(-\otimes\theta)$,
and $\mu_{2}^{\boldsymbol{\alpha},\boldsymbol{\beta}_{2},\boldsymbol{\beta}_{0}'}(-\otimes\theta)$
is $-1$, which follows from a combinatorial grading computation.
See Figure \ref{fig:drawing-genus1} for a local computation (the
general case is obtained by stabilizing): $\mu_{3}^{\boldsymbol{\beta}_{0},\boldsymbol{\beta}_{1},\boldsymbol{\beta}_{2},\boldsymbol{\beta}_{0}'}(\tau,\rho_{1}+\rho_{2},\sigma_{1}'+\sigma_{2}')$
is in the same homological grading as $\Theta^{+}$.
\end{proof}
\begin{figure}[h]
\begin{centering}
\includegraphics{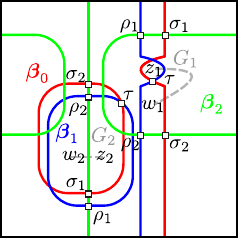}
\par\end{centering}
\caption{\label{fig:drawing-genus1}A genus $1$ diagram for Lemma \ref{lem:Let--be}.
The basepoint pair $(z_{2},w_{2})$ corresponds to the baseball $BB$.}
.
\end{figure}

\section{\label{sec:Swap-maps}Swap maps for balled links}

In the spirit of \cite[Section 3]{nahm2025unorientedskeinexacttriangle}
and Section \ref{sec:Band-maps-for}, we define another type of map
between balled links.
\begin{defn}
\label{def:Two-balled-links}Two balled links $L_{0},L_{1}$ are \emph{related
by a swap map $L_{0}\to L_{1}$} if there exists a setup as in Subsection
\ref{subsec:The-general-setup}, where $L_{0},L_{1}\in\boldsymbol{L}$
and their underlying links are identical.
\end{defn}

\begin{condition}
\label{cond:In-this-section,}In this section, further assume that
there are two baseballs $BB_{0},BB_{1}$ on the same component of
$L_{0}$ such that $LB(L_{1})\sqcup\{BB_{0}\}=LB(L_{0})\sqcup\{BB_{1}\}$
($LB(L)$ is the set of link baseballs of $L$).
\end{condition}

Fix $c\in H^{2}(Y;\mathbb{Z})$, and consider a $c$-strongly admissible
pair-pointed Heegaard diagram $(\Sigma,\boldsymbol{\alpha},\boldsymbol{\beta}_{0},\boldsymbol{\beta}_{1},\boldsymbol{p},\boldsymbol{G})$
as in Subsection \ref{subsec:The-Heegaard-diagram}. The definition
of the swap map $CF^{-}(\boldsymbol{\alpha},\boldsymbol{\beta}_{0})\to CF^{-}(\boldsymbol{\alpha},\boldsymbol{\beta}_{1})$
is similar to the definition of the merge and swap band maps from
Subsection \ref{subsec:Band-maps}. The sub Heegaard diagram for $\boldsymbol{\beta}_{0},\boldsymbol{\beta}_{1}$,
after isotopies, is a stabilization of the Heegaard diagram $K_{ori}$
of Figure \ref{fig:s1s2-important-1-1} (or equivalently, the right
hand side of Figure \ref{fig:local-ori}). Let $\theta\in CF_{top}^{-}(\boldsymbol{\beta}_{0},\boldsymbol{\beta}_{1})$
be the canonical element (Definition \ref{def:f-and-g}), and define
the swap map as
\[
\mu_{2}(-\otimes\theta):CF^{-}(\boldsymbol{\alpha},\boldsymbol{\beta}_{0})\to CF^{-}(\boldsymbol{\alpha},\boldsymbol{\beta}_{1}).
\]
This map is homogeneous with respect to the relative homological gradings
(see Subsection \ref{subsec:ainf}).
\begin{rem}
We can relax the assumption on $\boldsymbol{\alpha}$ similarly to
Remark \ref{rem:relax-band}; we only require that the basepoint pairs
that are contained in $BB_{0},BB_{1}$ are free basepoint pairs for
$\boldsymbol{\alpha}$.
\end{rem}

The same proof as Proposition \ref{prop:For-any-} gives the following.
\begin{prop}
\label{prop:For-any--1}For any $q\in L_{i}$, the swap map 
\[
F:HFL'{}^{-}(Y,L_{0};c)\to HFL'{}^{-}(Y,L_{1};c)
\]
is equivariant with respect to the action of $q$ (Definition \ref{def:hflaction}).
\end{prop}

\begin{proof}
Assume that $q$ is on the component of $L_{0}$ that the baseballs
$BB_{0},BB_{1}$ are on, since the proposition is immediate otherwise.
In this case, the statement is that $F(U_{0}^{1/2}\cdot-)=U_{1}^{1/2}F(-)$
on homology. This is a formal consequence of that the element $\theta\in HF^{-}(\boldsymbol{\beta}_{0},\boldsymbol{\beta}_{1})$
satisfies $U_{0}^{1/2}\theta=U_{1}^{1/2}\theta$.
\end{proof}
\begin{figure}[h]
\begin{centering}
\includegraphics{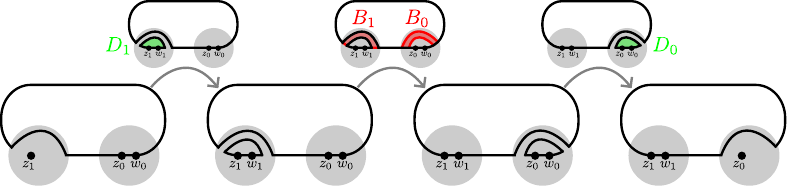}
\par\end{centering}
\caption{\label{fig:swap-schematic}We view swap maps as the composition of
three maps.}
\end{figure}

As in Subsection \ref{subsec:Merge-and-split}, the swap map can be
viewed as the composition of three maps: a creation, a mergesplit
pair, and an annihilation map (see Figure \ref{fig:swap-schematic}).
Hence, it can appropriately\footnote{For instance, if we want to view the swap map as the composition of
a split map on balled links and an annihilation map, then we have
to redefine the baseball $BB_{0}$ for the split map, i.e. ``shrink''
$BB_{0}$ such that $L_{1}\sqcup\partial D_{0}$ is a balled link.} also be viewed as the composition of a creation map and a merge map,
or a split map and an annihilation map.

We have the following.
\begin{prop}
\label{prop:Let--be}Let $L_{0},L_{1}$ be balled links in $S^{3}$
such that they are related by a swap move as above, and their underlying
link $L$ (which is identical) is an unlink. Then, under the identification
$HFL^{-}(S^{3},L_{i})\simeq Kh^{-}(L)$ of Proposition \ref{prop:Let--be-1},
the swap map is the identity.
\end{prop}

\begin{proof}
From the definition, the swap map is a composition of a creation map
and a merge map, and also a composition of a split map and an annihilation
map. We can compute the swap map for planar links using either description.
Alternatively, we can use Propositions \ref{prop:For-any--1} and
\ref{prop:swap-quasi-iso-balled}.
\end{proof}
\begin{prop}
\label{prop:swap-quasi-iso-balled}Swap maps are quasi-isomorphisms.\footnote{It is possible to show this by reducing to a local computation, as
in \cite[Theorem 4.3 and Section 9]{nahm2025unorientedskeinexacttriangle}
(but simpler, as we can use the computation for planar links). However,
this involves actually considering non-trivial local systems (more
precisely, we need to consider two attaching curves with non-trivial
local systems); for ease of exposition, we avoid mentioning local
systems in this paper, and present the following proof instead.}
\end{prop}

\begin{proof}
\begin{figure}[h]
\begin{centering}
\includegraphics{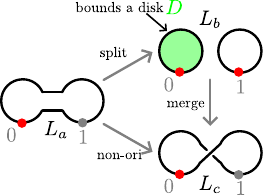}
\par\end{centering}
\caption{\label{fig:merge-swap-ncom-1}A schematic for an unoriented skein
exact triangle for Proposition \ref{prop:swap-quasi-iso-balled}.
The disk $D$ is disjoint from all the other components. The component
of $L_{b}$ that the baseball $BB_{1}$ is on may be knotted, and
may also be linked with other components. Also, only the components
that the baseball $BB_{0}$ or $BB_{1}$ is on are drawn.}
\end{figure}

For simplicity, we omit the Chern class from the notation. Let us
first show that the swap map $HFL'{}^{-}(Y,L_{0})\to HFL'{}^{-}(Y,L_{1})$
is surjective. As above, we view the swap map as the composite of
a split map $L_{0}\to L_{1}\sqcup\partial D$ for a disk $D$ and
an annihilation map $L_{1}\sqcup\partial D\to L_{1}$. Complete the
split map into an unoriented skein exact triangle as in Figure \ref{fig:merge-swap-ncom-1}:
we have $L_{a}=L_{0}$ and $L_{b}=L_{1}\sqcup\partial D$.

Let us call the split map $F_{s}$, and the merge map $F_{m}.$ We
have 
\[
HFL'{}^{-}(Y,L_{b})\simeq HFL'{}^{-}(Y,L_{1})\oplus U_{0}^{1/2}HFL'{}^{-}(Y,L_{1}),
\]
where this splitting is given by the disk $D$. On homology, $U_{0}^{1/2}F_{s}=U_{1}^{1/2}F_{s}$,
and so if $(a,U_{0}^{1/2}b)\in HFL'{}^{-}(Y,L_{b})$ is in the image
of $F_{s}$ for $a,b\in HFL'{}^{-}(Y,L_{1})$, then $a=U_{1}^{1/2}b$.
Hence, 
\[
{\rm Im}F_{s}\le(U_{0}^{1/2}+U_{1}^{1/2})HFL'(Y,L_{1}).
\]

Also, $F_{m}(U_{0}^{1/2}x)=F_{m}(U_{1}^{1/2}x)$ for all $x\in HFL'{}^{-}(Y,L_{b})$,
and so we have 
\[
(U_{0}^{1/2}+U_{1}^{1/2})HFL'(Y,L_{1})\le{\rm ker}F_{m}.
\]
Since Figure \ref{fig:merge-swap-ncom-1} forms an exact triangle,
we have ${\rm ker}F_{m}={\rm Im}F_{s}$, and so 
\[
{\rm Im}F_{s}=(U_{0}^{1/2}+U_{1}^{1/2})HFL'(Y,L_{1})={\rm ker}F_{m}.
\]

Hence, the swap map 
\[
HFL'{}^{-}(Y,L_{0})\xrightarrow{F_{s}}HFL'{}^{-}(Y,L_{1})\oplus U_{0}^{1/2}HFL'{}^{-}(Y,L_{1})\xrightarrow{\pi}U_{0}^{1/2}HFL'{}^{-}(Y,L_{1})\xrightarrow{U_{0}^{-1/2}}HFL'{}^{-}(Y,L_{1})
\]
is surjective.

Let us show that the swap map $HFL'{}^{-}(Y,L_{0})\to HFL'{}^{-}(Y,L_{1})$
is injective. We view the swap map as the composite of a creation
map $L_{0}\to L_{0}\sqcup\partial D$ for a disk $D$ and a merge
map $L_{0}\sqcup\partial D\to L_{1}$. Complete the merge map into
an unoriented skein exact triangle as in Figure \ref{fig:merge-swap-ncom-1}
(but swap the baseballs $BB_{0}$ and $BB_{1}$): $L_{b}=L_{0}\sqcup\partial D$
and $L_{c}=L_{1}$. The swap map is the composition
\[
HFL'{}^{-}(Y,L_{0})\hookrightarrow HFL'{}^{-}(Y,L_{0})\oplus U_{0}^{1/2}HFL'{}^{-}(Y,L_{0})\xrightarrow{F_{m}}HFL'{}^{-}(Y,L_{1}),
\]
where the first map is inclusion into the first summand. This is injective
since ${\rm ker}F_{m}=(U_{0}^{1/2}+U_{1}^{1/2})HFL'(Y,L_{0})$, and
so ${\rm ker}F_{m}\cap HFL'(Y,L_{0})=\{0\}$.
\end{proof}

\section{\label{sec:Some-important-balled}The unoriented link Floer homologies
of some simple links}

\begin{figure}[h]
\begin{centering}
\includegraphics{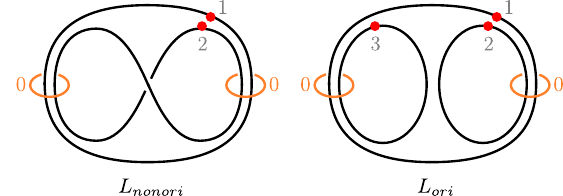}
\par\end{centering}
\caption{\label{fig:Some-important-balled}The balled links $L_{nonori},L_{ori}\subset\#^{2}S^{1}\times S^{2}$}
\end{figure}

In Sections \ref{sec:Band-maps-for} and \ref{sec:Swap-maps}, we
defined band maps and swap maps under some simplifying conditions
(Conditions \ref{cond:If--is} and \ref{cond:In-this-section,}).
To define these maps, we first studied the unoriented link Floer homology
of $K_{nonori}$ and $K_{ori}$ in $S^{1}\times S^{2}$ and defined
canonical elements $\theta$ in their unoriented link Floer homology
groups. For the general case, we also need to study the links $L_{nonori},L_{ori}$
in $\#^{2}S^{1}\times S^{2}$ (see Figure \ref{fig:Some-important-balled}).

For $L=L_{nonori},L_{ori}$, any Heegaard diagram that represents
$(\#^{2}S^{1}\times S^{2},L)$ is Alexander $\mathbb{Z}/2$-splittable,
by Lemma \ref{lem:alexander-two}. Also, there are exactly two ${\rm Spin}^{c}$-structures
with $c_{1}=0$, which differ by a meridian. Hence, the ${\rm Spin}^{c}$-splitting
on $HFL'{}^{-}(\#^{2}S^{1}\times S^{2},L;0)$ coincides with the relative
Alexander $\mathbb{Z}/2$-grading.

The following are the main results of this section.
\begin{prop}
\label{prop:For-each-balled}For the balled links $L=L_{nonori},L_{ori}$
of Figure \ref{fig:Some-important-balled}, $HFL'{}_{top}^{-}(\#^{2}S^{1}\times S^{2},L;0)$
is as follows, where the elements $a,b,c,d$ are homogeneous with
respect to the ${\rm Spin}^{c}$-splitting (equivalently, the relative
Alexander $\mathbb{Z}/2$-grading). 
\[\begin{tikzcd}[sep=large]
	a & b & a & b \\
	c & d & c & d
	\arrow["{\Phi _2}", from=1-1, to=1-2]
	\arrow["{\Phi _1 ,\Phi _2 }"{description}, from=1-1, to=2-1]
	\arrow["{\Phi _1 ,\Phi _2 }"{description}, from=1-2, to=2-2]
	\arrow["{\Phi _2 ,\Phi _3}", from=1-3, to=1-4]
	\arrow["{\Phi _1 ,\Phi _2 }"{description}, from=1-3, to=2-3]
	\arrow["{\Phi _1 ,\Phi _2 }"{description}, from=1-4, to=2-4]
	\arrow["{\Phi _2}", from=2-1, to=2-2]
	\arrow["{L_{nonori}}", shift left=5, draw=none, from=2-1, to=2-2]
	\arrow["{\Phi _2 ,\Phi _3}", from=2-3, to=2-4]
	\arrow["{L_{ori}}", shift left=5, draw=none, from=2-3, to=2-4]
\end{tikzcd}\]We have 
\[
{\rm gr}_{A}^{\mathbb{Z}/2}(a)={\rm gr}_{A}^{\mathbb{Z}/2}(d)\neq{\rm gr}_{A}^{\mathbb{Z}/2}(b)={\rm gr}_{A}^{\mathbb{Z}/2}(c).
\]

Furthermore, for all $\lambda\in H_{1}(\#^{2}S^{1}\times S^{2})$,
the degree $-1$ homology action $A_{\lambda,-1}=B_{\lambda,-1}$
vanishes on $HFL'{}_{top}^{-}(\#^{2}S^{1}\times S^{2},L_{ori};0)$.
\end{prop}

\begin{prop}
For every balled link $(Z,L)$ of Figures \ref{fig:s1s2-important-1-1}
and \ref{fig:Some-important-balled}, $HFL_{\mathbb{F}[U^{1/2}]}^{\prime-}(Z,L;0)$
(Definition \ref{def:unoriented-collapse}) is a free $\mathbb{F}[U^{1/2}]$-module.
\end{prop}

\begin{proof}
We already know this for $K_{nonori}$ and $K_{ori}$. We prove the
statement for $L_{ori}$ in Lemma \ref{lem:The-homology-}. We never
need the statement for $L_{nonori}$, but it can be deduced from Appendix
\ref{subsec:Direct-computations-for}.
\end{proof}
The $\Phi$ actions cannot distinguish $a,a+d\in HFL'{}_{top}^{-}(\#^{2}S^{1}\times S^{2},L_{ori};0)$,
but we can use the relative $H_{1}$ actions.
\begin{prop}
\label{prop:83}The composition
\[
HFL'{}^{-}(\#^{2}S^{1}\times S^{2},L_{ori};0)\xrightarrow{\rho}HFL'{}_{\mathbb{F}[U^{1/2}]}^{-}(\#^{2}S^{1}\times S^{2},L_{ori};0)\xrightarrow{\cdot U^{1/2}}HFL'{}_{\mathbb{F}[U^{1/2}]}^{-}(\#^{2}S^{1}\times S^{2},L_{ori};0)
\]
is injective on $HFL'{}_{top}^{-}(\#^{2}S^{1}\times S^{2},L_{ori};0)$.

Let $i\neq j$, $i,j\in\{1,2,3\}$, and let $\lambda$ be any path
between a link basepoint in the $i$th link baseball and a link basepoint
in the $j$th link baseball. For any $x\in HFL'{}_{top}^{-}(\#^{2}S^{1}\times S^{2},L_{ori};0)$,
there exists a (unique) $y\in HFL'{}_{top}^{-}(\#^{2}S^{1}\times S^{2},L_{ori};0)$
such that 
\[
A_{\mathbb{F}[U^{1/2}],\lambda,-1}(\rho(x))=U^{1/2}\rho(y).
\]
Furthermore, $y$ only depends on $x,i,j$, and does not depend on
the path $\lambda$. An analogous statement holds for $B_{\mathbb{F}[U^{1/2}],\lambda,-1}$.
In particular, the maps $A_{ii},B_{ii}$ are identically $0$.
\end{prop}

Hence, we introduce the following definition.
\begin{defn}
\label{def:relative-hom}Define 
\[
A_{ij},B_{ij}:HFL'{}_{top}^{-}(\#^{2}S^{1}\times S^{2},L_{ori};0)\to HFL'{}_{top}^{-}(\#^{2}S^{1}\times S^{2},L_{ori};0)
\]
such that for $x\in HFL'{}_{top}^{-}(\#^{2}S^{1}\times S^{2},L_{ori};0)$
and $\lambda$ a path between a link basepoint in the $i$th link
baseball and a link basepoint in the $j$th link baseball, 
\[
A_{\mathbb{F}[U^{1/2}],\lambda,-1}(\rho(x))=U^{1/2}\rho(A_{ij}(x)),\ B_{\mathbb{F}[U^{1/2}],\lambda,-1}(\rho(x))=U^{1/2}\rho(B_{ij}(x)).
\]
\end{defn}

We can compute these maps $A_{ij},B_{ij}$.
\begin{prop}
\label{prop:85}We can choose a basis $a,b,c,d$ of $HFL'{}_{top}^{-}(\#^{2}S^{1}\times S^{2},L_{ori};0)$
that is homogeneous with respect to the ${\rm Spin}^{c}$-splitting
(equivalently, the relative Alexander $\mathbb{Z}/2$-grading), such
that the $A_{ij}$ and $B_{ij}$ maps on $HFL'{}_{top}^{-}(\#^{2}S^{1}\times S^{2},L_{ori};0)$
are as follows.
\[\begin{tikzcd}[sep=large]
	a & b & a & b & a & b & a & b \\
	c & d & c & d & c & d & c & d
	\arrow["{\Phi _2 ,\Phi _3}", from=1-1, to=1-2]
	\arrow["{\Phi _1 ,\Phi _2 }"{description}, from=1-1, to=2-1]
	\arrow["{\Phi _1 ,\Phi _2 }"{description}, from=1-2, to=2-2]
	\arrow["{B_{23}}"{description}, curve={height=-12pt}, from=1-3, to=1-4]
	\arrow["{A_{12}}"{description}, curve={height=12pt}, from=1-3, to=2-3]
	\arrow["{B_{23}}"{description}, curve={height=-12pt}, from=1-4, to=1-3]
	\arrow["{A_{12}}"{description}, curve={height=12pt}, from=1-4, to=2-4]
	\arrow["{A_{23}}"{description}, curve={height=-12pt}, from=1-5, to=1-6]
	\arrow["{A_{23}}"{description}, from=1-5, to=2-5]
	\arrow["{A_{23}}"{description}, curve={height=-12pt}, from=1-6, to=1-5]
	\arrow["{A_{23}}"{description}, from=1-6, to=2-6]
	\arrow["{A_{13}}"{description}, curve={height=-12pt}, from=1-7, to=1-8]
	\arrow["{A_{13}}"{description}, curve={height=-12pt}, from=1-8, to=1-7]
	\arrow["{\Phi _2 ,\Phi _3}", from=2-1, to=2-2]
	\arrow["{A_{12}}"{description}, curve={height=12pt}, from=2-3, to=1-3]
	\arrow["{B_{23}}"{description}, curve={height=-12pt}, from=2-3, to=2-4]
	\arrow["{A_{12}}"{description}, curve={height=12pt}, from=2-4, to=1-4]
	\arrow["{B_{23}}"{description}, curve={height=-12pt}, from=2-4, to=2-3]
	\arrow["{A_{23}}"{description}, curve={height=-12pt}, from=2-5, to=2-6]
	\arrow["{A_{23}}"{description}, curve={height=-12pt}, from=2-6, to=2-5]
	\arrow["{A_{13}}"{description}, from=2-7, to=1-7]
	\arrow["{A_{13}}"{description}, curve={height=-12pt}, from=2-7, to=2-8]
	\arrow["{A_{13}}"{description}, from=2-8, to=1-8]
	\arrow["{A_{13}}"{description}, curve={height=-12pt}, from=2-8, to=2-7]
\end{tikzcd}\]
\end{prop}

It is possible to prove the above propositions by finding Heegaard
diagrams for $L_{nonori}$ and $L_{ori}$ such that the differential
does not depend on the almost complex structure and directly computing
the differentials (and the domains that contribute to the differentials):
see Appendix \ref{subsec:Direct-computations-for}. In this section,
however, we present proofs that instead utilize a couple of unoriented
skein exact sequences and avoid direct computation of the differentials.
Throughout this section, by $HF^{-}$ or $HFL'{}^{-}$, we always
mean the $c_{1}=0$ summands of these homology groups.

\begin{figure}[h]
\begin{centering}
\includegraphics{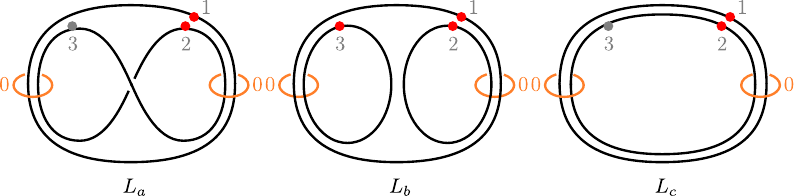}
\par\end{centering}
\caption{\label{fig:s1s2-links}An unoriented skein triple in $\#^{2}S^{1}\times S^{2}$.
The link baseballs are colored red, and the free baseballs are colored
grey.}
\end{figure}

The first unoriented skein triple we use is Figure \ref{fig:s1s2-links}.
We can compute the entirety of the unoriented link Floer homology
of $L_{c}$, since $(\#^{2}S^{1}\times S^{2},L_{c})$ is the connected
sum of $(S^{1}\times S^{2},K_{ori})$ and $(S^{1}\times S^{2},\emptyset)$.
Hence, we get 
\[
HFL'{}^{-}(\#^{2}S^{1}\times S^{2},L_{c};0)=\left(\mathbb{F}[U^{1/2}]f\oplus\mathbb{F}[U^{1/2}]g\right)\oplus\left(\mathbb{F}[U^{1/2}]f'\oplus\mathbb{F}[U^{1/2}]g'\right)
\]
where $f,g$ are as in Definition \ref{def:f-and-g} (they lie in
the top homological $\mathbb{Z}$-grading), and ${\rm gr}(f)-{\rm gr}(f')={\rm gr}(g)-{\rm gr}(g')=1$.
In fact, $f'=A_{\gamma}f$ and $g'=A_{\gamma}g$ where $\gamma$ is
a curve that goes through exactly one of the $0$-framed, orange circles,
and goes through it exactly once. Also, $f$ and $g$ lie in different
${\rm Spin}^{c}$-structures, and $\Phi_{1}f=\Phi_{2}f=g$ and $\Phi_{1}g=\Phi_{2}g=0$.

\begin{figure}[h]
\begin{centering}
\includegraphics{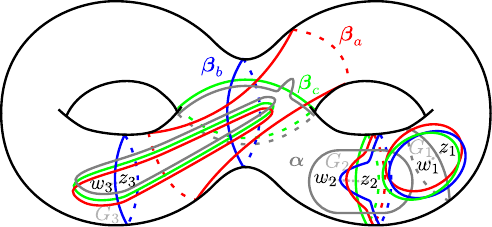}
\par\end{centering}
\caption{\label{fig:s1s2-heegaard}A weakly admissible pair-pointed Heegaard
diagram for the band maps between $L_{a},L_{b},L_{c}$; for each attaching
curve $\boldsymbol{\gamma}$ and each basepoint pair $(z_{i},w_{i})$,
let $(z_{i},w_{i})$ be a link basepoint pair for $\boldsymbol{\gamma}$
if and only if $\boldsymbol{\gamma}$ intersects the forbidding arc
$G_{i}$. }
\end{figure}

We will use the unoriented skein exact triangle to partially compute
the unoriented link Floer homology of $L_{a}$ and $L_{b}$. Consider
the pair-pointed Heegaard diagram given in Figure \ref{fig:s1s2-heegaard}.
The Heegaard diagram is weakly admissible, and all the intersection
points have $c_{1}=0$. Lift the relative homological $\mathbb{Z}$-gradings
to an absolute homological $\mathbb{Z}$-grading such that the top
grading of $CF^{-}(\boldsymbol{\alpha},\boldsymbol{\beta}_{i})$ is
$0$. Then, it is possible to check that the non-orientable band map
$L_{c}\to L_{a}$ and the split map $L_{a}\to L_{b}$ have degree
$0$ and that the merge map $L_{b}\to L_{c}$ has degree $-1$. Hence,
we have a long exact sequence
\begin{equation}
\cdots\to HF_{i}^{-}(\boldsymbol{\alpha},\boldsymbol{\beta}_{a})\to HF_{i}^{-}(\boldsymbol{\alpha},\boldsymbol{\beta}_{b})\to HF_{i-1}^{-}(\boldsymbol{\alpha},\boldsymbol{\beta}_{c})\to HF_{i-1}^{-}(\boldsymbol{\alpha},\boldsymbol{\beta}_{a})\to\cdots.\label{eq:unoriented-exact-sequence1}
\end{equation}

\subsection{\label{subsec:A-lower-bound-of-la}A lower bound of \texorpdfstring{$\dim_{\mathbb{F}}HFL_{top}^{-}(Y,L_{nonori};0)$}{dim HF-top(Y, Lnonori; 0)}}

Consider the composition map
\begin{equation}
\mu_{2}:HF_{0}^{-}(\boldsymbol{\alpha},\boldsymbol{\beta}_{c})\otimes HF_{0}^{-}(\boldsymbol{\beta}_{c},\boldsymbol{\beta}_{a})\to HF_{0}^{-}(\boldsymbol{\alpha},\boldsymbol{\beta}_{a}).\label{eq:mu2-topgrading-nonori}
\end{equation}
Let $f,g\in HF_{0}^{-}(\boldsymbol{\alpha},\boldsymbol{\beta}_{c})$
and $f,g\in HF_{0}^{-}(\boldsymbol{\beta}_{c},\boldsymbol{\beta}_{a})$
be as in Definition \ref{def:f-and-g}. Then, $HF_{0}^{-}(\boldsymbol{\alpha},\boldsymbol{\beta}_{c})\otimes HF_{0}^{-}(\boldsymbol{\beta}_{c},\boldsymbol{\beta}_{a})$
has rank $4$ over $\mathbb{F}$, and the $\Phi$ actions are as follows.
\[\begin{tikzcd}
	{f\otimes f} & {f\otimes g} \\
	{g\otimes f} & {g\otimes g}
	\arrow["{\Phi _2 }", from=1-1, to=1-2]
	\arrow["{\Phi _1,\Phi _2 }"', from=1-1, to=2-1]
	\arrow["{\Phi _1,\Phi _2 }", from=1-2, to=2-2]
	\arrow["{\Phi _2 }"', from=2-1, to=2-2]
\end{tikzcd}\]Since $HF_{1}^{-}(\boldsymbol{\alpha},\boldsymbol{\beta}_{b})=0$,
the non-orientable band map $\mu_{2}(-\otimes g):HF_{0}^{-}(\boldsymbol{\alpha},\boldsymbol{\beta}_{c})\to HF_{0}^{-}(\boldsymbol{\alpha},\boldsymbol{\beta}_{a})$
is injective. Since $\mu_{2}$ is equivariant with respect to the
$\Phi$ actions, we can deduce that Equation (\ref{eq:mu2-topgrading-nonori})
is injective, and so $\dim_{\mathbb{F}}HF_{0}^{-}(\boldsymbol{\alpha},\boldsymbol{\beta}_{a})\ge4$.

Also, we can check that the sub Heegaard diagram with attaching curves
$\boldsymbol{\alpha},\boldsymbol{\beta}_{c},\boldsymbol{\beta}_{a}$
is Alexander $\mathbb{Z}/2$-splittable (in fact, there is a general
statement, Proposition \ref{prop:combi-claims} (\ref{enu:The--category--1})),
and so Equation (\ref{eq:mu2-topgrading-nonori}) is homogeneous with
respect to the relative Alexander $\mathbb{Z}/2$-grading. Hence,
the images $\mu_{2}(f\otimes f)$, $\mu_{2}(f\otimes g)$, $\mu_{2}(g\otimes f)$,
and $\mu_{2}(g\otimes g)$ are also homogeneous with respect to the
relative Alexander $\mathbb{Z}/2$-grading.

\subsection{\label{subsec:An-upper-bound}An upper bound of \texorpdfstring{$\dim_{\mathbb{F}}HFL_{top}^{-}(Y,L_{nonori};0)$}{dim HF-top(Y, Lnonori; 0)}}

We claim that $\dim_{\mathbb{F}}HF_{0}^{-}(\boldsymbol{\alpha},\boldsymbol{\beta}_{a})\le4$,
and hence that Equation \ref{eq:mu2-topgrading-nonori} is an isomorphism.

This follows from a purely algebraic statement. First, we can destabilize
the diagram $(\boldsymbol{\alpha},\boldsymbol{\beta}_{a})$ (remove
the region with $w_{3},z_{3}$) and consider a genus $2$ diagram
with three circles for each attaching curve. Call the corresponding
attaching curves $\boldsymbol{\alpha}^{u},\boldsymbol{\beta}_{a}^{u}$.
All the generators are in the top homological $\mathbb{Z}$-grading;
let this grading be $0$. Write the differential $\partial=U_{1}^{1/2}\partial_{1}+U_{2}^{1/2}\partial_{2}$
where $\partial_{1},\partial_{2}$ are $\mathbb{F}[U_{1}^{1/2},U_{2}^{1/2}]$-linear,
and they map $CF_{0}^{-}(\boldsymbol{\alpha}^{u},\boldsymbol{\beta}_{a}^{u})$
into $CF_{0}^{-}(\boldsymbol{\alpha}^{u},\boldsymbol{\beta}_{a}^{u})$.
Then, $\partial^{2}=0$ gives 
\[
\partial_{1}^{2}=0,\ \partial_{2}^{2}=0,\ \partial_{1}\partial_{2}+\partial_{2}\partial_{1}=0.
\]
Let $x\in CF_{0}^{-}(\boldsymbol{\alpha}^{u},\boldsymbol{\beta}_{a}^{u})$
satisfy $\partial x=0$. Then, since multiplication by $U_{1}$ is
homotopic to multiplication by $U_{2}$, there exists some $y\in CF^{-}(\boldsymbol{\alpha}^{u},\boldsymbol{\beta}_{a}^{u})$
such that $\partial y=(U_{1}+U_{2})x$. Let $y_{1},y_{2}\in CF_{0}^{-}(\boldsymbol{\alpha}^{u},\boldsymbol{\beta}_{a}^{u})$
be such that $y=U_{1}^{1/2}y_{1}+U_{2}^{1/2}y_{2}$. Expanding 
\[
(U_{1}^{1/2}\partial_{1}+U_{2}^{1/2}\partial_{2})(U_{1}^{1/2}y_{1}+U_{2}^{1/2}y_{2})=(U_{1}+U_{2})x,
\]
we get 
\[
U_{1}\partial_{1}y_{1}+U_{1}^{1/2}U_{2}^{1/2}(\partial_{2}y_{1}+\partial_{1}y_{2})+U_{2}\partial_{2}y_{2}=(U_{1}+U_{2})x.
\]
Hence, in particular, $x\in{\rm Im}\partial_{1}$. Since 
\[
\dim_{\mathbb{F}}({\rm Im}\partial_{1}\cap CF_{0}^{-}(\boldsymbol{\alpha}^{u},\boldsymbol{\beta}_{a}^{u}))\le\frac{1}{2}\dim_{\mathbb{F}}CF_{0}^{-}(\boldsymbol{\alpha}^{u},\boldsymbol{\beta}_{a}^{u})=4,
\]
we get that $\dim_{\mathbb{F}}HF_{0}^{-}(\boldsymbol{\alpha}^{u},\boldsymbol{\beta}_{a}^{u})\le4$.
The statement of Proposition \ref{prop:For-each-balled} for $L_{nonori}$
follows.

\subsection{\label{subsec:An-upper-bound-1}An upper bound for $L_{ori}$}

We can in fact say a bit more in the setting of Subsection \ref{subsec:An-upper-bound}:
for any nonzero $x\in HF_{0}^{-}(\boldsymbol{\alpha}^{u},\boldsymbol{\beta}_{a}^{u})$,
we have $U_{1}^{1/2}x,U_{2}^{1/2}x\neq0$ in $HF_{-1}^{-}(\boldsymbol{\alpha}^{u},\boldsymbol{\beta}_{a}^{u})$.
Indeed, assume $U_{1}^{1/2}x=0$. Then there exists some $y\in CF_{0}^{-}(\boldsymbol{\alpha}^{u},\boldsymbol{\beta}_{a}^{u})$
such that $\partial_{1}y=x$, $\partial_{2}y=0$. However, since $\dim_{\mathbb{F}}{\rm ker}\partial\cap CF_{0}^{-}(\boldsymbol{\alpha}^{u},\boldsymbol{\beta}_{a}^{u})=4$
and 
\[
{\rm ker}\partial\cap CF_{0}^{-}(\boldsymbol{\alpha}^{u},\boldsymbol{\beta}_{a}^{u})\le{\rm Im}\partial_{i}\cap CF_{0}^{-}(\boldsymbol{\alpha}^{u},\boldsymbol{\beta}_{a}^{u})
\]
for $i=1,2$, they are equal. Hence, 
\[
{\rm Im}\partial_{i}\cap CF_{0}^{-}(\boldsymbol{\alpha}^{u},\boldsymbol{\beta}_{a}^{u})={\rm ker}\partial_{i}\cap CF_{0}^{-}(\boldsymbol{\alpha}^{u},\boldsymbol{\beta}_{a}^{u})
\]
for $i=1,2$ as well. Now, since $y\in{\rm ker}\partial_{2}$, we
have $y\in\ker\partial_{1}$, and so $x=0$.

Let us return to the original setting (Figure \ref{fig:s1s2-heegaard}).
We claim $\dim_{\mathbb{F}}HF_{0}^{-}(\boldsymbol{\alpha},\boldsymbol{\beta}_{b})\le4$.
The homology group $HF_{-1}^{-}(\boldsymbol{\alpha},\boldsymbol{\beta}_{c})$
has basis $U^{1/2}f$, $U^{1/2}g$, $f'$, $g'$. We have shown that
any nonzero element in the $2$-dimensional subspace ${\rm span}_{\mathbb{F}}\left\langle U^{1/2}f,U^{1/2}g\right\rangle $
does not vanish under $HF_{-1}^{-}(\boldsymbol{\alpha},\boldsymbol{\beta}_{c})\to HF_{-1}^{-}(\boldsymbol{\alpha},\boldsymbol{\beta}_{a})$.
Hence, the kernel of $HF_{-1}^{-}(\boldsymbol{\alpha},\boldsymbol{\beta}_{c})\to HF_{-1}^{-}(\boldsymbol{\alpha},\boldsymbol{\beta}_{a})$
has dimension $\le2$. The claim follows from the unoriented skein
exact sequence, Equation (\ref{eq:unoriented-exact-sequence1}). 

\subsection{\label{subsec:relative-homology}A lower bound for $L_{ori}$ and
actions on \texorpdfstring{$HFL_{top}^{-}(Y,L_{ori};0)$}{dim HF-top(Y, Lori; 0)}}

\begin{figure}[h]
\begin{centering}
\includegraphics{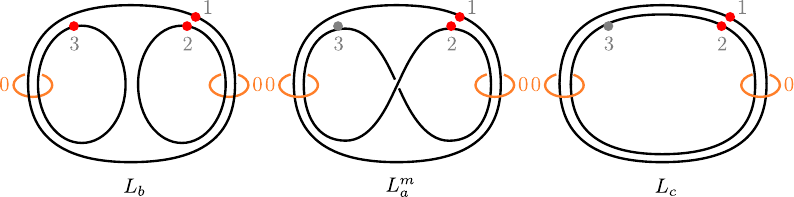}
\par\end{centering}
\caption{\label{fig:s1s2-links-1}Another unoriented skein triple in $\#^{2}S^{1}\times S^{2}$.
The link baseballs are colored red, and the free baseballs are colored
grey. In fact, $L_{a}^{m}$ and $L_{a}$ are isotopic.}
\end{figure}

\begin{figure}[h]
\begin{centering}
\includegraphics{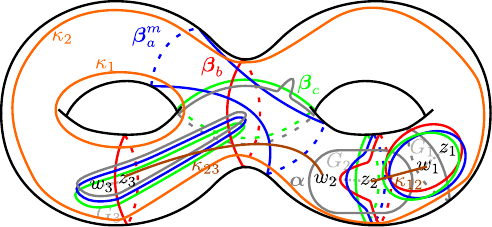}
\par\end{centering}
\caption{\label{fig:s1s2-heegaard-1}A weakly admissible pair-pointed Heegaard
diagram for the band maps between the links in Figure \ref{fig:s1s2-links-1}
together with loops $\kappa_{1},\kappa_{2}$ and paths $\kappa_{12},\kappa_{23}$
on the Heegaard surface; for each attaching curve $\boldsymbol{\gamma}$
and each basepoint pair $(z_{i},w_{i})$, let $(z_{i},w_{i})$ be
a link basepoint pair for $\boldsymbol{\gamma}$ if and only if $\boldsymbol{\gamma}$
intersects the forbidding arc $G_{i}$.}
\end{figure}

We will use the unoriented skein triple in Figure \ref{fig:s1s2-links-1}
to show that $\dim_{\mathbb{F}}HF_{0}^{-}(\boldsymbol{\alpha},\boldsymbol{\beta}_{b})=4$
and compute the $\Phi$, $H_{1}$, and relative $H_{1}$ actions on
$HF_{0}^{-}(\boldsymbol{\alpha},\boldsymbol{\beta}_{b})$. Note that
the attaching curves $\boldsymbol{\alpha},\boldsymbol{\beta}_{b},\boldsymbol{\beta}_{c}$
of Figures \ref{fig:s1s2-heegaard} and \ref{fig:s1s2-heegaard-1}
are the same.

All the generators have $c_{1}=0$; lift the relative homological
$\mathbb{Z}$-grading to an absolute homological $\mathbb{Z}$-grading
such that the top homological $\mathbb{Z}$-grading generators of
$CF^{-}(\boldsymbol{\alpha},\boldsymbol{\beta}_{c})$, $CF^{-}(\boldsymbol{\alpha},\boldsymbol{\beta}_{b})$,
and $CF^{-}(\boldsymbol{\alpha},\boldsymbol{\beta}_{a}^{m})$ have
$\mathbb{Z}$-grading $0$. Then, we have an unoriented skein exact
sequence
\[
\cdots\to HF_{i}^{-}(\boldsymbol{\alpha},\boldsymbol{\beta}_{b})\to HF_{i}^{-}(\boldsymbol{\alpha},\boldsymbol{\beta}_{a}^{m})\to HF_{i-1}^{-}(\boldsymbol{\alpha},\boldsymbol{\beta}_{c})\to HF_{i-1}^{-}(\boldsymbol{\alpha},\boldsymbol{\beta}_{b})\to\cdots.
\]
In particular, the band map $HF_{0}^{-}(\boldsymbol{\alpha},\boldsymbol{\beta}_{c})\to HF_{0}^{-}(\boldsymbol{\alpha},\boldsymbol{\beta}_{b})$
is injective. Since $HF_{0}^{-}(\boldsymbol{\alpha},\boldsymbol{\beta}_{c})={\rm span}_{\mathbb{F}}\left\langle f,g\right\rangle \neq0$,
the top homological $\mathbb{Z}$-grading of $HF^{-}(\boldsymbol{\alpha},\boldsymbol{\beta}_{b})$
is $0$.

We can now prove the following lemma.
\begin{lem}
\label{lem:The-homology-}The homology $HFL'{}_{\mathbb{F}[U^{1/2}]}^{-}(\#^{2}S^{1}\times S^{2},L_{ori};0)$
(Definition \ref{def:unoriented-collapse}) is a free $\mathbb{F}[U^{1/2}]$-module,
and $U^{1/2}\rho$ is injective on $HFL'{}_{top}^{-}(\#^{2}S^{1}\times S^{2},L_{ori};0)$.
\end{lem}

\begin{proof}
Consider the attaching curves $\boldsymbol{\alpha},\boldsymbol{\beta}_{b}$
of Figures \ref{fig:s1s2-heegaard} and \ref{fig:s1s2-heegaard-1}.
First, $\rho$ is injective on $HFL'{}_{top}^{-}(\#^{2}S^{1}\times S^{2},L_{ori};0)$
since all the intersection points $\boldsymbol{\alpha}\cap\boldsymbol{\beta}_{b}$
lie in the top homological grading.

Now, we show that $HFL'{}_{\mathbb{F}[U^{1/2}]}^{-}(\#^{2}S^{1}\times S^{2},L_{ori};0)$
is a free $\mathbb{F}[U^{1/2}]$-module; this implies that multiplication
by $U^{1/2}$ is injective. In fact, we claim that the differential
on $CF_{\mathbb{F}[U^{1/2}]}^{-}(\boldsymbol{\alpha},\boldsymbol{\beta}_{b})$
is identically zero. If we work over the ring ${\cal R}=\mathbb{F}[\{Z_{i}\},\{W_{i}\}]$,
for any two intersection points $a$ and $b$, the coefficient of
$b$ in $a$ is some linear combination of the $Z_{i}$'s and the
$W_{i}$'s. We would like to show that if we identify $Z_{i}=W_{i}=U^{1/2}$
for all $i$, then the coefficient is identically zero. Since all
the intersection points lie in the same homological grading, the coefficient
is an $\mathbb{F}$-linear combination of the $Z_{i}$'s and the $W_{i}$'s,
and so it is sufficient to show that the coefficient is identically
zero if we identify $Z_{i}=W_{i}=1$.

Observe that $(\Sigma,\boldsymbol{\alpha},\boldsymbol{\beta}_{b},\{z_{1},z_{2},z_{3}\})$
is weakly admissible and that all the intersection points have $c_{1}=0$.
Its homology group $HF_{\mathbb{F}[U],\boldsymbol{z}}^{-}(\boldsymbol{\alpha},\boldsymbol{\beta}_{b})$
is the Heegaard Floer homology of $\#^{2}S^{1}\times S^{2}$ with
three basepoints, so it is free over $\mathbb{F}[U]$, of rank $16$.
Hence, the differential is zero; in other words, the differential
over ${\cal R}$ is zero after identifying $Z_{i}=U$, $W_{i}=1$
for all $i$.
\end{proof}
Consider the composition 
\begin{equation}
\mu_{2}:HF_{0}^{-}(\boldsymbol{\alpha},\boldsymbol{\beta}_{c})\otimes HF_{0}^{-}(\boldsymbol{\beta}_{c},\boldsymbol{\beta}_{b})\to HF_{0}^{-}(\boldsymbol{\alpha},\boldsymbol{\beta}_{b}).\label{eq:mu2-topgrading-nonori-1}
\end{equation}
Using the $\Phi$ and relative $H_{1}$-actions, we will show that
Equation (\ref{eq:mu2-topgrading-nonori-1}) is an isomorphism, and
prove Propositions \ref{prop:For-each-balled}, \ref{prop:83}, and
\ref{prop:85}.

\subsubsection{\label{subsec:841}The $\Phi$ and relative $H_{1}$ actions on $HF_{0}^{-}(\boldsymbol{\alpha},\boldsymbol{\beta}_{b})$}

Recall that the coefficient rings of $HF^{-}(\boldsymbol{\alpha},\boldsymbol{\beta}_{c})$,
$HF^{-}(\boldsymbol{\beta}_{c},\boldsymbol{\beta}_{b})$, and $HF^{-}(\boldsymbol{\alpha},\boldsymbol{\beta}_{b})$
are not the same; the composition $\mu_{2}^{\boldsymbol{\alpha},\boldsymbol{\beta}_{c},\boldsymbol{\beta}_{b}}$
is defined as the composition (Equation (\ref{eq:mud-defn}))
\[
HF^{-}(\boldsymbol{\alpha},\boldsymbol{\beta}_{c})\otimes HF^{-}(\boldsymbol{\beta}_{c},\boldsymbol{\beta}_{b})\xrightarrow{\iota\otimes\iota}HF_{R}^{-}(\boldsymbol{\alpha},\boldsymbol{\beta}_{c})\otimes HF_{R}^{-}(\boldsymbol{\beta}_{c},\boldsymbol{\beta}_{b})\xrightarrow{\mu_{2,R}}HF_{R}^{-}(\boldsymbol{\alpha},\boldsymbol{\beta}_{b})=HF^{-}(\boldsymbol{\alpha},\boldsymbol{\beta}_{b}),
\]
where $R=\mathbb{F}[U_{1}^{1/2},U_{2}^{1/2},U_{3}^{1/2}]$. Note that
$\iota$ restricts to isomorphisms $HF_{0}^{-}(\boldsymbol{\alpha},\boldsymbol{\beta}_{c})\to HF_{R,0}^{-}(\boldsymbol{\alpha},\boldsymbol{\beta}_{c})$
and $HF_{0}^{-}(\boldsymbol{\beta}_{c},\boldsymbol{\beta}_{b})\to HF_{R,0}^{-}(\boldsymbol{\beta}_{c},\boldsymbol{\beta}_{b})$.
Let  
\[
f,g\in HF_{0}^{-}(\boldsymbol{\alpha},\boldsymbol{\beta}_{c})=HF_{R,0}^{-}(\boldsymbol{\alpha},\boldsymbol{\beta}_{c}),\ f,g\in HF_{0}^{-}(\boldsymbol{\beta}_{c},\boldsymbol{\beta}_{b})=HF_{R,0}^{-}(\boldsymbol{\beta}_{c},\boldsymbol{\beta}_{b})
\]
be as in Definition \ref{def:f-and-g}.

We understand the action $A_{\mathbb{F}[U^{1/2}],\kappa_{12},-1}$
on $HF_{\mathbb{F}[U^{1/2}],0}^{-}(\boldsymbol{\alpha},\boldsymbol{\beta}_{c})$
and the action $B_{\mathbb{F}[U^{1/2}],\kappa_{23},-1}$ on $HF_{\mathbb{F}[U^{1/2}],0}^{-}(\boldsymbol{\beta}_{c},\boldsymbol{\beta}_{b})$
by Subsections \ref{subsec:Model-computations} and \ref{subsec:Connected-sums}.
We have 
\begin{gather*}
A_{\mathbb{F}[U^{1/2}],\kappa_{12},-1}\rho(f)=U^{1/2}\rho(g),\ A_{\mathbb{F}[U^{1/2}],\kappa_{12},-1}\rho(g)=U^{1/2}\rho(f)\in HF_{\mathbb{F}[U^{1/2}],-1}^{-}(\boldsymbol{\alpha},\boldsymbol{\beta}_{c}),\\
B_{\mathbb{F}[U^{1/2}],\kappa_{23},-1}\rho(f)=U^{1/2}\rho(g),\ B_{\mathbb{F}[U^{1/2}],\kappa_{23},-1}\rho(g)=U^{1/2}\rho(f)\in HF_{\mathbb{F}[U^{1/2}],-1}^{-}(\boldsymbol{\beta}_{c},\boldsymbol{\beta}_{b}).
\end{gather*}
Hence, by Lemma \ref{lem:mu2-ab} and since $\rho$ commutes with
$\mu_{2}$, for any $x\in\mu_{2}(HF_{0}^{-}(\boldsymbol{\alpha},\boldsymbol{\beta}_{c})\otimes HF_{0}^{-}(\boldsymbol{\beta}_{c},\boldsymbol{\beta}_{b}))$,
there exists some $y\in\mu_{2}(HF_{0}^{-}(\boldsymbol{\alpha},\boldsymbol{\beta}_{c})\otimes HF_{0}^{-}(\boldsymbol{\beta}_{c},\boldsymbol{\beta}_{b}))$
such that $A_{\mathbb{F}[U^{1/2}],\kappa_{12},-1}\rho(x)=U^{1/2}\rho(y)$.
Define $A_{\kappa_{12},0}(x):=y$; then 
\[
A_{\kappa_{12},0}:\mu_{2}(HF_{0}^{-}(\boldsymbol{\alpha},\boldsymbol{\beta}_{c})\otimes HF_{0}^{-}(\boldsymbol{\beta}_{c},\boldsymbol{\beta}_{b}))\to\mu_{2}(HF_{0}^{-}(\boldsymbol{\alpha},\boldsymbol{\beta}_{c})\otimes HF_{0}^{-}(\boldsymbol{\beta}_{c},\boldsymbol{\beta}_{b}))
\]
is well-defined since $U^{1/2}\rho$ is injective on $HF_{0}^{-}(\boldsymbol{\alpha},\boldsymbol{\beta}_{b})$.
The same statement holds for $B_{\mathbb{F}[U^{1/2}],\kappa_{23},-1}$;
similarly define $B_{\kappa_{23,0}}$.

The $\Phi$, $A_{\kappa_{12},0}$, and $B_{\kappa_{23},0}$ actions
on $\mu_{2}(HF_{0}^{-}(\boldsymbol{\alpha},\boldsymbol{\beta}_{c})\otimes HF_{0}^{-}(\boldsymbol{\beta}_{c},\boldsymbol{\beta}_{b}))$
are as in Diagram (\ref{diag:8actions}). 
\begin{equation}
\label{diag:8actions}
\begin{tikzcd}[sep=large]
	{\mu_2 (f\otimes f)} & {\mu_2 (f\otimes g)} & {\mu_2 (f\otimes f)} & {\mu_2 (f\otimes g)} \\
	{\mu_2 (g\otimes f)} & {\mu_2 (g\otimes g)} & {\mu_2 (g\otimes f)} & {\mu_2 (g\otimes g)}
	\arrow["{\Phi _2 ,\Phi _3}", from=1-1, to=1-2]
	\arrow["{\Phi _1 ,\Phi _2 }"{description}, from=1-1, to=2-1]
	\arrow["{\Phi _1 ,\Phi _2 }"{description}, from=1-2, to=2-2]
	\arrow["{B_{\kappa_{23} , 0}}"{description}, curve={height=-12pt}, from=1-3, to=1-4]
	\arrow["{A_{\kappa _{12} , 0}}"', curve={height=12pt}, from=1-3, to=2-3]
	\arrow["{B_{\kappa_{23} , 0}}"{description}, curve={height=-12pt}, from=1-4, to=1-3]
	\arrow["{A_{\kappa _{12} , 0}}"', curve={height=12pt}, from=1-4, to=2-4]
	\arrow["{\Phi _2 ,\Phi _3}", from=2-1, to=2-2]
	\arrow["{A_{\kappa _{12} , 0}}"', curve={height=12pt}, from=2-3, to=1-3]
	\arrow["{B_{\kappa_{23} , 0}}"{description}, curve={height=-12pt}, from=2-3, to=2-4]
	\arrow["{A_{\kappa _{12} , 0}}"', curve={height=12pt}, from=2-4, to=1-4]
	\arrow["{B_{\kappa_{23} , 0}}"{description}, curve={height=-12pt}, from=2-4, to=2-3]
\end{tikzcd}
\end{equation}Also, the band map $HF_{0}^{-}(\boldsymbol{\alpha},\boldsymbol{\beta}_{c})\to HF_{0}^{-}(\boldsymbol{\alpha},\boldsymbol{\beta}_{b})$
is $\mu_{2}(-\otimes f)$, and so $\mu_{2}(-\otimes f)$ is injective
on $HF_{0}^{-}(\boldsymbol{\alpha},\boldsymbol{\beta}_{c})$. From
this together with the algebraic structures in Diagram (\ref{diag:8actions}),
one can deduce that Equation (\ref{eq:mu2-topgrading-nonori-1}) is
injective. Hence, $\dim_{\mathbb{F}}HF_{0}^{-}(\boldsymbol{\alpha},\boldsymbol{\beta}_{b})\ge4$,
and so by Subsection \ref{subsec:An-upper-bound-1}, Equation (\ref{eq:mu2-topgrading-nonori-1})
is an isomorphism.

Similarly to Subsection \ref{subsec:A-lower-bound-of-la}, we can
check that the sub Heegaard diagram with attaching curves $\boldsymbol{\alpha},\boldsymbol{\beta}_{c},\boldsymbol{\beta}_{b}$
is Alexander $\mathbb{Z}/2$-splittable, and so $\mu_{2}(f\otimes f)$,
$\mu_{2}(f\otimes g)$, $\mu_{2}(g\otimes f)$, and $\mu_{2}(g\otimes g)$
are homogeneous with respect to the relative Alexander $\mathbb{Z}/2$-grading.

\subsubsection{The $H_{1}(\#^{2}S^{1}\times S^{2})$ action on $HF_{0}^{-}(\boldsymbol{\alpha},\boldsymbol{\beta}_{b})$}

By Subsection \ref{subsec:Model-computations}, the action $A_{R,\kappa_{2},-1}$
vanishes on $HF_{R,0}^{-}(\boldsymbol{\alpha},\boldsymbol{\beta}_{c})$,
and the action $B_{R,\kappa_{1},-1}$ vanishes on $HF_{R,0}^{-}(\boldsymbol{\beta}_{c},\boldsymbol{\beta}_{b})$.
Since Equation (\ref{eq:mu2-topgrading-nonori-1}) is an isomorphism,
the actions $A_{R,\kappa_{2},-1}$ and $B_{R,\kappa_{1},-1}$ vanish
on $HF_{R,0}^{-}(\boldsymbol{\alpha},\boldsymbol{\beta}_{b})$ by
Lemma \ref{lem:mu2-ab}. Hence, the $H_{1}(\#^{2}S^{1}\times S^{2})$
action is trivial on $HF_{0}^{-}(\boldsymbol{\alpha},\boldsymbol{\beta}_{b})$
by Lemma \ref{lem:ab-additive}. This concludes the proof of Proposition
\ref{prop:For-each-balled}.

\subsubsection{The maps $A_{ij}$ and $B_{ij}$ do not depend on $\lambda$}

Fix $i,j\in\{1,2,3\}$ and $x\in HF_{0}^{-}(\boldsymbol{\alpha},\boldsymbol{\beta}_{b})$.
Let us consider paths $\lambda$ between $u_{i},u_{j}$ such that
$u_{i}\in\{z_{i},w_{i}\}$ and $u_{j}\in\{z_{j},w_{j}\}$. We claim
that $A_{\mathbb{F}[U^{1/2}],\lambda,-1}(\rho(x))$ and $B_{\mathbb{F}[U^{1/2}],\lambda,-1}(\rho(x))$
do not depend on $\lambda,u_{i},u_{j}$. First, it is sufficient to
prove this for $A_{\mathbb{F}[U^{1/2}],\lambda,-1}\rho(x)$ by Lemma
\ref{lem:a+b}. We can find a path $\kappa$ between $z_{i},w_{i}$
such that $\kappa$ is disjoint from $\boldsymbol{\alpha}$, and so
$A_{\mathbb{F}[U^{1/2}],\kappa,-1}\rho(x)=0$. Also, since the $H_{1}(\#^{2}S^{1}\times S^{2})$
action is trivial on $HF_{0}^{-}(\boldsymbol{\alpha},\boldsymbol{\beta}_{b})$,
if $\kappa$ is a loop, then $A_{\mathbb{F}[U^{1/2}],\kappa,-1}(\rho(x))=\rho(A_{R,\lambda,-1}(x))=0$.
Hence, $A_{\mathbb{F}[U^{1/2}],\lambda,-1}(\rho(x))$ does not depend
on the $\lambda,u_{i},u_{j}$ by Lemma \ref{lem:ab-additive}. This
concludes the proof of Proposition \ref{prop:83}.

\subsubsection{The maps $A_{23}$ and $A_{13}$}

We have computed $A_{12}$ and $B_{23}$. The map $A_{23}=B_{23}+\Phi_{2}+\Phi_{3}$
by Lemma \ref{lem:a+b}, and $A_{13}=A_{12}+A_{23}$ by Lemma \ref{lem:ab-additive}.
This concludes the proof of Proposition \ref{prop:85}.

\section{\label{sec:Gradings-for-the}Gradings for the setup}

In this section, we study gradings of the Heegaard diagrams defined
in Section \ref{sec:The-setup-and}.

\subsection{\label{subsec:Standard-generators}Standard generators and the main
proposition}

We call an intersection point ${\bf x}\in\boldsymbol{\beta}(L_{1})\cap\boldsymbol{\beta}(L_{2})$
of the Heegaard diagram ${\cal H}$ \emph{standard} if it lies in
the top homological $\mathbb{Z}$-grading. More precisely:
\begin{defn}
\label{def:standard-gen}Let $L_{1}<L_{2}$. Note that an intersection
point (\emph{``generator''}) ${\bf x}\in\boldsymbol{\beta}(L_{1})\cap\boldsymbol{\beta}(L_{2})$
of ${\cal H}$ consists of one point in each special region. We call
${\bf x}$ a \emph{standard generator} if for each special region,
if $\boldsymbol{\beta}(L_{1})$ and $\boldsymbol{\beta}(L_{2})$ in
that region are standard translates of each other, then the corresponding
point in that special region is the top grading intersection point.

We wound the attaching curves along arcs to obtain ${\cal H}^{wind}$
from ${\cal H}$. This might have created new intersection points,
but the intersection points of ${\cal H}$ still exist. Let the \emph{standard
generators} of ${\cal H}^{wind}$ be the ones that correspond to the
standard generators of ${\cal H}$.
\end{defn}

In this section, we show the following proposition. Henceforth, we
consider ${\cal B}$ as a homologically $\mathbb{Z}$-graded $A_{\infty}$-category,
where the homological grading is defined as in Proposition \ref{prop:combi-claims}
(\ref{enu:It-is-possible}).
\begin{prop}
\label{prop:combi-claims}We have the following.
\begin{enumerate}
\item \label{enu:The--category-}The $A_{\infty}$-category ${\cal B}$
is homologically $\mathbb{Z}$-gradable.
\item \label{enu:The--category--1}The $A_{\infty}$-category ${\cal B}$
can be equipped with an Alexander $\mathbb{Z}/2$-splitting.
\item \label{enu:The--structures-of}The standard generators have $c_{1}=0$.
\item \label{enu:It-is-possible}For $L_{1}<L_{2}$, the standard generators
of $CF^{-}(\boldsymbol{\beta}(L_{1}),\boldsymbol{\beta}(L_{2}))$
lie in the same relative homological $\mathbb{Z}$-grading. Lift the
relative homological $\mathbb{Z}$-gradings of $CF^{-}(\boldsymbol{\beta}(L_{1}),\boldsymbol{\beta}(L_{2}))$
to \emph{absolute homological $\mathbb{Z}$-gradings} such that the
standard generators lie in absolute homological $\mathbb{Z}$-grading
$0$. Then, $\mu_{d}$ has degree $d-2$ (i.e. this defines a homological
$\mathbb{Z}$-grading on ${\cal B}$).
\item \label{enu:If-the-underlying}If the underlying links of $L_{1}<L_{2}$
differ in at most one crossing ball, then the top absolute homological
$\mathbb{Z}$-grading of $HF^{-}(\boldsymbol{\beta}(L_{1}),\boldsymbol{\beta}(L_{2}))$
is $0$.
\item \label{enu:If-,-then}If $L_{1}<L_{2}$, then $HF^{-}(\boldsymbol{\beta}(L_{1}),\boldsymbol{\beta}(L_{2}))$
is supported in absolute homological $\mathbb{Z}$-grading $\le0$.
\end{enumerate}
\end{prop}

\begin{figure}[h]
\begin{centering}
\includegraphics{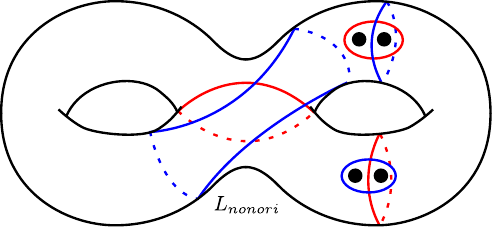}\includegraphics{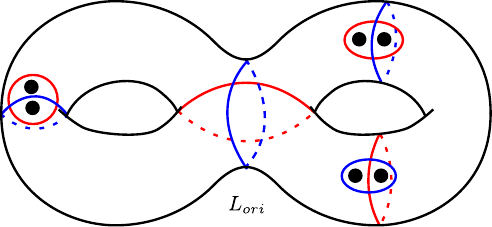}
\par\end{centering}
\caption{\label{fig:s1s2-important-1}Standard (not weakly admissible) Heegaard
diagrams}
\end{figure}

First, Proposition \ref{prop:combi-claims} (\ref{enu:The--category-})
follows from Subsection \ref{subsec:ainf}.

If the underlying links of $L_{1}<L_{2}$ differ in at most one crossing
ball, then the sub Heegaard diagram of ${\cal H}$ with attaching
curves $\boldsymbol{\beta}(L_{1})$ and $\boldsymbol{\beta}(L_{2})$
is a stabilization of the disjoint union of one of the Heegaard diagrams
in Figures \ref{fig:s1s2-important-1-1} and \ref{fig:s1s2-important-1}
(but we might have to switch the red and blue curves; also, the labels
of the basepoints are omitted) and some number of Heegaard diagrams
labelled $K_{ori}$ in Figure \ref{fig:s1s2-important-1-1}. Hence,
Proposition \ref{prop:combi-claims} (\ref{enu:If-the-underlying})
follows from Subsection \ref{subsec:Model-computations} and Section
\ref{sec:Some-important-balled}.

\subsection{\label{subsec:Proposition--()}Proposition \ref{prop:combi-claims}
(\ref{enu:The--category--1}) and (\ref{enu:The--structures-of})}

In this subsection, we prove Proposition \ref{prop:combi-claims}
(\ref{enu:The--category--1}) and (\ref{enu:The--structures-of}).
It is possible to enlarge $\boldsymbol{L}$ (ignoring Condition (\ref{enu:For-every-directed})
of Subsection \ref{subsec:The-general-setup}) such that whenever
$L_{1}<L_{2}$ are consecutive, then their underlying links differ
in at most one crossing ball. Without loss of generality, we can assume
this in this subsection.

For $L_{1}<L_{2}$, we can directly check that there is a domain between
any two generators in $\boldsymbol{\beta}(L_{1})\cap\boldsymbol{\beta}(L_{2})$
in ${\cal H}$, and so there is a domain between any two standard
generators in $\boldsymbol{\beta}(L_{1})\cap\boldsymbol{\beta}(L_{2})$
of ${\cal H}^{wind}$. Note that we do not need any admissibility
hypothesis to define the ${\rm Spin}^{c}$-structure of a generator
or the relative homological grading of two generators and that the
${\rm Spin}^{c}$-structure of a generator and the relative homological
grading of two generators are invariant under winding the curves.

If $L_{1}<L_{2}$ are consecutive, then we can check that every intersection
point ${\bf x}\in\boldsymbol{\beta}(L_{1})\cap\boldsymbol{\beta}(L_{2})$
in ${\cal H}$ has $c_{1}=0$, and so every intersection point ${\bf x}\in\boldsymbol{\beta}(L)\cap\boldsymbol{\beta}(L')$
in ${\cal H}$ for any $L<L'$ has $c_{1}=0$ by the first part of
Lemma \ref{lem:If--is-1}. Hence, we get Proposition \ref{prop:combi-claims}
(\ref{enu:The--structures-of}).

We can similarly show Proposition \ref{prop:combi-claims} (\ref{enu:The--category--1}).
Given a cornerless two-chain, we can write it as the sum of cornerless
$\boldsymbol{\beta}(L_{1}),\boldsymbol{\beta}(L_{2})$-domains ${\cal D}^{wind}$
(i.e. its boundary is contained in $\boldsymbol{\beta}(L_{1})\cup\boldsymbol{\beta}(L_{2})$,
or equivalently, it is in $D({\bf x},{\bf x})$ for some ${\bf x}\in\boldsymbol{\beta}(L_{1})\cap\boldsymbol{\beta}(L_{2})$)
of ${\cal H}^{wind}$ for consecutive $L_{1}<L_{2}$'s (\cite[Lemma 2.64]{nahm2025unorientedskeinexacttriangle}),
and $P({\cal D}^{wind})$ is even for cornerless $\boldsymbol{\beta}(L_{1}),\boldsymbol{\beta}(L_{2})$-domains
${\cal D}^{wind}$ of ${\cal H}^{wind}$: one way to show this is
to use Lemma \ref{lem:alexander-two}. Another way is to consider
the corresponding cornerless domain ${\cal D}$ of ${\cal H}$ that
has the ``same boundary''; we have $P({\cal D}^{wind})=P({\cal D})$,
and we can directly check that $P({\cal D})$ is even.

\subsection{An absolute homological $\mathbb{Z}$-grading}

Proposition \ref{prop:combi-claims} (\ref{enu:It-is-possible}) follows
from the following lemma. 
\begin{lem}
\label{lem:wind-maslov}Let $L_{1},L_{2},L_{3}\in\boldsymbol{L}$
be such that $L_{1}<L_{2}<L_{3}$ and let ${\cal D}$ be a $\boldsymbol{\beta}(L_{1})\boldsymbol{\beta}(L_{2})\boldsymbol{\beta}(L_{3})$-domain
whose vertices are standard generators. Then, $P({\cal D})-\mu({\cal D})$
equals the number of link basepoint pairs of $\boldsymbol{\beta}(L_{2})$
that are not link basepoint pairs for both $\boldsymbol{\beta}(L_{1})$
and $\boldsymbol{\beta}(L_{3})$.
\end{lem}

\begin{proof}
One can first directly check that for any $L<L'$, the standard generators
in $\boldsymbol{\beta}(L)\cap\boldsymbol{\beta}(L')$ have the same
relative homological grading. Hence, we only have to show that for
each such $L_{1}<L_{2}<L_{3}$, there exists such a domain ${\cal D}$.
We can check this by checking the corresponding statement in ${\cal H}$
instead of ${\cal H}^{wind}$. Indeed, as in Subsection \ref{subsec:Proposition--()},
if ${\cal D}$ is a two-chain (or a domain) in ${\cal H}$, then we
can consider the corresponding domain ${\cal D}^{wind}$ in ${\cal H}^{wind}$,
which is characterized by the property that ${\cal D}$ and ${\cal D}^{wind}$
have the ``same boundary''. Since the winding arcs do not intersect
any basepoints nor intersection points, the total multiplicity and
Maslov indices of ${\cal D}$ and ${\cal D}^{wind}$ are the same.

For each of the special regions, we have one of the following cases.
\begin{enumerate}
\item All three circles $\boldsymbol{\beta}(L_{1})$, $\boldsymbol{\beta}(L_{2})$,
and $\boldsymbol{\beta}(L_{3})$ are standard translates of each other.
\item $\boldsymbol{\beta}(L_{1})$ and $\boldsymbol{\beta}(L_{2})$, or
$\boldsymbol{\beta}(L_{2})$ and $\boldsymbol{\beta}(L_{3})$ are
standard translates but the third one is not.
\item $\boldsymbol{\beta}(L_{1})$ and $\boldsymbol{\beta}(L_{3})$ are
standard translates but the third one is not.
\item None of the above
\end{enumerate}
In the special regions that correspond to the one-handles, only the
first case can happen. In each crossing ball, the first, second, and
fourth cases can happen. In each baseball, the first three cases can
happen, and the third case corresponds to the link baseballs in $L_{2}$
that are not link baseballs in both $L_{1}$ and $L_{3}$. For each
of these cases, we can find a domain supported in these special regions
whose vertices are the standard intersection points. For the first,
second, and fourth cases, we can let $P({\cal D})=\mu({\cal D})=0$;
for the third case, we can let $P({\cal D})=1$ and $\mu({\cal D})=0$.
See Figure \ref{fig:zstandard-heegaard-triple} for the third and
fourth cases.

\begin{figure}[h]
\begin{centering}
\includegraphics{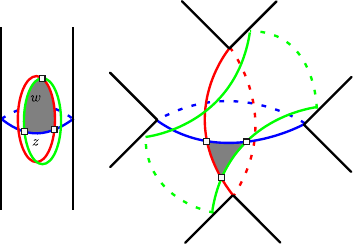}
\par\end{centering}
\caption{\label{fig:zstandard-heegaard-triple}Domains for the third and fourth
cases}
\end{figure}

\end{proof}

\subsection{\label{sec:Higher-compositions-vanish}A bound of the top homological
$\mathbb{Z}$-grading of \texorpdfstring{$HF^{-}(\boldsymbol{\beta}(L_{1}),\boldsymbol{\beta}(L_{2}))$}{HF-(beta(L1), beta(L2))}}

In this subsection, we show Proposition \ref{prop:combi-claims} (\ref{enu:If-,-then}).
Let us only focus on the sub Heegaard diagram of ${\cal H}$ and ${\cal H}^{wind}$
(which we also denote as ${\cal H}$ and ${\cal H}^{wind}$) with
attaching curves $\boldsymbol{\beta}(L_{1})$ and $\boldsymbol{\beta}(L_{2})$.
Let $\gamma$ be the set of winding arcs we used to get ${\cal H}^{wind}$
from ${\cal H}$. For simplicity, let us call $\boldsymbol{\beta}(L_{1}),\boldsymbol{\beta}(L_{2})$
in ${\cal H}$ (resp. ${\cal H}^{wind}$) $\boldsymbol{\alpha},\boldsymbol{\beta}$
(resp. $\boldsymbol{\alpha}^{wind},\boldsymbol{\beta}^{wind}$). In
this section, all Heegaard Floer homology groups refer to the $c_{1}=0$
summands.

We want to show that the top homological $\mathbb{Z}$-grading of
$HF^{-}(\boldsymbol{\alpha}^{wind},\boldsymbol{\beta}^{wind})$ is
lower than or equal to the homological $\mathbb{Z}$-grading of the
standard generators. We would be done if ${\cal H}$ is weakly admissible,
since the standard generators are the top homological $\mathbb{Z}$-grading
generators, but ${\cal H}$ is in general not weakly admissible. It
seems likely that there is always a way to wind the curves such that
the only generators with $c_{1}=0$ are the standard ones, which would
imply the claim, but we prove the claim using a different method.

\begin{figure}[h]
\begin{centering}
\includegraphics{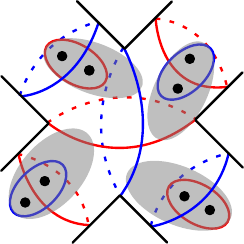}\qquad{}\includegraphics{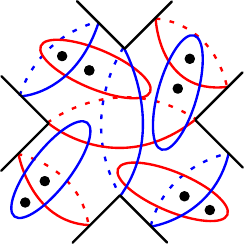}
\par\end{centering}
\caption{\label{fig:quasi-stabilize-4val}Quasi-stabilize four times inside
each crossing ball}
\end{figure}

Let us quasi-stabilize the diagram in each special region that corresponds
to the crossing balls. More precisely, let $(\Sigma,\boldsymbol{\alpha},\boldsymbol{\beta},\boldsymbol{u}\sqcup\boldsymbol{v})$
be the associated Heegaard diagram of ${\cal H}$ (Definition \ref{def:translate}),
and modify this Heegaard diagram in each special region that corresponds
to the crossing balls, as in the left side of Figure \ref{fig:quasi-stabilize-4val}
(for each crossing ball, add four attaching circles for each attaching
curve and eight basepoints), and also such that the winding arcs $\gamma$
do not intersect the gray regions. Let $(\Sigma,\boldsymbol{\alpha}^{qs},\boldsymbol{\beta}^{qs},\boldsymbol{u}^{qs}\sqcup\boldsymbol{v})$
be this new Heegaard diagram.

Below, we will consider various attaching curves $\widetilde{\boldsymbol{\alpha}^{qs}}$,
$\widetilde{\boldsymbol{\beta}^{qs}}$ on $(\Sigma,\boldsymbol{u}^{qs}\sqcup\boldsymbol{v})$.
Let $Y$, $L^{qs}$ be the three-manifold and the link that $\widetilde{\boldsymbol{\alpha}^{qs}},\widetilde{\boldsymbol{\beta}^{qs}}$
represents. Unlike the setting of Subsection \ref{subsec:Unoriented-links-and},
$L^{qs}$ is not minimally pointed. However, we can still define ${\rm Spin}^{c}$-structures
$\mathfrak{s}:\boldsymbol{\alpha}^{qs}\cap\boldsymbol{\beta}^{qs}\to{\rm Spin}^{c}(Y(L^{qs}))$
and the first Chern class $c_{1}:{\rm Spin}^{c}(Y(L^{qs}))\to H^{2}(Y;\mathbb{Z})$
similarly to Subsection \ref{subsec:Unoriented-links-and}. If $(\Sigma,\widetilde{\boldsymbol{\alpha}^{qs}},\widetilde{\boldsymbol{\beta}^{qs}},\boldsymbol{u}^{qs}\sqcup\boldsymbol{v})$
is weakly admissible, then we define the chain complex $CF^{-}(\widetilde{\boldsymbol{\alpha}^{qs}},\widetilde{\boldsymbol{\beta}^{qs}})$
as follows:
\begin{itemize}
\item Let the underlying module of $CF^{-}(\widetilde{\boldsymbol{\alpha}^{qs}},\widetilde{\boldsymbol{\beta}^{qs}})$
be the free $R$-module freely generated by the intersection points
${\bf x}\in\widetilde{\boldsymbol{\alpha}^{qs}}\cap\widetilde{\boldsymbol{\beta}^{qs}}$
with $c_{1}(\mathfrak{s}({\bf x}^{qs}))=0$, where $R:=R_{\boldsymbol{\alpha}^{wind},\boldsymbol{\beta}^{wind}}$
is the coefficient ring of $CF^{-}(\boldsymbol{\alpha}^{wind},\boldsymbol{\beta}^{wind})$.
\item Recall that if $\boldsymbol{p}$ is the set of baseball pairs of ${\cal H}$,
then $R=\mathbb{F}[\{U_{p}^{k_{p}}\}_{p\in\boldsymbol{p}}]$ for some
$k_{p}\in\{1/2,1\}$. Assign a weight $w(v)\in R$ to each basepoint
$v\in\boldsymbol{u}^{qs}\sqcup\boldsymbol{v}$ as follows:
\begin{itemize}
\item For $v\in\boldsymbol{v}$, let $w(v)=U_{p}$ where $p\in\boldsymbol{p}$
is the basepoint pair that $v$ is in.
\item If $v\in\boldsymbol{u}^{qs}$, then $v\in L^{qs}$. Let $p\in\boldsymbol{p}$
be the basepoint pair that is on the same link component of $L^{qs}$
as $v$. Let $w(v)=U_{p}^{1/2}$. 
\end{itemize}
\item Define the differential $\partial$ on $CF^{-}(\widetilde{\boldsymbol{\alpha}^{qs}},\widetilde{\boldsymbol{\beta}^{qs}})$
as
\[
\partial{\bf x}=\sum_{{\cal D}\in D({\bf x},{\bf y}),\ \mu({\cal D})=1}\#{\cal M}({\cal D})\prod_{v\in\boldsymbol{u}^{qs}\sqcup\boldsymbol{v}}w(v)^{n_{v}({\cal D})}{\bf y}.
\]
\end{itemize}

Call a generator \emph{${\bf x}\in\boldsymbol{\alpha}^{qs}\cap\boldsymbol{\beta}^{qs}$
standard} if it is standard if we ignore the attaching circles that
we added to quasi-stabilize. (Note that each of the red (resp. blue)
circles contained in the gray region intersect exactly one blue (resp.
red) circle.) We will obtain attaching curves from $\boldsymbol{\alpha}^{qs},\boldsymbol{\beta}^{qs}$
by winding them along arcs. Call a generator of such Heegaard diagram
\emph{standard} if it comes from a standard generator ${\bf x}\in\boldsymbol{\alpha}^{qs}\cap\boldsymbol{\beta}^{qs}$.
All the standard generators have $c_{1}=0$ and also the same relative
homological $\mathbb{Z}$-grading. Lift the relative homological $\mathbb{Z}$-grading
to an absolute homological $\mathbb{Z}$-grading such that the standard
generators have homological $\mathbb{Z}$-grading $0$.

Wind the curves along the winding arcs $\gamma$, and call these new
attaching curves $\boldsymbol{\alpha}^{qs,wind},\boldsymbol{\beta}^{qs,wind}$.
This Heegaard diagram can be obtained from ${\cal H}^{wind}$ by a
number of \emph{standard quasi-stabilizations}, where a \emph{standard
quasi-stabilization} is adding two basepoints and a circle for each
attaching curve in a way described in \cite[Subsection 5.1]{MR3709653}
(but the role of $\alpha$ and $\beta$ can be reversed; also see
\cite[Subsection 4.1]{MR3905679}). Now, we can unwind all the curves
along the winding arcs $\gamma$ and slightly isotope (wind) the circles
inside the gray region to get a diagram that looks like the right
side of Figure \ref{fig:quasi-stabilize-4val} inside each crossing
ball, and agrees with ${\cal H}$ inside each baseball and also in
the special regions that correspond to the one-handles. Call these
attaching curves $\boldsymbol{\alpha}^{qs,st},\boldsymbol{\beta}^{qs,st}$;
the Heegaard diagram $(\Sigma,\boldsymbol{\alpha}^{qs,st},\boldsymbol{\beta}^{qs,st},\boldsymbol{u}^{qs}\sqcup\boldsymbol{v})$
is weakly admissible. All the generators in this diagram are standard
generators, and so $HF^{-}(\boldsymbol{\alpha}^{qs,st},\boldsymbol{\beta}^{qs,st})$
is supported in homological $\mathbb{Z}$-gradings $\le0$.

We claim that $HF^{-}(\boldsymbol{\alpha}^{qs,wind},\boldsymbol{\beta}^{qs,wind})$
is also supported in homological $\mathbb{Z}$-gradings $\le0$. This
would be immediate if the Heegaard diagram with attaching curves $\boldsymbol{\alpha}^{qs,st}$,
$\boldsymbol{\beta}^{qs,st}$, $\boldsymbol{\alpha}^{qs,wind}$, $\boldsymbol{\beta}^{qs,wind}$
is weakly admissible, since then we could consider the canonical isomorphism
$HF^{-}(\boldsymbol{\alpha}^{qs,st},\boldsymbol{\beta}^{qs,st})\simeq HF^{-}(\boldsymbol{\alpha}^{qs,wind},\boldsymbol{\beta}^{qs,wind})$
which has degree $0$. If it is not weakly admissible, we further
wind $\boldsymbol{\alpha}^{qs,st},\boldsymbol{\beta}^{qs,st}$ (and
call them $\boldsymbol{\alpha}^{qs,tmp},\boldsymbol{\beta}^{qs,tmp}$)
such that the Heegaard diagram with attaching curves $\boldsymbol{\alpha}^{qs,st}$,
$\boldsymbol{\beta}^{qs,st}$, $\boldsymbol{\alpha}^{qs,tmp}$, $\boldsymbol{\beta}^{qs,tmp}$
is admissible, and also that the Heegaard diagram with attaching curves
$\boldsymbol{\alpha}^{qs,tmp}$, $\boldsymbol{\beta}^{qs,tmp}$, $\boldsymbol{\alpha}^{qs,wind}$,
$\boldsymbol{\beta}^{qs,wind}$ is admissible. Then, the canonical
isomorphisms 
\[
HF^{-}(\boldsymbol{\alpha}^{qs,st},\boldsymbol{\beta}^{qs,st})\simeq HF^{-}(\boldsymbol{\alpha}^{qs,tmp},\boldsymbol{\beta}^{qs,tmp})\simeq HF^{-}(\boldsymbol{\alpha}^{qs,wind},\boldsymbol{\beta}^{qs,wind})
\]
have degree zero. 

By \cite[Proposition 5.3]{MR3709653}, $HF^{-}(\boldsymbol{\alpha}^{wind},\boldsymbol{\beta}^{wind})$
is also supported in homological $\mathbb{Z}$-gradings $\le0$.

\section{\label{sec:Canonical-generators}Canonical elements for the setup}

In this section, we define\emph{ canonical elements }$\theta\in HFL_{top}^{-}(Y,L;0)$
for $(Y,L)$ that satisfies Condition \ref{cond:We-consider-balled}.
\begin{condition}
\label{cond:We-consider-balled}We consider balled links $L\subset Y$
for which $(Y,L)$ is obtained by adding some number of free baseballs
and performing some number of $0$-surgeries on the disjoint union
of the following:
\begin{enumerate}
\item \label{enu:canon1}$(S^{1}\times S^{2},K_{nonori})$, $(S^{1}\times S^{2},K_{ori})$,
$(\#^{2}S^{1}\times S^{2},L_{nonori})$, or $(\#^{2}S^{1}\times S^{2},L_{ori})$
\item \label{enu:some-number-of}Some number of $(S^{1}\times S^{2},K_{ori})$'s
\item \label{enu:canon1-1}Some number of $(S^{3},{\rm unlink})$'s
\end{enumerate}
\end{condition}

In Definition \ref{def:f-and-g}, we in particular defined the canonical
elements of $HFL'{}^{-}(S^{1}\times S^{2},K_{nonori};0)$ and $HFL'{}^{-}(S^{1}\times S^{2},K_{ori};0)$
(see Remark \ref{rem:canonical-simple}). Let us first define the
canonical elements of $HFL'{}^{-}(\#^{2}S^{1}\times S^{2},L_{nonori};0)$
and $HFL'{}^{-}(\#^{2}S^{1}\times S^{2},L_{ori};0)$.
\begin{defn}
Recall the notations from Proposition \ref{prop:For-each-balled}.
\begin{itemize}
\item The \emph{canonical element} of $HFL'{}^{-}(\#^{2}S^{1}\times S^{2},L_{nonori};0)$
is $b$.
\item The \emph{canonical element} of $HFL'{}^{-}(\#^{2}S^{1}\times S^{2},L_{ori};0)$
is either $a$ or $a+d$, and we need to use the relative $H_{1}$-action
to pin it down. The canonical element satisfies $A_{13}\Phi_{1}x=x$
(Definition \ref{def:relative-hom}), i.e. $A_{\mathbb{F}[U^{1/2}],\lambda,-1}\rho(\Phi_{1}x)=U^{1/2}\rho(x)$
for some curve $\lambda$ (hence for all such curves, by Proposition
\ref{prop:83}) between a link basepoint in the first link baseball
and a link basepoint in the third link baseball.
\end{itemize}
\end{defn}

The top homological $\mathbb{Z}$-grading part $HFL_{top}^{-}(Y,L;0)$
is the tensor product (over $\mathbb{F}$) of the $HFL_{top}^{-}$'s
of the summands of $(Y,L)$ ((\ref{enu:canon1}), (\ref{enu:some-number-of}),
and (\ref{enu:canon1-1}) of Condition \ref{cond:We-consider-balled});
we define the canonical element $\theta\in HFL_{top}^{-}(Y,L;0)$
as the tensor product of the canonical elements. Note that the stabilization
map $S^{+}$ (Definition \ref{def:A-pair-pointed-Heegaard-1}) maps
$\theta$ to $\theta$.

In Subsections \ref{subsec:The-simpler-cases} and \ref{subsec:The-case},
we show that the canonical element is indeed well-defined, and we
set up conventions in Subsection \ref{subsec:Canonical-generators}.

\subsection{\label{subsec:The-simpler-cases}The simpler cases}

Let us consider the case $L_{?}\neq L_{ori}$: the canonical element
$\theta$ is simpler to characterize in this case. One can obtain
a Heegaard diagram for $(Y,L)$ by stabilizing a disjoint union of
a Heegaard diagram for $(\#^{\ell}S^{1}\times S^{2},L_{?})$ and some
number of Heegaard diagrams for $(S^{1}\times S^{2},K_{ori})$. Using
this, we can compute $HFL_{top}^{-}(Y,L;0)$ and the $\Phi$-actions
on it; the \emph{canonical element} can be canonically characterized
as follows. Here, consider the homology class of a link component
of $L$ as an element of $H_{1}(Y;\mathbb{Z})$ up to sign.
\begin{defn}
Let $L_{?}=K_{nonori}$, $K_{ori}$, or $L_{nonori}$. The \emph{canonical
element $\theta$} is the unique nonzero element in $HFL_{top}^{-}(Y,L;0)$
such that:
\begin{itemize}
\item $\theta$ is homogeneous with respect to the ${\rm Spin}^{c}$-splitting.
\item ($K_{ori})$ For each link component $L_{k}$ of $L$ such that there
exists another link component $L_{\ell}$ of $L$ such that their
homology classes $0\neq[L_{k}]=\pm[L_{\ell}]\in H_{1}(Y;\mathbb{Z})$,
we have $\Phi_{k}\theta\neq0$, where $\Phi_{k}$ is the $\Phi$-action
of (the link baseball of) $L_{k}$ (Definition \ref{def:phi-balled}).
\item ($K_{nonori}$) For any link component $L_{k}$ whose homology class
is nonzero and divisible by $2$, $\Phi_{k}\theta=0$. (Note that
such link component exists if and only if $L_{?}=L_{k}$.)
\item ($L_{nonori}$) If there are two link components $L_{k},L_{\ell}$
whose homology classes are nonzero and do not fit into the above categories,
then $\Phi_{k}\theta=\Phi_{\ell}\theta\neq0$. (Note that there are
either no such link components or exactly two; the latter case happens
if and only if $L_{?}=L_{nonori}$.) 
\end{itemize}
\end{defn}

\subsection{\label{subsec:The-case}The case $L_{?}=L_{ori}$}

Let $L_{?}=L_{ori}$. In this case, we need to make an additional
choice: an order on the three components of $L_{?}$\footnote{The canonical element only depends on the parity of this order.}.
Given such an order, call them $L_{1}<L_{2}<L_{3}$.

Let us study the $H_{1}$ and relative $H_{1}$ actions on $HFL_{top}^{\prime-}(Y,L;0)$
and $HFL_{\mathbb{F}[U^{1/2}],top}^{\prime-}(Y,L;0)$ (Definition
\ref{def:unoriented-collapse}). One can obtain a Heegaard diagram
for it by stabilizing a disjoint union of a Heegaard diagram for $(\#^{2}S^{1}\times S^{2},L_{ori})$
and some number of Heegaard diagrams for $(S^{1}\times S^{2},K_{ori})$.
Using this, one can show that $HFL_{\mathbb{F}[U^{1/2}]}^{\prime-}(Y,L;0)$
is a free $\mathbb{F}[U^{1/2}]$-module, and that the base change
map
\[
\rho:HFL^{\prime-}(Y,L;0)\to HFL_{\mathbb{F}[U^{1/2}]}^{\prime-}(Y,L;0)
\]
maps the top homological grading to the top homological grading, and
is injective on the top homological grading. Hence $U^{1/2}\rho$
is injective on $HFL_{top}^{\prime-}(Y,L;0)$.

We understand the $A_{\mathbb{F}[U^{1/2}],\lambda,-1}$-action on
$HFL_{\mathbb{F}[U^{1/2}],top}^{\prime-}(Y,L;0)$, for closed loops
$\lambda$ on the Heegaard surface: if it homologically intersects
nontrivially (mod $2$) with any of the cocores of the $0$-surgery
(i.e. the boundary of the cocores of the $1$-handle of the trace
of the surgery), then for any nonzero $x\in HFL_{\mathbb{F}[U^{1/2}],top}^{\prime-}(Y,L;0)$,
$A_{\mathbb{F}[U^{1/2}],\lambda,-1}x\notin U^{1/2}HFL_{\mathbb{F}[U^{1/2}],top}^{\prime-}(Y,L;0)$.
\begin{defn}
Let $L_{?}=L_{ori}$. There are exactly three link components of $L$
whose homology class is nonzero and unique (among the homology classes
of link components of $L$) in $H_{1}(Y;\mathbb{F})$. Assume that
we are given an order of these three link components; call them $L_{1}<L_{2}<L_{3}$.
The \emph{canonical element $\theta$} is the unique nonzero element
in $HFL_{top}^{-}(Y,L;0)$ such that
\begin{itemize}
\item $\theta$ is homogeneous with respect to the ${\rm Spin}^{c}$-splitting.
\item For each link component $L_{k}$ of $L$ such that there exists another
link component $L_{\ell}\in L$ such that their homology classes $0\neq[L_{k}]=\pm[L_{\ell}]\in H_{1}(Y;\mathbb{Z})$,
we have $\Phi_{k}\theta\neq0$. (Note that $L_{k},L_{\ell}\neq L_{1},L_{2},L_{3}$.)
\item $\Phi_{1}\theta,\Phi_{2}\theta,\Phi_{3}\theta\neq0$, and there exists
a curve $\lambda$ between a link basepoint of $L_{1}$ and a link
basepoint of $L_{3}$ such that $A_{\mathbb{F}[U^{1/2}],\lambda,-1}\rho(\Phi_{1}\theta)=U^{1/2}\rho(\theta).$
\end{itemize}
\end{defn}

\subsection{\label{subsec:Canonical-generators}Canonical elements for the setup}

Let us return to the setup of Section \ref{sec:The-setup-and}. We
consider the case where the underlying links of $L_{0}<L_{1}$ differ
in at most one crossing ball. The attaching curves $\boldsymbol{\beta}(L_{0}),\boldsymbol{\beta}(L_{1})$
(which we simply write $\boldsymbol{\beta}_{0},\boldsymbol{\beta}_{1}$
respectively) represent a balled link $(Y,L)$ that satisfies Condition
\ref{cond:We-consider-balled}.

\subsubsection*{The simpler cases}

If $L_{?}=K_{nonori}$, $K_{ori}$, or $L_{nonori}$, then define
the \emph{canonical element} $\theta\in HF_{0}^{-}(\boldsymbol{\beta}_{0},\boldsymbol{\beta}_{1})$
as the canonical element in $HFL'{}_{top}^{-}(Y,L;0)$ (by Proposition
\ref{prop:combi-claims} (\ref{enu:If-the-underlying}), the top homological
$\mathbb{Z}$-grading is $0$).

\subsubsection*{The case $L_{?}=L_{ori}$}

This case corresponds to the following case:
\begin{itemize}
\item the underlying links of $L_{0},L_{1}$ differ in exactly one crossing
ball,
\item the associated band $B:L_{0}\to L_{1}$ is a merge or split band,
and
\item the three link components of $L_{0}$, $L_{1}$ that $B$ intersects
have \emph{pairwise distinct link baseballs}. 
\end{itemize}
Recall from Subsection \ref{subsec:The-case} that we have to make
an additional choice: an order of these three link components of $L_{0},L_{1}$
that $B$ intersects, or equivalently an order of the link baseballs
on these three link components.

In this paper, whenever we have to define the canonical element $\theta\in HF_{0}^{-}(\boldsymbol{\beta}_{0},\boldsymbol{\beta}_{1})$
for the case $L_{?}=L_{ori}$, the baseballs will be labelled as $(i,x)$
(i.e. called $BB_{i,x}$) for some $i,x$ (sometimes it will be called
$BB_{i,a,b}$, in which case $x:=a,b$). These labels and the type
of the baseballs will determine the order of the three link baseballs.
The following conditions will be satisfied:
\begin{enumerate}
\item \label{enu:merge-lori}If $B$ is a merge band, then there exist $x,y,z$
such that $BB_{i,x}$ and $BB_{j,y}$ are the link baseballs of the
two link components of $L_{0}$ that $B$ intersects, and $BB_{i,z}$
is the link baseball of the link component of $L_{1}$ that $B$ intersects. 
\item \label{enu:split-lori}If $B$ is a split band, then there exist $x,y,z$
such that $BB_{i,x}$ is the link baseball of the link component of
$L_{0}$ that $B$ intersects, and $BB_{i,y}$ and $BB_{j,z}$ are
the link baseballs of the two link components of $L_{1}$ that $B$
intersects.
\end{enumerate}
In fact, the $i$'s will be totally ordered, and we will have $i<j$
for the above cases.

If $B$ is a merge band, then let $i,j,x,y,z$ be as in Condition
(\ref{enu:merge-lori}), and choose the order $BB_{i,x}<BB_{j,y}<BB_{i,z}$.
If $B$ is a split band, then let $i,j,x,y,z$ be as in Condition
(\ref{enu:split-lori}), and choose the order $BB_{i,x}<BB_{j,z}<BB_{i,y}$.
Define the \emph{canonical element} $\theta\in HF_{0}^{-}(\boldsymbol{\beta}_{0},\boldsymbol{\beta}_{1})$
as the canonical element in $HFL'{}_{top}^{-}(Y,L;0)$ determined
by the above order of the three link components of $L_{0},L_{1}$
that $B$ intersects.

\begin{defn}
\label{def:If--are}If $L_{0},L_{1}$ are related by a swap map (Definition
\ref{def:Two-balled-links}), call $\theta$ \emph{swap}. If $L_{0},L_{1}$
are related by a band map (Definition \ref{def:Two-balled-links-1}),
call $\theta$ the type of the band (\emph{non-orientable}, \emph{merge},
split). We also say that 
\[
\theta\in HF_{0}^{-}(\boldsymbol{\beta}_{0},\boldsymbol{\beta}_{1}),\ \mu_{2}(-\otimes\theta):HF^{-}(\boldsymbol{\alpha},\boldsymbol{\beta}_{0})\to HF^{-}(\boldsymbol{\alpha},\boldsymbol{\beta}_{1})
\]
are \emph{swap maps} (resp. \emph{non-orientable band map}s, \emph{merge
map}s, \emph{split maps}).
\end{defn}

It is easy to see that a swap map is a composition of the restricted
type of swap maps that we considered in Section \ref{sec:Swap-maps}.
Similarly, using Proposition \ref{prop:swap-merge-2}, one can show
that a band map is the composition of a restricted type of band map
that we considered in Section \ref{sec:Band-maps-for} and a swap
map. Hence, we get the following.
\begin{prop}
\label{prop:The-swap,-merge,}The swap, merge, and split maps for
planar links and planar bands coincide with Khovanov homology in the
sense of Propositions \ref{prop:Let--be-1} and \ref{prop:Let--be}.
\end{prop}

\begin{proof}
By the above discussion, the statement for swap maps follows from
Proposition \ref{prop:Let--be}; the statements for merge and split
maps follow from Propositions \ref{prop:swap-merge-2} and \ref{prop:Let--be-1}.
\end{proof}
As in Propositions \ref{prop:For-any-} and \ref{prop:For-any--1},
the swap maps and band maps are equivariant with respect to the action
of any $q\in L_{i}$ outside the crossing ball. We do not need to
use this fact for band maps; however, we need the following lemma
(which in particular shows that the swap maps are equivariant).
\begin{lem}
\label{lem:swap-equivariant}Let $L_{0}$ and $L_{1}$ be related
by a swap map. For any $q\in L_{i}$ outside the crossing balls and
for any $x\in HF_{0}^{-}(\boldsymbol{\beta}_{0},\boldsymbol{\beta}_{1})$,
\[
\mu_{2}(-\otimes x):HF^{-}(\boldsymbol{\alpha},\boldsymbol{\beta}_{0})\to HF^{-}(\boldsymbol{\alpha},\boldsymbol{\beta}_{1})
\]
is equivariant with respect to the action of $q$ (Definition \ref{def:hflaction}).
\end{lem}

\begin{proof}
Let $BB_{i}$ be the link baseball of the component of $L_{i}$ that
$q$ is on. We may assume $BB_{0}\neq BB_{1}$, since the lemma is
immediate otherwise. Recall that we simply call $U_{i}:=U_{BB_{i}}$.
Then, as in Propositions \ref{prop:For-any-} and \ref{prop:For-any--1},
it is sufficient to check $U_{0}^{1/2}x=U_{1}^{1/2}x\in HF_{-1}^{-}(\boldsymbol{\beta}_{0},\boldsymbol{\beta}_{1})$.
After isotoping $\boldsymbol{\beta}_{0},\boldsymbol{\beta}_{1}$,
we may assume that the sub Heegaard diagram with attaching curves
$\boldsymbol{\beta}_{0},\boldsymbol{\beta}_{1}$ is a stabilization
of a disjoint union of some number of Heegaard diagram for $(S^{1}\times S^{2},K_{ori})$;
call this $(\Sigma^{u},\boldsymbol{\beta}_{0}^{u},\boldsymbol{\beta}_{1}^{u},\boldsymbol{p}^{u},\boldsymbol{G}^{u})$.
The top homological grading part $HF_{0}^{-}(\boldsymbol{\beta}_{0},\boldsymbol{\beta}_{1})$
is contained in the image of the stabilization map 
\[
S^{+}:HF^{-}(\boldsymbol{\beta}_{0}^{u},\boldsymbol{\beta}_{1}^{u})\to HF^{-}(\boldsymbol{\beta}_{0},\boldsymbol{\beta}_{1}).
\]
Say $x=S^{+}(y)$. Then, $U_{0}^{1/2}x=U_{1}^{1/2}x$ since $U_{0}^{1/2}y=U_{1}^{1/2}y$,
which follows from Subsection \ref{subsec:Model-computations}.
\end{proof}

\section{\label{sec:Iterating-the-unoriented}Iterating the unoriented skein
exact triangle}

In this section, we set up the stage, construct the giant twisted
complex, and show Theorem \ref{thm:main-thm}, assuming a series of
local computation results that we prove throughout the remainder of
this paper. To define the spectral sequence, we work in the $A_{\infty}$-category
${\cal A}_{c}$ defined as in Subsection \ref{subsec:The-Heegaard-diagram},
and define a twisted complex $\underline{\boldsymbol{\beta}}$ that
has a\emph{ cube filtration}, and a quotient $\underline{\boldsymbol{\beta}_{01}}$
of $\underline{\boldsymbol{\beta}}$. The induced filtration on $CF^{-}(\boldsymbol{\alpha},\underline{\boldsymbol{\beta}_{01}})$
will give the desired spectral sequence. Compare \cite{MR2141852,MR2964628}.

\subsection{\label{subsec:The-setup-for}The setup for the spectral sequence}

Let $L$ be a link in a three-manifold $Y$ with some crossing balls,
such that $L\cap CB_{r}\subset CB_{r}$ is $2$ of Figure \ref{fig:skein-moves-1}
for all $r$. For simplicity, assume that $Y$ is connected.

\subsubsection*{The cube of resolutions}

Let $X=\{1,\cdots,s\}$ and for each $f:X\to\{0,1,2\}$, let $L(f)$
be the link obtained from $L$ by replacing $CB_{r}\cap L$ with $f(r)$
of Figure \ref{fig:skein-moves-1} for each $r$. Consider the partial
order on $\{0,1,2\}^{X}$ given by $f\le g$ if and only if $f(r)\le g(r)$
for all $r$. Define $|f|:=\sum_{r}|f(r)|$, $|g-f|:=\sum_{r}|g(r)-f(r)|$.
For consecutive $f<g$, let $B(f,g):L(f)\to L(g)$ be the band given
by Figure \ref{fig:skein-moves-1}.

\subsubsection*{The baseballs}

Let $\boldsymbol{CC}$ be the set of connected components of $L\backslash\bigsqcup_{r}CB_{r}$.
For each connected component $CC\in\boldsymbol{CC}$ and $f:X\to\{0,1,2\}$,
put a baseball $BB_{CC,f}\subset Y$ such that $BB_{CC,f}$ is disjoint
from $\bigsqcup_{r}CB_{r}$ and intersects $CC$ (and these baseballs
$BB_{CC,f}$ are pairwise distinct, hence pairwise disjoint).

\subsubsection*{A total order on $\boldsymbol{CC}$}

Fix a bijective function $\sigma:\{1,\cdots,N\}\to\boldsymbol{CC}$.
We identify bijective functions $\{1,\cdots,N\}\to\boldsymbol{CC}$
with total orders on $\boldsymbol{CC}$. Write $BB_{i,f}:=BB_{\sigma(i),f}$
for simplicity.

\subsubsection*{The balled links $L(f)$}

For each link component of $L(f)$, we declare that $BB_{i,f}$ is
a \emph{link baseball}, where $i$ is the smallest $i$ such that
$\sigma(i)$ is on that link component. All the baseballs that we
have not declared as link baseballs are \emph{free baseballs}.

It turns out that we have to slightly modify the band maps (also called
the canonical element) $\theta\in HF_{0}^{-}(\boldsymbol{\beta}(f),\boldsymbol{\beta}(g))$,
$F_{B}:=\mu_{2}(-\otimes\theta)$ (Definition \ref{def:If--are})
for consecutive $f<g$ to be able to iterate the unoriented skein
exact triangle. We call these modified band maps $\delta$, $G_{B}:=\mu_{2}(-\otimes\delta)$;
they also give rise to an unoriented skein exact triangle (Proposition
\ref{prop:We-have-the} (\ref{enu:6})). Sometimes, we write $\theta(f,g)$
(resp. $\delta(f,g)$) to emphasize that we are referring to the canonical
element (resp. the modified band map) of $HF_{0}^{-}(\boldsymbol{\beta}(f),\boldsymbol{\beta}(g))$.
\begin{defn}
\label{def:For-each-consecutive}For each consecutive pair $f<g$,
define the \emph{modified band map} $\delta(f,g)\in HF_{0}^{-}(\boldsymbol{\beta}(f),\boldsymbol{\beta}(g))$
as follows. Let $\theta(f,g)\in HF_{0}^{-}(\boldsymbol{\beta}(f),\boldsymbol{\beta}(g))$
be the canonical element, and consider the band $B(f,g):L(f)\to L(g)$.
\begin{itemize}
\item If $B(f,g)$ is non-orientable, let $\delta(f,g):=\theta(f,g)$.
\item If $B(f,g)$ is a merge band, then let $i,j$ ($i<j)$ be such that
the band intersects the $i$th and $j$th link components of $L(f)$.
Let
\[
\delta(f,g):=\theta(f,g)+\Phi_{j,f}\left(\sum_{k=i+1}^{j-1}\Phi_{k,f}\right)\theta(f,g)\in HF_{0}^{-}(\boldsymbol{\beta}(f),\boldsymbol{\beta}(g)).
\]
\item If $B(f,g)$ is a split band, then let $i,j$ ($i<j)$ be such that
the band intersects the $i$th and $j$th link components of $L(g)$.
Let
\[
\delta(f,g):=\theta(f,g)+\Phi_{j,g}\left(\sum_{k=i+1}^{j-1}\Phi_{k,f}\right)\theta(f,g)\in HF_{0}^{-}(\boldsymbol{\beta}(f),\boldsymbol{\beta}(g)).
\]
\end{itemize}
\end{defn}

\begin{rem}
In Subsection \ref{subsec:The-reduced-hat}, we consider a slight
variant of the setup of this subsection to define the spectral sequence
for the reduced hat version. Definition \ref{def:For-each-consecutive}
still applies word for word in that setting. 
\end{rem}

Let us state our main theorem for the minus version $HFL'{}^{-}$
and the unreduced hat version $\widehat{HFL'}$. For the unreduced
hat version, choose a baseball $BB$. Let us define the \emph{modified
band maps} for $\widehat{CFL'}(Y,L(f),BB;c)$: recall that 
\[
\widehat{CFL'}(Y,L(f),BB;c)=CFL'{}^{-}(Y,L(f);c)/U_{BB}.
\]
Define the \emph{modified band map }
\[
G_{B(f,g)}:\widehat{HFL'}(Y,L(f),BB;c)\to\widehat{HFL'}(Y,L(g),BB;c)
\]
as the map induced by 
\[
\mu_{2}(-\otimes\delta):CFL'{}^{-}(Y,L(f);c)\to CFL'(Y,L(g);c),
\]
by quotienting out by $U_{BB}$ and taking the map induced on homology.
\begin{thm}
\label{thm:spectral-1}Let $c\in H^{2}(Y;\mathbb{Z})$. By $HFL'$,
we mean $\widehat{HFL'}$ (for a choice of a baseball $BB$) or $HFL'{}^{-}$.
There exists a spectral sequence that converges to $HFL'(Y,L;c)$,
whose $E_{1}$ page is 
\[
\left(\bigoplus_{f:X\to\{0,1\}}HFL'(Y,L(f);c),d_{1}\right)
\]
where $d_{1}$ is the sum of the\emph{ modified band maps}
\[
G_{B(f,g)}:HFL'(Y,L(f);c)\to HFL'(Y,L(g);c)
\]
for consecutive $f<g$. 
\end{thm}

The remainder of this section is organized as follows: we show Theorem
\ref{thm:spectral-1} in Subsection \ref{subsec:The-twisted-complex},
and a similar statement for $\widetilde{HFL'}$ (Theorem \ref{thm:spectral-1-1})
in Subsection \ref{subsec:The-reduced-hat}. We study the basepoint
actions (Definition \ref{def:hflaction}) in Subsection \ref{subsec:The-action}.
In Subsection \ref{subsec:The-Khovanov-chain}, we show that if the
crossing balls correspond to the crossings of a planar diagram of
a link in $S^{3}$, then the $E_{1}$ page of the spectral sequence
is the Khovanov cube of resolutions chain complex.

\subsection{\label{subsec:The-twisted-complex}The twisted complex}

Fix $c\in H^{2}(Y;\mathbb{Z})$, and consider the Heegaard diagram
and the $A_{\infty}$-categories ${\cal A}_{c},{\cal B}$ constructed
in Subsection \ref{subsec:The-Heegaard-diagram} for the setting of
Subsection \ref{subsec:The-setup-for}. We work in the $A_{\infty}$-category
${\cal A}_{c}$ and construct a differential $\delta\in{\rm Hom}_{\Sigma{\cal A}_{c}}(\underline{\boldsymbol{\beta}},\underline{\boldsymbol{\beta}})$
on $\underline{\boldsymbol{\beta}}:=\bigoplus_{f:X\to\{0,1,2\}}\boldsymbol{\beta}(f)$
such that $\mu_{1}(\delta)=0$, i.e. $(\underline{\boldsymbol{\beta}},\delta)$
is a twisted complex. To build the twisted complex, we only work with
the beta-attaching curves; hence, $\delta\in{\rm Hom}_{\Sigma{\cal B}}(\underline{\boldsymbol{\beta}},\underline{\boldsymbol{\beta}})$.
Recall that the subcategory ${\cal B}$ with objects $\boldsymbol{\beta}(f)$
is homologically $\mathbb{Z}$-graded (Proposition \ref{prop:combi-claims}).
The differential $\delta$ will consist of elements 
\[
\delta(f,g)\in CF_{|g-f|-1}^{-}(\boldsymbol{\beta}(f),\boldsymbol{\beta}(g)),
\]
which we call the \emph{$f\to g$ component of $\delta$}, for $f<g$. 

First, for consecutive $f<g$, we abuse notation and let $\delta(f,g)\in CF_{0}^{-}(\boldsymbol{\beta}(f),\boldsymbol{\beta}(g))$
be any cycle that represents the homology class (``modified band
map'') $\delta(f,g)\in HF_{0}^{-}(\boldsymbol{\beta}(f),\boldsymbol{\beta}(g))$.
\begin{rem}
We are also abusing notation in the sense that if $|g-f|\ge2$, then
$\delta(f,g)\in CF_{|g-f|-1}^{-}(\boldsymbol{\beta}(f),\boldsymbol{\beta}(g))$
does \emph{not} mean the modified band map. This should be unambiguous
since we only consider the modified band map for consecutive $f<g$.
\end{rem}

Since $\delta(f,g)=0$ for $f<g$, we can filter $\underline{\boldsymbol{\beta}}$
by the \emph{cube filtration}, or also the \emph{``flattened'' $\mathbb{Z}$-filtration}
\[
\underline{\boldsymbol{\beta}}=\bigoplus_{|f|\ge0}\boldsymbol{\beta}(f)\supset\cdots\supset\bigoplus_{|f|\ge d}\boldsymbol{\beta}(f)\supset\bigoplus_{|f|\ge d+1}\boldsymbol{\beta}(f)\supset\cdots.
\]
By abuse of notation, denote the twisted complex $(\underline{\boldsymbol{\beta}},\delta)$
also as $\underline{\boldsymbol{\beta}}$. 
\begin{rem}
\label{rem:homologically-graded-twisted}If we want to view $\underline{\boldsymbol{\beta}}$
as a homologically graded twisted complex, i.e. such that $\delta$
has homological degree $-1$, then simply shift the $\boldsymbol{\beta}(f)$'s
in homological grading and consider $\bigoplus_{f:X\to\{0,1,2\}}\boldsymbol{\beta}(f)[|f|]$\footnote{The convention is that for a homogeneous $f\in CF^{-}(\boldsymbol{\alpha},\boldsymbol{\beta})$,
the corresponding element in $CF^{-}(\boldsymbol{\alpha}[k],\boldsymbol{\beta}[\ell])$
has degree ${\rm deg}f+k-\ell$.}. To avoid confusions, we work with the unshifted twisted complex.
\end{rem}

\begin{prop}
\label{prop:We-have-the}We have the following.
\begin{enumerate}
\item \label{enu:4}(Modified band maps commute) Let $f_{00}<f_{01}<f_{11}$
be consecutive, and assume that $f_{00},f_{11}$ differ in two $x\in X$'s.
Then there exists exactly one $f_{10}\neq f_{01}$ such that $f_{00}<f_{10}<f_{11}$
are consecutive. We have 
\[
\mu_{2}(\delta(f_{00},f_{01})\otimes\delta(f_{01},f_{11}))=\mu_{2}(\delta(f_{00},f_{10})\otimes\delta(f_{10},f_{11}))\in HF_{0}^{-}(\boldsymbol{\beta}(f_{00}),\boldsymbol{\beta}(f_{11})).
\]
\item \label{enu:6}(Modified band maps give rise to an unoriented skein
exact triangle) Let $f_{0}<f_{1}<f_{2}$ be consecutive and differ
in only one $x\in X$. Then 
\[
\mu_{2}(\delta(f_{0},f_{1})\otimes\delta(f_{1},f_{2}))=0\in HF_{0}^{-}(\boldsymbol{\beta}(f_{0}),\boldsymbol{\beta}(f_{2})).
\]
Let $\delta(f_{0},f_{2})\in CF_{1}^{-}(\boldsymbol{\beta}(f_{0}),\boldsymbol{\beta}(f_{2}))$
be such that
\[
\mu_{1}(\delta(f_{0},f_{2}))=\mu_{2}(\delta(f_{0},f_{1})\otimes\delta(f_{1},f_{2})),
\]
where we abuse notation and let $\delta(f_{i},f_{i+1})\in CF_{0}^{-}(\boldsymbol{\beta}(f_{i}),\boldsymbol{\beta}(f_{i+1}))$
be any cycle that represents the corresponding homology class. Then,
\[\begin{tikzcd}
	{\underline{\boldsymbol{\beta} (f_{0},f_{1},f_{2})}:=} & {\boldsymbol{\beta}(f_{0})} & {\boldsymbol{\beta}(f_{1})} & {\boldsymbol{\beta}(f_{2})}
	\arrow["{\delta(f_{0},f_{1})}"', from=1-2, to=1-3]
	\arrow["{\delta(f_{0},f_{2})}", curve={height=-12pt}, from=1-2, to=1-4]
	\arrow["{\delta(f_{1},f_{2})}"', from=1-3, to=1-4]
\end{tikzcd}\]is a twisted complex, and for any such $\delta(f_{0},f_{2})$, we
have $HF^{-}(\boldsymbol{\alpha},\underline{\boldsymbol{\beta}(f_{0},f_{1},f_{2})})=0$.
\end{enumerate}
\end{prop}

\begin{proof}
We prove (\ref{enu:4}) in Section \ref{sec:Composition-of-band},
and (\ref{enu:6}) in Section \ref{sec:Proofs-of-claims}.
\end{proof}
\begin{rem}
It is necessary to modify the band maps as some band maps do \emph{not}
commute. See Remark \ref{rem:Similarly,-we-can}.
\end{rem}

Let us first assume these claims and prove Theorem \ref{thm:spectral-1}
for $HFL'{}^{-}$ (the statement for $\widehat{HFL'}$ follows from
this: see Remark \ref{rem:unreduced-hat}). We induct on $d$ and
define $\delta(f,g)\in CF_{d-1}^{-}(\boldsymbol{\beta}(f),\boldsymbol{\beta}(g))$
(the $f\to g$ component of $\delta\in{\rm Hom}_{\Sigma{\cal B}}(\underline{\boldsymbol{\beta}},\underline{\boldsymbol{\beta}})$)
for $f,g$ such that $f<g$ and $|g-f|=d$. More precisely, we recursively
define $\delta_{\le d}\in{\rm Hom}_{\Sigma{\cal B}}(\underline{\boldsymbol{\beta}},\underline{\boldsymbol{\beta}})$,
such that
\begin{enumerate}
\item $\delta_{\le d}(f,g)=0$ (here, $\delta_{\le d}(f,g)$ means the $f\to g$
component of $\delta_{\le d}$) for $f,g$ such that $f\not<g$, or
$f<g$ and $|g-f|>d$,
\item $\delta_{\le d}(f,g)=\delta_{\le(d-1)}(f,g)$ for $f<g$ such that
$|g-f|\le d-1$, and
\item $\mu_{1}(\delta_{\le d})(f,g)=0$ for $f<g$ such that $|g-f|\le d$.
\end{enumerate}
We let $\delta:=\delta_{\le2s}$; then we have $\mu_{1}(\delta)=0$.

The base case is $d=1$: for consecutive $f<g$, we have already defined
$\delta(f,g)\in CF_{0}^{-}(\boldsymbol{\beta}(f),\boldsymbol{\beta}(g))$
as any cycle that represents the homology class $\delta(f,g)$; let
$\delta_{\le1}$ consist of them.

Let $d\ge2$, and assume the induction hypotheses for $\delta_{\le(d-1)}$.
Let us define $\delta_{\le d}$. By the first and second induction
hypotheses, this reduces to defining $\delta_{\le d}(f,g)$ for $f<g$
such that $|g-f|=d$. Let $f<g$, $|g-f|=d$. First, $\mu_{1}(\delta_{\le(d-1)})(f,g)\in CF_{d-2}^{-}(\boldsymbol{\beta}(f),\boldsymbol{\beta}(g))$
is a cycle, by the third induction hypothesis and that $\mu_{1}\circ\mu_{1}=0$.
We claim that $\mu_{1}(\delta_{\le(d-1)})(f,g)$ is a boundary. For
$d=2$, this follows from Proposition \ref{prop:We-have-the} (\ref{enu:4})
and the first part of (\ref{enu:6}); for $d\ge3$, this follows from
Proposition \ref{prop:combi-claims} (\ref{enu:If-,-then}). Now,
let $\delta_{\le d}(f,g)\in CF_{d-1}^{-}(\boldsymbol{\beta}(f),\boldsymbol{\beta}(g))$
be such that 
\[
\mu_{1}(\delta_{\le d}(f,g))=\mu_{1}(\delta_{\le(d-1)})(f,g).
\]
Then $\delta_{\le d}$ satisfies the induction hypotheses.

Let $\underline{\boldsymbol{\beta}_{01}}$ be the quotient of $\underline{\boldsymbol{\beta}}$
whose underlying object is $\bigoplus_{f:X\to\{0,1\}}\boldsymbol{\beta}(f)$,
and let $\boldsymbol{2}:X\to\{0,1,2\}$ be the constant function with
value $2$. Then, the cube filtration (or the ``flattened'' $\mathbb{Z}$-filtration)
on $HF^{-}(\boldsymbol{\alpha},\underline{\boldsymbol{\beta}_{01}})$
gives our desired spectral sequence: it remains to show that $HF^{-}(\boldsymbol{\alpha},\underline{\boldsymbol{\beta}_{01}})\simeq HF^{-}(\boldsymbol{\alpha},\boldsymbol{\beta}(\boldsymbol{2}))$.
This follows from Proposition \ref{prop:We-have-the} (\ref{enu:6}),
as in \cite[Proof of Theorem 4.1]{MR2141852}: there in fact exists
a chain map $CF^{-}(\boldsymbol{\alpha},\underline{\boldsymbol{\beta}_{01}})\to CF^{-}(\boldsymbol{\alpha},\boldsymbol{\beta}(\boldsymbol{2}))$
that is a quasi-isomorphism. We briefly recall the proof. Let $F_{k}$
for $k=0,1,\cdots,s$ be the set of $f:X\to\{0,1,2\}$ such that $f(1),\cdots,f(k)=2$
and $f(k+1),\cdots,f(s)\in\{0,1\}$. Let $\underline{\boldsymbol{\beta}^{k}}$
be the twisted complex whose underlying object is $\bigoplus_{f\in F_{k}}\boldsymbol{\beta}(f)$,
and whose differential consists of the $f\to g$ components of $\delta\in{\rm Hom}_{\Sigma{\cal B}}(\underline{\boldsymbol{\beta}},\underline{\boldsymbol{\beta}})$
for $f,g\in F_{k}$. In particular, $\underline{\boldsymbol{\beta}^{0}}=\underline{\boldsymbol{\beta}_{01}}$
and $\underline{\boldsymbol{\beta}^{s}}=\boldsymbol{\beta}(\boldsymbol{2})$.
Consider the cycle $\underline{\delta^{k}}\in CF^{-}(\underline{\boldsymbol{\beta}^{k}},\underline{\boldsymbol{\beta}^{k+1}})$
given by the $f\to g$ components of $\delta\in{\rm Hom}_{\Sigma{\cal B}}(\underline{\boldsymbol{\beta}},\underline{\boldsymbol{\beta}})$
for $f\in F_{k}$ and $g\in F_{k+1}$. Proposition \ref{prop:We-have-the}
(\ref{enu:6}) implies that 
\[
\mu_{2}(-\otimes\underline{\delta^{k}}):CF^{-}(\boldsymbol{\alpha},\underline{\boldsymbol{\beta}^{k}})\to CF^{-}(\boldsymbol{\alpha},\underline{\boldsymbol{\beta}^{k+1}})
\]
is a quasi-isomorphism; hence their composite $CF^{-}(\boldsymbol{\alpha},\underline{\boldsymbol{\beta}^{0}})\to CF^{-}(\boldsymbol{\alpha},\underline{\boldsymbol{\beta}^{s}})$
is a quasi-isomorphism.
\begin{rem}
\label{rem:unreduced-hat}Let $BB$ be a baseball. Then, $CF^{-}(\boldsymbol{\alpha},\underline{\boldsymbol{\beta}_{01}})$
is a filtered, free chain complex over $\mathbb{F}[U_{BB}]$ and the
above map $CF^{-}(\boldsymbol{\alpha},\underline{\boldsymbol{\beta}_{01}})\to CF^{-}(\boldsymbol{\alpha},\boldsymbol{\beta}(\boldsymbol{2}))$
is an $\mathbb{F}[U_{BB}]$-linear chain map between free chain complexes
over $\mathbb{F}[U_{BB}]$, and so tensoring with $\mathbb{F}[U_{BB}]/(U_{BB})$
over $\mathbb{F}[U_{BB}]$ gives the spectral sequence for $\widehat{HFL'}$.
Indeed, if $C$ is a free chain complex over $\mathbb{F}[U]$ and
has zero homology, then $C\otimes_{\mathbb{F}[U]}\mathbb{F}[U]/U$
has zero homology as well, since we have a short exact sequence of
chain complexes 
\[
0\to C\xrightarrow{\cdot U}C\to C\otimes_{\mathbb{F}[U]}\mathbb{F}[U]/U\to0.
\]
Hence the chain map $CF^{-}(\boldsymbol{\alpha},\underline{\boldsymbol{\beta}_{01}})\to CF^{-}(\boldsymbol{\alpha},\boldsymbol{\beta}(\boldsymbol{2}))$
is a quasi-isomorphism after tensoring with $\mathbb{F}[U_{BB}]/U_{BB}$
over $\mathbb{F}[U_{BB}]$ as well.
\end{rem}

\begin{rem}
Note that we indeed need to consider a slight variant to show an analogue
of Theorem \ref{thm:spectral-1} for $\widetilde{HFL'}$: we would
like to quotient out $CF^{-}(\boldsymbol{\alpha},\underline{\boldsymbol{\beta}})$
by $U_{BB^{link}}^{1/2}$ for some baseball $BB^{link}$, but $CF^{-}(\boldsymbol{\alpha},\underline{\boldsymbol{\beta}})$
is not an $\mathbb{F}[U_{BB^{link}}^{1/2}]$-module for any baseball
$BB^{link}$ since there is no baseball $BB^{link}$ that is a link
baseball for every link $L(f)$.
\end{rem}

\subsection{\label{subsec:The-reduced-hat}The reduced hat version}

In this subsection, we prove Theorem \ref{thm:spectral-1-1}, an analogue
of Theorem \ref{thm:spectral-1} for $\widetilde{HFL'}$.

Let $\sigma:\{1,\cdots,N\}\to\boldsymbol{CC}$ be a bijection. We
consider the following modified setting.
\begin{itemize}
\item For each $i=2,\cdots,N$ and $f:X\to\{0,1,2\}$, put a baseball $BB_{i,f}$
on $\sigma(i)$ .
\item Put a baseball $BB_{1}$ on $\sigma(1)\in\boldsymbol{CC}$. For convenience,
also denote $BB_{1}$ as $BB_{1,f}$.
\end{itemize}
As in Subsection \ref{subsec:The-setup-for}, let $L(f)$ be the balled
link with underlying link $L(f)$ such that for each link component
of $L(f)$, $BB_{i,f}$ is its \emph{link baseball}, where $i$ is
the smallest $i$ such that $\sigma(i)$ is on that link component.
Define the modified band maps $\delta(f,g)\in HF_{0}^{-}(\boldsymbol{\beta}(f),\boldsymbol{\beta}(g))$
for consecutive $f<g$ in the same way as Definition \ref{def:For-each-consecutive}.
\begin{prop}
\label{prop:(Modified-band-maps-2}The same statement as Proposition
\ref{prop:We-have-the} holds in this case as well.
\end{prop}

\begin{proof}
The counterpart of Proposition \ref{prop:We-have-the} (\ref{enu:4})
is shown in Section \ref{sec:Computations-for-the}, and the counterpart
of Proposition \ref{prop:We-have-the} (\ref{enu:6}) is shown in
Section \ref{sec:Proofs-of-claims}.
\end{proof}
Hence, we can build a twisted complex $\underline{\boldsymbol{\beta}}$
as in Subsection \ref{subsec:The-twisted-complex}. Since $BB_{1}$
is a link baseball for all $L(f)$, the chain complex $CF^{-}(\boldsymbol{\alpha},\underline{\boldsymbol{\beta}})$
is a free chain complex over $\mathbb{F}[U_{1}^{1/2}]$ (recall from
Notation \ref{nota:Recall-from-Notation} that $U_{1}:=U_{BB_{1}}$).
Hence, by tensoring it with $\mathbb{F}[U_{1}^{1/2}]/(U_{1}^{1/2})$
over $\mathbb{F}[U_{1}^{1/2}]$, we get Theorem \ref{thm:spectral-1-1}
(recall Remark \ref{rem:unreduced-hat}).

Define the \emph{modified band map }
\[
G_{B(f,g)}:\widetilde{HFL'}(Y,L(f),BB_{1};c)\to\widetilde{HFL'}(Y,L(g),BB_{1};c)
\]
as the map induced by 
\[
\mu_{2}(-\otimes\delta):CFL'{}^{-}(Y,L(f);c)\to CFL'(Y,L(g);c),
\]
by quotienting out by $U_{1}^{1/2}$ and taking the map induced on
homology.
\begin{thm}
\label{thm:spectral-1-1}Let $c\in H^{2}(Y;\mathbb{Z})$. There exists
a spectral sequence that converges to $\widetilde{HFL'}(Y,L,BB_{1};c)$,
whose $E_{1}$ page is 
\[
\left(\bigoplus_{f:X\to\{0,1\}}\widetilde{HFL'}(Y,L(f),BB_{1};c),d_{1}\right),
\]
where $d_{1}$ is the sum of the\emph{ modified band maps}
\[
G_{B(f,g)}:\widetilde{HFL'}(Y,L(f),BB_{1};c)\to\widetilde{HFL'}(Y,L(g),BB_{1};c)
\]
for consecutive $f<g$. 
\end{thm}

\subsection{\label{subsec:The-action}The basepoint actions}

In this subsection, we show that the spectral sequences constructed
in Subsection \ref{subsec:The-twisted-complex} respect the basepoint
actions (Definition \ref{def:hflaction}). We first make this precise.

Let $R_{\boldsymbol{CC}}=\mathbb{F}[\{X_{CC}\}_{CC\in\boldsymbol{CC}}]$.
For $f:X\to\{0,1,2\}$ and a connected component $CC$ of $L\backslash\bigsqcup_{r}CB_{r}$,
define the action of $X_{CC}$ on $HFL'{}^{-}(Y,L(f);c)$ as multiplication
by $U_{i,f}^{1/2}$, where $BB_{i,f}$ is the link baseball on the
connected component of $L(f)$ that $CC$ is contained in. Then, $HFL'{}^{-}(Y,L(f);c)$
is an $R_{\boldsymbol{CC}}$-module, and the action of $X_{CC}^{2}$
is the same as multiplication by $U_{BB}$ for any baseball $BB$
(recall that $Y$ is connected, so multiplication by $U_{BB}$ for
different baseballs $BB$ are chain homotopic). Define the $R_{\boldsymbol{CC}}$-action
on $\widehat{HFL'}$ and $\widetilde{HFL'}$ similarly.
\begin{thm}
\label{thm:spectral-1-2}Let $c\in H^{2}(Y;\mathbb{Z})$. For $\widehat{HFL'}$
(for a choice of a baseball $BB$) and $HFL'{}^{-}$, there exist
spectral sequences of $R_{\boldsymbol{CC}}$-modules that satisfy
the conditions of Theorem \ref{thm:spectral-1}.
\end{thm}

We focus on the minus version; the corresponding statement for the
unreduced hat version follows by tensoring the chain complexes with
$\mathbb{F}[U_{BB}]/(U_{BB})$ over $\mathbb{F}[U_{BB}]$.
\begin{rem}
We do not know whether the spectral sequence for the reduced hat version
is a spectral sequence of $R_{\boldsymbol{CC}}$-modules.
\end{rem}

We show Theorem \ref{thm:spectral-1-2} in two steps: first, we show
that the spectral sequence constructed as in Subsection \ref{subsec:The-twisted-complex}
is a spectral sequence of $\mathbb{F}[X_{\sigma(1)}]$-modules, and
then show that it is a spectral sequence of $R_{\boldsymbol{CC}}$-modules
by showing that the spectral sequences given by different $\sigma$'s
are isomorphic.

\subsubsection{\label{subsec:Step-1}Step 1}

Fix a total order $\sigma:\{1,\cdots,N\}\to\boldsymbol{CC}$ as before.
\begin{itemize}
\item For each $i=2,\cdots,N$ and $f:X\to\{0,1,2\}$, put two baseballs
$BB_{i,f,1}$ and $BB_{i,f,2}$ on $\sigma(i)$.
\item Put a baseball $BB_{1}$ on $\sigma(1)$, and for each $f:X\to\{0,1,2\}$,
put a baseball $BB_{1,f,2}$ on $\sigma(1)$. For convenience, also
denote $BB_{1}$ as $BB_{1,f,1}$.
\end{itemize}
Let $L(f,1)$ and $L(f,2)$ be the balled links with underlying link
$L(f)$ such that for each link component of $L(f,k)$ ($k=1,2$),
$BB_{i,f,k}$ is its \emph{link baseball}, where $i$ is the smallest
$i$ such that $\sigma(i)$ is on that link component.

Define differentials on $\bigoplus_{f:X\to\{0,1,2\}}\boldsymbol{\beta}(f,k)$
as in Subsections \ref{subsec:The-twisted-complex} and \ref{subsec:The-reduced-hat},
and let $\underline{\boldsymbol{\beta}(k)}$ be the corresponding
twisted complex. By Proposition \ref{prop:(Modified-band-maps-1},
we can define a cycle $\underline{\theta}\in CF^{-}(\underline{\boldsymbol{\beta}(1)},\underline{\boldsymbol{\beta}(2)})$
such that for each $f$, the $\boldsymbol{\beta}(f,1)\to\boldsymbol{\beta}(f,2)$
component of $\underline{\theta}$ is a cycle in $CF_{0}^{-}(\boldsymbol{\beta}(f,1),\boldsymbol{\beta}(f,2))$
that represents the canonical element (``swap map'') $\theta\in HF_{0}^{-}(\boldsymbol{\beta}(f,1),\boldsymbol{\beta}(f,2))$,
and $\underline{\theta}$ preserves the cube filtration (if $f\not\le g$,
then the $\boldsymbol{\beta}(f,1)\to\boldsymbol{\beta}(g,2)$ component
of $\underline{\theta}$ is $0$). 
\begin{prop}
\label{prop:(Modified-band-maps-1}(Modified band maps sometimes commute
with swap maps) If $f<g$ are consecutive, then 
\[
\mu_{2}^{\boldsymbol{\beta}(f,1),\boldsymbol{\beta}(g,1),\boldsymbol{\beta}(g,2)}(\delta\otimes\theta)=\mu_{2}^{\boldsymbol{\beta}(f,1),\boldsymbol{\beta}(f,2),\boldsymbol{\beta}(g,2)}(\theta\otimes\delta)\in HF_{0}^{-}(\boldsymbol{\beta}(f,1),\boldsymbol{\beta}(g,2)).
\]
\end{prop}

\begin{proof}
If the band $B:\boldsymbol{\beta}(f,1)\to\boldsymbol{\beta}(g,1)$
is non-orientable, then $\delta=\theta$, and so this follows from
Lemma \ref{lem:summary-theta}. If $B$ is a merge or a split band,
this follows from Lemma \ref{lem:summary-delta}.
\end{proof}
Consider Diagram (\ref{diag:spectral-com}), where $\underline{\boldsymbol{\beta}(k)_{01}}$
is the quotient of $\underline{\boldsymbol{\beta}(k)}$ whose underlying
object is $\bigoplus_{f:X\to\{0,1\}}\boldsymbol{\beta}(f,k)$, and
$\boldsymbol{\beta}(k)(\boldsymbol{2}):=\boldsymbol{\beta}(\boldsymbol{2},k)$
where $\boldsymbol{2}:X\to\{0,1,2\}$ is the constant function $2$.
Here, $A$ and $D$ are the chain maps that are quasi-isomorphisms
obtained as in Subsection \ref{subsec:The-twisted-complex}, and $B$
(resp. $C$) are the chain maps $\mu_{2}(-\otimes\underline{\theta})$
(resp. $\mu_{2}(-\otimes\theta)$), which are quasi-isomorphisms as
well, since the $\boldsymbol{\beta}(f,1)\to\boldsymbol{\beta}(f,2)$
components of $\underline{\theta}$ are swap maps and we can use Proposition
\ref{prop:Something-is-an}. 
\begin{equation}\label{diag:spectral-com}
\begin{tikzcd}
	{CF^- (\boldsymbol{\alpha},\underline{\boldsymbol{\beta}(1)_{01}})} & {CF^- (\boldsymbol{\alpha},\boldsymbol{\beta}(1)(\boldsymbol{2}))} \\
	{CF^- (\boldsymbol{\alpha},\underline{\boldsymbol{\beta}(2)_{01}})} & {CF^- (\boldsymbol{\alpha},\boldsymbol{\beta}(2)(\boldsymbol{2}))}
	\arrow["A", from=1-1, to=1-2]
	\arrow["B", from=1-1, to=2-1]
	\arrow["C", from=1-2, to=2-2]
	\arrow["D", from=2-1, to=2-2]
\end{tikzcd}
\end{equation}The chain map $B$ induces an isomorphism between the associated spectral
sequences, and also an isomorphism of filtered $\mathbb{F}[U]$-modules
between $HF^{-}(\boldsymbol{\alpha},\underline{\boldsymbol{\beta}(1)_{01})}$
and $HF^{-}(\boldsymbol{\alpha},\underline{\boldsymbol{\beta}(2)_{01}})$.
Diagram (\ref{diag:spectral-com}) commutes up to homotopy.

The basepoint pair corresponding to the baseball $BB_{1}$ is a link
basepoint pair for all $\boldsymbol{\beta}(f,1)$. Hence, the filtered
chain complex $CF^{-}(\boldsymbol{\alpha},\underline{\boldsymbol{\beta}(1)})$
is a chain complex of $\mathbb{F}[U_{1}^{1/2}]$-modules, and in particular
$A$ is an $\mathbb{F}[U_{1}^{1/2}]$-linear chain map. Hence, the
spectral sequence induced by the cube filtration on $CF^{-}(\boldsymbol{\alpha},\underline{\boldsymbol{\beta}(1)_{01}})$
is a spectral sequence of $\mathbb{F}[X_{\sigma(1)}]$-modules, and
it converges to $HF^{-}(\boldsymbol{\alpha},\boldsymbol{\beta}(1)(\boldsymbol{2}))$
as an $\mathbb{F}[X_{\sigma(1)}]$-module. Here, to be precise, we
are considering the filtration (of $\mathbb{F}[U_{1}^{1/2}]$-modules)
on $HF^{-}(\boldsymbol{\alpha},\boldsymbol{\beta}(1)(\boldsymbol{2}))$
induced by $A$ and the cube filtration on $HF^{-}(\boldsymbol{\alpha},\underline{\boldsymbol{\beta}(1)_{01}})$.

The chain map $B$ is a quasi-isomorphism. Use $B$ to equip an $\mathbb{F}[X_{\sigma(1)}]$-module
structure on the cube filtration of $HF^{-}(\boldsymbol{\alpha},\underline{\boldsymbol{\beta}(2)_{01}})$.
The map $B$ induces an $\mathbb{F}[X_{\sigma(1)}]$-linear (in fact
$R_{\boldsymbol{CC}}$-linear by Lemma \ref{lem:swap-equivariant})
isomorphism on the underlying modules of the $E_{1}$ page, and so
all the differentials $d_{k}$ of the spectral sequence induced by
the cube filtration on $CF^{-}(\boldsymbol{\alpha},\underline{\boldsymbol{\beta}(2)_{01}})$
are $\mathbb{F}[X_{\sigma(1)}]$-linear, and this spectral sequence
is a spectral sequence of $\mathbb{F}[X_{\sigma(1)}]$-modules.

Now, we have to identify the filtered $\mathbb{F}[X_{\sigma(1)}]$-module
$HF^{-}(\boldsymbol{\alpha},\underline{\boldsymbol{\beta}(2)_{01}})$
with the $\mathbb{F}[X_{\sigma(1)}]$-module $HF^{-}(\boldsymbol{\alpha},\underline{\boldsymbol{\beta}(2)(\boldsymbol{2})})$.
The chain map $C$ is a quasi-isomorphism, and is $\mathbb{F}[X_{\sigma(1)}]$-linear
(in fact $R_{\boldsymbol{CC}}$-linear by Lemma \ref{lem:swap-equivariant})
on homology. Since the above diagram commutes up to homotopy, the
filtered $\mathbb{F}[X_{\sigma(1)}]$-module structure on $HF^{-}(\boldsymbol{\alpha},\boldsymbol{\beta}(2)(\boldsymbol{2}))$
induced by $C$ agrees with the filtration induced by $D$, and so
$D$ on homology is in particular an isomorphism of filtered $\mathbb{F}[X_{\sigma(1)}]$-modules.

\subsubsection{\label{subsec:Step-2}Step 2}

Let $\boldsymbol{TO}$ be the set of bijective $\sigma:\{1,\cdots,N\}\to\boldsymbol{CC}$'s
(``total orders on $\boldsymbol{CC}$''). Put a baseball $BB_{CC,f,\sigma}$
on $CC$ for each $CC\in\boldsymbol{CC}$, $f:X\to\{0,1,2\}$, and
$\sigma\in\boldsymbol{TO}$, and let $L(f,\sigma)$ be the balled
link with underlying link $L(f)$ that is obtained as in Subsection
\ref{subsec:The-setup-for} but using the total order $\sigma$. Order
the $L(f,\sigma)$'s using the lexicographic order on $(f,\sigma)$
(choose any total order on $\boldsymbol{TO}$). Consider the associated
Heegaard diagram with attaching curves $\boldsymbol{\beta}(f,\sigma)$.
For consecutive $f<g$, let $\delta\in HF_{0}^{-}(\boldsymbol{\beta}(f,\sigma),\boldsymbol{\beta}(g,\sigma))$
be the modified band map defined as in Definition \ref{def:For-each-consecutive},
using the total order $\sigma$. Then, for each $\sigma\in\boldsymbol{TO}$,
we obtain a twisted complex $\underline{\boldsymbol{\beta}(\sigma)}$
as in Subsection \ref{subsec:The-twisted-complex}.
\begin{defn}
\label{def:sigmatau}Let $\sigma<\tau$ be such that there exists
some $b=1,\cdots,N-1$ such that $\tau=\sigma\circ(b\ b+1)$ for some
$b=1,\cdots,N-1$, where $(b\ b+1)\in S_{N}$ is the transposition
that swaps $b$ and $b+1$. Define the \emph{modified swap map} $\chi\in HF_{0}^{-}(\boldsymbol{\beta}(f,\sigma),\boldsymbol{\beta}(f,\tau))$
as follows.
\begin{itemize}
\item If both $BB_{\sigma(b),f,\sigma}$ and $BB_{\sigma(b+1),f,\sigma}$
are link baseballs of $L(f,\sigma)$, then 
\[
\chi:=\theta+\Phi_{\sigma(b),f,\sigma}\Phi_{\sigma(b+1),f,\sigma}\theta.
\]
\item Otherwise, $\chi:=\theta$.
\end{itemize}
\end{defn}

Let $\sigma<\tau=\sigma\circ(b\ b+1)$ be as in Definition \ref{def:sigmatau}.
By Proposition \ref{prop:(Modified-band-maps}, we can define a cycle
$\underline{{\rm \chi}}\in CF^{-}(\underline{\boldsymbol{\beta}(\sigma)},\underline{\boldsymbol{\beta}(\tau)})$
such that for each $f$, the $\boldsymbol{\beta}(f,\sigma)\to\boldsymbol{\beta}(f,\tau)$
component of $\underline{\chi}$ is a cycle in $CF_{0}^{-}(\boldsymbol{\beta}(f,\sigma),\boldsymbol{\beta}(f,\tau))$
that represents the corresponding modified swap map $\chi$, and $\underline{\chi}$
preserves the cube filtration.
\begin{prop}
\label{prop:(Modified-band-maps}If $f<g$ are consecutive, then 
\[
\mu_{2}^{\boldsymbol{\beta}(f,\sigma),\boldsymbol{\beta}(g,\sigma),\boldsymbol{\beta}(g,\tau)}(\delta\otimes\chi)=\mu_{2}^{\boldsymbol{\beta}(f,\sigma),\boldsymbol{\beta}(f,\tau),\boldsymbol{\beta}(g,\tau)}(\chi\otimes\delta)\in HF_{0}^{-}(\boldsymbol{\beta}(f,\sigma),\boldsymbol{\beta}(g,\tau)).
\]
\end{prop}

\begin{proof}
We prove this in Subsection \ref{subsec:Modified-band-maps}.
\end{proof}
Using an argument similar to that of Subsubsection \ref{subsec:Step-1}
(use Lemma \ref{lem:swap-equivariant} for $x=\chi$; and Proposition
\ref{prop:Something-is-an-1} in place of Proposition \ref{prop:Something-is-an}),
we deduce that the spectral sequence induced by the cube filtration
on $CF^{-}(\boldsymbol{\alpha},\boldsymbol{\beta}(\sigma)_{01})$
converges to $HF^{-}(\boldsymbol{\alpha},\boldsymbol{\beta}(\sigma)(\boldsymbol{2}))$
as $R_{\boldsymbol{CC}}$-modules.

\subsection{\label{subsec:The-Khovanov-chain}The Khovanov chain complex}

In this subsection, we specialize to the case where the crossing balls
correspond to the crossings of a planar diagram of a link in $S^{3}$.
We identify the $E_{1}$ page of the spectral sequences that we constructed
in Subsections \ref{subsec:The-twisted-complex} and \ref{subsec:The-reduced-hat}
with the Khovanov cube of resolutions chain complex, in all the versions
(reduced hat, unreduced hat, and minus), and prove Corollary \ref{cor:khovanov-main}.
Define the relative homological grading and the relative Alexander
$\mathbb{Z}/2$-grading on the chain complexes $\widetilde{CFL'}$
and $\widehat{CFL'}$ as the gradings induced by the corresponding
gradings on $CFL'{}^{-}$. 
\begin{cor}
\label{cor:khovanov-main}Let $L$ be a balled link in $S^{3}$. For
each version (reduced hat, unreduced hat, and minus) of unoriented
link Floer homology, there exists a spectral sequence that converges
to it, whose $E_{2}$ page is the corresponding version of Khovanov
homology (in fact the $E_{1}$ page is the corresponding version of
the Khovanov chain complex):
\[
\widetilde{Kh}(m(L,x))\Rightarrow\widetilde{HFL'}(S^{3},L,BB^{link}),\ \widehat{Kh}(m(L))\Rightarrow\widehat{HFL'}(S^{3},L,BB),\ Kh^{-}(m(L))\Rightarrow HFL'{}^{-}(S^{3},L),
\]
where, for the reduced hat version, we choose a link baseball $BB^{link}$
of $L$ and a distinguished point $x\in L\cap BB^{link}$, and for
the unreduced hat version, we choose a baseball $BB$ of $L$. The
spectral sequences for $\widehat{Kh}$ and $Kh^{-}$ are spectral
sequences of $R_{\boldsymbol{CC}}$-modules.

For $k\ge1$, the differential $d_{k}$ on the $E_{k}$ page has $(q,h)$
bidegree $(2k-2,k)$, and $d_{2\ell}\equiv0$ for $\ell\ge1$. The
relative $\delta=q/2-h$ grading descends to the relative homological
grading (``the $\delta$-grading'') on unoriented link Floer homology,
and the relative $q/2$ grading (i.e. the quantum grading divided
by $2$) modulo $2$ descends to the relative Alexander $\mathbb{Z}/2$-grading
on unoriented link Floer homology.
\end{cor}

\begin{rem}
Recall (Remark \ref{rem:For-knots-}) that for knots $K\subset S^{3}$,
$\widetilde{HFL'}(S^{3},K,BB^{link})$ is the knot Floer homology
$\widehat{HFK}(S^{3},K)$. Under this identification, the relative
homological grading of $\widetilde{HFL'}$ indeed corresponds to the
relative $\delta$-grading $M-A$ of $\widehat{HFK}$, where $M$
and $A$ are the Maslov and Alexander gradings, respectively.
\end{rem}

We prove Corollary \ref{cor:khovanov-main} in several steps. Assume
that the crossing balls correspond to the crossings of a planar diagram
of a link in $S^{3}$; in particular, the underlying links of $L(f)$
for $f:X\to\{0,1\}$ are planar, and the bands $B(f,g):L(f)\to L(g)$
for consecutive $f<g$, $f,g:X\to\{0,1\}$ are planar as well.

\subsubsection*{The underlying modules of the $E_{1}$ page and the Khovanov chain
complex}

Let us first identify the underlying $R_{\boldsymbol{CC}}$-module
of the $E_{1}$ page with the underlying $R_{\boldsymbol{CC}}$-module
of the Khovanov cube of resolutions chain complex. For the minus version,
this follows from Proposition \ref{prop:Let--be-1}. For the reduced
and unreduced hat versions, this is Lemma \ref{prop:Let--be-1-1}.
\begin{lem}[{\cite[Section 5]{nahm2025unorientedskeinexacttriangle}}]
\label{prop:Let--be-1-1}Let $L$ be a balled link in $S^{3}$, whose
underlying link is a planar link. Let $BB$ be a baseball, and let
$BB^{link}$ be a link baseball. Consider $\widehat{HFL'}(S^{3},L,BB)$
and $\widetilde{HFL'}(S^{3},L,BB^{link})$ as $Kh^{-}(L)$-modules.
Also, consider $\widehat{Kh}(L)$ and $\widetilde{Kh}(L,x)$ for $x\in BB^{link}\cap L$
as $Kh^{-}(L)$-modules. Then, 
\[
\widehat{HFL'}(S^{3},L,BB)\simeq\widehat{Kh}(L),\ \widetilde{HFL'}(S^{3},L,BB^{link})\simeq\widetilde{Kh}(L,x)
\]
as relatively graded $Kh^{-}(L)$-modules (in the sense of Proposition
\ref{prop:Let--be-1}; note that such isomorphisms are unique).
\end{lem}

\begin{proof}
A pair-pointed Heegaard diagram for $(S^{3},L)$ can be obtained from
performing some number of $0$-surgeries (Subsection \ref{subsec:Connected-sums})
to the Heegaard diagram given by the disjoint union of some number
of $S^{2}$'s, where each connected component has exactly one basepoint
pair. Hence, one can compute the corresponding chain complex.
\end{proof}

\subsubsection*{The differential $d_{1}$ on the $E_{1}$ page and the differential
of the Khovanov chain complex}

We show that the differential $d_{1}$ on the $E_{1}$ page agrees
with the differential of the Khovanov chain complex under the above
identification of the underlying modules of the $E_{1}$ page and
the Khovanov chain complex. Recall that $d_{1}$ is the sum of the
modified band maps $G_{B(f,g)}$ (for consecutive $f<g$, $f,g:X\to\{0,1\}$),
which are induced by $\mu_{2}(-\otimes\delta(f,g))$.

First, we show that it is sufficient to show this for the minus version.
Indeed, recall that the chain complexes for the unreduced and reduced
hat versions are defined as 
\[
\widehat{CFL'}(S^{3},L,BB)=CFL'{}^{-}(S^{3},L)/U_{BB},\ \widetilde{CFL'}(S^{3},L,BB^{link})=CFL'{}^{-}(S^{3},L)/U_{BB^{link}}.
\]
Hence, we have base change maps 
\[
HFL'{}^{-}(S^{3},L)\to\widehat{HFL'}(S^{3},L,BB),\ HFL'{}^{-}(S^{3},L)\to\widetilde{HFL'}(S^{3},L,BB^{link}).
\]
If the underlying link of $L$ is a planar link, then these maps are
surjective, and are the same as the base change maps (let $x\in BB^{link}\cap L$)
\[
Kh^{-}(L)\to\widehat{Kh}(L),\ Kh^{-}(L)\to\widetilde{Kh}(L,x)
\]
as maps between $Kh^{-}(L)$-modules. Hence, we only have to prove
that $d_{1}$ agrees with the differential on $CKh$ for the minus
version since the modified band maps $G_{B(f,g)}$ commute with the
base change maps.

Now, let us show the statement for the minus version. By Proposition
\ref{prop:The-swap,-merge,}, the sum of the composition maps $\mu_{2}(-\otimes\theta(f,g))$
for consecutive $f<g$, $f,g:X\to\{0,1\}$ agrees with the differential
on $CKh^{-}$. Lemma \ref{lem:planar-phi} implies that 
\[
\mu_{2}(-\otimes\theta(f,g))=\mu_{2}(-\otimes\delta(f,g)):HF^{-}(\boldsymbol{\alpha},\boldsymbol{\beta}(f))\to HF^{-}(\boldsymbol{\alpha},\boldsymbol{\beta}(g))
\]
 for consecutive $f<g$, $f,g:X\to\{0,1\}$. Hence, the differential
$d_{1}$ agrees with the differential on $CKh^{-}$.
\begin{lem}
\label{lem:planar-phi}Let $\boldsymbol{\alpha},\boldsymbol{\beta}_{0}$
and $\boldsymbol{\alpha},\boldsymbol{\beta}_{1}$ represent planar
links in $S^{3}$. Then, for any $x\in HF^{-}(\boldsymbol{\beta}_{0},\boldsymbol{\beta}_{1})$
and any basepoint pair $p$,
\[
\mu_{2}(-\otimes\Phi_{p}x):HF^{-}(\boldsymbol{\alpha},\boldsymbol{\beta}_{0})\to HF^{-}(\boldsymbol{\alpha},\boldsymbol{\beta}_{1})
\]
is identically zero. 
\end{lem}

\begin{proof}
We claim that $\Phi_{p}$ is identically zero on $HF^{-}(\boldsymbol{\alpha},\boldsymbol{\beta}_{i})$.
The lemma follows from this since then $\mu_{2}({\rm Id}\otimes\Phi_{p})=\Phi_{p}\mu_{2}({\rm Id}\otimes{\rm Id})=0$.

It is easy to compute $\Phi_{p}$ directly, but we present a grading
argument. The map $\Phi_{p}$ has homological degree $0$ and Alexander
$\mathbb{Z}/2$ degree $1$. There is a unique nonzero top homological
grading element $x\in HF^{-}(\boldsymbol{\alpha},\boldsymbol{\beta}_{i})$
($x$ is homogeneous with respect to the relative Alexander $\mathbb{Z}/2$-grading).
Hence, $\Phi_{p}x=0$. Since $HF^{-}(\boldsymbol{\alpha},\boldsymbol{\beta}_{i})$
is generated by $x$ over the coefficient ring, $\Phi_{p}$ is identically
zero on $HF^{-}(\boldsymbol{\alpha},\boldsymbol{\beta}_{i})$.
\end{proof}
Let us study gradings. We focus on the minus version, as the statements
for the reduced and unreduced hat versions follow from the below argument
by considering the induced relative homological $\mathbb{Z}$-gradings
and the relative Alexander $\mathbb{Z}/2$-gradings on $CF^{-}(\boldsymbol{\alpha},\underline{\boldsymbol{\beta}})/U_{1}^{1/2}$
(for the setting of Subsection \ref{subsec:The-reduced-hat}) and
$CF^{-}(\boldsymbol{\alpha},\underline{\boldsymbol{\beta}})/U_{BB}$.
The argument below applies to both the settings of Subsections \ref{subsec:The-twisted-complex}
and \ref{subsec:The-reduced-hat}.

\subsubsection*{Homological $\mathbb{Z}$-grading}

Recall from Subsection \ref{subsec:ainf} that the $A_{\infty}$-category
${\cal A}_{0}$ (recall Subsection \ref{subsec:The-Heegaard-diagram};
its objects are the alpha curve $\boldsymbol{\alpha}$ and all the
beta curves $\boldsymbol{\beta}(f)$) is homologically $\mathbb{Z}$-gradable.
Choose a homological $\mathbb{Z}$-grading on ${\cal A}_{0}$ that
agrees with the homological $\mathbb{Z}$-grading on ${\cal B}$ when
restricted to ${\cal B}$ (i.e. the standard generators have homological
$\mathbb{Z}$-grading $0$).

As in Remark \ref{rem:homologically-graded-twisted}, consider $\underline{\boldsymbol{\beta}}$
as a homologically graded twisted complex, i.e. such that $\delta$
has homological degree $-1$, by shifting the $\boldsymbol{\beta}(f)$'s
in homological grading. Then, the relative homological $\mathbb{Z}$-grading
on $CF^{-}(\boldsymbol{\alpha},\underline{\boldsymbol{\beta}_{01}})$
on the $E_{1}$ page agrees with the relative $\delta=q/2-h$ grading
on $CKh^{-}$. Every differential $d_{k}$ in the spectral sequence
has $\delta$ degree $-1$. If $k\ge1$, then its $(q,h)$ bidegree
is defined, and since its $h$ degree is $k$, its $(q,h)$ bidegree
is $(2k-2,k)$. The chain map $CF^{-}(\boldsymbol{\alpha},\underline{\boldsymbol{\beta}_{01}})\to CF^{-}(\boldsymbol{\alpha},\boldsymbol{\beta}(\boldsymbol{2}))$
(defined as in Subsection \ref{subsec:The-twisted-complex}) is a
quasi-isomorphism and is homogeneous with respect to the relative
homological grading.

\subsubsection*{Alexander $\mathbb{Z}/2$-grading}

The $A_{\infty}$-category ${\cal A}_{0}$ has an Alexander $\mathbb{Z}/2$-splitting,
by the same argument as in Subsection \ref{subsec:Proposition--()}:
indeed, any cornerless two-chain is the sum of a cornerless $\boldsymbol{\alpha},\boldsymbol{\beta}(f)$-domain
(i.e. the boundary of the domain lies in $\boldsymbol{\alpha}\cup\boldsymbol{\beta}(f)$)
for some $f:X\to\{0,1,2\}$ and cornerless $\boldsymbol{\beta}(g),\boldsymbol{\beta}(h)$-domains
for consecutive $g<h$'s (\cite[Lemma 2.64]{nahm2025unorientedskeinexacttriangle}).
Hence, we only have to additionally check that $P(D)$ is even for
any cornerless $\boldsymbol{\alpha},\boldsymbol{\beta}(f)$-domain.
This holds since $\boldsymbol{\alpha},\boldsymbol{\beta}(f)$ represent
a link in $S^{3}$ (see Lemma \ref{lem:alexander-two}).

For each $f<g$, the relative Alexander $\mathbb{Z}/2$-grading of
$CF^{-}(\boldsymbol{\beta}(f),\boldsymbol{\beta}(g))$ is uniquely
determined, but there are many ways to lift them to an Alexander $\mathbb{Z}/2$-splitting
of ${\cal A}_{0}$.
\begin{lem}
\label{lem:alexz/2}We can choose an Alexander $\mathbb{Z}/2$-splitting
of ${\cal A}_{0}$ and let the differential $\delta$ of the twisted
complex $\underline{\boldsymbol{\beta}}$ be such that all the components
$\delta(f,g)$ ($f<g$) of the differential $\delta$ have Alexander
$\mathbb{Z}/2$-grading $0$.
\end{lem}

\begin{proof}
Let us first consider consecutive $f<g$. Since the canonical elements
$\theta(f,g)\in HF_{0}^{-}(\boldsymbol{\beta}(f),\boldsymbol{\beta}(g))$
are homogeneous with respect to the relative Alexander $\mathbb{Z}/2$-grading
and the $\Phi$ actions have Alexander $\mathbb{Z}/2$ degree $1$,
the homology classes $\delta(f,g)\in HF_{0}^{-}(\boldsymbol{\beta}(f),\boldsymbol{\beta}(g))$
are also homogeneous with respect to the relative Alexander $\mathbb{Z}/2$-grading.
By Proposition \ref{prop:We-have-the} (\ref{enu:4}) and the corresponding
statement of Proposition \ref{prop:(Modified-band-maps-2}, we can
choose an Alexander $\mathbb{Z}/2$-splitting of ${\cal A}_{0}$ such
that all the homology classes $\delta(f,g)$ have Alexander $\mathbb{Z}/2$-grading
$0$. Hence, we could have let the cycles $\delta(f,g)\in CF_{0}^{-}(\boldsymbol{\beta}(f),\boldsymbol{\beta}(g))$
that represent the homology classes $\delta(f,g)$ have Alexander
$\mathbb{Z}/2$-grading $0$ as well. The lemma follows from this
since when we recursively defined $\delta(f,g)$ for $f<g$, $|g-f|\ge2$,
we could have let $\delta(f,g)$ have Alexander $\mathbb{Z}/2$-grading
$0$.
\end{proof}
Choose an Alexander $\mathbb{Z}/2$-splitting of ${\cal A}_{0}$ and
$\delta$ as in Lemma \ref{lem:alexz/2}. (Also, let $\delta$ have
homological degree $-1$ as above.) This induces an Alexander $\mathbb{Z}/2$-splitting
on $CF^{-}(\boldsymbol{\alpha},\underline{\boldsymbol{\beta}})$.
The differential on $CF^{-}(\boldsymbol{\alpha},\underline{\boldsymbol{\beta}_{01}})$
has Alexander $\mathbb{Z}/2$-degree $0$, and the chain map $CF^{-}(\boldsymbol{\alpha},\underline{\boldsymbol{\beta}_{01}})\to CF^{-}(\boldsymbol{\alpha},\boldsymbol{\beta}(\boldsymbol{2}))$
(defined as in Subsection \ref{subsec:The-twisted-complex}) is a
quasi-isomorphism, is homogeneous with respect to the relative homological
grading, and preserves the Alexander $\mathbb{Z}/2$-grading. 

The underlying module of the $E_{1}$ page is $\bigoplus_{f:X\to\{0,1\}}HF^{-}(\boldsymbol{\alpha},\boldsymbol{\beta}(f))$,
which we identified with $CKh^{-}$. Under this identification, for
each $f:X\to\{0,1\}$, the relative $q/2$-grading modulo $2$ on
$HF^{-}(\boldsymbol{\alpha},\boldsymbol{\beta}(f))$ agrees with the
relative Alexander $\mathbb{Z}/2$-grading. Also, since $d_{1}$ preserves
the Alexander $\mathbb{Z}/2$-grading, we can deduce that $q/2$ modulo
$2$ agrees with the Alexander $\mathbb{Z}/2$-grading up to an overall
shift. Since the differentials $d_{k}$ preserve the Alexander $\mathbb{Z}/2$-grading,
$d_{2\ell}$ vanishes for $\ell\ge1$.

\section{\label{sec:Some-simple-examples}Some simple examples}

For certain links, it is possible to compute the spectral sequence
using its formal properties. For instance, for alternating knots,
the spectral sequence for the reduced hat version 
\[
\widetilde{Kh}(m(K))\Rightarrow\widetilde{HFL'}(S^{3},K)=\widehat{HFK}(S^{3},K)
\]
is trivial as the ranks of $\widetilde{Kh}$ and $\widehat{HFK}$
are the same. However, we will see that the spectral sequence for
the unreduced hat version is already nontrivial for the trefoils.
Also, we can sometimes pin down the homological filtration on $HFL'{}^{-}(Y,L)$
using the $R_{\boldsymbol{CC}}$-module structure.

\subsection{The Hopf link}

\begin{figure}[h]
\begin{centering}
\includegraphics{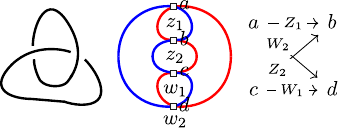}
\par\end{centering}
\caption{\label{fig:hopf}The Hopf link $H$, a Heegaard diagram for it, and
its chain complex over $\mathbb{F}[Z_{1},Z_{2},W_{1},W_{2}]$ }
\end{figure}

Let $H$ be the Hopf link. See Figure \ref{fig:hopf}: we have 
\[
HFL'{}^{-}(S^{3},H)=\mathbb{F}[U_{1}^{1/2}]\{b\}\oplus\mathbb{F}[U_{1}^{1/2}]\{b+d\}
\]
as $\mathbb{F}[U_{1}^{1/2}]$-modules, and in the basis $(b,b+d)$,
$U_{2}^{1/2}$ acts as
\[
U_{1}^{1/2}\begin{pmatrix}1 & 0\\
1 & 1
\end{pmatrix}.
\]
On the other hand, $Kh^{-}(m(H))=\mathbb{F}[U_{1}^{1/2}]\{x\}\oplus\mathbb{F}[U_{1}^{1/2}]\{y\}$,
and the action of $U_{2}^{1/2}$ is the same as $U_{1}^{1/2}$.

The spectral sequence collapses on the $E_{2}$ page, but the $R_{\boldsymbol{CC}}$-module
structures of $Kh^{-}(m(H))$ and $HFL'{}^{-}(S^{3},H)$ are different.
From this, we can deduce that the induced homological filtration on
the $HFL'{}^{-}(S^{3},H)$ (the $E_{\infty}$ page) is (the \emph{absolute}
homological filtration degree is not defined since we have not fixed
an orientation on the link)
\[
HFL'{}^{-}(S^{3},H)=F_{-2}=F_{-1}\supset F_{0}=\mathbb{F}[U_{1}^{1/2}]\{b+d\}\supset F_{1}=\{0\}.
\]

\subsection{A trefoil}

\begin{figure}[h]
\begin{centering}
\includegraphics{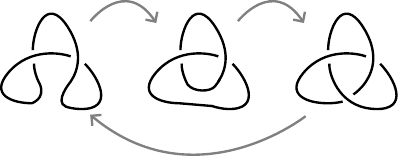}
\par\end{centering}
\caption{\label{fig:hopf-1}An unoriented skein exact triple (unknot, Hopf
link $H$, right handed trefoil $T$)}
\end{figure}

Let $T$ be the right handed trefoil. As observed above, the spectral
sequence for the reduced hat version collapses on the $E_{2}$ page.
However, since the unoriented link Floer chain complex for the trefoil
$T$ is 
\[
\mathbb{F}[U^{1/2}]\{b\}\oplus\left(\mathbb{F}[U^{1/2}]\{c\}\xrightarrow{\cdot U^{1/2}}\mathbb{F}[U^{1/2}]\{a\}\right),
\]
($a,b,c$ are in the same $\delta$-grading), the unreduced hat version
$\widehat{HFL'}(S^{3},T)$ has rank $4$, and the minus version $HFL'{}^{-}(S^{3},T)\simeq\mathbb{F}[U^{1/2}]\oplus\mathbb{F}[U^{1/2}]/(U^{1/2})$.
On the other hand, $\widehat{Kh}(m(T))$ is as in Figure \ref{fig:hopf-1-1}
and 
\[
Kh^{-}(m(T))=\mathbb{F}[U^{1/2}]\{x\}\oplus\mathbb{F}[U^{1/2}]\{y\}\oplus\mathbb{F}[U^{1/2}]\{z\},
\]
and so the spectral sequence does \emph{not} collapse on the $E_{2}$
page for both the unreduced hat version and the minus version.

Let us focus on the unreduced hat version. Since the differential
$d_{k}$ has bidegree $(q,h)=(2k-2,k)$ and $d_{2\ell}$ vanishes
for $\ell\ge1$, there is only one possibility for the spectral sequence:
the only nontrivial differential is $d_{3}:z\mapsto U^{1/2}x$.
\begin{figure}[h]
\begin{centering}
\includegraphics{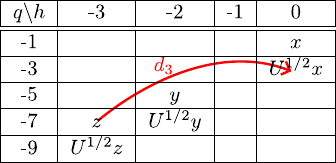}
\par\end{centering}
\caption{\label{fig:hopf-1-1}The Khovanov homology of the left handed trefoil
$\widehat{Kh}(m(T))$ and the differential $d_{3}$}
\end{figure}

\begin{rem}
Since the filtered chain complex for the Hopf link is a filtered sub
chain complex of the filtered chain complex for the right handed trefoil,
we get a map between spectral sequences, from the spectral sequence
of the Hopf link to the spectral sequence of the trefoil. On the $E_{2}$
page, this map is $\widehat{Kh}(m(H))\to\widehat{Kh}(m(T))$, which
is an isomorphism onto the homological grading $-2,-1,0$ summands.
\end{rem}

Using the map from the minus version of the spectral sequence to the
unreduced hat version, we see that the spectral sequence for the minus
version is given by $d_{3}z=U^{1/2}x$ (and $d_{3}$ is $\mathbb{F}[U^{1/2}]$-linear).

Since for knots, the $E_{\infty}$ page for the minus version has
one $\mathbb{F}[U^{1/2}]$-tower, we can ask what its homological
filtration degree is. For the trefoils, it is $\pm2$. (We can also
instead consider the infinity version of the spectral sequence, for
which the $E_{\infty}$ page is one copy of $\mathbb{F}[U^{1/2},U^{-1}]$.)
\begin{rem}
It is interesting to compare this spectral sequence with Rasmussen's
$E(-1)$ spectral sequence. For knots, the $E_{\infty}$ page is $\mathbb{F},\mathbb{F}[x]/x^{2},\mathbb{F}[x]$
for the reduced, unreduced, and the minus version, respectively. For
the left handed trefoil, in all these versions, there exists a nontrivial
$d_{3}$, from $z$ to $x$.

The $t$-invariant defined by Ballinger is the homological filtration
degree of the generator on the $E_{\infty}$ page for the $E(-1)$
spectral sequence. This is a concordance invariant, and gives non-orientable
four ball genus bounds like the $v$-invariant defined by Ozsv\'{a}th,
Stipsicz, and Szab\'{o} in \cite{MR3694597}, using unoriented link
Floer homology. For the trefoils, the $t$-invariant is $\pm2$, and
in fact $t(K)=v(K)$ for alternating knots.
\end{rem}

\section{\label{sec:Interpreting-schematics}Interpreting schematics}

Throughout the remainder of the paper, we will claim and show that
certain diagrams commute. We state these statements in the form of
schematics, as in Figures \ref{fig:swap-swap} and \ref{fig:nori-swap-com}
(see the later sections for more examples). In this section, we explain
how to interpret these schematics and introduce notations. 

Each schematic is either a triangle or a square, as in Figure \ref{fig:schematic}\footnote{In the first case, $0<1<2$, in the second case, $00<01<11$, $00<10<11$,
and $01$ and $10$ are incomparable.}. For each schematic, we are in the setting of Subsection \ref{subsec:The-general-setup}.
Since we will only consider beta-attaching curves, the ambient manifold
is unimportant. The schematics describe balled links $L_{v}$ for
each vertex $v$. The collection of crossing balls $CB_{1},\cdots,CB_{s}$
will be specified in the schematics\footnote{With the exception that in Subsection \ref{subsec:Modified-band-maps},
not all crossing balls are shown.} ($s=0$ for Figure \ref{fig:swap-swap} and $s=1$ for Figure \ref{fig:nori-swap-com}).
For each $L_{v}$ and $CB_{s}$, we identify $L_{v}\cap CB_{s}$ in
$CB_{s}$ with either $0$ or $1$ of Figure \ref{fig:skein-moves-1}.

The baseballs can be placed anywhere on the specified connected component
of $L\backslash\bigsqcup_{r}CB_{r}$. Hence, for convenience, we might
draw multiple link basepoints at the same place. The link baseballs
are colored red, and the free baseballs are sometimes not drawn. We
are free to add any number of baseballs that are free baseballs for
all $L_{v}$.

\begin{figure}[h]
\begin{centering}
\includegraphics{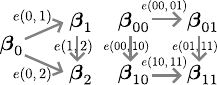}
\par\end{centering}
\caption{\label{fig:schematic}Commutative diagrams}
\end{figure}

Subsection \ref{subsec:The-Heegaard-diagram} gives us a Heegaard
diagram ${\cal H}^{wind}$ with Heegaard surface $-\partial H(L)$
and attaching curve $\boldsymbol{\beta}_{v}$ for each vertex $v$
of the schematic. The following are some concrete examples of the
corresponding Heegaard diagrams: Figure \ref{fig:z11-3-heegaard}
(after performing a $0$-surgery along $z_{1}$ and $w_{1}$) corresponds
to Figure \ref{fig:z11-3}, and Figure \ref{fig:z21-2-heegaard-2}
(after performing a $0$-surgery along $z_{1}$ and $w_{1}$) corresponds
to Figure \ref{fig:z21-2}.

An element $e(v,v')\in HF^{-}(\boldsymbol{\beta}_{v},\boldsymbol{\beta}_{v'})$
will be assigned for each edge $v\to v'$; $e(v,v')$ is sometimes
described directly in the schematic (recall from Definition \ref{def:If--are}
that ``swap'', ``non-orientable'', ``merge'', and ``split''
refer to the canonical element $\theta$). We sometimes write $\theta(v,v')$
(resp. $\delta(v,v')$) to emphasize that it is the canonical element
$\theta$ (resp. the modified band map $\delta$) in $HF^{-}(\boldsymbol{\beta}_{v},\boldsymbol{\beta}_{v'})$.

If the schematic is a triangle, the statement is that 
\[
\mu_{2}(e(0,1)\otimes e(1,2))=e(0,2)\in HF^{-}(\boldsymbol{\beta}_{0},\boldsymbol{\beta}_{2}).
\]
If the schematic is a square, we consider the compositions $\mu_{2}^{\boldsymbol{\beta}_{00},\boldsymbol{\beta}_{01},\boldsymbol{\beta}_{11}}$
and $\mu_{2}^{\boldsymbol{\beta}_{00},\boldsymbol{\beta}_{10},\boldsymbol{\beta}_{11}}$
which will sometimes be denoted as $\mu_{2}^{rd}$ and $\mu_{2}^{dr}$,
respectively; the statement is that 
\[
x^{rd}:=\mu_{2}^{rd}(e(00,01)\otimes e(01,11))=\mu_{2}^{dr}(e(00,10)\otimes e(10,11))=:x^{dr}\in HF^{-}(\boldsymbol{\beta}_{00},\boldsymbol{\beta}_{11}).
\]

\begin{rem}
\label{rem:alex-homogeneous}The composition maps $\mu_{2}$ are homogeneous
with respect to the relative Alexander $\mathbb{Z}/2$-grading by
Proposition \ref{prop:combi-claims} (\ref{enu:The--category--1}).
Note that $e(v,v')\in HF^{-}(\boldsymbol{\beta}_{v},\boldsymbol{\beta}_{v'})$
will always be homogeneous with respect to the relative Alexander
$\mathbb{Z}/2$-grading. Hence, the compositions $\mu_{2}(e(0,1)\otimes e(1,2))$,
$x^{rd}$, and $x^{dr}$ are also homogeneous with respect to the
relative Alexander $\mathbb{Z}/2$-grading.
\end{rem}

\begin{rem}
As above, we work on the homology level (instead of the chain level)
in Sections \ref{sec:A-nonvanishing-argument}, \ref{sec:Band-maps-and},
\ref{sec:Modified-band-maps}, \ref{sec:Composition-of-band}, and
\ref{sec:Computations-for-the}; in particular, $\mu_{2}$ means the
induced composition map on homology.
\end{rem}

\section{\label{sec:A-nonvanishing-argument}A nonvanishing argument}

In this section, we introduce an argument that we will use often,
which shows that some composition $\mu_{2}(x\otimes y)$ is nonzero
on homology.

Say we are in the setting of Section \ref{sec:The-setup-and} and
that there are beta-attaching curves $\boldsymbol{\beta}<\boldsymbol{\beta}_{0}<\boldsymbol{\beta}_{1}$
such that $\boldsymbol{\beta}_{0},\boldsymbol{\beta}_{1}$ differ
in exactly one crossing ball $CB$, in which $\boldsymbol{\beta}_{i}\cap CB$
is $i$ of Figure \ref{fig:skein-moves-1}\footnote{This is similar to Definition \ref{def:Two-balled-links-1} but stronger
since $\boldsymbol{\beta}$ is included. Indeed, in particular, $\boldsymbol{\beta}\cap CB$
must be $0$ of Figure \ref{fig:skein-moves-1} since $\boldsymbol{\beta}<\boldsymbol{\beta}_{0}$.}, and $\boldsymbol{\beta}_{0},\boldsymbol{\beta}_{1}$ satisfy Condition
\ref{cond:If--is}. Then, we can find a $\boldsymbol{\beta}_{2}>\boldsymbol{\beta}_{1}$
such that $\boldsymbol{\beta}_{0},\boldsymbol{\beta}_{1},\boldsymbol{\beta}_{2}$
form an unoriented skein triple (Definition \ref{def:unoriented-skein-triple}).
Let $BB$ be the baseball from Condition \ref{cond:If--is} and Definition
\ref{def:unoriented-skein-triple}. If the basepoint pair for $BB$
is a free basepoint pair for $\boldsymbol{\beta}$, then by Theorem
\ref{thm:Let--be}, 
\[
HF_{1}^{-}(\boldsymbol{\beta},\boldsymbol{\beta}_{2})\to HF_{0}^{-}(\boldsymbol{\beta},\boldsymbol{\beta}_{0})\xrightarrow{\mu_{2}(-\otimes\theta)}HF_{0}^{-}(\boldsymbol{\beta},\boldsymbol{\beta}_{1})
\]
is exact (the homological $\mathbb{Z}$-gradings are pinned down by
Proposition \ref{prop:combi-claims} (\ref{enu:It-is-possible}) and
Lemma \ref{lem:Let--be}). Furthermore $HF_{1}^{-}(\boldsymbol{\beta},\boldsymbol{\beta}_{2})=0$
by Proposition \ref{prop:combi-claims} (\ref{enu:If-,-then}), and
so $\mu_{2}(-\otimes\theta):HF_{0}^{-}(\boldsymbol{\beta},\boldsymbol{\beta}_{0})\to HF_{0}^{-}(\boldsymbol{\beta},\boldsymbol{\beta}_{1})$
is injective.

Similarly, if there are beta-attaching curves $\boldsymbol{\beta}_{1}<\boldsymbol{\beta}_{2}<\boldsymbol{\beta}$,
such that $\boldsymbol{\beta}_{1},\boldsymbol{\beta}_{2}$ differ
in exactly one crossing ball $CB$, in which $\boldsymbol{\beta}_{i}\cap CB$
is $i$ of Figure \ref{fig:skein-moves-1} and if $\boldsymbol{\beta}_{1},\boldsymbol{\beta}_{2}$
satisfy Condition \ref{cond:If--is}, then $\mu_{2}(\theta\otimes-):HF_{0}^{-}(\boldsymbol{\beta}_{2},\boldsymbol{\beta})\to HF_{0}^{-}(\boldsymbol{\beta}_{1},\boldsymbol{\beta})$
is injective. (Note that if $\boldsymbol{\beta}_{i}\cap CB$ is $i-1$
of Figure \ref{fig:skein-moves-1} and $\boldsymbol{\beta}\cap CB$
is $1$ of Figure \ref{fig:skein-moves-1}, then we can reparametrize
$CB$ such that the above conditions are satisfied.)
\begin{rem}
We have in fact used this exact argument in Subsections \ref{subsec:A-lower-bound-of-la}
and \ref{subsec:relative-homology}.
\end{rem}

\begin{rem}
\label{rem:If-we-only}If we only want to show that $\mu_{2}^{\boldsymbol{\beta}_{u},\boldsymbol{\beta}_{v},\boldsymbol{\beta}_{w}}(\theta\otimes\theta)\neq0$,
another way that works in certain cases is using that the composition
of the corresponding band maps is nonzero for planar links. In other
words, we find an attaching curve $\boldsymbol{\alpha}$ such that
\[
\mu_{2}(-\otimes\mu_{2}(\theta\otimes\theta)):HF^{-}(\boldsymbol{\alpha},\boldsymbol{\beta}_{u})\to HF^{-}(\boldsymbol{\alpha},\boldsymbol{\beta}_{w})
\]
is nonzero (and $(\boldsymbol{\alpha},\boldsymbol{\beta}_{s})$ represent
planar links and the bands are planar bands). In this case, we also
get that $\mu_{2}(\theta\otimes\theta)$ is not in the subspace of
$HF^{-}(\boldsymbol{\beta}_{u},\boldsymbol{\beta}_{w})$ spanned by
the image of the $\Phi$-actions, by Lemma \ref{lem:planar-phi}.
\end{rem}

\section{\label{sec:Band-maps-and}Swap maps and band maps}

\subsection{Swap maps}

In Proposition \ref{prop:swap-quasi-iso-balled}, we showed that swap
maps, in the restrictive sense of Section \ref{sec:Swap-maps} (Condition
\ref{cond:In-this-section,}), are quasi-isomorphisms. From this,
we deduce the following.
\begin{prop}
\label{prop:Something-is-an}Let $L_{0}$, $L_{1}$ be two balled
links that are related by a swap map (Definition \ref{def:Two-balled-links}),
and consider a weakly admissible pair-pointed Heegaard diagram $(\Sigma,\boldsymbol{\beta}_{0},\boldsymbol{\beta}_{1},\boldsymbol{p},\boldsymbol{G})$
as in Subsection \ref{subsec:The-Heegaard-diagram}. Let $\boldsymbol{\alpha}$
be an attaching curve such that $(\Sigma,\boldsymbol{\alpha},\boldsymbol{\beta}_{0},\boldsymbol{\beta}_{1},\boldsymbol{p},\boldsymbol{G})$
is $c$-strongly admissible, and that a basepoint pair $p$ is a free
basepoint pair for $\boldsymbol{\alpha}$ whenever $p$ is a free
basepoint pair for one of $\boldsymbol{\beta}_{0},\boldsymbol{\beta}_{1}$
and a link basepoint pair for the other. Then the map
\[
\mu_{2}(-\otimes\theta):HF^{-}(\boldsymbol{\alpha},\boldsymbol{\beta}_{0})\to HF^{-}(\boldsymbol{\alpha},\boldsymbol{\beta}_{1})
\]
is an isomorphism. Similarly, (if $\boldsymbol{\beta}_{0},\boldsymbol{\beta}_{1}<\boldsymbol{\alpha}$,
then)
\[
\mu_{2}(\theta\otimes-):HF^{-}(\boldsymbol{\beta}_{1},\boldsymbol{\alpha})\to HF^{-}(\boldsymbol{\beta}_{0},\boldsymbol{\alpha})
\]
is an isomorphism.
\end{prop}

\begin{proof}
We can reduce to the case where there exists exactly one component
for which the link baseball of $L_{0}$ and $L_{1}$ are different.
This is possible since the $\mu_{2}(-\otimes\theta)$ of the general
case is a composition of these cases. This case follows from Proposition
\ref{prop:swap-quasi-iso-balled}.
\end{proof}

\subsection{Swap maps and swap maps}

We show that swap maps always commute.
\begin{prop}
\label{prop:swap-commute}The diagram in Figure \ref{fig:swap-swap}
commutes.
\end{prop}

\begin{figure}[h]
\begin{centering}
\includegraphics{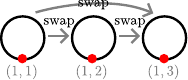}
\par\end{centering}
\caption{\label{fig:swap-swap}Swap maps commute}
\end{figure}

\begin{proof}
We use the argument of Remark \ref{rem:If-we-only}. Considering the
composition of swap maps for planar links (which is nonzero), we can
deduce that $\mu_{2}(\theta(0,1)\otimes\theta(1,2))$ is nonzero and
is in the same homological $\mathbb{Z}$-grading as $\theta(0,2)$.
Furthermore, by Lemma \ref{lem:planar-phi}, $\mu_{2}(\theta(0,1)\otimes\theta(1,2))$
is not in the image of any $\Phi$-action either, and so it is in
the same relative Alexander $\mathbb{Z}/2$-grading as $\theta(0,2)$\footnote{We can also show these statements about gradings directly from the
Heegaard diagrams.} (recall Remark \ref{rem:alex-homogeneous}). Hence they must be equal.
\end{proof}
We record the following lemma.
\begin{lem}
\label{lem:phi-swap}Consider the attaching curves given by Figure
\ref{fig:swap-swap}. The map 
\[
\mu_{2}:HF^{-}(\boldsymbol{\beta}_{0},\boldsymbol{\beta}_{1})\otimes HF^{-}(\boldsymbol{\beta}_{1},\boldsymbol{\beta}_{2})\to HF_{}^{-}(\boldsymbol{\beta}_{0},\boldsymbol{\beta}_{2})
\]
satisfies 
\[
\mu_{2}(\Phi_{1,1}\otimes{\rm Id})=\mu_{2}(\Phi_{1,2}\otimes{\rm Id})=\mu_{2}({\rm Id}\otimes\Phi_{1,2})=\mu_{2}({\rm Id}\otimes\Phi_{1,3})=\Phi_{1,1}\mu_{2}({\rm Id}\otimes{\rm Id})=\Phi_{1,3}\mu_{2}({\rm Id}\otimes{\rm Id}).
\]
\end{lem}

\begin{proof}
Since $\Phi_{1,1}=0$ on $HF^{-}(\boldsymbol{\beta}_{1},\boldsymbol{\beta}_{2})$,
we have $\mu_{2}(\Phi_{1,1}\otimes{\rm Id})=\Phi_{1,1}\mu_{2}({\rm Id}\otimes{\rm Id})$.
Since $\Phi_{1,2}=0$ on $HF^{-}(\boldsymbol{\beta}_{0},\boldsymbol{\beta}_{2})$,
we have $\mu_{2}(\Phi_{1,2}\otimes{\rm Id})=\mu_{2}({\rm Id}\otimes\Phi_{1,2})$.
Since $\Phi_{1,1}=\Phi_{1,2}$ on $HF^{-}(\boldsymbol{\beta}_{0},\boldsymbol{\beta}_{1})$,
we have $\mu_{2}(\Phi_{1,1}\otimes{\rm Id})=\mu_{2}(\Phi_{1,2}\otimes{\rm Id})$.
The rest follow similarly.
\end{proof}

\subsection{Swap maps and band maps}

In this subsection, we show that band maps and swap maps commute in
certain cases. Non-orientable band maps and swap maps always commute,
but swap maps do not always commute with merge or split bands (Proposition
\ref{prop:swap-merge-2-1}).

The non-orientable case is simpler.
\begin{prop}
\label{prop:non-ori-swap}The diagrams in Figure \ref{fig:nori-swap-com}
commute. Furthermore, the compositions are the canonical element defined
in Subsection \ref{subsec:Canonical-generators}. 
\begin{figure}[h]
\begin{centering}
\includegraphics{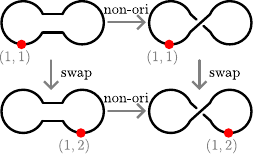}\qquad{}\includegraphics{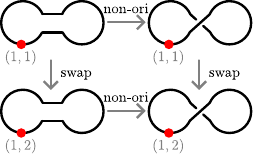}
\par\end{centering}
\caption{\label{fig:nori-swap-com}Non-orientable band maps and swap maps commute}
\end{figure}
\end{prop}

\begin{proof}
Note that the sub Heegaard diagram with attaching curves $\boldsymbol{\beta}_{00},\boldsymbol{\beta}_{11}$
represents $(\#^{2}S^{1}\times S^{2},L_{nonori})$. Let us prove that
the left hand side of Figure \ref{fig:nori-swap-com} commutes (the
right hand side follows similarly).

Consider 
\begin{equation}
\mu_{2}^{dr}:HF_{0}^{-}(\boldsymbol{\beta}_{00},\boldsymbol{\beta}_{10})\otimes HF_{0}^{-}(\boldsymbol{\beta}_{10},\boldsymbol{\beta}_{11})\to HF_{0}^{-}(\boldsymbol{\beta}_{00},\boldsymbol{\beta}_{11}).\label{eq:mu2drnonori}
\end{equation}
Let $f,g\in HF_{0}^{-}(\boldsymbol{\beta}_{00},\boldsymbol{\beta}_{10})$,
$f,g\in HF_{0}^{-}(\boldsymbol{\beta}_{10},\boldsymbol{\beta}_{11})$
be as in Definition \ref{def:f-and-g}. Then, the $\Phi$ actions
on the left hand side of Equation (\ref{eq:mu2drnonori}) are as follows:
\[\begin{tikzcd}
	{f\otimes f} & {f\otimes g} \\
	{g\otimes f} & {g\otimes g}
	\arrow["{\Phi _{1,2}}", from=1-1, to=1-2]
	\arrow["{\Phi _{1,1},\Phi _{1,2}}"{description}, from=1-1, to=2-1]
	\arrow["{\Phi _{1,1},\Phi _{1,2}}"{description}, from=1-2, to=2-2]
	\arrow["{\Phi _{1,2}}"', from=2-1, to=2-2]
\end{tikzcd}\]The map $\mu_{2}^{dr}(-\otimes\theta)$ is injective on $HF_{0}^{-}(\boldsymbol{\beta}_{00},\boldsymbol{\beta}_{10})$
by Section \ref{sec:A-nonvanishing-argument}. Also, recall that $\mu_{2}^{dr}$
is equivariant with respect to the $\Phi$ actions. Hence, we can
deduce (in the same way as in Subsection \ref{subsec:A-lower-bound-of-la})
that Equation (\ref{eq:mu2drnonori}) is injective and hence an ($\Phi$-equivariant)
isomorphism that is homogeneous with respect to the relative Alexander
$\mathbb{Z}/2$-grading. From this, it follows that $x^{dr}$ is the
canonical element.

Similarly, the $\Phi$-actions on $HF_{0}^{-}(\boldsymbol{\beta}_{00},\boldsymbol{\beta}_{01})\otimes HF_{0}^{-}(\boldsymbol{\beta}_{01},\boldsymbol{\beta}_{11})$
are as follows:
\[\begin{tikzcd}
	{f\otimes f} & {f\otimes g} \\
	{g\otimes f} & {g\otimes g}
	\arrow["{\Phi _{1,1},\Phi _{1,2}}", from=1-1, to=1-2]
	\arrow["{\Phi _{1,1}}"{description}, from=1-1, to=2-1]
	\arrow["{\Phi _{1,1}}"{description}, from=1-2, to=2-2]
	\arrow["{\Phi _{1,1},\Phi _{1,2}}"', from=2-1, to=2-2]
\end{tikzcd}\]Hence, we can show that
\[
\mu_{2}^{rd}:HF_{0}^{-}(\boldsymbol{\beta}_{00},\boldsymbol{\beta}_{01})\otimes HF_{0}^{-}(\boldsymbol{\beta}_{01},\boldsymbol{\beta}_{11})\to HF_{0}^{-}(\boldsymbol{\beta}_{00},\boldsymbol{\beta}_{11})
\]
is an ($\Phi$-equivariant) isomorphism that is homogeneous with respect
to the relative Alexander $\mathbb{Z}/2$-grading, and hence that
$x^{rd}=\mu_{2}^{rd}(g\otimes f)$ is the canonical element.
\end{proof}

We can similarly show the following.
\begin{prop}
\label{prop:swap-merge-1}The diagrams in Figure \ref{fig:merge-swap-com}
commute.
\end{prop}

\begin{figure}[h]
\begin{centering}
\includegraphics{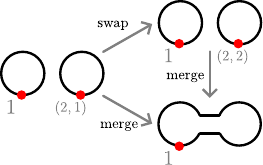}\qquad{}\includegraphics{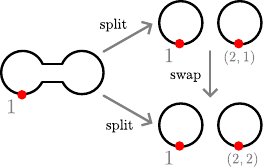}
\par\end{centering}
\caption{\label{fig:merge-swap-com}Some orientable band maps and swap maps
commute}
\end{figure}

\begin{proof}
This follows from the exact same argument as Proposition \ref{prop:swap-commute}.
\end{proof}
We use the relative homology actions $A_{ij}$ and $B_{ij}$'s from
Definition \ref{def:relative-hom} to show the next proposition.
\begin{prop}
\label{prop:swap-merge-2}The diagrams in Figure \ref{fig:split-swap-com}
commute. Furthermore, the compositions are the canonical element defined
in Subsection \ref{subsec:Canonical-generators}.
\end{prop}

\begin{figure}[h]
\begin{centering}
\includegraphics{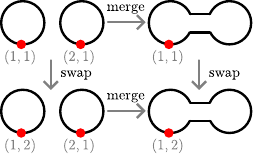}\qquad{}\includegraphics{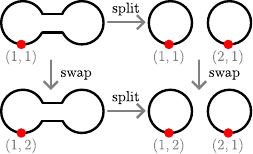}
\par\end{centering}
\caption{\label{fig:split-swap-com}More orientable band maps and swap maps
that commute}
\end{figure}

\begin{proof}
Note that the sub Heegaard diagram with attaching curves $\boldsymbol{\beta}_{00},\boldsymbol{\beta}_{11}$
represents $(\#^{2}S^{1}\times S^{2},L_{ori})$. Let us prove that
the right hand side of Figure \ref{fig:split-swap-com} commutes (the
left hand side follows similarly). Recall from Remark \ref{rem:alex-homogeneous}
that $x^{rd},x^{dr}$ are homogeneous with respect to the relative
Alexander $\mathbb{Z}/2$-grading. By Remark \ref{rem:If-we-only},
$x^{rd}$ and $x^{dr}$ are nonzero and do not lie in the span of
the images of the $\Phi$-actions. Hence, they both lie in the relative
Alexander $\mathbb{Z}/2$-grading summand ${\rm span}_{\mathbb{F}}\left\langle a,d\right\rangle $
(using the notations of Propositions \ref{prop:For-each-balled} and
\ref{prop:85}). We claim that both $x^{rd}$ and $x^{dr}$ are the
unique nonzero element $x\in HF_{0}^{-}(\boldsymbol{\beta}_{00},\boldsymbol{\beta}_{11})$
in the corresponding relative Alexander $\mathbb{Z}/2$-grading, such
that $A_{(1,1)(1,2)}\Phi_{1,1}x=x$. 

Similarly to Subsubsection \ref{subsec:841}, we can compute the $\Phi$
and $A_{(1,1),(1,2)}$ actions on $\mu_{2}^{dr}(HF_{0}^{-}(\boldsymbol{\beta}_{00},\boldsymbol{\beta}_{10})\otimes HF_{0}^{-}(\boldsymbol{\beta}_{10},\boldsymbol{\beta}_{11}))$.
(In fact,
\[
\mu_{2}^{dr}:HF_{0}^{-}(\boldsymbol{\beta}_{00},\boldsymbol{\beta}_{10})\otimes HF_{0}^{-}(\boldsymbol{\beta}_{10},\boldsymbol{\beta}_{11})\to HF_{0}^{-}(\boldsymbol{\beta}_{00},\boldsymbol{\beta}_{11})
\]
is an isomorphism. Also, in fact, the attaching curves $\boldsymbol{\alpha},\boldsymbol{\beta}_{c},\boldsymbol{\beta}_{b}$
that we considered in Subsection \ref{subsec:relative-homology} correspond
to the attaching curves $\boldsymbol{\beta}_{00},\boldsymbol{\beta}_{10},\boldsymbol{\beta}_{11}$
(respectively), where the basepoint pairs $(z_{1},w_{1}),(z_{2},w_{2}),(z_{3},w_{3})$
of Figure \ref{fig:s1s2-heegaard-1} correspond to the baseballs $BB_{1,1},BB_{1,2},BB_{2,1}$,
respectively.) Let $f,g\in HF_{0}^{-}(\boldsymbol{\beta}_{00},\boldsymbol{\beta}_{10})$,
$f,g\in HF_{0}^{-}(\boldsymbol{\beta}_{10},\boldsymbol{\beta}_{11})$
be as in Definition \ref{def:f-and-g}; then we have the following:
\[\begin{tikzcd}
	{\mu_2 ^{dr}(f\otimes f)} & {\mu_2 ^{dr}(f\otimes g)} & {\mu_2 ^{dr}(f\otimes f)} & {\mu_2 ^{dr}(f\otimes g)} \\
	{\mu_2 ^{dr}(g\otimes f)} & {\mu_2 ^{dr}(g\otimes g)} & {\mu_2 ^{dr}(g\otimes f)} & {\mu_2 ^{dr}(g\otimes g)}
	\arrow["{\Phi_{1,2},\Phi_{2,1}}", from=1-1, to=1-2]
	\arrow["{\Phi_{1,1},\Phi_{1,2}}"{description}, from=1-1, to=2-1]
	\arrow["{\Phi_{1,1},\Phi_{1,2}}"{description}, from=1-2, to=2-2]
	\arrow[curve={height=12pt}, from=1-3, to=2-3]
	\arrow[curve={height=12pt}, from=1-4, to=2-4]
	\arrow["{\Phi_{1,2},\Phi_{2,1}}"', from=2-1, to=2-2]
	\arrow["{A_{(1,1),(1,2)}}"', curve={height=12pt}, from=2-3, to=1-3]
	\arrow[curve={height=12pt}, from=2-4, to=1-4]
\end{tikzcd}\]Since $x^{dr}=\mu_{2}^{dr}(f\otimes f)$, we have $A_{(1,1)(1,2)}\Phi_{1,1}x^{dr}=x^{dr}$.

Similarly, we can compute the $\Phi$ and $B_{(1,1),(1,2)}$ actions
on $\mu_{2}^{rd}(HF_{0}^{-}(\boldsymbol{\beta}_{00},\boldsymbol{\beta}_{01})\otimes HF_{0}^{-}(\boldsymbol{\beta}_{01},\boldsymbol{\beta}_{11}))$,
and hence the map $A_{(1,1),(1,2)}=B_{(1,1),(1,2)}+\Phi_{1,1}+\Phi_{1,2}$.
They are as follows:
\[\begin{tikzcd}
	{\mu_2 ^{rd}(f\otimes f)} & {\mu_2 ^{rd}(f\otimes g)} & {\mu_2 ^{rd}(f\otimes f)} & {\mu_2 ^{rd}(f\otimes g)} & {\mu_2 ^{rd}(f\otimes f)} & {\mu_2 ^{rd}(f\otimes g)} \\
	{\mu_2 ^{rd}(g\otimes f)} & {\mu_2 ^{rd}(g\otimes g)} & {\mu_2 ^{rd}(g\otimes f)} & {\mu_2 ^{rd}(g\otimes g)} & {\mu_2 ^{rd}(g\otimes f)} & {\mu_2 ^{rd}(g\otimes g)}
	\arrow["{\Phi_{1,1},\Phi_{1,2}}", from=1-1, to=1-2]
	\arrow["{\Phi_{1,1},\Phi_{2,1}}"{description}, from=1-1, to=2-1]
	\arrow["{\Phi_{1,1},\Phi_{2,1}}"{description}, from=1-2, to=2-2]
	\arrow[curve={height=-12pt}, from=1-3, to=1-4]
	\arrow["{B_{(1,1),(1,2)}}", curve={height=-12pt}, from=1-4, to=1-3]
	\arrow[curve={height=-12pt}, from=1-5, to=1-6]
	\arrow[from=1-5, to=2-5]
	\arrow["{A_{(1,1),(1,2)}}", curve={height=-12pt}, from=1-6, to=1-5]
	\arrow[from=1-6, to=2-6]
	\arrow["{\Phi_{1,1},\Phi_{1,2}}"', from=2-1, to=2-2]
	\arrow[curve={height=-12pt}, from=2-3, to=2-4]
	\arrow[curve={height=-12pt}, from=2-4, to=2-3]
	\arrow[curve={height=-12pt}, from=2-5, to=2-6]
	\arrow[curve={height=-12pt}, from=2-6, to=2-5]
\end{tikzcd}\]Hence, $A_{(1,1)(1,2)}\Phi_{1,1}x^{rd}=x^{rd}$.

\end{proof}
Similarly, we can show the following.
\begin{prop}
\label{prop:swap-merge-2-1}The diagrams in Figure \ref{fig:split-swap-com-1}
do not commute. In fact, we have $x^{rd}=x^{dr}+\Phi_{1}\Phi_{2}x^{dr}$.
\end{prop}

\begin{figure}[h]
\begin{centering}
\includegraphics{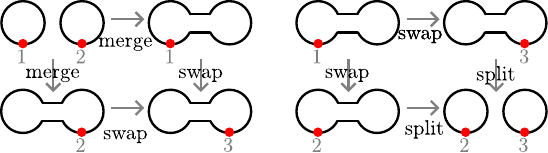}
\par\end{centering}
\caption{\label{fig:split-swap-com-1}Some orientable band maps and swap maps
that do not commute}
\end{figure}

\begin{figure}[h]
\begin{centering}
\includegraphics{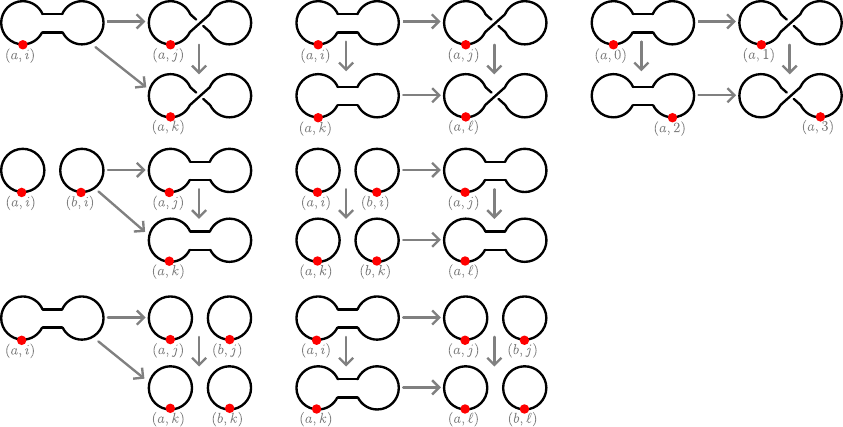}
\par\end{centering}
\caption{\label{fig:composite-phi-all-cases-2}Commutative diagrams ($a<b$
are positive integers)}
\end{figure}

Let us summarize some of our findings.
\begin{lem}
\label{lem:summary-theta}Consider Figure \ref{fig:composite-phi-all-cases-2}.
For each case, there is exactly one crossing ball, and for each $L_{v}$,
there may be some components that are not drawn (but these do not
depend on $v$). For the first and second columns, respectively, let
us consider the following cases:
\[
(i,j,k)=(0,0,1),(0,1,2),\ {\rm resp.}\ (i,j,k,\ell)=(0,0,1,1),(0,0,1,2),(0,1,2,3).
\]
All these cases commute if we assign the canonical element $\theta$
to every edge.
\end{lem}

\begin{proof}
\textbf{First row}: these follow from Propositions \ref{prop:non-ori-swap}
and \ref{prop:swap-commute}.

\textbf{Second and third rows, first column}: $(i,j,k)=(0,0,1)$ is
Proposition \ref{prop:swap-merge-2}, and $(i,j,k)=(0,1,2)$ follows
from a \emph{diagram chasing argument} using the $(0,0,1)$ case and
Proposition \ref{prop:swap-commute}. Indeed, let $(i,j,k)=(0,1,2)$,
and let us consider the second row. Let $L_{v}$ and $\boldsymbol{\beta}_{v}$
for $v\in\{0,1,2\}$ be as in Section \ref{sec:Interpreting-schematics}.
Let $L_{aux}$ be the same as $L_{1}$ but let the link baseball of
the component of $L_{aux}$ that $BB_{a,1}$ is on be $BB_{a,0}$
instead of $BB_{a,1}$. Let $\boldsymbol{\beta}_{aux}$ be the corresponding
attaching curve. Then, 
\begin{multline}
\mu_{2}^{\boldsymbol{\beta}_{0},\boldsymbol{\beta}_{1},\boldsymbol{\beta}_{2}}(\theta\otimes\theta)=\mu_{2}^{\boldsymbol{\beta}_{0},\boldsymbol{\beta}_{1},\boldsymbol{\beta}_{2}}\left(\mu_{2}^{\boldsymbol{\beta}_{0},\boldsymbol{\beta}_{aux},\boldsymbol{\beta}_{1}}(\theta\otimes\theta)\otimes\theta\right)\\
=\mu_{2}^{\boldsymbol{\beta}_{0},\boldsymbol{\beta}_{aux},\boldsymbol{\beta}_{2}}\left(\theta\otimes\mu_{2}^{\boldsymbol{\beta}_{aux},\boldsymbol{\beta}_{1},\boldsymbol{\beta}_{2}}(\theta\otimes\theta)\right)=\mu_{2}^{\boldsymbol{\beta}_{0},\boldsymbol{\beta}_{aux},\boldsymbol{\beta}_{2}}\left(\theta\otimes\theta\right)=\theta,\label{eq:diagram-chase}
\end{multline}
where the first and fourth equalities are the $(0,0,1)$ case, and
the third equality is Proposition \ref{prop:swap-commute}. The third
row follows similarly.

\textbf{Second and third rows, second column}: $(i,j,k,\ell)=(0,0,1,1)$
follows from by Propositions \ref{prop:swap-merge-1}, \ref{prop:swap-merge-2},
and \ref{prop:swap-commute}. The other $(i,j,k,\ell)$'s follow from
a diagram chasing argument as in Equation (\ref{eq:diagram-chase}),
using the $(0,0,1,1)$ case, the $(0,0,1)$ case of the first column,
and Proposition \ref{prop:swap-commute}.
\end{proof}

\section{\label{sec:Modified-band-maps}Modified band maps and modified swap
maps}

In this section, we show that modified band maps commute with modified
swap maps, thus showing Proposition \ref{prop:(Modified-band-maps}.
This follows from Section \ref{sec:Band-maps-and} by purely algebraic
arguments.

\subsection{Modified band maps and swap maps}

Recall that band maps sometimes commuted with swap maps but not always.
Similarly, modified band maps sometimes commute with swap maps, but
they do not always commute.
\begin{lem}
\label{lem:summary-delta}Consider the second and third rows of Figure
\ref{fig:composite-phi-all-cases-2}. For each case, there is exactly
one crossing ball. For each $L_{v}$, there are some components that
are not drawn; call these components $C_{t}$ (does not depend on
$v$) for some positive integers $t$ such that $a<t<b$. Assume that
the link baseballs of the $C_{t}$ component of $L_{v}$ are pairwise
distinct; call it $BB_{t,v}$. Define $\Phi_{ab}:HF^{-}(\boldsymbol{\beta}_{v},\boldsymbol{\beta}_{v'})\to HF^{-}(\boldsymbol{\beta}_{v},\boldsymbol{\beta}_{v'})$
as $\Phi_{ab}:=\sum_{k=a+1}^{b-1}\Phi_{k,v}$. For each horizontal
or diagonal edge $e:v\to v'$, let $t$ be the unique $t\in\{i,j,k,\ell\}$
such that $BB_{b,t}$ is a link baseball of $L_{v}$ or $L_{v'}$,
and define the \emph{modified band map} 
\[
\delta:=\theta+\Phi_{b,t}\Phi_{ab}\theta\in HF_{0}^{-}(\boldsymbol{\beta}_{v},\boldsymbol{\beta}_{v'}).
\]
Assign the swap map $\theta$ to each vertical edge and the modified
band map $\delta$ to each horizontal or diagonal edge.

For the first and second columns, respectively, all the following
cases commute:
\[
(i,j,k)=(0,0,1),(0,1,2),\ {\rm resp.}\ (i,j,k,\ell)=(0,0,1,1),(0,0,1,2),(0,1,2,3).
\]
\end{lem}

\begin{proof}
We deduce these from Lemma \ref{lem:summary-theta} together with
some purely algebraic arguments. First, note that by Lemma \ref{lem:phi-swap},
\begin{equation}
\mu_{2}(\Phi_{ab}\otimes{\rm Id})=\mu_{2}({\rm Id}\otimes\Phi_{ab})=\Phi_{ab}\mu_{2}({\rm Id}\otimes{\rm Id}).\label{eq:swapphi}
\end{equation}

\textbf{First column}: let us consider the $(i,j,k)=(0,0,1)$ cases;
the general cases are similar. For the second row, we have
\[
\mu_{2}(\delta\otimes\theta)=\mu_{2}(({\rm Id}+\Phi_{b,0}\Phi_{ab})\theta\otimes\theta)=({\rm Id}+\Phi_{b,0}\Phi_{ab})\mu_{2}(\theta\otimes\theta)=({\rm Id}+\Phi_{b,0}\Phi_{ab})\theta=\delta,
\]
where the second equality follows from Equation (\ref{eq:swapphi})
and that $\mu_{2}(\Phi_{b,0}\otimes{\rm Id})=\Phi_{b,0}\mu_{2}({\rm Id}\otimes{\rm Id})$
(which holds since $\Phi_{b,0}=0$ on $HF^{-}(\boldsymbol{\beta}_{1},\boldsymbol{\beta}_{2})$).
The third row follows similarly from $\mu_{2}(\Phi_{b,0}\otimes{\rm Id})=\Phi_{b,1}\mu_{2}({\rm Id}\otimes{\rm Id})$.

\textbf{Second column}: the $(i,j,k,\ell)=(0,0,1,1)$ cases follow
from
\begin{gather*}
\mu_{2}^{rd}(\Phi_{b,0}\otimes{\rm Id})=\Phi_{b,0}\mu_{2}^{rd}({\rm Id}\otimes{\rm Id}),\ \mu_{2}^{dr}({\rm Id}\otimes\Phi_{b,1})=\Phi_{b,0}\mu_{2}^{dr}({\rm Id}\otimes{\rm Id}),\ {\rm resp.}\\
\mu_{2}^{rd}(\Phi_{b,0}\otimes{\rm Id})=\Phi_{b,1}\mu_{2}^{rd}({\rm Id}\otimes{\rm Id}),\ \mu_{2}^{dr}({\rm Id}\otimes\Phi_{b,1})=\Phi_{b,1}\mu_{2}^{dr}({\rm Id}\otimes{\rm Id}).
\end{gather*}
The general cases are similar.
\end{proof}
Similarly, we can show the following, using Proposition \ref{prop:swap-merge-2-1}.
\begin{lem}
\label{lem:composite-phi-1}Consider Figure \ref{fig:composite-phi-all-cases-1}.
Some components are not drawn as in Lemma \ref{lem:summary-delta}.
Assign the swap map $\theta$ to each vertical edge and the modified
band map $\delta$ (defined as in Lemma \ref{lem:summary-delta})
to each horizontal edge. Both diagrams do not commute, and $x^{dr}=x^{rd}+\Phi_{a,0}\Phi_{b,3}x^{rd}$
for both cases.
\end{lem}

\begin{figure}[h]
\begin{centering}
\includegraphics{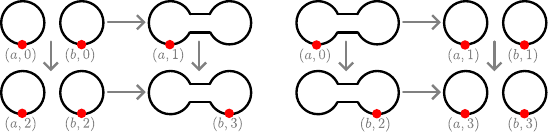}
\par\end{centering}
\caption{\label{fig:composite-phi-all-cases-1}Noncommutative diagrams}
\end{figure}

\subsection{\label{subsec:Modified-band-maps}Modified band maps and modified
swap maps}

In this subsection, we show Proposition \ref{prop:(Modified-band-maps}.
Recall that the bijections (``total order'') $\sigma,\tau:\{1,\cdots,N\}\to\boldsymbol{CC}$
satisfy $\tau=\sigma\circ(b\ b+1)$ for some $b=1,\cdots,N-1$. For
notational convenience, let us rename the balled links as $L_{00}:=L(f,\sigma)$,
$L_{01}:=L(g,\sigma)$, $L_{10}:=L(f,\tau)$, and $L_{11}:=L(g,\tau)$;
and let the attaching curve $\boldsymbol{\beta}_{s}$ for $s\in\{00,01,10,11\}$
be the attaching curve for $L_{s}$. Let $CB$ be the crossing ball
that the underlying links differ in. Identify (by reparametrizing
$CB$) $L_{jk}\cap CB\subset CB$ ($j,k=0,1$) with $k$ of Figure
\ref{fig:skein-moves-1}, and let $B\subset CB$ be the corresponding
band (from $0$ to $1$).

Rename the modified band maps (Subsubsection \ref{subsec:Step-2}
and Definition \ref{def:For-each-consecutive}) as $\delta_{\sigma}:=\delta\in HF^{-}(\boldsymbol{\beta}_{00},\boldsymbol{\beta}_{01})$
and $\delta_{\tau}:=\delta\in HF^{-}(\boldsymbol{\beta}_{10},\boldsymbol{\beta}_{11})$
to avoid confusion, and let us use the notations $\mu_{2}^{rd}$,
$\mu_{2}^{dr}$ as in Subsection \ref{sec:Interpreting-schematics}.
The statement of Proposition \ref{prop:(Modified-band-maps} is 
\begin{equation}
\mu_{2}^{rd}(\delta_{\sigma}\otimes\chi)=\mu_{2}^{dr}(\chi\otimes\delta_{\tau}).\label{eq:modified-band-maps-statement}
\end{equation}

\begin{figure}[h]
\begin{centering}
\includegraphics{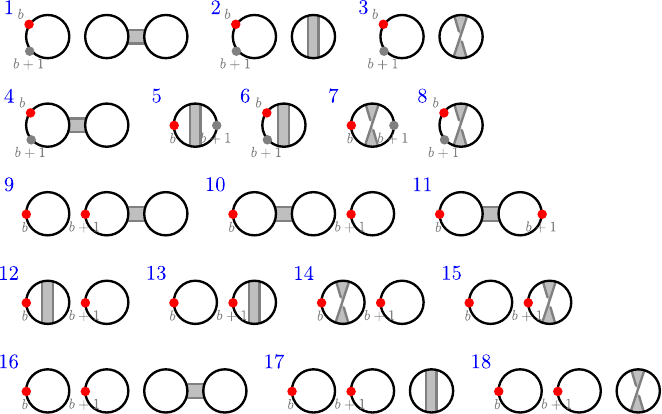}
\par\end{centering}
\caption{\label{fig:one-band}All the cases for Proposition \ref{prop:(Modified-band-maps}}
\end{figure}

For notational convenience, let us rename the baseballs as $BB_{i,00}:=BB_{\sigma(i),f,\sigma}$,
$BB_{i,01}:=BB_{\sigma(i),g,\sigma}$, $BB_{i,10}:=BB_{\tau(i),f,\tau}$,
and $BB_{i,11}:=BB_{\tau(i),g,\tau}$. Figure \ref{fig:one-band}
lists all possible configurations of the link $L_{00}$ (ignoring
the components of $L_{00}$ that do not intersect $B\cup BB_{b,00}\cup BB_{b+1,00}$),
the band $B$, and the baseballs $BB_{b,00}$ and $BB_{b+1,00}$ (simply
labelled $b$ and $b+1$), up to an orientation-preserving diffeomorphism
of $\overline{N}(L_{00}\cup B)\cup BB_{b,00}\cup BB_{b+1,00}$.

We will show Equation (\ref{eq:modified-band-maps-statement}) using
Lemmas \ref{lem:summary-theta} and \ref{lem:summary-delta} together
with a series of reductions. 
\begin{defn}
Let $\boldsymbol{CD}$ be the set of connected components of $L_{00}\backslash CB$.
Recall that $\boldsymbol{CC}$ is the set of connected components
of $L_{00}\backslash\bigsqcup CB_{r}$ where $CB_{r}$ ranges over
all the crossing balls. For each $CC\in\boldsymbol{CC}$, there exists
a unique $CD\in\boldsymbol{CD}$ such that $CC\subset CD$.

Given a total order $\eta$ on $\boldsymbol{CC}$, let ${\rm min}_{\eta}CD\in\boldsymbol{CC}$
for $CD\in\boldsymbol{CD}$ be the smallest $CC\in\boldsymbol{CC}$
(with respect to $\eta$) such that $CC\subset CD$. Define the \emph{induced
total order }$<_{\eta}$ on $\boldsymbol{CD}$ by letting $CD<_{\eta}CD'$
if and only if ${\rm min}_{\eta}CD$ is smaller than ${\rm min}_{\eta}CD'$
with respect to $\eta$.

For $\eta\in\{\sigma,\tau\}$ and $a=1,\cdots,N$, say $a$ is \emph{$\eta$-minimal
}if $\eta(a)={\rm \min}_{\eta}CD$ for the connected component $CD\in\boldsymbol{CD}$
such that $\eta(a)\subset CD$.

Let $CD_{x},CD_{y}\in\boldsymbol{CD}$ be the two connected components
of $L_{00}\backslash CB$ whose closure intersects $CB$, such that
$CD_{x}<_{\sigma}CD_{y}$. Say \emph{$\sigma$ and $\tau$ are compatible
with respect to $B$} if $B$ is non-orientable, or if $B$ is a merge
or a split band and $CD_{x}<_{\tau}CD_{y}$.
\end{defn}

If $B$ is a merge band (resp. a split band), then the link baseball
of $L_{01}$ (resp. $L_{00}$) is on $CD_{x}$. If $CD_{x}<_{\tau}CD_{y}$,
then similarly if $B$ is a merge band (resp. a split band), then
the link baseball of $L_{11}$ (resp. $L_{10}$) is on $CD_{x}$.
Hence, if $CD_{x}<_{\tau}CD_{y}$, then 
\begin{equation}
\mu_{2}^{rd}(\theta\otimes\theta)=\mu_{2}^{dr}(\theta\otimes\theta)\label{eq:thetas-statement}
\end{equation}
by Lemma \ref{lem:summary-theta}. If $B$ is non-orientable, then
Equation (\ref{eq:thetas-statement}) follows from Lemma \ref{lem:summary-theta}
as well. Hence, Equation (\ref{eq:thetas-statement}) holds if $\sigma$
and $\tau$ are compatible with respect to $B$.
\begin{defn}
\label{def:Assume-that-}Assume that $\sigma$ and $\tau$ are compatible
with respect to $B$. For $\eta\in\{\sigma,\tau\}$ and $i\in\{0,1\}$,
define $\delta_{\eta}\in HF^{-}(\boldsymbol{\beta}_{i0},\boldsymbol{\beta}_{i1})$
as follows:
\begin{itemize}
\item If $B$ is non-orientable, then let $\delta_{\eta}:=\theta$.
\item If $B$ is a merge band (resp. split band), let $\Phi_{y}$ be the
$\Phi$-action of the link component of $L_{i0}$ (resp. $L_{i1}$)
that contains $CD_{y}$. For $CD\in\boldsymbol{CD}\backslash\{CD_{x},CD_{y}\}$,
define $\Phi_{CD}$ as the $\Phi$-action of the link component $CD$
of $L_{i0}$. Also, let $\Phi_{xy,\eta}$ be the sum of $\Phi_{CD}$
for $CD\in\boldsymbol{CD}$ such that $CD_{x}<_{\eta}CD<_{\eta}CD_{y}$.
Define
\[
\delta_{\eta}:=\theta+\Phi_{y}\Phi_{xy,\eta}\theta.
\]
\end{itemize}
\end{defn}

\begin{rem}
The elements $\delta_{\sigma}\in HF^{-}(\boldsymbol{\beta}_{00},\boldsymbol{\beta}_{01})$
and $\delta_{\tau}\in HF^{-}(\boldsymbol{\beta}_{10},\boldsymbol{\beta}_{11})$
from Definition \ref{def:Assume-that-} agree with the modified band
maps that we have already defined.
\end{rem}

If $\sigma$ and $\tau$ are compatible with respect to $B$, then
\begin{equation}
\mu_{2}^{rd}(\delta_{\sigma}\otimes\theta)=\mu_{2}^{dr}(\theta\otimes\delta_{\sigma})\label{eq:deltasigma-theta}
\end{equation}
by Lemmas \ref{lem:summary-theta} (if $B$ is non-orientable) and
\ref{lem:summary-delta} (if $B$ is a merge or split band).
\begin{lem}
\label{lem:minimal-reduce}Equation (\ref{eq:modified-band-maps-statement})
holds if $b$ or $b+1$ is not $\sigma$-minimal.
\end{lem}

\begin{proof}
Say $b$ or $b+1$ is not $\sigma$-minimal. Then, $b$ or $b+1$
is not $\tau$-minimal either, and so for each $s\in\{00,01,10,11\}$,
either $BB_{b,s}$ or $BB_{b+1,s}$ is not a link baseball of $L_{s}$.
Hence, $\theta=\chi$ for $HF^{-}(\boldsymbol{\beta}_{00},\boldsymbol{\beta}_{10})$
and $HF^{-}(\boldsymbol{\beta}_{01},\boldsymbol{\beta}_{11})$.

If $B$ is non-orientable, then $\delta_{\sigma}=\theta\in HF^{-}(\boldsymbol{\beta}_{00},\boldsymbol{\beta}_{10})$
and $\delta_{\tau}=\theta\in HF^{-}(\boldsymbol{\beta}_{01},\boldsymbol{\beta}_{11})$,
and so Equation (\ref{eq:modified-band-maps-statement}) follows from
Equation (\ref{eq:thetas-statement}).

Let $B$ be a merge or a split band. Since $b$ or $b+1$ is not minimal,
$<_{\sigma}=<_{\tau}$ on $\boldsymbol{CD}$. Hence, $\sigma$ and
$\tau$ are compatible with respect to $B$ and $\delta_{\sigma}=\delta_{\tau}$.
Thus, Equation (\ref{eq:modified-band-maps-statement}) follows from
Equation (\ref{eq:deltasigma-theta}).
\end{proof}
By Lemma \ref{lem:minimal-reduce}, we may assume that both $b$ and
$b+1$ are $\sigma$-minimal; in particular, the baseballs ($BB_{b,00}$
and/or $BB_{b+1,00}$) colored red in Figure \ref{fig:one-band} are
indeed link baseballs on $L_{00}$.

For all but the fifth and eleventh cases, $\sigma$ and $\tau$ are
compatible with respect to $B$. Hence, Equation (\ref{eq:deltasigma-theta})
holds.

\subsubsection{The fifth and eleventh cases}

The fifth and eleventh cases demonstrate why we should modify the
vertical swap maps: see Figure \ref{fig:one-band-2}. For both cases,
we have $\delta_{\sigma}(00,01)=\theta(00,01)$ and $\delta_{\tau}(10,11)=\theta(10,11)$.
The fifth case follows from the following, where the second equality
is Lemma \ref{lem:composite-phi-1}:
\[
\mu_{2}^{dr}(\chi\otimes\delta_{\tau})=\mu_{2}^{dr}(\theta\otimes\delta_{\sigma})=({\rm Id}+\Phi_{b+1,11}\Phi_{b,11})\mu_{2}^{rd}(\delta_{\sigma}\otimes\theta)=\mu_{2}^{rd}(\delta_{\sigma}\otimes\chi).
\]
The eleventh case follows similarly.

\begin{figure}[h]
\begin{centering}
\includegraphics{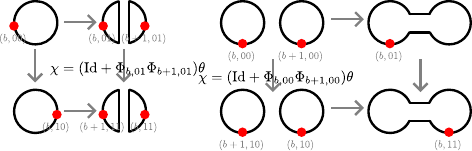}
\par\end{centering}
\caption{\label{fig:one-band-2}The fifth and eleventh cases of Figure \ref{fig:one-band}
(the components that do not intersect $B\cup BB_{b,00}\cup BB_{b+1,00}$
are not drawn)}
\end{figure}

\subsubsection{The ninth, tenth, twelfth, and thirteenth cases}

Let us consider the ninth case (the other three cases follow similarly).
We derive Equation (\ref{eq:modified-band-maps-statement}) from Equation
(\ref{eq:deltasigma-theta}). Let $CD\in\{CD_{x},CD_{y}\}$ be the
connected component of $L_{00}\backslash CB$ whose closure intersects
$CB$ but does not intersect $BB_{b,00},BB_{b+1,00}$. Let $z$ be
the minimum $z$ such that $BB_{z,00}$ is on $CD$. There are two
cases to consider, depending on whether $z<b$ or $z>b+1$: see Figure
\ref{fig:one-band-2}. Here, $a<b$, $b+1<c$.

\begin{figure}[h]
\begin{centering}
\includegraphics{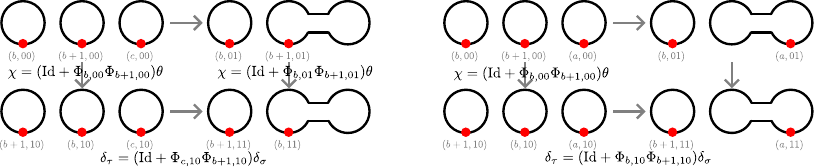}
\par\end{centering}
\caption{\label{fig:one-band-2-1}The ninth case of Figure \ref{fig:one-band}
(the components that do not intersect $B\cup BB_{b,00}\cup BB_{b+1,00}$
are not drawn)}
\end{figure}

In the left hand side case, we have 
\begin{gather*}
\mu_{2}^{dr}(\Phi_{b,00}\otimes{\rm Id})=\Phi_{b,00}\mu_{2}^{dr}({\rm Id}\otimes{\rm Id})=\mu_{2}^{dr}({\rm Id}\otimes\Phi_{b+1,10}),\\
\mu_{2}^{dr}(\Phi_{b+1,00}\otimes{\rm Id})=\Phi_{b+1,00}\mu_{2}^{dr}({\rm Id}\otimes{\rm Id}),\ \mu_{2}^{dr}({\rm Id}\otimes\Phi_{c,10})=\Phi_{c,00}\mu_{2}^{dr}({\rm Id}\otimes{\rm Id}),
\end{gather*}
and so 
\[
\mu_{2}^{dr}(({\rm Id}+\Phi_{b,00}\Phi_{b+1,00})\otimes({\rm Id}+\Phi_{c,10}\Phi_{b+1,10}))=({\rm Id}+\Phi_{b,00}\Phi_{b,11})\mu_{2}^{dr}({\rm Id}\otimes{\rm Id}).
\]
Also, we have 
\[
\mu_{2}^{rd}({\rm Id}\otimes({\rm Id}+\Phi_{b,01}\Phi_{b+1,01}))=({\rm Id}+\Phi_{b,00}\Phi_{b,11})\mu_{2}^{rd}({\rm Id}\otimes{\rm Id}).
\]

In the right hand side case, we have 
\[
\mu_{2}^{dr}(\Phi_{b,00}\otimes{\rm Id})=\Phi_{b,00}\mu_{2}^{dr}({\rm Id}\otimes{\rm Id})=\mu_{2}^{dr}({\rm Id}\otimes\Phi_{b+1,10}),\ \mu_{2}^{dr}({\rm Id}\otimes\Phi_{b,10})=\Phi_{b+1,00}\mu_{2}^{dr}({\rm Id}\otimes{\rm Id}),
\]
and so 
\[
\mu_{2}^{dr}(({\rm Id}+\Phi_{b,00}\Phi_{b+1,00})\otimes({\rm Id}+\Phi_{b,10}\Phi_{b+1,10}))=\mu_{2}^{dr}({\rm Id}\otimes{\rm Id}).
\]

\subsubsection{The fourteenth, fifteenth, sixteenth, seventeenth, and eighteenth
cases}

For these cases, $\delta_{\sigma}(10,11)=\delta_{\tau}(10,11)$, $\chi(00,10)=\Phi_{b,00}\Phi_{b+1,00}\theta(00,10)$,
and $\chi(01,11)=\Phi_{b,01}\Phi_{b+1,01}\theta(01,11)$. Hence, Equation
(\ref{eq:modified-band-maps-statement}) follows from Equation (\ref{eq:deltasigma-theta})
and 
\begin{gather*}
\mu_{2}^{rd}({\rm Id}\otimes\Phi_{b,01}\Phi_{b+1,01})=\Phi_{b,00}\Phi_{b+1,00}\mu_{2}^{rd}({\rm Id}\otimes{\rm Id}),\\
\mu_{2}^{dr}(\Phi_{b,00}\Phi_{b+1,00}\otimes{\rm Id})=\Phi_{b,00}\Phi_{b+1,00}\mu_{2}^{dr}({\rm Id}\otimes{\rm Id}).
\end{gather*}

\subsubsection{The first, second, third, fourth, sixth, seventh, and eighth cases}

For the remaining cases, $\delta_{\sigma}(10,11)=\delta_{\tau}(10,11)$
and $\chi=\theta$. (In fact, $b+1$ is not minimal for all but the
seventh case.) Hence, these cases follow from Equation (\ref{eq:deltasigma-theta}).

\subsection{Modified swap maps and modified swap maps}

The same statement as Proposition \ref{prop:Something-is-an} holds
for modified swap maps.
\begin{prop}
\label{prop:Something-is-an-1}Assume the conditions of Proposition
\ref{prop:Something-is-an}, but further assume that there exist two
components of $L_{0}$ such that their link baseballs in $L_{0}$
are different from their link baseballs in $L_{1}$. Let $\Phi_{b,0},\Phi_{b+1,0}$
be the $\Phi$-actions for these two components of $L_{0}$. Then,
\[
\mu_{2}(-\otimes({\rm Id}+\Phi_{b,0}\Phi_{b+1,0})\theta):HF^{-}(\boldsymbol{\alpha},\boldsymbol{\beta}_{0})\to HF^{-}(\boldsymbol{\alpha},\boldsymbol{\beta}_{1})
\]
is an isomorphism. Similarly, (if $\boldsymbol{\beta}_{0},\boldsymbol{\beta}_{1}<\boldsymbol{\alpha}$,
then)
\[
\mu_{2}(({\rm Id}+\Phi_{b,0}\Phi_{b+1,0})\theta\otimes-):HF^{-}(\boldsymbol{\beta}_{1},\boldsymbol{\alpha})\to HF^{-}(\boldsymbol{\beta}_{0},\boldsymbol{\alpha})
\]
is an isomorphism.
\end{prop}

\begin{proof}
Using Proposition \ref{prop:swap-commute} and Lemma \ref{lem:phi-swap},
we can show that Figure \ref{fig:one-band-2-1-1} commutes. Hence
the conclusion follows from Proposition \ref{prop:Something-is-an}.
\end{proof}
\begin{figure}[h]
\begin{centering}
\includegraphics{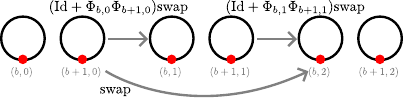}
\par\end{centering}
\caption{\label{fig:one-band-2-1-1}Modified swap maps and modified swap maps}
\end{figure}

\section{\label{sec:Composition-of-band}Modified band maps and modified band
maps}

In this section, we show Proposition \ref{prop:We-have-the} (\ref{enu:4}),
i.e. that modified band maps commute.

\subsection{\label{subsec:The-setup}The setup}

In this subsection, we describe what kind of local computations we
should do to show Proposition \ref{prop:We-have-the} (\ref{enu:4}).

Let $L_{00}=L$ be a link, and let there be two crossing balls $CB_{1}$
and $CB_{2}$. Let $CB_{r}\cap L_{00}\subset CB_{r}$ ($r=1,2$) be
$0$ of Figure \ref{fig:skein-moves-1}, and let $L_{ij}$ ($ij\in\{00,01,10,11\}$)
be the link obtained from $L_{00}$ by replacing $CB_{1}\cap L_{00}\subset CB_{1}$
with $i$ of Figure \ref{fig:skein-moves-1}, and $CB_{2}\cap L_{00}\subset CB_{2}$
with $j$ of Figure \ref{fig:skein-moves-1}. For $r=1,2$, let $B_{r}\subset CB_{r}$
be the band on $L_{00}$ from $0$ to $1$ of Figure \ref{fig:skein-moves-1}.
(Hence $L_{10}$ (resp. $L_{01}$) is obtained from $L_{00}$ by surgering
along $B_{1}$ (resp. $B_{2}$).)

Choose a bijection (``total order'') $\sigma:\{1,\cdots,N\}\to\boldsymbol{CC}$
where $\boldsymbol{CC}$ is the set of connected components of $L_{00}\backslash(CB_{1}\cup CB_{2})$,
and put a baseball $BB_{i,s}$ on $\sigma(i)$ for each $i=1,\cdots,N$
and $s\in\{00,01,10,11\}$. For each link component of $L(f)$, declare
$BB_{i,f}$ as its \emph{link baseball}, where $i$ is the smallest
$i$ such that $\sigma(i)$ is on that link component.

As in Subsection \ref{subsec:The-Heegaard-diagram}, let $H(L_{00})$
be the union of $CB_{1}\cup CB_{2}$ and a tubular neighborhood of
$L_{00}$, and consider a weakly admissible Heegaard diagram with
Heegaard surface $-\partial H(L_{00})$ and attaching curves $\boldsymbol{\beta}_{s}$
for $s\in\{00,01,10,11\}$. Define the modified band maps $\delta\in HF_{0}^{-}(\boldsymbol{\beta}_{s},\boldsymbol{\beta}_{t})$
for consecutive $s<t$ as in Definition \ref{def:For-each-consecutive}.
In this section, we will show
\[
x^{rd}:=\mu_{2}^{rd}(\delta\otimes\delta)=\mu_{2}^{dr}(\delta\otimes\delta)=:x^{dr}.
\]
This implies Proposition \ref{prop:We-have-the} (\ref{enu:4}), since
in the setting of Proposition \ref{prop:We-have-the} (\ref{enu:4}),
the sub Heegaard diagram with the four attaching curves $\boldsymbol{\beta}(f_{00}),\boldsymbol{\beta}(f_{01}),\boldsymbol{\beta}(f_{10}),\boldsymbol{\beta}(f_{11})$,
after isotopies, is a stabilization of a Heegaard diagram obtained
as above.

\begin{figure}[h]
\begin{centering}
\includegraphics{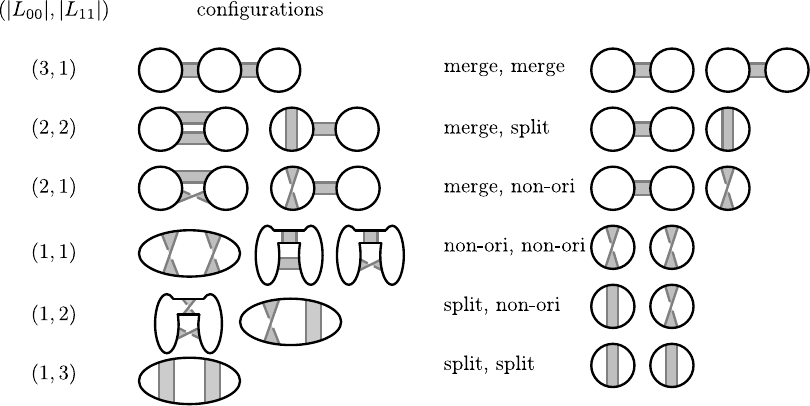}
\par\end{centering}
\caption{\label{fig:two-bands}Every configuration of two bands $B_{1}\cup B_{2}$
on $L_{00}$, such that every link component of $L_{00}$ intersects
$B_{1}\cup B_{2}$. Left: $L_{00}\cup B_{1}\cup B_{2}$ is connected;
right: $L_{00}\cup B_{1}\cup B_{2}$ is disconnected}
\end{figure}

Figure \ref{fig:two-bands} lists all the possible configurations
of two bands $B_{1},B_{2}$ on $L_{00}$, such that every component
of $L_{00}$ intersects $B_{1}\cup B_{2}$, up to an orientation preserving
diffeomorphism of $\overline{N}(L_{00}\cup B_{1}\cup B_{2})$.

We claim that for each configuration of $B_{1},B_{2}$ on $L_{00}$,
it is sufficient to check $x^{rd}=x^{dr}$ for one chosen total order
$\sigma$ of $\boldsymbol{CC}$. Indeed, this follows from Propositions
\ref{prop:(Modified-band-maps} and \ref{prop:Something-is-an-1}:
let $\sigma,\tau$ be total orders of $\boldsymbol{CC}$ and assume
that $x^{rd}=x^{dr}$ for $\sigma$. We show that $x^{rd}=x^{dr}$
for $\tau$. It is sufficient to show this for $\tau=\sigma\circ(b\ b+1)$
for some $b=1,\cdots,N-1$. As in Subsubsection \ref{subsec:Step-2},
for each $CC\in\boldsymbol{CC}$, $s\in\{00,01,10,11\}$, and $\eta\in\{\sigma,\tau\}$,
put a baseball $BB_{CC,s,\eta}$ on $CC$. Define the attaching curves
$\boldsymbol{\beta}_{s,\eta}$, the modified band maps $\delta$,
and the modified swap maps $\chi$ as in Subsubsection \ref{subsec:Step-2}.
Define $\mu_{2}^{rd,\eta}$, $\mu_{2}^{dr,\eta}$ as the $\mu_{2}^{rd}$,
$\mu_{2}^{dr}$'s for the $\boldsymbol{\beta}_{s,\eta}$'s, and define
$x^{rd,\eta},x^{dr,\eta}$ similarly. We assume $x^{rd,\sigma}=x^{rd,\sigma}$
and show $x^{rd,\tau}=x^{rd,\tau}$. By Proposition \ref{prop:(Modified-band-maps},
we have 
\begin{equation}
\mu_{2}^{\boldsymbol{\beta}_{00,\sigma},\boldsymbol{\beta}_{11,\sigma},\boldsymbol{\beta}_{11,\tau}}((x^{rd,\sigma}+x^{dr,\sigma})\otimes\chi)=\mu_{2}^{\boldsymbol{\beta}_{00,\sigma},\boldsymbol{\beta}_{00,\tau},\boldsymbol{\beta}_{11,\tau}}(\chi\otimes(x^{rd,\tau}+x^{dr,\tau})).\label{eq:sumandchi}
\end{equation}
By Proposition \ref{prop:Something-is-an-1}, $\mu_{2}^{\boldsymbol{\beta}_{00,\sigma},\boldsymbol{\beta}_{00,\tau},\boldsymbol{\beta}_{11,\tau}}(\chi\otimes-)$
is an isomorphism, and so $x^{rd,\tau}=x^{rd,\tau}$.

For each configuration of $B_{1},B_{2}$ on $L_{00}$, we will choose
a total order $\sigma$ such that (in particular) $\delta=\theta\in HF_{0}^{-}(\boldsymbol{\beta}_{s},\boldsymbol{\beta}_{t})$
for every consecutive $s<t$. By a ``diagram chasing argument''
(similar to Equation (\ref{eq:diagram-chase})) using Lemma \ref{lem:summary-theta}
and Proposition \ref{prop:swap-commute}, we can reduce to the cases
that we check in Subsections \ref{subsec:new-cases} and \ref{subsec:actions}.

\subsection{\label{subsec:new-cases}New cases}

For most cases, we already know $HF_{0}^{-}(\boldsymbol{\beta}_{00},\boldsymbol{\beta}_{11})$
well enough; let us first cover the other cases. In all the below
cases, the composites $x^{rd},x^{dr}$ are nonzero by Section \ref{sec:A-nonvanishing-argument},
and are homogeneous with respect to the relative Alexander $\mathbb{Z}/2$-grading
by Remark \ref{rem:alex-homogeneous}.

In this subsection, for simplicity, we work with Heegaard diagrams
given by destabilizing the Heegaard diagrams obtained as in Section
\ref{sec:Interpreting-schematics}: for instance, compare Figures
\ref{fig:z11-3} and \ref{fig:z11-3-heegaard}. Define the homological
$\mathbb{Z}$-gradings for the destabilized Heegaard diagrams such
that the stabilization map $S^{+}$ (Definition \ref{def:A-pair-pointed-Heegaard-1})
has degree $0$.

\subsubsection{\label{subsec:113}The third case of $(|L_{00}|,|L_{11}|)=(1,1)$}

\begin{figure}[h]
\begin{centering}
\includegraphics{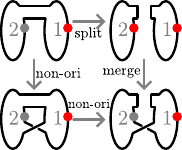}
\par\end{centering}
\caption{\label{fig:z11-3}The third case of $(|L_{00}|,|L_{11}|)=(1,1)$}
\end{figure}

\begin{figure}[h]
\begin{centering}
\includegraphics{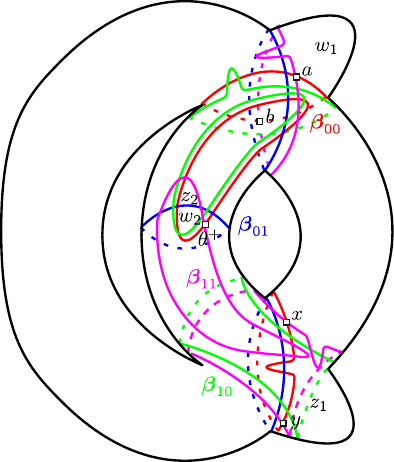}\qquad{}\includegraphics{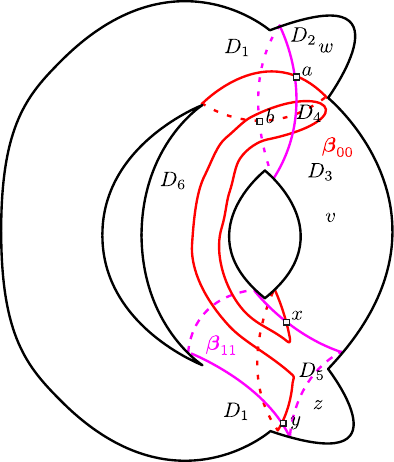}
\par\end{centering}
\caption{\label{fig:z11-3-heegaard}Destabilized Heegaard diagrams for Subsubsection
\ref{subsec:113}.}
\end{figure}

We show that Figure \ref{fig:z11-3} commutes. We claim that $x^{rd},x^{dr}$
is the unique nonzero element in the corresponding relative Alexander
$\mathbb{Z}/2$-grading such that $\Phi_{1}x=0$.

Let us first study the right hand side of Figure \ref{fig:z11-3-heegaard}.
We can check that $(\Sigma,\boldsymbol{\beta}_{00},\boldsymbol{\beta}_{11})$
is weakly admissible both with respect to $\{w,z\}$ and $\{v\}$.
There are four generators, and all of them have $c_{1}=0$. Since
the homology with respect to the basepoint $v$ represents $\#^{2}S^{1}\times S^{2}$,
whose Heegaard Floer homology has rank $4$, the differential is trivial.
Considering gradings, we can deduce that the differential over $\mathbb{F}[W,Z]$
is also trivial. Hence, $HF_{\mathbb{F}[U^{1/2}]}^{-}(\boldsymbol{\beta}_{00},\boldsymbol{\beta}_{11})$
is freely generated over $\mathbb{F}[U^{1/2}]$ by $ax,ay,bx,by$.
There are exactly two Maslov index $1$ domains in $D(ax,bx)$: $D_{2}$
and $D_{5}+D_{6}$, and they both have an odd number of holomorphic
representatives. Hence, $\Phi_{1}(ax)\neq0$, and so $\ker\Phi_{1}\cap{\rm span}_{\mathbb{F}}\left\langle ax,by\right\rangle $
has rank at most $1$.

In the left hand side of Figure \ref{fig:z11-3-heegaard}, the composites
$x^{rd}$ and $x^{dr}$ lie in the same relative Alexander $\mathbb{Z}/2$-grading
as $by\theta^{+}$. Since they are nonzero, we are left to show that
they lie in $\ker\Phi_{1}$. We have 
\begin{multline*}
\Phi_{1}x^{rd}=\Phi_{1}\mu_{2}^{rd}(\theta\otimes\theta)=\mu_{2}^{rd}(\Phi_{1}\theta\otimes\theta)+\mu_{2}^{rd}(\theta\otimes\Phi_{1}\theta)\\
=\mu_{2}^{rd}(\Phi_{2}\theta\otimes\theta)+\mu_{2}^{rd}(\theta\otimes\Phi_{2}\theta)=\Phi_{2}\mu_{2}^{rd}(\theta\otimes\theta)=0,
\end{multline*}
\begin{equation}
\Phi_{1}x^{dr}=\Phi_{1}\mu_{2}^{dr}(\theta\otimes\theta)=\mu_{2}^{dr}(\Phi_{1}\theta\otimes\theta)+\mu_{2}^{dr}(\theta\otimes\Phi_{1}\theta)=0.\label{eq:phi1xdr=00003D0}
\end{equation}

\subsubsection{\label{subsec:111}The first case of $(|L_{00}|,|L_{11}|)=(1,1)$}

We show that Figure \ref{fig:z11-1-heegaard} commutes. We also claim
that $x^{rd},x^{dr}$ is the unique nonzero element in the corresponding
relative Alexander $\mathbb{Z}/2$-grading such that $\Phi_{1}x=0$.
The same argument as in Subsection \ref{subsec:113} works. First,
$\Phi_{1}x^{rd},\Phi_{1}x^{dr}=0$ by the same argument as Equation
(\ref{eq:phi1xdr=00003D0}).

Consider the Heegaard diagram in Figure \ref{fig:z11-1-heegaard}.
We can check that $(\Sigma,\boldsymbol{\beta}_{00},\boldsymbol{\beta}_{11})$
is weakly admissible both with respect to $\{w,z\}$ and $\{v\}$,
and there are exactly four generators, $ax,ay,bx,by$. Hence, the
differential is $0$. The relevant relative Alexander $\mathbb{Z}/2$-grading
is the one that $ax$ and $by$ live in. In Figure \ref{fig:z11-1-heegaard},
there are exactly two Maslov index $1$ domains in $D(ax,bx)$: $D_{1}$
and $D_{2}$, and they both have an odd number of holomorphic representatives.
Hence, $\Phi_{1}(ax)\neq0$, and so $\ker\Phi_{1}\cap{\rm span}_{\mathbb{F}}\left\langle ax,by\right\rangle $
has rank at most $1$.

\begin{figure}[h]
\begin{centering}
\includegraphics{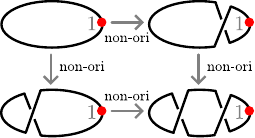}\qquad{}\includegraphics{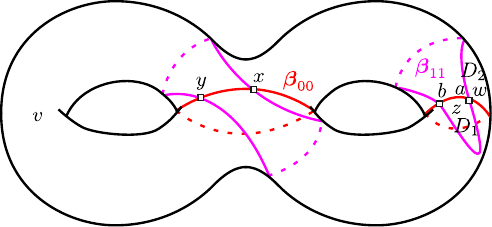}
\par\end{centering}
\caption{\label{fig:z11-1-heegaard}A case for Subsubsection \ref{subsec:111}
and a destabilized Heegaard diagram for it}
\end{figure}

\subsubsection{\label{subsec:112}The second case of $(|L_{00}|,|L_{11}|)=(1,1)$}

\begin{figure}[h]
\begin{centering}
\includegraphics{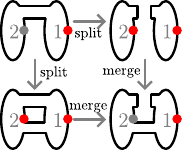}
\par\end{centering}
\caption{\label{fig:z21-2}The second case of $(|L_{00}|,|L_{11}|)=(1,1)$}
\end{figure}

\begin{figure}[h]
\begin{centering}
\includegraphics{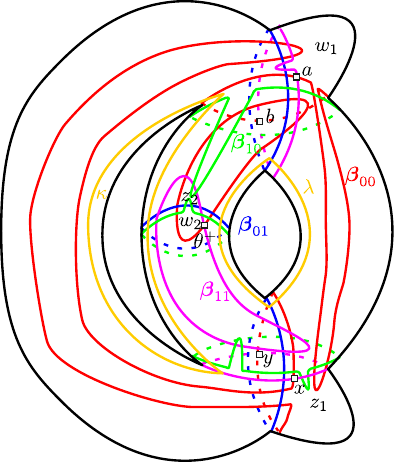}\qquad{}\includegraphics{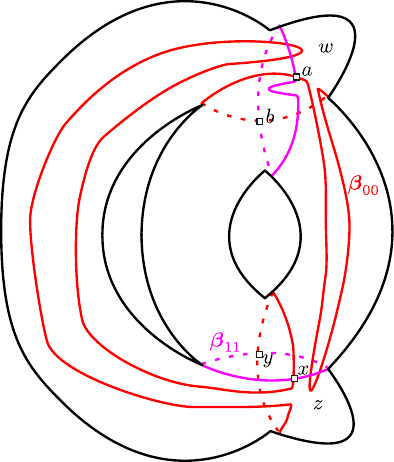}
\par\end{centering}
\caption{\label{fig:z21-2-heegaard-2}Destabilized Heegaard diagrams for Subsubsection
\ref{subsec:112}.}
\end{figure}

We show that Figure \ref{fig:z21-2} commutes. In this case, the $\Phi$
actions on $HF_{0}^{-}(\boldsymbol{\beta}_{00},\boldsymbol{\beta}_{11})$
vanishes; we use the $H_{1}$ action instead. We claim that $x^{rd},x^{dr}$
is the unique nonzero element in the corresponding relative Alexander
$\mathbb{Z}/2$-grading such that the $H_{1}$ action is trivial.

In this case, $(\boldsymbol{\beta}_{00},\boldsymbol{\beta}_{11})$
represents a doubly pointed null-homologous knot in $\#^{2}S^{1}\times S^{2}$.
Lemma \ref{lem:structure-z11-2} is the structure theorem that we
need for $HF_{\mathbb{F}[U^{1/2}],0}^{-}(\boldsymbol{\beta}_{00},\boldsymbol{\beta}_{11})$;
let us show that $x^{rd}=x^{dr}$ assuming Lemma \ref{lem:structure-z11-2}.
First, we can directly check that $x^{rd}$ and $x^{dr}$ lie in the
same relative Alexander $\mathbb{Z}/2$-grading as both $ay\theta^{+}$
and $bx\theta^{+}$. Let $R=\mathbb{F}[U_{1}^{1/2},U_{2}^{1/2}]$,
$S=\mathbb{F}[U_{1}^{1/2},U_{2}]$. By Subsection \ref{subsec:Connected-sums},
the same statement as Lemma \ref{lem:structure-z11-2} holds for the
relative Alexander $\mathbb{Z}/2$-grading ${\rm gr}_{A}^{\mathbb{Z}/2}(ay\theta^{+})$
summand of $HF_{S,0}^{-}(\boldsymbol{\beta}_{00},\boldsymbol{\beta}_{11})$.
Now, consider the homology classes $\kappa,\lambda\in H_{1}(\#^{2}S^{1}\times S^{2};\mathbb{F})$
represented by the yellow circles in Figure \ref{fig:z21-2-heegaard-2}.
We have $A_{R,\kappa,-1}(\theta)=0\in HF_{R}^{-}(\boldsymbol{\beta}_{00},\boldsymbol{\beta}_{01})$,
and so $A_{R,\kappa,-1}\mu_{2,R}^{rd}(\theta\otimes\theta)=0$. Since
$A_{R,\kappa,-1}$ preserves the splitting $HF_{R}^{-}(\boldsymbol{\beta}_{00},\boldsymbol{\beta}_{11})=HF_{S}^{-}(\boldsymbol{\beta}_{00},\boldsymbol{\beta}_{11})\oplus U_{2}^{1/2}HF_{S}^{-}(\boldsymbol{\beta}_{00},\boldsymbol{\beta}_{11})$
and agrees with $A_{S,\kappa,-1}$ on the summands, we have $A_{S,\kappa,-1}x^{rd}=0$.
Similarly, $A_{S,\lambda,-1}x^{dr}=0$ since $A_{R,\lambda,-1}(\theta)=0\in HF_{R}^{-}(\boldsymbol{\beta}_{00},\boldsymbol{\beta}_{10})$.
Hence, $x^{rd}=x^{dr}$.
\begin{lem}
\label{lem:structure-z11-2}Consider the right hand side of Figure
\ref{fig:z21-2-heegaard-2}. Let $C_{0}\le HF_{\mathbb{F}[U^{1/2}],0}^{-}(\boldsymbol{\beta}_{00},\boldsymbol{\beta}_{11})$
be the relative Alexander $\mathbb{Z}/2$-grading ${\rm gr}_{A}^{\mathbb{Z}/2}(ay)$
summand. Then, $C_{0}$ has rank $2$, and for nontrivial $\gamma\in H_{1}(\#^{2}S^{1}\times S^{2};\mathbb{F})$,
${\rm ker}A_{\gamma,-1}\cap C_{0}$ has rank $1$, and it does not
depend on $\gamma$.
\end{lem}

\begin{proof}
Consider the weakly admissible Heegaard diagram from the right hand
side of Figure \ref{fig:z21-2-heegaard-2}. First, we can check that
the Heegaard diagram is weakly admissible with respect to both $\{z\}$
and $\{w\}$. Although there are eight generators, only four of them,
$ax,ay,bx,by$ have $c_{1}=0$. It is possible to show the lemma directly
(see Appendix \ref{subsec:A-direct-computation}). Here, we present
the following, curve count free argument. 
\begin{table}[h]
\begin{centering}
\begin{tabular}{|c|c|c|c|c|}
\hline 
 & $ax$ & $ay$ & $bx$ & $by$\tabularnewline
\hline 
\hline 
${\rm gr}_{z}$ & $0$ & $1$ & $-1$ & $0$\tabularnewline
\hline 
${\rm gr}_{w}$ & $0$ & $-1$ & $1$ & $0$\tabularnewline
\hline 
\end{tabular}
\par\end{centering}
\caption{\label{tab:gradings}The gradings $({\rm gr}_{z},{\rm gr}_{w})$ for
the generators in the torsion ${\rm Spin}^{c}$-structure}
\end{table}

The relative homological $\mathbb{Z}$-gradings ${\rm gr}_{z}$ and
${\rm gr}_{w}$ (arbitrarily lifted to an absolute $\mathbb{Z}$-grading)
are as in Table \ref{tab:gradings}. Let ${\cal R}=\mathbb{F}[Z,W]$,
and let us compute the differential $\partial_{{\cal R}}$ on ${\cal CF}^{-}(\boldsymbol{\beta}_{00},\boldsymbol{\beta}_{11})$
and also partially compute the degree $-1$ $H_{1}$-action. We will
use the following \emph{key fact} repeatedly: for any two generators
$\tau,\sigma\in\{ax,ay,bx,by\}$, by considering the gradings $({\rm gr}_{z},{\rm gr}_{w})$,
one can show that all the Maslov index $1$ domains $\phi\in D(\sigma,\tau)$
have the same weight, which is either $Z$ or $W$.

The differential $\partial_{{\cal R}}$ vanishes since it vanishes
if we identify $Z=1$ or $W=1$ (and by using the key fact): indeed,
the homology of 
\[
{\cal CF}^{-}(\boldsymbol{\beta}_{00},\boldsymbol{\beta}_{11})\otimes\mathcal{R}/(Z-1),\ {\cal CF}^{-}(\boldsymbol{\beta}_{00},\boldsymbol{\beta}_{11})\otimes{\cal R}/(W-1)
\]
are $HF^{-}(\#^{2}S^{1}\times S^{2})$, which has rank $4$ (over
$\mathbb{F}[W]$ and $\mathbb{F}[Z]$, respectively).

Let $\kappa,\lambda\in H_{1}(\#^{2}S^{1}\times S^{2};\mathbb{F})$
span the homology group. To compute the $A_{\kappa,-1},A_{\lambda,-1}$
maps on ${\cal CF}^{-}(\boldsymbol{\beta}_{00},\boldsymbol{\beta}_{11})$,
we only have to compute them in $HF_{{\cal R}/(Z-1)}^{-}$, by the
key fact. In $HF_{\mathcal{R}/(W-1)}^{-}(\boldsymbol{\beta}_{00},\boldsymbol{\beta}_{11})$,
the map 
\[
H_{1}(\#^{2}S^{1}\times S^{2};\mathbb{F})\to{\rm span}_{\mathbb{F}}\left\langle ax,by\right\rangle :\gamma\mapsto A_{\mathcal{R}/(W-1),\gamma,-1}(ay)
\]
is an isomorphism. We have
\[
HF_{{\cal R}/(Z-1),-1}^{-}(\boldsymbol{\beta}_{00},\boldsymbol{\beta}_{11})={\rm span}_{\mathbb{F}}\left\langle Wbx,A_{\kappa,-1}A_{\lambda,-1}(bx)\right\rangle ={\rm span}_{\mathbb{F}}\left\langle Wbx,ay\right\rangle .
\]
Since 
\[
W^{-1}A_{{\cal R}/(Z-1),\gamma,-1}(ay)=A_{\mathcal{R}/(W-1),\gamma,-1}(ay)\in{\rm span}_{\mathbb{F}}\left\langle ax,by\right\rangle ,
\]
the $H_{1}$-action on $ay$ is in particular nontrivial in $HF_{{\cal R}/(Z-1)}^{-}$,
and so we have 
\[
ay=Wbx+A_{\kappa,-1}A_{\lambda,-1}(bx)\in CF_{{\cal R}/(Z-1)}^{-}(\boldsymbol{\beta}_{00},\boldsymbol{\beta}_{11}).
\]
Hence, the maps $A_{\kappa,-1}$ and $A_{\lambda,-1}$ on ${\cal CF}^{-}(\boldsymbol{\beta}_{00},\boldsymbol{\beta}_{11})$
are as follows, where $c,d$ are such that ${\rm span}_{\mathbb{F}}\left\langle c,d\right\rangle ={\rm span}_{\mathbb{F}}\left\langle ax,by\right\rangle $:
\[\begin{tikzcd}
	ay & c & ay & c \\
	d & bx & d & bx
	\arrow["W"{description}, from=1-1, to=1-2]
	\arrow["W"{description}, from=1-3, to=2-3]
	\arrow["Z"{description}, from=1-4, to=1-3]
	\arrow["W"{description}, from=1-4, to=2-4]
	\arrow["Z"{description}, from=2-1, to=1-1]
	\arrow["W"{description}, from=2-1, to=2-2]
	\arrow["{A_{\kappa,-1}}"{description}, shift left=5, draw=none, from=2-1, to=2-2]
	\arrow["Z"{description}, from=2-2, to=1-2]
	\arrow["{A_{\lambda,-1}}"{description}, shift left=5, draw=none, from=2-3, to=2-4]
	\arrow["Z"{description}, from=2-4, to=2-3]
\end{tikzcd}\]In particular, the statement of the lemma follows.
\end{proof}

\subsection{\label{subsec:actions}The remaining cases}

For the remaining cases, we can reduce to cases for which we already
know $HF_{0}^{-}(\boldsymbol{\beta}_{00},\boldsymbol{\beta}_{11})$
well enough. In all the below cases, the composites $x^{rd},x^{dr}$
are nonzero by Section \ref{sec:A-nonvanishing-argument}, and are
homogeneous with respect to the relative Alexander $\mathbb{Z}/2$-grading
by Remark \ref{rem:alex-homogeneous}.

\subsubsection{$(|L_{00}|,|L_{11}|)=(2,1),(1,2)$}

\begin{figure}[h]
\begin{centering}
\includegraphics{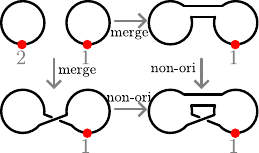}\qquad{}\includegraphics{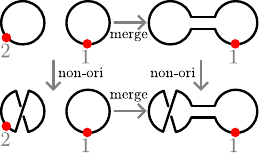}\qquad{}\includegraphics{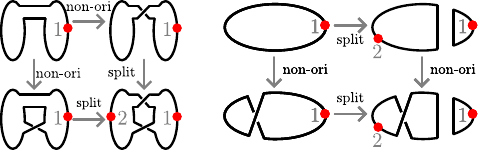}
\par\end{centering}
\caption{\label{fig:z21}The cases for $(|L_{00}|,|L_{11}|)=(2,1),(1,2)$}
\end{figure}

We check the cases of Figure \ref{fig:z21}. In all four cases, $(\boldsymbol{\beta}_{00},\boldsymbol{\beta}_{11})$
represents $(\#^{2}S^{1}\times S^{2},L_{nonori})\#(S^{1}\times S^{2},\emptyset)$,
but we can destabilize and work in a Heegaard diagram where $(\boldsymbol{\beta}_{00},\boldsymbol{\beta}_{11})$
represents $(\#^{2}S^{1}\times S^{2},L_{nonori})$. The argument is
exactly the same as Proposition \ref{prop:non-ori-swap}. The composites
are the canonical element $\theta$.

\subsubsection{\label{subsec:31}$(|L_{00}|,|L_{11}|)=(2,2),(3,1),(1,3)$}

\begin{figure}[h]
\begin{centering}
\includegraphics{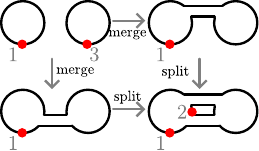}\qquad{}\includegraphics{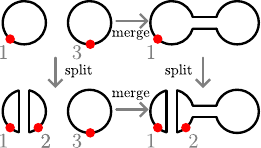}\qquad{}\includegraphics{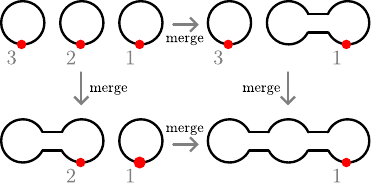}\qquad{}\includegraphics{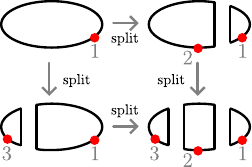}
\par\end{centering}
\caption{\label{fig:merge-merge-cases}The cases for $(|L_{00}|,|L_{11}|)=(2,2),(3,1),(1,3)$}
\end{figure}

We check the cases of Figure \ref{fig:merge-merge-cases}. In all
four cases, $(\boldsymbol{\beta}_{00},\boldsymbol{\beta}_{11})$ represents
$(\#^{2}S^{1}\times S^{2},L_{ori})\#(S^{1}\times S^{2},\emptyset)$,
but we can destabilize and work in a Heegaard diagram where $(\boldsymbol{\beta}_{00},\boldsymbol{\beta}_{11})$
represents $(\#^{2}S^{1}\times S^{2},L_{ori})$. The argument is exactly
the same as Proposition \ref{prop:swap-merge-2}.
\begin{rem}
\label{rem:Similarly,-we-can}Similarly, we can check that the left
hand side Figure \ref{fig:merge-merge-cases-1} does \emph{not} commute.
This is why we had to modify the band maps: the right hand side commutes.
\begin{figure}[h]
\begin{centering}
\includegraphics{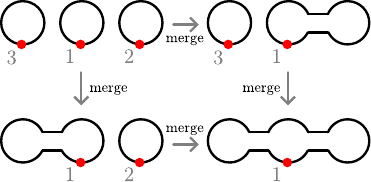}\qquad{}\includegraphics{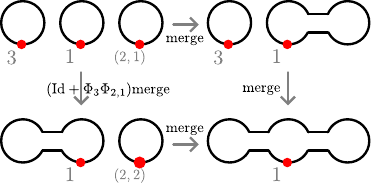}
\par\end{centering}
\caption{\label{fig:merge-merge-cases-1}Left: some merge maps that do \emph{not}
commute; right: we can modify the maps so that they commute}
\end{figure}
\end{rem}

\subsubsection{\label{subsec:Merge,-merge}The remaining cases}

\begin{figure}[h]
\begin{centering}
\includegraphics{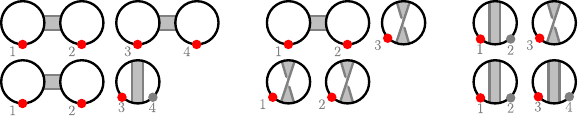}
\par\end{centering}
\caption{\label{fig:z42}The cases for Subsubsection \ref{subsec:Merge,-merge}.}
\end{figure}

We check the cases of Figure \ref{fig:z42}, which are immediate.

\section{\label{sec:Computations-for-the}Modified band maps and modified
band maps for Subsection \ref{subsec:The-reduced-hat}}

In this subsection, we show the analogue of Proposition \ref{prop:We-have-the}
(\ref{enu:4}) for Proposition \ref{prop:(Modified-band-maps-2}.

Recall the setup of Subsection \ref{subsec:The-setup}. We modify
this similarly to Subsubsection \ref{subsec:Step-1}:
\begin{itemize}
\item For each $i=2,\cdots,N$ and $s\in\{00,01,10,11\}$, put two baseballs
$BB_{i,s,1}$ and $BB_{i,s,2}$ on $\sigma(i)$.
\item Put a baseball $BB_{1}$ on $\sigma(1)$, and for each $s\in\{00,01,10,11\}$,
put a baseball $BB_{1,s,2}$ on $\sigma(1)$. For convenience, also
denote $BB_{1}$ as $BB_{1,s,1}$.
\end{itemize}
Define the balled links $L_{s,k}$ and attaching curves $\boldsymbol{\beta}_{s,k}$
for $k=1,2$ similarly to Subsubsection \ref{subsec:Step-1}, and
define $\mu_{2}^{rd,k}$, $\mu_{2}^{dr,k}$ as the $\mu_{2}^{rd}$,
$\mu_{2}^{dr}$'s for the $\boldsymbol{\beta}_{s,k}$'s, and define
$x^{rd,k},x^{dr,k}$ similarly.

Proposition \ref{prop:We-have-the} (\ref{enu:4}) implies that $x^{rd,2}=x^{dr,2}$,
and hence as in Equation (\ref{eq:sumandchi}), one can deduce from
Lemmas \ref{lem:summary-theta} and \ref{lem:summary-delta} that
\[
\mu_{2}^{\boldsymbol{\beta}_{00,1},\boldsymbol{\beta}_{11,1},\boldsymbol{\beta}_{11,2}}((x^{rd,1}+x^{dr,1})\otimes\theta)=\mu_{2}^{\boldsymbol{\beta}_{00,1},\boldsymbol{\beta}_{00,2},\boldsymbol{\beta}_{11,2}}(\theta\otimes(x^{rd,2}+x^{dr,2}))=0.
\]
Although $\theta\in HF^{-}(\boldsymbol{\beta}_{11,1},\boldsymbol{\beta}_{11,2})$
is a swap map, we cannot use Proposition \ref{prop:Something-is-an}
since $BB_{1}$ is a link baseball for $L_{00,1}$ and $L_{11,1}$,
but not for $L_{11,2}$. We show that $\mu_{2}^{\boldsymbol{\beta}_{00,1},\boldsymbol{\beta}_{11,1},\boldsymbol{\beta}_{11,2}}(-\otimes\theta)$
is injective on $HF_{0}^{-}(\boldsymbol{\beta}_{00,1},\boldsymbol{\beta}_{11,1})$
by interpreting it as a band map.

Define $L_{11,tmp}$ as $L_{11,1}$ but where the link baseball for
the component that intersects $BB_{1}$ is $BB_{1,11,2}$ instead,
and let $\boldsymbol{\beta}_{11,tmp}$ be the corresponding attaching
curve. Since $\mu_{2}^{\boldsymbol{\beta}_{11,1},\boldsymbol{\beta}_{11,tmp},\boldsymbol{\beta}_{11,2}}(\theta\otimes\theta)=\theta$
(these are all swap maps), we have 
\[
\mu_{2}^{\boldsymbol{\beta}_{00,1},\boldsymbol{\beta}_{11,1},\boldsymbol{\beta}_{11,2}}(-\otimes\theta)=\mu_{2}^{\boldsymbol{\beta}_{00,1},\boldsymbol{\beta}_{11,tmp},\boldsymbol{\beta}_{11,2}}\left(\mu_{2}^{\boldsymbol{\beta}_{00,1},\boldsymbol{\beta}_{11,1},\boldsymbol{\beta}_{11,tmp}}(-\otimes\theta)\otimes\theta\right).
\]
Since $\mu_{2}^{\boldsymbol{\beta}_{00,1},\boldsymbol{\beta}_{11,tmp},\boldsymbol{\beta}_{11,2}}(-\otimes\theta)$
is an isomorphism by Proposition \ref{prop:Something-is-an}, it is
sufficient to show that $\mu_{2}^{\boldsymbol{\beta}_{00,1},\boldsymbol{\beta}_{11,1},\boldsymbol{\beta}_{11,tmp}}(-\otimes\theta)$
is injective on $HF_{0}^{-}(\boldsymbol{\beta}_{00,1},\boldsymbol{\beta}_{11,1})$.

Let $(z_{1},w_{1})$ be the basepoint pair that corresponds to $BB_{1}$.
Let $\boldsymbol{\beta}_{11,tmp}^{link}$ be the same as $\boldsymbol{\beta}_{11,tmp}$
but let the basepoint pair $(z_{1},w_{1})$ be a link basepoint pair
for $\boldsymbol{\beta}_{11,tmp}^{link}$. Then, $\mu_{2}^{\boldsymbol{\beta}_{00,1},\boldsymbol{\beta}_{11,1},\boldsymbol{\beta}_{11,tmp}}$
and $\mu_{2}^{\boldsymbol{\beta}_{00,1},\boldsymbol{\beta}_{11,1},\boldsymbol{\beta}_{11,tmp}^{link}}$
are identical (and the chain complexes that are involved are identical),
since $(z_{1},w_{1})$ is a link basepoint pair for $\boldsymbol{\beta}_{00,1}$
and $\boldsymbol{\beta}_{11,1}$. Hence, it is sufficient to show
that $\mu_{2}^{\boldsymbol{\beta}_{00,1},\boldsymbol{\beta}_{11,1},\boldsymbol{\beta}_{11,tmp}^{link}}(-\otimes\theta)$
is injective on $HF_{0}^{-}(\boldsymbol{\beta}_{00,1},\boldsymbol{\beta}_{11,1})$.
This follows since if we perform a $0$-surgery along $z_{1}$ and
$w_{1}$ (the right side of Figure \ref{fig:composite-split-merge-1})\footnote{We can ignore the avoiding arc for $BB_{1}$ since the corresponding
basepoint pair is a link basepoint pair for every relevant attaching
curve.}, then this fits into the framework of Section \ref{sec:A-nonvanishing-argument},
where $(\boldsymbol{\beta},\boldsymbol{\beta}_{0},\boldsymbol{\beta}_{1})=(\boldsymbol{\beta}_{00,1},\boldsymbol{\beta}_{11,1},\boldsymbol{\beta}_{11,tmp}^{link})$
($\boldsymbol{\beta}_{2}=\boldsymbol{\beta}_{11,aux}$ of Figure \ref{fig:composite-split-merge-1}).
(Note that performing a $0$-surgery is unnecessary for the argument
to work; we perform a $0$-surgery just so that we can apply Proposition
\ref{prop:combi-claims} instead of computing the relevant homological
gradings again.)

\begin{figure}[h]
\begin{centering}
\includegraphics{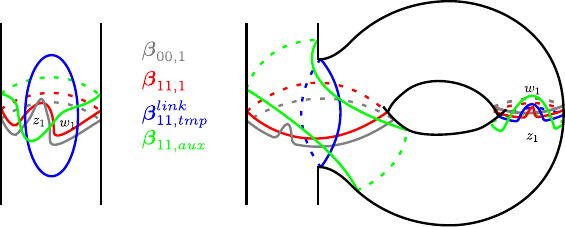}
\par\end{centering}
\caption{\label{fig:composite-split-merge-1}Introduce another attaching curve
$\boldsymbol{\beta}_{11,aux}$ to form a unoriented skein exact triangle
(the avoiding arcs are not drawn)}
\end{figure}

\section{\label{sec:Proofs-of-claims}Modified band maps give rise to an unoriented
skein exact triangle}

In this section, we show Proposition \ref{prop:We-have-the} (\ref{enu:6}).
The argument in this section also works for the corresponding statement
for Proposition \ref{prop:(Modified-band-maps-2}.

For each $i\in\{2,\cdots,N\}$, add a new baseball $BB_{i,aux}$ to
$\sigma(i)\in\boldsymbol{CC}$. If we are in the setting of Subsection
\ref{subsec:The-twisted-complex} (resp. Subsection \ref{subsec:The-reduced-hat}),
then also add a new baseball $BB_{1,aux}$ to $\sigma(1)$ (resp.
define $BB_{1,aux}:=BB_{1}$). For $k=0,1,2$, let $L_{k}'$ be the
balled link with the same underlying link as $L(f_{k})$ but whenever
$BB_{i,f_{k}}$ is a link baseball of $L(f_{k})$, let $BB_{i,aux}$
be the link baseball of the corresponding component of $L_{k}'$ instead.
Then, $L_{0}',L_{1}',L_{2}'$ form an unoriented skein exact triple
(Definition \ref{def:unoriented-skein-triple}). Let $\boldsymbol{\beta}_{k}'$
be the corresponding attaching curve. Also, simply call $\boldsymbol{\beta}_{k}:=\boldsymbol{\beta}(f_{k})$.

First, let us show that 
\begin{equation}
\mu_{2}^{\boldsymbol{\beta}_{0},\boldsymbol{\beta}_{1},\boldsymbol{\beta}_{2}}(\theta\otimes\theta)=0\in HF_{0}^{-}(\boldsymbol{\beta}_{0},\boldsymbol{\beta}_{2}).\label{eq:thetathetazero}
\end{equation}
Using Lemma \ref{lem:summary-theta}, we can deduce that
\[
\mu_{2}^{\boldsymbol{\beta}_{0}',\boldsymbol{\beta}_{0},\boldsymbol{\beta}_{2}}\left(\theta\otimes\mu_{2}^{\boldsymbol{\beta}_{0},\boldsymbol{\beta}_{1},\boldsymbol{\beta}_{2}}(\theta\otimes\theta)\right)=\mu_{2}^{\boldsymbol{\beta}_{0}',\boldsymbol{\beta}_{2}',\boldsymbol{\beta}_{2}}\left(\mu_{2}^{\boldsymbol{\beta}_{0}',\boldsymbol{\beta}_{1}',\boldsymbol{\beta}_{2}'}(\theta\otimes\theta)\otimes\theta\right)\in HF_{0}^{-}(\boldsymbol{\beta}_{0}',\boldsymbol{\beta}_{2}).
\]
By Theorem \ref{thm:Let--be}, we have 
\[
\mu_{2}^{\boldsymbol{\beta}_{0}',\boldsymbol{\beta}_{1}',\boldsymbol{\beta}_{2}'}(\theta\otimes\theta)=0\in HF_{0}^{-}(\boldsymbol{\beta}_{0}',\boldsymbol{\beta}_{2}'),
\]
and by Proposition \ref{prop:Something-is-an}, $\mu_{2}^{\boldsymbol{\beta}_{0}',\boldsymbol{\beta}_{0},\boldsymbol{\beta}_{2}}(\theta\otimes-)$
is an isomorphism on homology. Hence, Equation (\ref{eq:thetathetazero})
holds.

Now, that $\mu_{2}^{\boldsymbol{\beta}_{0},\boldsymbol{\beta}_{1},\boldsymbol{\beta}_{2}}(\delta\otimes\delta)=0$
on homology follows from Equation (\ref{eq:thetathetazero}) using
an algebraic argument, just like the proof of Lemma \ref{lem:summary-delta}.

Let us now show the second part of Proposition \ref{prop:We-have-the}
(\ref{enu:6}). This proof is similar to the proof of Theorem \ref{thm:Let--be}.
By Theorem \ref{thm:Let--be}, there exists some element $\zeta\in CF_{1}^{-}(\boldsymbol{\beta}_{0}',\boldsymbol{\beta}_{2}')$
such that 
\[\begin{tikzcd}
	{\underline {\boldsymbol{\beta}'_{012}}:=} & {\boldsymbol{\beta}_{0}'} & {\boldsymbol{\beta}_{1}'} & {\boldsymbol{\beta}_{2}'}
	\arrow["\theta"{description}, from=1-2, to=1-3]
	\arrow["{\zeta }"{description, pos=0.3}, curve={height=-12pt}, from=1-2, to=1-4]
	\arrow["\theta"{description}, from=1-3, to=1-4]
\end{tikzcd}\]is a twisted complex, and also $HF^{-}(\boldsymbol{\alpha},\underline{\boldsymbol{\beta}'_{012}})=0$.
Choose an element $\zeta\in CF_{1}^{-}(\boldsymbol{\beta}_{0},\boldsymbol{\beta}_{2})$
such that the following is a twisted complex (i.e. $\mu_{1}(\zeta)=\mu_{2}(\delta\otimes\delta)$).
\[\begin{tikzcd}
	{\underline {\boldsymbol{\beta}_{012}}:=} & {\boldsymbol{\beta}_{0}} & {\boldsymbol{\beta}_{1}} & {\boldsymbol{\beta}_{2}}
	\arrow["\delta"{description}, from=1-2, to=1-3]
	\arrow["\zeta"{description, pos=0.3}, curve={height=-12pt}, from=1-2, to=1-4]
	\arrow["\delta"{description}, from=1-3, to=1-4]
\end{tikzcd}\]

We will use the following lemma.

\begin{figure}[h]
\begin{centering}
\includegraphics{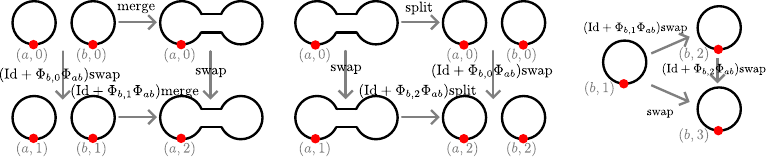}
\par\end{centering}
\caption{\label{fig:composite-split-merge}Three commutative diagrams}
\end{figure}

\begin{lem}
\label{lem:composite-split-merge}Consider Figure \ref{fig:composite-split-merge}.
Some components are not drawn as in Lemma \ref{lem:summary-delta},
and define $\Phi_{ab}:HF^{-}(\boldsymbol{\beta}_{v},\boldsymbol{\beta}_{v'})\to HF^{-}(\boldsymbol{\beta}_{v},\boldsymbol{\beta}_{v'})$
as $\Phi_{ab}:=\sum_{k=a+1}^{b-1}\Phi_{k,v}$. All three cases commute.
\end{lem}

\begin{proof}
By Proposition \ref{prop:swap-commute} and Lemma \ref{lem:summary-theta},
Figure \ref{fig:composite-split-merge} commutes if the edges are
labelled $\theta$. We can deduce the lemma from this using algebraic
arguments as in the proof of Lemma \ref{lem:summary-delta}.
\end{proof}
We will define the vertical and diagonal maps of Diagram (\ref{eq:swap-twisted}),
such that if $\underline{{\rm swap}}\in CF^{-}(\underline{\boldsymbol{\beta}'_{012}},\underline{\boldsymbol{\beta}_{012}})$
is the element that consists of the vertical and diagonal maps, then
$\underline{{\rm swap}}$ is a cycle. Let the vertical components
$\boldsymbol{\beta}_{k}'\to\boldsymbol{\beta}_{k}$ be swap maps,
except that we modify $\boldsymbol{\beta}_{k}'\to\boldsymbol{\beta}_{k}$
as in the left vertical map of the first diagram of Figure \ref{fig:composite-split-merge}
if $\boldsymbol{\beta}_{k-1}\to\boldsymbol{\beta}_{k}$ is a merge;
and as in the right vertical map of the second diagram of Figure \ref{fig:composite-split-merge}
if $\boldsymbol{\beta}_{k-1}\to\boldsymbol{\beta}_{k}$ is a split.
(In fact, these modifications are the same.) Then, by Lemmas \ref{lem:summary-theta}
and \ref{lem:composite-split-merge}, and Proposition \ref{prop:combi-claims}
(\ref{enu:If-,-then}), we can fill in the diagonal maps such that
$\underline{{\rm swap}}$ is a cycle.
\begin{equation}\label{eq:swap-twisted}
\begin{tikzcd}
	{\boldsymbol{\beta}_{0}'} & {\boldsymbol{\beta}_{1}'} & {\boldsymbol{\beta}_{2}'} \\
	{\boldsymbol{\beta}_0} & {\boldsymbol{\beta}_1} & {\boldsymbol{\beta}_2}
	\arrow["\theta"', from=1-1, to=1-2]
	\arrow[curve={height=-12pt}, from=1-1, to=1-3]
	\arrow[from=1-1, to=2-1]
	\arrow[from=1-1, to=2-2]
	\arrow[from=1-1, to=2-3]
	\arrow["\theta"', from=1-2, to=1-3]
	\arrow[from=1-2, to=2-2]
	\arrow[from=1-2, to=2-3]
	\arrow[from=1-3, to=2-3]
	\arrow["\delta", from=2-1, to=2-2]
	\arrow[curve={height=12pt}, from=2-1, to=2-3]
	\arrow["\delta", from=2-2, to=2-3]
\end{tikzcd}
\end{equation}

The vertical maps of Figure \ref{fig:composite-split-merge}, paired
with $\boldsymbol{\alpha}$, are quasi-isomorphisms, by the third
diagram of Figure \ref{fig:composite-split-merge} and Proposition
\ref{prop:Something-is-an} (this is similar to the argument of Proposition
\ref{prop:Something-is-an-1}). Hence, the cycle $\underline{{\rm swap}}$
induces a quasi-isomorphism 
\[
\mu_{2}(-\otimes\underline{{\rm swap}}):CF^{-}(\boldsymbol{\alpha},\underline{\boldsymbol{\beta}_{012}'})\to CF^{-}(\boldsymbol{\alpha},\underline{\boldsymbol{\beta}_{012}}),
\]
and so $HF^{-}(\alpha,\underline{\boldsymbol{\beta}_{012}})=0$.

\appendix

\section{\label{sec:Direct-computations}Direct computations}

Let ${\cal R}_{1}=\mathbb{F}[W_{1},W_{2},Z_{1},Z_{2}]$ and ${\cal R}_{2}=\mathbb{F}[W_{1},W_{2},W_{3},Z_{1},Z_{2},Z_{3}]$.
We compute $\mathcal{CFL}_{{\cal R}_{1}}^{-}(\#^{2}S^{1}\times S^{2},L_{ononori})$
and ${\cal CFL}_{{\cal R}_{2}}^{-}(\#^{2}S^{1}\times S^{2},L_{ori})$
using the Heegaard diagrams of Figure \ref{fig:Lnonori-ori}. They
are the tensor product of the three complexes on the left, resp. the
four complexes on the right.
\[\begin{tikzcd}[column sep=large,row sep=scriptsize]
	{\mathcal{R} _1 b} & {\mathcal{R} _1a} & {\mathcal{R} _2 d} & {\mathcal{R} _2 c} \\
	{\mathcal{R} _1 x} & {\mathcal{R} _1 y} & {\mathcal{R} _2 z} & {\mathcal{R} _2 w} \\
	{\mathcal{R} _1 z} & {\mathcal{R} _1 w} & {\mathcal{R} _2 b} & {\mathcal{R} _2 a} \\
	&& {\mathcal{R} _2 x} & {\mathcal{R} _2 y}
	\arrow["{z_1+z_2 }"{description}, curve={height=-12pt}, from=1-1, to=1-2]
	\arrow["{w_1+w_2}"{description}, curve={height=-12pt}, from=1-2, to=1-1]
	\arrow["{z_1 + z_2}"{description}, curve={height=-12pt}, from=1-3, to=1-4]
	\arrow["{w_1 +w_3}"{description}, curve={height=-12pt}, from=1-4, to=1-3]
	\arrow["{z_1+z_2 }"{description}, curve={height=-12pt}, from=2-1, to=2-2]
	\arrow["{w_1+w_2}"{description}, curve={height=-12pt}, from=2-2, to=2-1]
	\arrow["{z_1 +z_2}"{description}, curve={height=-12pt}, from=2-3, to=2-4]
	\arrow["{w_1 +w_3}"{description}, curve={height=-12pt}, from=2-4, to=2-3]
	\arrow["{z_1+w_2}"{description}, from=3-1, to=3-2]
	\arrow["{z_2 +z_3}"{description}, curve={height=-12pt}, from=3-3, to=3-4]
	\arrow["{w_2 + w_3}"{description}, curve={height=-12pt}, from=3-4, to=3-3]
	\arrow["{z_2 +z_3 }"{description}, curve={height=-12pt}, from=4-3, to=4-4]
	\arrow["{w_2 + w_3}"{description}, curve={height=-12pt}, from=4-4, to=4-3]
\end{tikzcd}\]

\begin{figure}[h]
\begin{centering}
\includegraphics{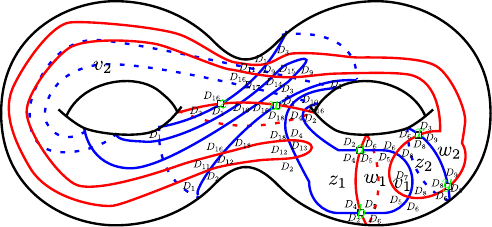}\includegraphics{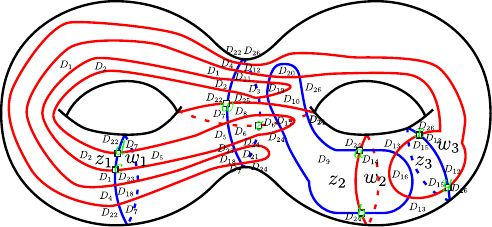}
\par\end{centering}
\caption{\label{fig:Lnonori-ori}Heegaard diagrams for $L_{nonori},L_{ori}\subset\#^{2}S^{1}\times S^{2}$.}
\end{figure}

\subsection{\label{subsec:Direct-computations-for}Direct computations for Section
\ref{sec:Some-important-balled}}

Let us study the left hand side Heegaard diagram of Figure \ref{fig:Lnonori-ori}.
There are $8$ generators, which are $u_{1}u_{2}u_{3}$ for $u_{1}\in\{a,b\}$,
$u_{2}\in\{z,w\}$, $u_{3}\in\{x,y\}$. For two such generators ${\bf u}:=u_{1}u_{2}u_{3}$
and ${\bf v}:=v_{1}v_{2}v_{3}$, there exists a Maslov index $1$
domain in $D({\bf u},{\bf v})$ only the $u_{i}$'s and $v_{i}$'s
differ in exactly one $i$. The following are all the Maslov index
$1$ domains: (by an annulus, we mean a domain as in Figure \ref{fig:annulus})
\begin{itemize}
\item $a\to b$\footnote{By this, we mean that for any $u_{2},u_{3}$, these are all the Maslov
index $1$ domains in $D(au_{2}u_{3},bu_{2}u_{3})$.}: $D_{5}+D_{7}$ (bigon, $w_{1}$), $D_{1}+D_{2}+D_{3}+D_{9}+D_{11}+D_{12}+D_{16}+D_{17}+D_{18}$
($w_{2}$)
\item $z\to w$: $D_{3}+D_{9}+D_{14}+D_{15}+D_{17}$ (annulus, $w_{2}$),
$D_{2}+D_{4}+D_{12}+D_{13}+D_{18}$ (annulus, $z_{1}$)
\item $x\to y$: $D_{7}+D_{8}$ (bigon $z_{2}$), $D_{1}+D_{2}+D_{3}+D_{4}+D_{10}+D_{14}+D_{16}+D_{17}+D_{18}$
($z_{1}$)
\item $b\to a$: $D_{6}+D_{8}$ (annulus, $z_{2}$), $D_{4}+D_{10}+D_{13}+D_{14}+D_{15}$
(bigon, $z_{1}$)
\item $w\to z$: $D_{1}+D_{3}+D_{9}+D_{10}+D_{11}+D_{14}+D_{15}+D_{16}+D_{17}$
($w_{2}$), $D_{1}+D_{2}+D_{4}+D_{10}+D_{11}+D_{12}+D_{13}+D_{16}+D_{18}$
($z_{1}$)
\item $y\to x$: $D_{5}+D_{6}$ (annulus, $w_{1}$), $D_{9}+D_{11}+D_{12}+D_{13}+D_{15}$
(bigon, $w_{2}$)
\end{itemize}
Since the Heegaard diagram is weakly admissible with respect to $\{v_{1},v_{2}\}$,
for each row, there are in total an even number of holomorphic representatives.
Since bigons and annuli always have an odd number of holomorphic representatives,
we can deduce that every domain above has an odd number of holomorphic
representatives, except the two domains for $w\to z$. For these,
we can use $\partial_{{\cal R}_{1}}^{2}=0$ to conclude that both
have an even number of holomorphic representatives.

Let us study the right hand side Heegaard diagram of Figure \ref{fig:Lnonori-ori}.
Similarly to above, there are $16$ generators, given by choosing
one of $\{a,b\}$, $\{c,d\}$, $\{z,w\}$, and $\{x,y\}$. The following
are all the Maslov index $1$ domains.
\begin{itemize}
\item $a\to b$: $D_{14}+D_{16}$ (bigon, $w_{2}$), $D_{3}+D_{6}+D_{8}+D_{11}+D_{12}+D_{21}+D_{24}+D_{25}+D_{26}$
($w_{3}$)
\item $c\to d$: $D_{5}+D_{6}+D_{17}$ (bigon, $w_{1}$), $D_{1}+D_{4}+D_{10}+D_{11}+D_{12}+D_{20}+D_{22}+D_{25}+D_{26}$
($w_{3}$)
\item $z\to w$: $D_{6}+D_{8}+D_{9}+D_{17}+D_{21}+D_{24}$ (annulus, $z_{2}$),
$D_{1}+D_{2}+D_{4}+D_{22}$ (annulus, $z_{1}$)
\item $x\to y$: $D_{15}+D_{16}$ (bigon, $z_{3}$), $D_{1}+D_{7}+D_{8}+D_{9}+D_{10}+D_{11}+D_{22}+D_{23}+D_{24}+D_{25}+D_{26}$
($z_{2}$)
\item $b\to a$: $D_{13}+D_{15}$ (annulus, $z_{3}$), $D_{9}+D_{10}+D_{17}+D_{19}+D_{20}$
(bigon, $z_{2}$)
\item $d\to c$: $D_{2}+D_{3}+D_{19}$ (bigon, $z_{1}$), $D_{7}+D_{8}+D_{9}+D_{18}+D_{21}+D_{23}+D_{24}$
($z_{2}$)
\item $w\to z$: $D_{5}+D_{7}+D_{18}+D_{23}$ (annulus, $w_{1}$), $D_{3}+D_{10}+D_{11}+D_{12}+D_{19}+D_{20}+D_{25}+D_{26}$
(annulus, $w_{3}$)
\item $y\to x$: $D_{13}+D_{14}$ (annulus, $w_{2}$), $D_{4}+D_{12}+D_{18}+D_{20}+D_{21}$
(bigon, $w_{3}$)
\end{itemize}
Similarly, using that the Heegaard diagram with respect to $\{z_{1},z_{2},z_{3}\}$
and $\{w_{1},w_{2},w_{3}\}$ are both admissible, we can show that
all the above domains have an odd number of holomorphic representatives.

\subsection{\label{subsec:A-direct-computation}A direct computation for Lemma
\ref{lem:structure-z11-2} }

\begin{figure}[h]
\begin{centering}
\includegraphics{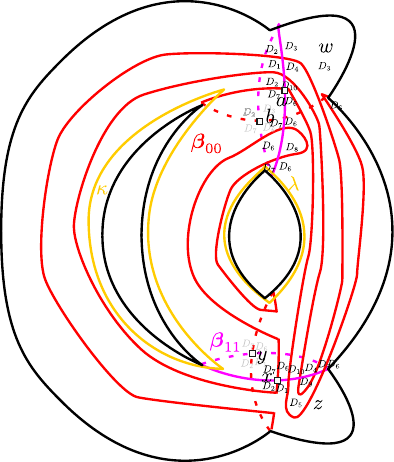}
\par\end{centering}
\caption{\label{fig:z112}A Heegaard diagram for Lemma \ref{lem:structure-z11-2}.}
\end{figure}

Let us study Figure \ref{fig:z112}. There are $4$ generators with
$c_{1}=0$: $ax$, $ay$, $bx$, $by$. The following are all the
Maslov index $1$ domains.
\begin{itemize}
\item $a\to b$: $D_{3}+D_{4}+D_{6}+D_{7}+D_{8}+D_{10}$ ($A_{\kappa,-1}$),
$D_{3}+D_{4}+D_{5}+D_{9}+D_{10}$ (bigon)
\item $x\to y$: $D_{1}+D_{4}+D_{6}+D_{7}+D_{8}+D_{9}$ ($A_{\kappa,-1}$),
$D_{1}+D_{4}+D_{5}+2D_{9}$ (bigon), $D_{1}+D_{2}+D_{6}+2D_{7}$
\item $b\to a$: $D_{1}+D_{4}+D_{6}+2D_{8}+D_{9}$ ($A_{\lambda,-1}$, annulus),
$D_{1}+D_{2}+D_{6}+D_{7}+D_{8}$, $D_{1}+D_{2}+D_{5}+D_{9}$ ($A_{\kappa,-1}$,
annulus)
\item $y\to x$: $D_{3}+D_{4}+D_{6}+2D_{8}+D_{10}$ ($A_{\lambda,-1}$,
annulus), $D_{2}+D_{3}+D_{5}+D_{10}$ ($A_{\kappa,-1}$, annulus)
\end{itemize}
The Heegaard diagram is admissible with respect to $\{w\}$ and also
with respect to $\{z\}$, and so $\partial=0$. Using this, we can
compute the number of holomorphic representatives for all the domains,
except the ones for $x\to y$. Since $A_{\kappa,-1}^{2}\simeq0$ (\cite[Proposition 4.17]{MR2113019},
\cite[Lemma 5.5]{1512.01184}), both domains $D(ax,ay),D(bx,by)$
with underlying two-chain $D_{1}+D_{4}+D_{6}+D_{7}+D_{8}+D_{9}$ have
an odd number of holomorphic representatives, and the actions of $A_{\kappa,-1}$
and $A_{\lambda,-1}$ are as follows. 

\[\begin{tikzcd}[sep=large]
	ax & ay & ax & ay \\
	bx & by & bx & by
	\arrow["Z"{description}, curve={height=-12pt}, from=1-1, to=1-2]
	\arrow["W"{description}, curve={height=12pt}, from=1-1, to=2-1]
	\arrow["W"{description}, curve={height=-12pt}, from=1-2, to=1-1]
	\arrow["W"{description}, curve={height=12pt}, from=1-2, to=2-2]
	\arrow["W"{description}, from=1-4, to=1-3]
	\arrow["Z"{description}, curve={height=12pt}, from=2-1, to=1-1]
	\arrow["Z"{description}, curve={height=-12pt}, from=2-1, to=2-2]
	\arrow["Z"{description}, curve={height=12pt}, from=2-2, to=1-2]
	\arrow["W"{description}, curve={height=-12pt}, from=2-2, to=2-1]
	\arrow["{A_{\kappa,-1}}"', shift right=5, draw=none, from=2-2, to=2-1]
	\arrow["Z"{description}, from=2-3, to=1-3]
	\arrow["Z"{description}, from=2-4, to=1-4]
	\arrow["W"{description}, from=2-4, to=2-3]
	\arrow["{A_{\lambda,-1}}"', shift right=5, draw=none, from=2-4, to=2-3]
\end{tikzcd}\]Hence, the conclusion of Lemma \ref{lem:structure-z11-2} holds ($c=ax+by$,
$d=ax$).

\bibliographystyle{amsalpha}
\bibliography{essay_bib}

\providecommand{\bysame}{\leavevmode\hbox to3em{\hrulefill}\thinspace}
\providecommand{\MR}{\relax\ifhmode\unskip\space\fi MR }
\providecommand{\MRhref}[2]{%
  \href{http://www.ams.org/mathscinet-getitem?mr=#1}{#2}
}
\providecommand{\href}[2]{#2}
\begin{thebibliography}{Dow24}

\bibitem[Bal20]{ballinger2020concordance}
William Ballinger, \emph{Concordance invariants from the $ e (-1) $ spectral
  sequence on {K}hovanov homology}, arXiv preprint arXiv:2004.10807 (2020).

\bibitem[BL12]{MR2964628}
John~A. Baldwin and Adam~Simon Levine, \emph{A combinatorial spanning tree
  model for knot {F}loer homology}, Adv. Math. \textbf{231} (2012), no.~3-4,
  1886--1939. \MR{2964628}

\bibitem[BLZ24]{MR4845975}
Maciej Borodzik, Beibei Liu, and Ian Zemke, \emph{Heegaard {F}loer homology,
  knotifications of links, and plane curves with noncuspidal singularities},
  Algebr. Geom. Topol. \textbf{24} (2024), no.~9, 4837--4889. \MR{4845975}

\bibitem[Dow24]{MR4777638}
Nathan Dowlin, \emph{A spectral sequence from {K}hovanov homology to knot
  {F}loer homology}, J. Amer. Math. Soc. \textbf{37} (2024), no.~4, 951--1010.
  \MR{4777638}

\bibitem[Juh06]{MR2253454}
Andr\'as Juh\'asz, \emph{Holomorphic discs and sutured manifolds}, Algebr.
  Geom. Topol. \textbf{6} (2006), 1429--1457. \MR{2253454}

\bibitem[Kho06]{MR2232858}
Mikhail Khovanov, \emph{Link homology and {F}robenius extensions}, Fund. Math.
  \textbf{190} (2006), 179--190. \MR{2232858}

\bibitem[KM11]{MR2805599}
P.~B. Kronheimer and T.~S. Mrowka, \emph{{K}hovanov homology is an
  unknot-detector}, Publ. Math. Inst. Hautes \'{E}tudes Sci. (2011), no.~113,
  97--208. \MR{2805599}

\bibitem[LS22]{MR4504654}
Robert Lipshitz and Sucharit Sarkar, \emph{A mixed invariant of nonorientable
  surfaces in equivariant {K}hovanov homology}, Trans. Amer. Math. Soc.
  \textbf{375} (2022), no.~12, 8807--8849. \MR{4504654}

\bibitem[Nah25]{nahm2025unorientedskeinexacttriangle}
Gheehyun Nahm, \emph{An unoriented skein exact triangle in unoriented link
  {F}loer homology}, 2025.

\bibitem[Ni14]{MR3294567}
Yi~Ni, \emph{Homological actions on sutured {F}loer homology}, Math. Res. Lett.
  \textbf{21} (2014), no.~5, 1177--1197. \MR{3294567}

\bibitem[OS04a]{MR2023281}
Peter Ozsv\'ath and Zolt\'an Szab\'o, \emph{Holomorphic disks and genus
  bounds}, Geom. Topol. \textbf{8} (2004), 311--334. \MR{2023281}

\bibitem[OS04b]{MR2113019}
Peter Ozsv\'{a}th and Zolt\'{a}n Szab\'{o}, \emph{Holomorphic disks and
  topological invariants for closed three-manifolds}, Ann. of Math. (2)
  \textbf{159} (2004), no.~3, 1027--1158. \MR{2113019}

\bibitem[OS05]{MR2141852}
\bysame, \emph{On the {H}eegaard {F}loer homology of branched double-covers},
  Adv. Math. \textbf{194} (2005), no.~1, 1--33. \MR{2141852}

\bibitem[OS08]{MR2443092}
Peter Ozsv\'ath and Zolt\'an Szab\'o, \emph{Holomorphic disks, link invariants
  and the multi-variable {A}lexander polynomial}, Algebr. Geom. Topol.
  \textbf{8} (2008), no.~2, 615--692. \MR{2443092}

\bibitem[OSS17]{MR3694597}
Peter~S. Ozsv\'{a}th, Andr\'{a}s~I. Stipsicz, and Zolt\'{a}n Szab\'{o},
  \emph{Unoriented knot {F}loer homology and the unoriented four-ball genus},
  Int. Math. Res. Not. IMRN (2017), no.~17, 5137--5181. \MR{3694597}

\bibitem[Ras05]{MR2189938}
Jacob Rasmussen, \emph{Knot polynomials and knot homologies}, Geometry and
  topology of manifolds, Fields Inst. Commun., vol.~47, Amer. Math. Soc.,
  Providence, RI, 2005, pp.~261--280. \MR{2189938}

\bibitem[Ras15]{MR3447099}
\bysame, \emph{Some differentials on {K}hovanov-{R}ozansky homology}, Geom.
  Topol. \textbf{19} (2015), no.~6, 3031--3104. \MR{3447099}

\bibitem[Sar15]{MR3426686}
Sucharit Sarkar, \emph{Moving basepoints and the induced automorphisms of link
  {F}loer homology}, Algebr. Geom. Topol. \textbf{15} (2015), no.~5,
  2479--2515. \MR{3426686}

\bibitem[Zem15]{1512.01184}
Ian Zemke, \emph{Graph cobordisms and {H}eegaard {F}loer homology}, arXiv
  preprint arXiv:1512.01184 (2015).

\bibitem[Zem17]{MR3709653}
\bysame, \emph{Quasistabilization and basepoint moving maps in link {F}loer
  homology}, Algebr. Geom. Topol. \textbf{17} (2017), no.~6, 3461--3518.
  \MR{3709653}

\bibitem[Zem19]{MR3905679}
\bysame, \emph{Link cobordisms and functoriality in link {F}loer homology}, J.
  Topol. \textbf{12} (2019), no.~1, 94--220. \MR{3905679}

\bibitem[Zem23]{2308.15658}
\bysame, \emph{A general {H}eegaard {F}loer surgery formula}, arXiv preprint
  arXiv:2308.15658 (2023).

\end{thebibliography}

\end{document}